\documentclass[12pt]{amsart}

\usepackage[backref=page, colorlinks=true, linkcolor=magenta, citecolor=cyan]{hyperref}
\usepackage{mathtools}
\usepackage{soul}
\usepackage{times}

\RequirePackage[dvipsnames,usenames]{color}

\input{kmacros3.sty}
\usepackage{quiver}

\usepackage{mabliautoref}
\numberwithin{equation}{theorem}

\usepackage{tikz}
\usepackage{tikz-cd}
\usetikzlibrary{cd}
\usepackage{graphicx}
\usepackage[all,cmtip]{xy}

\RequirePackage[dvipsnames,usenames]{color}

\usepackage{bm}
\usepackage{pifont}
\usepackage{upgreek}
\usepackage{ulem}
\usepackage{stmaryrd}
\usepackage[utf8]{inputenc}

\newcommand{\newpt}[2]{{\mathrm{Pt}(#1, #2)}}

\usepackage[margin=1 in,marginpar=2cm]{geometry}

\usepackage{todonotes}

\usepackage{setspace}

\usepackage[skip=5pt plus1pt, indent=5pt]{parskip}

\usepackage{enumerate}
\usepackage{calc}

\usepackage{bbding}

\usepackage{verbatim}
\usepackage{alltt}

\newcommand{\perfd}{\mathrm{perfd}}
\newcommand{\perf}{\mathrm{perf}}
\newcommand{\rsftypfd}{R^{S_\infty}_\perfd}
\newcommand{\raftypfd}{R^{A_\infty}_\perfd}

\renewcommand{\epsilon}{\varepsilon}
\DeclareMathOperator{\depth}{depth}
\newcommand{\cOne}{{\mathrm{c1}}}

\DeclareMathOperator{\ppt}{ppt} 
 
\DeclareMathOperator{\pt}{Pt}
\DeclareMathOperator{\lct}{lct}
\DeclareMathOperator{\eh}{e}

\DeclareMathOperator{\np}{\mathcal{I}}
\DeclareMathOperator{\coff}{C}
\newcommand{\cotensor}[1]{%
  \mathbin{\mathop{\widehat{\otimes}}_{#1}}%
}

\usepackage{tikz}
\usepackage{tikz-cd}
\usetikzlibrary{cd}
\usepackage{graphicx}
\usepackage[all,cmtip]{xy}

\usepackage{mabliautoref}
\numberwithin{equation}{theorem}

\usepackage{bm}
\usepackage{pifont}
\usepackage{upgreek}

\usepackage{ulem}
\usepackage{stmaryrd}
\usepackage{cancel}

\numberwithin{equation}{theorem}

\usepackage{fullpage}

\usepackage{enumerate}
\usepackage{calc}

\usepackage{verbatim}
\usepackage{alltt}
\usepackage{scalerel,stackengine}

\definecolor{darkgreen}{rgb}{0,0.4,0}
\definecolor{forestgreen}{RGB}{34,139,34}
\definecolor{darkpurple}{rgb}{0.35,0,0.45}
\definecolor{darkorchid}{rgb}{0.38,0.08,0.32}
\definecolor{darkturquoise}{rgb}{0.01,0.55,0.7}

\newcommand{\Ainfty}{A_{\infty}}
\newcommand{\m}{\mathfrak{m}}
\newcommand{\n}{\mathfrak{n}}
\newcommand{\p}{\mathfrak{p}}
\newcommand{\q}{\mathfrak{q}}
\DeclareMathOperator{\ev}{\mathrm{mult}}
\newcommand{\un}{\mathrm{un}}

\newcommand{\frad}{\mathfrak{d}}
\newcommand{\fran}{\mathfrak{n}}

\stackMath
\newcommand\reallywidehat[1]{%
\savestack{\tmpbox}{\stretchto{%
  \scaleto{%
    \scalerel*[\widthof{\ensuremath{#1}}]{\kern.1pt\mathchar"0362\kern.1pt}%
    {\rule{0ex}{\textheight}}
  }{\textheight}%
}{2.4ex}}%
\stackon[-6.9pt]{#1}{\tmpbox}%
}
\newcommand\reallywidetilde[1]{%
\stackon[-7pt]{#1}{%
  \vstretch{1.2}{\hstretch{2.4}{%
    \widetilde{\phantom{\;\;\;\;\;\;\;\;}}}}}%
}
\begin{document}

\title{Perfectoid Pure and Signature Pairs}

\author[Cai, Fayolle, Ma, Schwede, Smirnov, Tucker, Waldron]{Hanlin Cai, Anne Fayolle, Linquan Ma, Karl Schwede, Ilya Smirnov, Kevin Tucker, Joe Waldron}

\address{Department of Mathematics, Columbia University, New York, NY, USA}
\email{hc3589@columbia.edu}
\address{Department of Mathematics, University of Utah, Salt Lake City, UT, USA}
\email{fayolle@math.utah.edu}
\address{Department of Mathematics, Purdue University, West Lafayette, IN, USA}
\email{ma326@purdue.edu}
\address{Department of Mathematics, University of Utah, Salt Lake City, UT, USA}
\email{schwede@math.utah.edu}
\address{BCAM -- Basque Center for Applied Mathematics, Bilbao, Spain \quad and \quad
IKERBASQUE, Basque Foundation for Science, Bilbao, Spain}
\email{ismirnov@bcamath.org}
\address{Department of Mathematics, University of Illinois at Chicago, Chicago, IL, USA}
\email{kftucker@uic.edu}
\address{Department of Mathematics, Michigan State University, East Lansing, MI, USA}
\email{waldro51@msu.edu}

\begin{abstract}
We extend perfectoid purity and related invariants to divisor pairs in mixed characteristic.
The resulting purity condition agrees with $F$-purity of pairs in characteristic $p$ as defined by Hara--Watanabe.
For normal (respectively, demi-normal)
log $\mathbb{Q}$-Gorenstein pairs,
we show that perfectoid pure implies log canonical (respectively, semi-log canonical). 
We similarly establish adjunction for perfectoid pure pairs and inversion of adjunction for principal pairs in complete intersections along a Cartier divisor. 
We also define perfectoid pure thresholds and prove localization and base change results, 
including a formula for local extensions with regular closed fiber. We establish upper bounds in terms of analytic spread, dimension, and log canonical thresholds.
Finally, we also introduce perfectoid signatures for divisor pairs in mixed characteristic whose positivity characterizes the BCM-regularity of log $\mathbb{Q}$-Gorenstein pairs. We prove a transformation rule under finite extensions and show continuity and convexity for scalings of a Cartier divisor. 
\end{abstract}

\maketitle
\tableofcontents

\section{Introduction}

Within birational algebraic geometry and characteristic $p > 0$ commutative algebra, instead of simply studying the singularities of a ring $R$, one often studies the singularities of a pair $(R, f^t)$, $(R, \Delta)$ or $(R, \fra^t)$.  Here $f \in R$ and $t \in \bQ_{\geq 0}$, $\Delta \geq 0$ is a $\bQ$-divisor on $\Spec R$, or $\fra \subseteq R$ is an ideal and $t \in \bQ_{\geq 0}$.  Pairs arise naturally as correction terms when changing rings or varieties and they have become an important tool for proving facts even about regular rings, see for instance \cite{EinLazarsfeldSmith.UniformBoundsAndSymbolicPowers}.  

Recently, by applying methods originating in $p$-adic Hodge theory, \cite{ScholzePerfectoidSpaces,AndreConjectureFacteurDirect,BhattScholzePrisms,Bhatt}, a theory of singularities in mixed characteristic has been developed analogous to the theories coming out of birational algebraic geometry and characteristic $p > 0$ commutative algebra, see for instance \cite{MaSchwedeSingularitiesMixedCharBCM,MaSchwedeTuckerWaldronWitaszekAdjointPLT,TakamatsuYoshikawa.MMPForSemiStableThreefolds,BMPSTWW1,CaiLeeMaSchwedeTuckerPerfectoidSignature,BMPSTWW3,Yoshikawa.ComputationMethodForPerfdPurity,CaiPandeQuinlanSchwedeTucker.PlusPureThresholds}.  Within those works, there is a suitable theory of singularities of pairs analogous to the theory of log terminal or strongly $F$-regular pairs (namely, BCM-regular pairs).   However, there is no general theory analogous to log canonical and $F$-pure pairs (\cite{KollarSingularitiesMMP, HaraWatanabe, TakagiPairsAndInversion}).  Furthermore, we also lack a theory analogous to log canonical (or $F$-pure) thresholds and a theory analogous to the theory of $F$-signature of pairs.   Indeed, there are aspects of the BCM-regular definitions that make some of these generalizations difficult to even formulate, and we discuss these below in \autoref{subsec.IntroOnDefinitions}.

In this paper we develop these missing theories in mixed characteristic.  In particular, we introduce and develop the following notions.
\begin{enumerate}
    \item A definition of perfectoid purity for pairs $(R, \Delta)$ (or $(R, f^t)$) agreeing with the definition of $F$-purity for pairs from \cite{HaraWatanabe}, see \autoref{sec.PerfectoidPurePrincipal} and \autoref{sec.PerfdPureDivisorPairs}.
    \item An analog of the $F$-pure threshold of \cite{TakagiWatanabe}, denoted $\ppt(\fra, \Delta)$, see \autoref{sec.PerfectoidPureThresholds}.
    \item An analog of the $F$-signature for pairs from \cite{BSTfsigpairs1,PolstraTuckerCombinedApproach}, see \autoref{sec.PerfectoidSignatureForPairs}.
\end{enumerate}

\subsection{Main results}
We prove the following results about these definitions.  First we show that perfectoid pure pairs are log canonical.  

\begin{theoremA*}[{\autoref{pfd pure implies lc}}]
Let $R$ be an excellent normal ring {admitting a dualizing complex}, with $p$ in its Jacobson radical. Let $\Delta$ be an effective $\Q$-divisor on $R$ such that $K_R+\Delta$ is $\Q$-Cartier.
    If $(R,\Delta)$ is perfectoid pure, then it is log canonical.
\end{theoremA*}
We in fact prove the non-normal semi-log canonical version as well where, instead of assuming that $R$ is normal, we assume that $R$ is demi-normal (\cite{KollarSingularitiesMMP})  and that $\Delta$ is a Mumford $\bQ$-divisor (\cite{KollarMumfordDivisors}), see \autoref{pfd pure implies slc}.

Indeed, the above result is new even when $\Delta = 0$ where the result had only been shown when $K_R$ was $\bQ$-Cartier with \emph{index not divisible by $p$}, see \cite[Corollary~5.11]{BMPSTWW3}.  In fact, it is the flexibility that pairs provides that lets us handle the case where the index is divisible by $p$.  

We also obtain the following adjunction-type result.
\begin{theoremB*}[\autoref{thm:perfectoid-pure-adjunction}]
Let $R$ be an excellent equidimensional demi-normal ring with $p$ in its Jacobson radical and which admits a dualizing complex. 
Let $D$ be a reduced effective Mumford divisor on $R$ and let $i\colon D=V(I_D)\hookrightarrow \Spec R$.  Assume that $S\coloneqq R/I_D$ is deminormal.  
Let $\Delta$ be an effective Mumford $\Q$-divisor on $\Spec(R)$ having no prime component in common with $D$.  

Assume that $K_R+D+\Delta$ is $\Q$-Cartier {at every height $1$-prime of $S$} and that $(R,D+\Delta)$ is perfectoid pure.  Then $(S, \Delta_D)$ is perfectoid pure where $\Delta_D$ is the different.  
\end{theoremB*}

Under moderate geometric hypotheses, we can weaken the hypothesis that $S$ is demi-normal to require instead that $S$ is $\mathbf{S}_2$ (and we can even remove that if $D$ is $\bQ$-Cartier), see \autoref{cor:geometric-adjunction}.  Indeed, if all the divisors are Cartier, then we can ignore these demi-normality and even equidimensionality assumptions completely and obtain a quite general adjunction statement in \autoref{quotient of pfd pure is pfd pure}.  If additionally we assume that the ambient ring is a complete intersection, we also obtain inversion of adjunction  \autoref{quotient of ci pfd pure}.

We now move to our analogs of $F$-pure thresholds.  We prove the following, which answers \cite[Conjecture 4.22(a)]{BMPSTWW3} and even generalizes the statement to the context of pairs.  

\begin{theoremC*}[\autoref{thm: ppt in regular fibers for pairs}]
Let $(R, \m) \to (S, \n)$ be a faithfully flat local map of Noetherian local rings of residue characteristic $p$ such that $S/\m S$ is regular. Let $f_1,\dots,f_n \in S$ be elements whose images form a regular system of parameters on $S/\m S$. 

Let $\ba$ be an ideal of $R$ and $\Delta$ be an effective Mumford $\Q$-divisor on $R$. 
If $\Delta \neq 0$, further assume that $R$ and $S$ are excellent reduced $\mathbf{S}_2$ rings and that the flat pullback $\Delta_S$ of $\Delta$ is a Mumford $\Q$-divisor.
Then 
\[
\ppt ((\ba S+(f_1,\dots,f_n))S; \Delta_S) = \ppt (\ba; \Delta) + \dim (S/\m S).
\]
In particular, if $(R,\Delta)$ is perfectoid pure, then so is $(S,\Delta_S)$. 
\end{theoremC*}

    Using this, in \autoref{cor.AnalyticSpreadPPT} we show that $\ppt (\ba) \leq \ell (\ba)$, where $\ell (\ba)$ denotes the analytic spread of $\ba$. 
    Moreover, we prove that 
    $$\ppt (\ba) \leq \min \{\dim (R_Q) \mid (p, \ba) \subseteq Q \in \Spec (R)\}.$$

We also compare the perfectoid pure threshold to the log canonical threshold.

\begin{theoremD*}[\autoref{prop.pptVsLCT}]
    Let $(R, \m)$ be an excellent normal local ring of residual characteristic $p > 0$ admitting a dualizing complex and let $\ba\subseteq R$ be an ideal. Let $\Delta$ be an effective $\Q$-divisor on $R$ such that $K_R+\Delta$ is $\Q$-Cartier.
    Then $\ppt(\ba; \Delta) \leq \lct(\ba; \Delta).$
\end{theoremD*}

    We also compute some interesting examples of perfectoid pure thresholds.  Suppose for instance that $R = \bZ_{p}\llbracket x,y,z \rrbracket/(x^3+y^3+z^3)$.  Then we prove that 
    \[
    \ppt(\fram) = 1 \;\;\;\;\; \text{if $p \equiv 1 (\text{mod}\,3)$}
    \]
    If $p \equiv 2 (\mathrm{mod}\,3)$ then 
    we know that $R$ is perfectoid pure by \cite{Yoshikawa.ComputationMethodForPerfdPurity} and in fact also that $\ppt(p) = 1 - p/(p^2-1)$.  Building upon this, and the theory we develop, we also show that 
    \[
        \ppt(\fram) = 1 - p/(p^2-1) \;\;\;\;\; \text{if $p \equiv 2 (\text{mod}\,3)$}.
    \]

    We now move on to the perfectoid signature of a pair $(R, \Delta)$, which we denote by $s^{\underline{x}}_{\perfd}(R, \Delta)$.  As in \cite{CaiLeeMaSchwedeTuckerPerfectoidSignature}, $R$ is a complete Noetherian normal domain and $\underline{x}$ is used to identify a certain choice of system of parameters from a Noether-Cohen normalization $A \subseteq R$.  We show a general transformation rule in the context of perfectoid signature, proving the following analog of \cite[Theorem 4.4] {CarvajalSchwedeTuckerEtaleFundFsignature} and \cite[Theorem 1.3]{XuZhuang.UniquenessOfMinimizer}.

    \begin{theoremE*}[\autoref{thm.transformationruleforpairs}]
        Let $\phi\colon (R,\m,k)\to (R',\n,l)$ be a finite {generically \'etale} extension of complete normal local domains of mixed characteristic $(0,p)$ with perfect residue fields and with fraction fields $K = K(R)$ and $L = K(R')$. Let $\tau\colon \Spec(R')\to \Spec(R)$ be the corresponding morphism of schemes. Suppose $\Delta_R\ge 0$ is a $\bQ$-divisor such that $\Delta_{R'}\coloneqq \tau^*\Delta_R-K_{R'/R}\ge 0$. Let $t$ be the maximal number of direct $\mathrm{Tr}$-summands\footnote{{That is, a rank-1 free $R$-module summand of $R'$ such that if $\rho : R' \to R$ is projection on that summand, then $\rho = \Tr \cdot s$ for some $s \in R'$, see also \cite[Definition 2.12]{CarvajalSchwedeTuckerEtaleFundFsignature}.}} of $R'$ and $r=[L:K]$. Then we have
        \[ 
            t\cdot  s^{\underline{x}}_\perfd(R',\Delta_{R'})={r}\cdot s^{\underline{x}}_\perfd(R,\Delta_R).
        \]
    \end{theoremE*}

    Just as in the non-pair case, we show that perfectoid signature of pairs measures BCM-regularity.  Indeed, in \autoref{thm.BCMRegular=>PositiveSignatureForPairs} we prove the following.  Suppose $\Delta \geq 0$ is a $\bQ$-divisor such that $K_R + \Delta$ is $\bQ$-Cartier.  Then  
     $(R, \Delta)$ is BCM-regular if and only if
    \[
        s^{\underline{x}}_{\perfd}(R, \Delta) > 0.
    \]

    Similar to \cite{BlickleSchwedeTucker.FSigPairs2}, we also study the function 
    $t \mapsto s^{\underline{x}}_{\perfd}(R, tD)$
    for $D = \Div f$.  We prove the following.

    \begin{theoremF*}[{\autoref{thm:continuity}, \autoref{thm.PerfdSignatureIsConvex}}]
  Suppose $(R,\fram,k)$ is a complete Noetherian normal local domain of mixed characteristic $(0,p)$ with perfect residue field and $f \in \fram$ is non-zero. Let $D = \divisor(f)$ on $\Spec(R)$. The function $\QQ_{\geq 0} \to \mathbb{R}$ given by
  \begin{equation*}
    t \mapsto s^{\underline{x}}_\perfd(R,tD)
  \end{equation*}
  is continuous. In particular, it extends naturally to a function $\mathbb{R}_{\geq 0} \to \mathbb{R}$.  Furthermore, that function is convex.
    \end{theoremF*}

    \subsection{On the definitions}
\label{subsec.IntroOnDefinitions}

Perhaps the most challenging part of the work above was developing each definition.  We describe the idea behind each one below.  

\subsubsection*{Perfectoid purity for pairs $(R, \Delta)$}

In order to mimic the idea of \cite{HaraWatanabe}, we want some map 
\[
R \to B((1-\epsilon)\Delta)
\]
to be pure for all $1 \gg \epsilon > 0$ and where $B$ is some perfectoid algebra.  The problem is interpreting twisting by a $\bQ$-divisor on $B$.  Unlike \cite{MaSchwedePerfectoidTestIdeal}, we cannot assume that $B$ is an $R^+$-algebra, as if it was then we would also have a splitting $R \to R^+ \to R^+((1-\epsilon) \Delta)$ which would imply that $R$ itself was a splinter (and hence is KLT if $\bQ$-Gorenstein).  While we may (and do) assume that $B$ is absolutely integrally closed, this means that in principle there can be multiple different ways to pullback a $\bQ$-divisor.  In this generality, $n$th roots of elements of $R$ are \emph{not} unique up to units.  Nevertheless, the $\epsilon$ perturbation saves us as long as we pass all the way to $B$.  

We sketch the idea which is developed carefully at the start of \autoref{sec.PerfdPureDivisorPairs}.  Suppose we have $\epsilon_1 > \epsilon_2 > 0$.  One can always pick a finite flat extension $R \subseteq S_1$ and a map $S_1 \to B$  and an integer $n_1 > 0$ such that $n_1 (1-\epsilon_1) \Delta$ is an integral divisor that pulls back to $n_1 \Delta_{S_1}$, where $\Delta_{S_1}$ is an effective Weil generalized divisor.  Such an $S_1$, $n_1$ and $\Delta_{S_1}$ exist: simply formally adjoin an appropriate root to $R$ and use that $B$ is absolutely integrally closed (\autoref{construction SnD}).  Note even if $R$ is normal, we cannot assume $S_1$ is normal while keeping $R \subseteq S_1$ flat, and normalizing leads to problems similar to those with $R^+$.  

We likewise construct $R \subseteq S_2 \to B$ and $\Delta_{S_2}$ an effective $\bQ$-divisor such that $n_2((1-\epsilon_2)\Delta) = n_2 \Delta_{S_2}$.  Set $\sM_i$ to be the sheaf on $R$ associated to $B \otimes_{S_i} S_i(\Delta_{S_i})$.   Then in \autoref{divisorial comparison on big open}, we show that there is a canonical pointed\footnote{sending $1 \mapsto 1$} map 
\[
    \sM_1|_U \to \sM_2|_U
\]
for some open subset $U \subseteq \Spec R$ whose complement has codimension $\geq 2$ (a \emph{big open set}).  The key observation that lets us construct this map is that $B$ is seminormal, since it is perfectoid, see \autoref{BS - root}.  We then show we can take a filtered colimit over all $\epsilon > 0$ and over all big open sets to obtain a pointed\footnote{meaning there is a given map $R \to M_B({<}\Delta)$} object $M_B({<}\Delta)$ (one can take other colimits too as we show in \autoref{divisorial module computed by towers}).  With that in mind, one has that $(R, \Delta)$ is perfectoid pure if and only if the structure map
\[
    R \to M_B({<}\Delta)
\]
is pure, see \autoref{universal module detects perfectoid purity}.  This is equivalent to asking that for every $\epsilon > 0$, there exists a finite flat $R$-algebra $S$ and $\Delta_S$ as above with $n (1-\epsilon) \Delta$ an integral divisor that pulls back to $n \Delta_S$ for some $n > 0$ as well as some map $S \to B_{\epsilon}$ for some perfectoid $S$-algebra, such that $R \to B_\epsilon \otimes_S S(\Delta_S)$ is pure.  In fact, this latter characterization is our original definition \autoref{def perfd pure div pair}.

\subsubsection*{Perfectoid pure thresholds}

In this section we introduce analogs of the $F$-pure threshold of an ideal $\fra \subseteq R$ (or even for a pair $(R, \Delta)$ as above), which we denote by $\ppt(R; \fra)$ or $\ppt(R, \Delta; \fra)$ and call it \emph{perfectoid pure threshold} of $\fra$.  

We explain the definition, set $\np(B,{<}\Delta)$ to be elements $b \in B$ such that $R \xrightarrow{1 \mapsto b} B \to M_B({<}\Delta)$ is not pure, see \autoref{def: non-pure ideals} and \autoref{nonpure ideal below delta}.  We then define 
\[
    \ppt(\fra; \Delta) = \sup_B \sup\{r \in \bQ_{\geq 0}\;|\; \overline{(\fra B)^{r}} \nsubseteq \np(B,{<}\Delta)\}.
\]
where $\overline{(\fra B)^r}$ is the fractional integral closure (see also \cite[Section 10.5]{HunekeSwansonIntegralClosure} and \cite{RVS}) and $B$ runs over perfectoid $R$-algebras where certain key elements have compatible systems of roots, (see \autoref{def.ThresholdForDivisorPairs}).  It turns out one can compute this from concretely defined perfectoid algebras $\widehat{R}^{A_{\infty}}_{\perfd}$ (\autoref{prop: perfectoid threshold from one algebra}).

Thanks to the Brian\c{c}on--Skoda theorem for perfectoid algebras (\cite[Corollary 3.20]{MMRGSBrianconSkoda}) one can reduce this computation to the principal case similar to the summation formula for multiplier ideals: supposing $R$ is a local normal domain, there are minimal generators
$0 \ne \fra=(x_1,\ldots,x_m)\subsetneq R$ with 
\[
    \ppt(\fra; \Delta) = \sup \Big\{ {|a| \over p^e} \mid e \in \mathbb{Z}_{\geq 0}, a \in \mathbb{Z}_{\geq 0}^m, 
    (R, \Delta + \frac{1}{p^e} \Div (x_1^{a_1}\cdots x_m^{a_m}))
    \text{ is perfectoid pure}
    \Big\}, 
\]
see \autoref{cor: perfectoid threshold via pure pairs}. This is one key result that underlies many other results in \autoref{sec.PerfectoidPureThresholds}.

\subsubsection*{Perfectoid signature pairs}
    
Our strategy here is slightly different.  We are working with perfectoid signature which measures BCM-regularity and not perfectoid purity.  Because of this, we can use finite (normal) extensions of integral domains $R \subseteq S$ instead of the finite flat extensions mentioned above.  

Suppose we are given a normal pair $(R, \Delta \geq 0)$ where $R$ is a complete local Noetherian normal local domain of mixed characteristic. 
We fix $A \subseteq R$ to be a Noether-Cohen normalization ($R$ is a finite $A$-module) with a chosen regular system of parameters $\underline{x}$.
 We set $S \supseteq R$ to be a finite extension such that $\Delta$ pulls back to a divisor with integer coefficients.  In \autoref{def.perfectoidsigforpairs}, we then define 
\[
\np^R (S^{A_\infty}_\perfd, \Delta_R) = \big\{x \in S^{A_\infty}_{\perfd} \mid R \xrightarrow{1 \mapsto x \otimes 1} S^{A_\infty}_{\perfd} \otimes_S S(\Delta_R) \mbox{ is not pure} \big \}.
\]
Then we define the \emph{perfectoid signature} of $(R, \Delta_R)$ to be 
$$s^{\underline{x}}_\perfd(R,\Delta_R)=\frac{1}{r}\cdot\lambda_\infty(S^{A_\infty}_\perfd/\np^R (S^{A_\infty}_\perfd, \Delta_R))$$
where $r=\rank_R S$.

We show that this is independent of the choice of $S$.  This computation is essentially also the proof of the transformation rule under finite maps for perfectoid signature.  Note that, like the non-pair case (\cite{CaiLeeMaSchwedeTuckerPerfectoidSignature}), we do not know whether perfectoid signature is independent of the choice of $A$ or \underline{x}.

\noindent
\textbf{AI Acknowledgements.}
The ideas of this paper were developed by the human authors over several years, and a large majority of the results were known to us before summer 2026. 
As we have worked to finalize the manuscript, various parts of this paper have been discussed with LLMs. These include OpenAI's ChatGPT 5.6 Sol and Astra 6.0, Anthropic's Fable 5.1 and Google Gemini 3.1 Pro. In some cases, these systems provided initial proof drafts, which the human authors subsequently wrote in their own words, see below. LLMs also contributed some mathematical ideas that have been incorporated into the proofs. Notably, our original proof of \autoref{divisorial comparison on big open} used the perfectoid mixed characteristic Briançon-Skoda Theorem of \cite{MMRGSBrianconSkoda} for gluing and an LLM suggested to us that we should use the fact that perfectoid rings are seminormal instead, using \autoref{BS - root}.
In \autoref{subsec:frombelow} and \autoref{subsec:nonpure}, these systems were used to flesh out the authors' initial conception of the modules and ideals ($M_B({<}\Delta)$ and $\np(B, {<}\Delta)$) and suggest an organization of the material. We then wrote the section ourselves and used LLMs to suggest minor edits such as finding typos.
In \autoref{subsec.PerfdPureVsLC}, the LLMs also pointed out a substantive mistake in an earlier version of \autoref{damped canonical module comparison} which we fixed and rewrote. LLMs were used in \autoref{subsec.PerfdPureAdjunction} to write a draft of the proof based on an author's understanding of the required outline (similar to other adjunction arguments), and subsequently to edit it for readability. We then rewrote the whole section ourselves.  More generally, we have also made use of LLMs when finalizing this document. They have been used to help identify typos and errors, and to suggest improvements in exposition.  
We take full responsibility for the mathematical content and final form of the article.

\noindent
\textbf{Other Acknowledgements.}
The authors are grateful to Bhargav Bhatt, Zsolt Patakfalvi, and Jakub Witaszek for valuable conversations.

{Cai was supported by the Simons Foundation Collaboration on Perfection in Algebra, Geometry, and Topology.}
{Fayolle was supported by an NSERC doctoral grant, a Simons dissertation fellowship and Simons Foundation grant SFI-MPS-MOV-00006719-07.}
{Ma was supported by NSF Grant DMS \#2302430 and \#2424441 when preparing this article.}
{Schwede was supported by NSF Grants \#2501903, \#2101800, NSF FRG Grant \#1952522, a Fellowship from the Simons Foundation and Simons Travel Support for Mathematicians SFI-MPS-TSM-00013051.}
{Tucker was partially supported by
NSF Grant \#2501904 and Simons Foundation Travel Support for Mathematicians SFI-MPS-TSM-00014083.}
{Smirnov was supported by Spanish Ministry of Science, Innovation and Universities under grants 
RYC2020-028976-I funded by MICIU/AEI/10.13039/501100011033 and EI ESF "ESF Investing in your future”, PCI2024-155055-2 funded by MICIU/AEI/10.13039/501100011033 and UE, and
PID2024-156181NB-C31 funded by MICIU/AEI/10.13039/501100011033 and FEDER, UE. }
{Waldron was supported by NSF CAREER Grant DMS \#2440240, NSF Grant DMS \#2401279, and the Simons Foundation Gift ID \#850684.}

\section{Preliminaries}

All rings considered are commutative with unity. If $R$ is a ring, $I$ is an ideal of $R$, and $M$ is an $R$-module then $M^{\wedge I}$ is the classical $I$-adic completion of $M$. We write
$\widetilde M$ for the associated quasicoherent sheaf on $X=\Spec R$.
Let $p$ be a prime and suppose that it is in the Jacobson radical of $R$. We denote by $I^-{=\bigcap_{n\geq1}(I+p^nR)}$ the $p$-adic closure of $I$ in $R$.

\subsection{Perfectoid rings}
We first recall the definition and various properties of a perfectoid ring as in \cite[Definitions 3.5]{BhattMorrowScholzeIntegralPAHodge,BhattIyengarMaRegularRings} (see also \cite{MaTuckerIntroSingularitiesPerfectoidBCM} for an introduction).

\begin{definition}
A ring $S$ is perfectoid if it is $\varpi$-adically complete for some element $\varpi\in S$ such that $\varpi^pu=p$ for some unit $u$, the Frobenius map $\phi\colon S/pS\to S/pS$ is surjective and the kernel of Fontaine's map $\theta\colon W(S^\flat)\to S$ is principal for $S^\flat=\varprojlim_\phi S/p$.
\end{definition}

\begin{proposition} \label{prop: properties of perfectoids}
  Let $R$ be a perfectoid ring, $f_1, \ldots, f_n \in R$, and $I = (f_1, \ldots, f_n)$.

  \begin{enumerate}
      \item $R$ is reduced \cite[Section 2.1.3]{CesnaviciusScholzePurityFlatCohomology} (see also \cite[Corollary 2.4]{MaTuckerIntroSingularitiesPerfectoidBCM});

      \item $R$ is absolutely weakly normal, and hence seminormal
      \cite[Lemma~4.28]{BMPSTWW3}; \label{seminormal property}

      \item for some unit $u$, $up$ has a compatible system of $p$-power roots \cite[Lemma~3.6]{BhattIyengarMaRegularRings};
      \item for every $f \in R$ there exists $g \in R$ admitting a compatible system of $p$-power roots such that $f - g \in pR$ \cite[Lemma~2.2]{IshizukaYoshikawa};
      \item if $p \in I$, then there are generators $g_1, \ldots, g_n$ of $I$, each admitting a compatible system of $p$-power roots \cite[Lemma~2.3]{IshizukaYoshikawa};
      \item if $p \in I$ then $R^{\wedge I}$ is perfectoid and the completion agrees with derived $I$-completion \cite[2.1.11(e)]{CesnaviciusScholzePurityFlatCohomology}. 
  \end{enumerate}
\end{proposition}

\begin{theorem}[{\cite[Theorem~7.14, Remark~7.15]{BhattScholzePrisms}, cf. \cite{AndreConjectureFacteurDirect}}]\label{thm: Andre}
Let $R$ be a perfectoid ring. Then there exists a $p$-completely faithfully flat map $R \to S$ of perfectoid rings such that $S$ is absolutely integrally closed
(i.e., every monic polynomial has a solution). In particular, every element of $S$ admits a compatible
system of $p$-power roots.
\end{theorem}

Next, let $S$ be a ring with $p$ in its Jacobson radical and denote its perfectoidization by $S_\perfd$, as in \cite[Definition 8.2]{BhattScholzePrisms} (see also \cite[Definition~3.10]{BMPSTWW3}). When $S$ admits a finite map from a perfectoid ring, $S_\perfd$ is a perfectoid ring by \cite[Theorem~10.11]{BhattScholzePrisms}{, and the natural map $S\to S_\perfd$ is universal among maps from $S$ to perfectoid rings by \cite[Corollary~8.14]{BhattScholzePrisms} (see also \cite[Proposition~3.14]{BMPSTWW3})}. {By \cite[Theorem~7.4]{BhattScholzePrisms}, if $S$ is perfectoid and $I\subset S$ is an ideal, then $S\to(S/I)_\perfd$ is surjective. We denote its kernel by $I_\perfd$, so $(S/I)_\perfd=S/I_\perfd$.}
{If $I$ is finitely generated, then $I_\perfd$
gives a theory of almost mathematics for derived $p$-complete
$S$-modules; see \cite[Section~10]{BhattScholzePrisms}. More precisely,
for every $n\geq1$, the pair $(S/p^nS,I_\perfd S/p^nS)$ is a basic
setup for almost mathematics in the sense of
\cite[Section~2.1.1]{GabberRameroAlmostRingTheory}.}

\begin{corollary}\label{pfd of ideal only depends on radical}
Let $S$ be a perfectoid ring and $I\subset S$ an ideal.
{Then $I_\perfd$ is a $p$-adically closed radical ideal and
$I_\perfd=(\sqrt I)_\perfd=(I^-)_\perfd$. Moreover,
if $I=(f_1,\ldots,f_n)$ and each $f_i$ admits a compatible system of
$p$-power roots $f_i^{1/p^\infty}$ in $S$, then
$I_\perfd=(\sqrt I)^-=
\bigl(f_1^{1/p^\infty},\ldots,f_n^{1/p^\infty}\bigr)^-$.}
\end{corollary}
\begin{proof}
{The first assertion follows because $I_\perfd$ is the
kernel of a map to a perfectoid ring. Moreover, every map from $S/I$
to a perfectoid ring kills both $\sqrt I$ and $I^-$, since perfectoid
rings are reduced and $p$-adically separated. The first two
equalities therefore follow from the universal property of
perfectoidization. The final equalities follow from
\cite[Corollary~2.8]{FayolleCentersPerfectoidPurity}, building on the principal case from \cite[Lemma 2.19]{CaiLeeMaSchwedeTuckerPerfectoidSignature}.}
\end{proof}

\begin{lemma}\label{BS - root}
Let $B$ be a seminormal ring. Then the following hold.
\begin{enumerate}[(a)]
    \item If $g,x\in B$ and $g^m\in x^nB$ for some
        $m<n\in\Z_{>0}$, then there is a unique $c\in\sqrt{xB}$ such that
        $g=cx$. \label{BS - root.baseCase}

    \item \label{BS - root.RadicalComputation}
    {Suppose that $f\in B$ and
        $\{f^{1/p^e}\}_{e\in\Z_{>0}}$ is a compatible system of $p$-power
        roots of $f$. Then
        \[
            \sqrt{fB}=(f^{1/p^\infty})B
            \cong\varinjlim \left(B \xrightarrow{\cdot f^{1-1/p}} B
            \xrightarrow{\cdot f^{1/p-1/p^2}} B
            \xrightarrow{\cdot f^{1/p^2-1/p^3}} \cdots \right).
        \]
        For any map $B\to B'$ of seminormal rings, the natural map
        $\sqrt{fB}\otimes_BB'\to\sqrt{fB'}$ is an isomorphism. Moreover,
        $(f^{1/p^\infty})B=(f'^{1/p^\infty})B$ for any other compatible
        system $\{f'^{1/p^e}\}_{e\in\Z_{>0}}$ of $p$-power roots of $f$.}

    \item {Suppose that $\pi,y,z\in B$ satisfy
        $y^N=u\pi^a$ and $z^N=v\pi^b$ for some $N\in\Z_{>0}$, units
        $u,v\in B^\times$, and integers $0\leq a<b$. Then there is a unique
        $w\in\sqrt{yB}$ such that $z=wy$, and
        $w^N=vu^{-1}\pi^{b-a}$.}
\end{enumerate}

{ \noindent In particular, all these assertions apply when the
rings involved are localizations of perfectoid rings.}
\end{lemma}

\begin{proof}
For (a), write $g^m=x^nh$ for some $h\in B$, and set
\[
C=\left(B[T]/(xT-g,T^m-x^{n-m}h)\right)_{\red}.
\]
The map $B\to C$ is finite and subintegral. Indeed, away from $V(x)$
the relation $xT=g$ determines $T$, while over a prime containing
$x$ (and hence $g$), every prime above it contains $T$ and has the
same residue field. Its kernel is contained in every prime of $B$,
and hence is zero because $B$ is reduced. Since $C$ is reduced and
$B$ is seminormal, the map is an isomorphism by
\cite[Tag~0H3G]{stacks-project}.

The image $c$ of $T$ in $B$ satisfies $g=cx$ and
$c^m=x^{n-m}h\in xB$, so $c\in\sqrt{xB}$. If
$c'\in\sqrt{xB}$ also satisfies $g=c'x$, then
$x(c-c')=0$ and $(c-c')^q\in xB$ for some $q>0$. Thus
$(c-c')^{q+1}=0$, and $c=c'$ because $B$ is reduced.

{
For (b), certainly $(f^{1/p^\infty})B\subseteq\sqrt{fB}$.
Conversely, this
ideal is radical.  Indeed, if $g^m\in(f^{1/p^\infty})B$, we may choose $e$ such that
$g^m\in f^{1/p^e}B$ and $k$ such that $p^k>m$. Applying (a) with
$x=f^{1/p^{e+k}}$ and $n=p^k$ gives
$g\in f^{1/p^{e+k}}B$. This proves the first equality and the
independence of the chosen root system.

The colimit description of $(f^{1/p^\infty})B$ in the statement follows from
\cite{AberbachHochsterFiniteTorDimension} and requires only that $B$ is reduced (see also Observation~(a)
in the proof of
\cite[Theorem~1.2]{MaTuckerIntroSingularitiesPerfectoidBCM}).
Tensor products commute with filtered colimits, giving
$(f^{1/p^\infty})B\otimes_BB'\cong(f^{1/p^\infty})B'$ for any map
$B\to B'$ of seminormal rings. Applying the first equality to $B'$
identifies the right side with $\sqrt{fB'}$.}

{
For (c), note that
$z^{(a+1)N}=y^{(b+1)N}v^{a+1}u^{-(b+1)}\pi^{b-a}$.
Since $(a+1)N<(b+1)N$, applying (a) with
$(g,x,m,n)=(z,y,(a+1)N,(b+1)N)$ gives a unique
$w\in\sqrt{yB}$ such that $z=wy$. Put
$d=w^N-vu^{-1}\pi^{b-a}$; then $y^Nd=z^N-v\pi^b=0$.
If $a=0$, then $y$ is a unit, so $d=0$. If $a>0$, then
$\pi^a=u^{-1}y^N\in yB$, so $\pi,d\in\sqrt{yB}$. Choose $\ell>0$
with $d^\ell\in yB$. Then $d^{\ell N+1}=0$, and hence $d=0$
because $B$ is reduced.}

The final assertion follows because every localization of a
perfectoid ring is seminormal by
\autoref{prop: properties of perfectoids}(\ref{seminormal property})
and \cite[Tag~0EUM]{stacks-project}.
\end{proof}

\begin{corollary}\label{root ideal tensor comparison}
{Let $R\to B$ be a ring map with $B$ perfectoid, and let
$f\in R$ be such that its image in $B$ admits a compatible system of
$p$-power roots. For every $p^\infty$-torsion $R$-module $M$, the
natural map
\[
M\tensor_R(f^{1/p^\infty})B
\longrightarrow M\tensor_R(fB)_\perfd
\]
is an isomorphism.}
\end{corollary}

\begin{proof}
{The ring $B/\sqrt{fB}$ is reduced, so its
$p^\infty$-torsion is killed by $p$: if $p^n\overline b=0$, then
$(p\overline b)^n=0$. By
\cite[Proposition~2.13]{CaiLeeMaSchwedeTuckerPerfectoidSignature}, the
derived $p$-completion of $B/\sqrt{fB}$ is therefore discrete and
agrees with $B/(\sqrt{fB})^-=B/(fB)_\perfd$ by
\autoref{pfd of ideal only depends on radical}.

A complex and its derived $p$-completion become equivalent after
derived tensoring with a $p^\infty$-torsion module. This is immediate
for a module killed by $p^N$, and the general case follows from
$M=\colim_NM[p^N]$. Applying this to $B/\sqrt{fB}$ and using
\autoref{BS - root} gives
$M\tensor_R^{\mathbb L}(f^{1/p^\infty})B\simeq
M\tensor_R^{\mathbb L}(fB)_\perfd$. Taking zeroth cohomology proves
the result.}
\end{proof}

We will often work with a special perfectoid ring, constructed as follows.

\begin{notation}\label{notation RAinfty}
Let $(R,\fram,k)$ be a complete Noetherian local ring of residue
characteristic $p>0$. Fix a Cohen ring $C_k$ for $k$ and a local map
$C_k\to W(k^{1/p^\infty})$ lifting the inclusion
$k\subseteq k^{1/p^\infty}$. Let
\[
\varphi\colon
A=C_k\llbracket x_1,\ldots,x_n\rrbracket\longrightarrow R
\]
be a local homomorphism inducing the identity on residue fields such
that $R$ is module-finite over $\varphi(A)$. The map $\varphi$ need
not be injective or surjective. When it is surjective, we call it a
\emph{Cohen presentation} of $R$; when it is injective, we call it a
\emph{Noether--Cohen normalization} of $R$. Set
\[
A_\infty\coloneqq
\left(A\cotensor{C_k}W(k^{1/p^\infty})\right)
\bigl[p^{1/p^\infty},x_1^{1/p^\infty},\ldots,
x_n^{1/p^\infty}\bigr]^{\wedge p}.
\]
Here both completions are $p$-adic. Then $A_\infty$ is perfectoid by
\cite[Example~3.8(4)]{BhattIyengarMaRegularRings}. Set
$\raftypfd\coloneqq(R\tensor_AA_\infty)_\perfd$, which is a perfectoid
$R$-algebra by \cite[Theorem~10.11]{BhattScholzePrisms}. Both
$A_\infty$ and $\raftypfd$ depend on the chosen regular system of
parameters $p,x_1,\ldots,x_n$ of $A$, and $\raftypfd$ also depends on
$\varphi$.
\end{notation}

\begin{lemma}\label{rafty is R torsion free}
Let $(R,\m)$ be a Noetherian complete local domain of mixed characteristic and let $A\to R$ be as in \autoref{notation RAinfty}. Then every nonzero element of $R$ is a nonzerodivisor on $R^{A_\infty}_\perfd$. Moreover, for every $0 \neq g \in R$, there exists an integer $N \geq 0$ such that $(p^nR^{A_\infty}_\perfd :_{R^{A_\infty}_\perfd}g)
\subseteq p^{n-N}R^{A_\infty}_\perfd$ for every $n > N$.
\end{lemma}
\begin{proof}
Let $0\neq g\in R$ and we will first show that $g$ is a nonzerodivisor on $R^{A_\infty}_\perfd$. { If necessary, we may adjoin variables to $A$ to obtain a new ring $S$ with $S\twoheadrightarrow R$.  We may further assume that $g$ admits a compatible system of $p$-power roots in $R^{S_{\infty}}_{\perfd}$ (by replacing $S$ by $S\llbracket x\rrbracket$ and considering $S\llbracket x\rrbracket\to R$ where $x\mapsto g$).  Because $R^{A_\infty}_\perfd\to R^{S_\infty}_\perfd$ is $p$-completely faithfully flat and in particular injective by Andr\'{e}'s flatness lemma, it suffices to show that $g$ is not a zerodivisor on $R^{S_{\infty}}_{\perfd}$.} Writing $R=S/Q$, we also have that $S_\infty/Q\twoheadrightarrow (S_\infty/Q)_\perfd \cong R^{S_\infty}_\perfd$ by \cite[Theorem 7.4]{BhattScholzePrisms} (essentially also Andr\'{e}'s flatness lemma).

In what follows, for a $p$-torsion free ring $T$, we denote $T'$ the $p$-root closure of $T$ in $T[1/p]$. That is, 
$$T'=\{ \eta \in T[1/p] \mid \eta^{p^e}\in T \text{ for some $e$}\}.$$
We consider the following commutative diagram
\[
\xymatrix{
S_\infty/Q \ar[r] \ar@{->>}[rrd] & (S_\infty/Q)' \ar[r] & (S_\infty/Q)'^{\wedge p} \ar[d]^\cong & \\
&& (S_\infty/Q)_\perfd 
}
\]
where the isomorphism follows from \cite[Corollary 5.9]{IshizukaCalculationPerfectoidization}. Note that all rings in the above diagrams are $p$-torsion free (see \cite[Lemma A.2]{MaSchwedeTuckerWaldronWitaszekAdjointPLT}).

Suppose $gz=0$ in $(S_\infty/Q)_\perfd$. By the surjection $S_\infty/Q\twoheadrightarrow (S_{\infty}/Q)_\perfd$, we will abuse notation and also use $g$ and $z$ as elements in $S_\infty/Q$. Since $(S_{\infty}/Q)_\perfd$ is reduced, we know that $g^{1/p^\infty}z=0$ in $(S_{\infty}/Q)_\perfd$ and thus 
$$g^{1/p^\infty}z \in \ker(S_\infty/Q \to (S_\infty/Q)'^{\wedge p})$$
Unwinding definitions, this means 
$$g^{1/p^\infty}z \in \cap_np^n (S_\infty/Q)' \cap (S_\infty/Q).$$
By the definition of $p$-root closure, this means 
$$\forall m, n, \exists e_0 \text{ such that } g^{p^{e-m}}z^{p^e} \in p^{np^e}(S_\infty/Q) \text{ for all $e\geq e_0$}.$$
It follows that for all $e\gg0$ (depending on $m,n$),
\begin{equation}
\label{eqn: containment in proof of g nonzerodivisor}
z^{p^e} \in p^{np^e}(S_\infty/Q): g^{p^{e-m}} =(p^{np^e}:_{S/Q} g^{p^{e-m}})(S_\infty/Q)
\end{equation}
by faithful flatness of $S/Q\to S_\infty/Q$. Let $v_1,\dots,v_t$ be the Rees valuations of $(p)\subseteq R$. For any $x\in p^{np^e}:_R g^{p^{e-m}}$, we have 
$v_i(x) + p^{e-m}v_i(g) \geq np^ev_i(p)$. Thus 
$$\frac{v_i(x)}{p^e} \geq nv_i(p) - \frac{v_i(g)}{p^m}.$$
Choose $m\gg0$ so that $v_i(g)/p^m\leq v_i(p)$ for all $i$.  We have that 
$$v_i(x)\geq (n-1)p^ev_i(p)$$
for all $i$ and thus $x\in \overline{p^{(n-1)p^e}R}\subseteq p^{(n-1)p^e-C}R$ by valuative criterion of integral closure and the uniform Brian\c{c}on--Skoda theorem ($C$ is a constant that does not depend on $n$ and $e$) \cite{HunekeUniformBounds,MMRGSBrianconSkoda}. Now \autoref{eqn: containment in proof of g nonzerodivisor} says that for any fixed $n$, for all $e\gg0$, we have
$$z^{p^e} \in p^{(n-1)p^e-C}(S_\infty/Q) \subseteq p^{(n-2)p^e}(S_\infty/Q).$$
It follows that $\frac{z}{p^{n-2}}\in(S_\infty/Q)'$ and thus $\frac{z}{p^{n-2}}\in (S_\infty/Q)_\perfd$ (note that we are abusing notations, it should be interpreted as the image of $\frac{z}{p^{n-2}}$). Since this is true for all $n$, we obtain that 
$$z\in \cap_n p^{n-2}(S_\infty/Q)_\perfd=0$$
since $(S_\infty/Q)_\perfd$ is $p$-adically complete.

For the final statement, choose $N \geq 1$ such that $v_i(g) \leq (N-1) v_i(p)$ for every $i$. If $n>N$, $z\in(S_\infty/Q)'$, and
$gz\in p^n(S_\infty/Q)'$, the calculation above applied to
\[
g^{p^e}z^{p^e}\in p^{np^e}(S_\infty/Q)
\]
for $e\gg0$ gives
$z^{p^e}\in p^{(n-N)p^e}(S_\infty/Q)$.
Thus $z/p^{n-N}\in(S_\infty/Q)'$.
Since $(S_\infty/Q)' \to (S_\infty/Q)_\perfd$ is surjective (the kernel is exactly $\cap_n p^n(S_\infty/Q)'$), passing to $p$-adic completion we obtain that
\[
(p^n (S_\infty/Q)_\perfd :_{(S_\infty/Q)_\perfd} g) \subseteq p^{n-N}(S_\infty/Q)_\perfd.
\]
Since $R^{A_\infty}_\perfd\to(S_\infty/Q)_\perfd$ is faithfully flat modulo $p^{n}$, 
\begin{align*}
(0:_{{\raftypfd}/p^n} g) \otimes_{\raftypfd} (S_\infty/Q)_\perfd & \cong 0:_{(S_\infty/Q)_\perfd/p^n} g \\
& \subseteq p^{n-N}(S_\infty/Q)_\perfd/p^n(S_\infty/Q)_\perfd\\
& \cong \big(p^{n-N}\raftypfd/p^n\raftypfd\big) \otimes_{\raftypfd} (S_\infty/Q)_\perfd,
\end{align*}
which implies that $0:_{{\raftypfd}/p^n} g \subseteq p^{n-N}\raftypfd/p^n\raftypfd$ as wanted.
\end{proof}

{

\begin{proposition}\label{componentwise BCM transfer}
In the setting of \autoref{notation RAinfty}, suppose also that
$0\neq g\in A$ is such that $A[g^{-1}]\to R[g^{-1}]$ is finite
\'etale. Let $N$ be a $p^\infty$-torsion $R$-module and
$\zeta\in N$. For each minimal prime $\mathfrak p_i$ of $R$, set
$R_i=R/\mathfrak p_i$ and let $\zeta_i$ be the image of $\zeta$ in
$R_i\tensor_RN$. Suppose that, whenever $g\notin\mathfrak p_i$, the
element $\zeta_i$ maps to zero in $B_i\tensor_RN$ for some perfectoid
balanced big Cohen--Macaulay $R_i^+$-algebra $B_i$. Then
$(gA_\infty)_\perfd(1\tensor\zeta)=0$ in
$\raftypfd\tensor_RN$.
\end{proposition}

\begin{proof}
We first show that $(gA_\infty)_\perfd(1\tensor\zeta)=0$ in
$(R_i)^{A_\infty}_\perfd\tensor_RN$ for every $i$. If
$g\in\mathfrak p_i$, then $g$ is in the kernel of $A_\infty\to(R_i)^{A_\infty}_\perfd$. Since the kernel is a perfectoid ideal, it must contain $(gA_\infty)_\perfd$, so the claim follows.
If $g\notin\mathfrak p_i$, the ideal
$\mathfrak p_iR[g^{-1}]$ is a minimal prime of the finite \'etale
$A[g^{-1}]$-algebra $R[g^{-1}]$ and contracts to $(0)$ in
$A[g^{-1}]$. Since $R[g^{-1}]$ is finite \'etale over the normal domain
$A[g^{-1}]$, it is normal. Thus $R[g^{-1}]$ is the product of the
$R_i[g^{-1}]$ ranging over all $i$ with $g\notin \p_i$ and applying \cite[Tag 00U2]{stacks-project} gives that
$A[g^{-1}]\to R_i[g^{-1}]$ is finite \'etale. Then, $A\subset R_i$ is a Noether-Cohen normalization so the argument of
\cite[Lemma~5.6]{CaiLeeMaSchwedeTuckerPerfectoidSignature}, which
applies to the present choice of \(A_\infty\) by
\cite[Lemma~7.9 and the proof of Claim~7.11]{BMPSTWW2}, gives $(gA_\infty)_\perfd(1\tensor\zeta)=0$ in
$(R_i)^{A_\infty}_\perfd\tensor_RN$, proving the componentwise claim.

Since perfectoidization commutes with finite products, the target of
\begin{equation*}
\raftypfd\longrightarrow\prod_i(R_i)^{A_\infty}_\perfd
\end{equation*}
is the perfectoidization of the finite $\raftypfd$-algebra
$\left(\prod_iR_i\right)\tensor_R\raftypfd$, which agrees with
$\raftypfd$ after inverting $g$. Thus the displayed map is
$p$-completely $(g\raftypfd)_\perfd$-almost pure by
\cite[Proposition~2.25]{CaiLeeMaSchwedeTuckerPerfectoidSignature}. Since $N$ is $p^\infty$-torsion, it follows that 
$$\ker\left(\raftypfd \otimes_R N \to \prod_i(R_i)^{A_\infty}_\perfd \otimes_R N\right)$$ is annihilated by $(g\raftypfd)_\perfd$. 
Now the componentwise vanishing shows that
$(gA_\infty)_\perfd(1\tensor\zeta)$ is in the above kernel. Thus we obtain that
$$(gA_\infty)_\perfd(1\tensor\zeta) = (gA_\infty)_\perfd \left((gA_\infty)_\perfd(1\tensor\zeta)\right) \subseteq (g\raftypfd)_\perfd\left((gA_\infty)_\perfd(1\tensor\zeta)\right) =0$$
where we used that $(gA_\infty)_\perfd$ is idempotent (after tensoring with $N$) and that $(gA_\infty)_\perfd\raftypfd\subseteq(g\raftypfd)_\perfd$ (which follows from universality as $A_\infty\to(\raftypfd/g\raftypfd)_\perfd$ kills $g$).
\end{proof}
}

\subsection{Generalized and Mumford divisors} 
We use Hartshorne's theory of generalized divisors in a form based on
codimension-two extension; see
\cite{HartshorneGeneralizedDivisorsOnGorensteinSchemes,HartshonreGeneralizedDivisorsAndBiliaison}.
Let $X$ be a Noetherian scheme. For a coherent $\sO_X$-module $\sF$,
we say that $\sF$ is $\mathbf{S}_k$ if
$\depth_{\sO_{X,x}}\sF_x\geq
\min\{k,\dim\bigl(\Supp(\sF_x)\bigr)\}$ for every $x\in X$.
We say that $X$ is $\mathbf{S}_k$ if $\sO_X$ is $\mathbf{S}_k$
\cite[Tag~0341]{stacks-project}. If every local ring of $X$ is
catenary and $X$ is $\mathbf{S}_2$, then $X$ is locally equidimensional
\cite[Tag~0FIW]{stacks-project}. This applies, in particular, when
$X$ is excellent and $\mathbf{S}_2$.
We say that a closed subset $Z\subseteq X$ has codimension at least
two if $\dim\sO_{X,x}\geq2$ for every $x\in Z$.
For a coherent $\sO_X$-module $\sF$, set
$\sF^{**}\coloneqq
\sHom_X(\sHom_X(\sF,\sO_X),\sO_X)$, and call $\sF$ \emph{reflexive}
if the natural map $\sF\to\sF^{**}$ is an isomorphism.
If $X$ is a Noetherian {scheme}, recall that a \emph{big open subset} $U \subseteq X$ is an open subset whose complement has codimension $\geq 2$.

\begin{definition}\label{codimension two extension property}
Let \(X\) be a Noetherian scheme and \(\sG\) an
\(\sO_X\)-module. We say that \(\sG\) has the \emph{codimension-two extension
property} if for every open immersion
$j\colon U\hookrightarrow X$ from a big open, the canonical map $\sG\longrightarrow j_*(\sG|_U)$
is an isomorphism.
Equivalently, for every \(\sO_X\)-module \(\sF\), restriction induces
an isomorphism
\[
\Hom_X(\sF,\sG)
\xrightarrow{\ \sim\ }
\Hom_U(\sF|_U,\sG|_U).
\]
\end{definition}

Let $\mathcal S_X\subseteq\sO_X$ be the subsheaf of local
nonzerodivisors, and set
$\sK_X\coloneqq\mathcal S_X^{-1}\sO_X$, the sheaf of total quotient
rings of $X$. For every affine open $V=\Spec A\subseteq X$, one has
$\Gamma(V,\sK_X)=K(A)$, the total quotient ring of $A$
\cite[Tag~02OV]{stacks-project}.
A \emph{fractional ideal} on $X$ is a coherent
$\sO_X$-submodule $\sI\subseteq\sK_X$; it is \emph{nondegenerate} if
$\sI_\eta=(\sK_X)_\eta$ at every generic point $\eta$ of an
irreducible component of $X$.
It is \emph{principal} if $\sI=g\sO_X$ for some
$g\in\Gamma(X,\sK_X)^\times$, \emph{locally principal} if it is
locally of this form, and \emph{locally principal in codimension one}
if it is locally principal on a big open subset.
A nondegenerate $\mathbf{S}_2$ fractional ideal
has the codimension-two extension property 
\cite[\href{https://stacks.math.columbia.edu/tag/0E9I}{Tag 0E9I}]{stacks-project}. 
Consequently, these
fractional ideals and their morphisms are determined by restriction
to any big open. This motivates the following construction.

\begin{definition}\label{codimension one envelope}
For an $\sO_X$-module $\sF$, set
\[
\sF^{\mathrm{c1}}
\coloneqq
\varinjlim_U
(j_U)_*\bigl(\sF|_U\bigr),
\]
where the filtered colimit is over big open subsets $U\subseteq X$, ordered by reverse
inclusion, and $j_U\colon U\hookrightarrow X$ is the open immersion.
There is a natural map $\sF\to\sF^{\mathrm{c1}}$, which is an
isomorphism at every point of codimension at most one. We call
$\sF^{\mathrm{c1}}$ the \emph{codimension-one envelope} of $\sF$.
If $X=\Spec R$ and $N$ is an $R$-module, we use the same notation for
\[
N^{\mathrm{c1}}
\coloneqq
\Gamma\bigl(X,(\widetilde N)^{\mathrm{c1}}\bigr)
=
\varinjlim_U
\Gamma(U,\widetilde N),
\]
further writing $N^{\mathrm{c1},R}$ to specify the base as necessary.
{ In particular, $\widetilde{N}^{\cOne} \cong \widetilde{N^{\cOne}}$ (which can be quickly deduced from \cite[\href{https://stacks.math.columbia.edu/tag/01LC}{Tag 01LC}]{stacks-project} and \cite[\href{https://stacks.math.columbia.edu/tag/01BH}{Tag 01BH}]{stacks-project}). In either setting, we observe that $(-)^{\cOne}$ is a functor and $M \to M^{\cOne}$ is a natural transformation.}
\end{definition}

{
\begin{proposition}
    \label{prop.SomeGeneralFactsAboutCodim1Envelopes}
    Fix $R$ to be a Noetherian ring and suppose that $M$ and $N$ are $R$-modules.  Then:
    \begin{enumerate}
        \item The sheaf $\widetilde{N^{\cOne}}$ has the codimension-2 extension property. \label{prop.SomeGeneralFactsAboutCodim1Envelopes.cOneHasExtensionProp}
        \item $N^{\cOne} \to (N^{\cOne})^{\cOne}$ is an isomorphism.\label{prop.SomeGeneralFactsAboutCodim1Envelopes.cOneIdempotence}
        \item For every height $\leq 1$ prime $P$ of $R$, the map $N_P \to (N^{\cOne})_P$ is an isomorphism.  \label{prop.SomeGeneralFactsAboutCodim1Envelopes.LocalizationStaysIso}        
        \item Suppose $M \to N$ is a map of $R$-modules, then $M^{\cOne} \to N^{\cOne}$ is an isomorphism if and only if $M_P \to N_P$ is an isomorphism for all height $\leq 1$ primes $P \subseteq R$.  \label{prop.SomeGeneralFactsAboutCodim1Envelopes.IsoCheckedCodim1}
        \item The natural map $(M\otimes N)^{\cOne} \to (M^{\cOne} \otimes N)^{\cOne}$ is an isomorphism. \label{prop.SomeGeneralFactsAboutCodim1Envelopes.TensorProduct}
        \item The formation of $(-)^{\cOne}$ commutes with filtered colimits.  In other words, for a filtered system $\{M_i\}$ with $\varinjlim_i M_i = M$, we have a natural isomorphism $\varinjlim_i (M_i^{\cOne}) \cong M^{\cOne}$.  \label{prop.SomeGeneralFactsAboutCodim1Envelopes.Colimit}
    \end{enumerate}

\end{proposition}
\begin{proof}
    Write $X = \Spec R$.  
    For \autoref{prop.SomeGeneralFactsAboutCodim1Envelopes.cOneHasExtensionProp}, 
    pick a big open subset $V \subseteq X$.  We need to show that 
    \[
        N^{\cOne} = \varinjlim_{U \subseteq X} \Gamma(U, \widetilde{N}) \to \Gamma(V, \varinjlim_{U \subseteq X} (j_U)_* \widetilde{N}|_U)
    \]
    is an isomorphism where the filtered colimits run over big open subsets.  But the right side is equal to $\displaystyle{\varinjlim_{U \subseteq V} \Gamma(U, \widetilde{N})}$ which clearly agrees with $N^{\cOne}$.  This proves the first statement.  But then \autoref{prop.SomeGeneralFactsAboutCodim1Envelopes.cOneIdempotence} follows immediately.

    For \autoref{prop.SomeGeneralFactsAboutCodim1Envelopes.LocalizationStaysIso}, 
    suppose we have two big open subsets $P \in U \subseteq V \subseteq X$.   Then by \cite[\href{https://stacks.math.columbia.edu/tag/0EHM}{Tag 0EHM}]{stacks-project}, \cite[\href{https://stacks.math.columbia.edu/tag/01LC}{Tag 01LC}]{stacks-project}, and the correspondence between modules and quasi-coherent sheaves, all the restriction maps $N_P = \Gamma(X, \widetilde{N})_P \to \Gamma(V, \widetilde{N})_P \to \Gamma(U, \widetilde{N})_P$ are isomorphisms.  It follows that $N_P \to (N^{\cOne})_P$ is also an isomorphism.  

    The ($\Rightarrow$) direction of \autoref{prop.SomeGeneralFactsAboutCodim1Envelopes.IsoCheckedCodim1} is immediate from \autoref{prop.SomeGeneralFactsAboutCodim1Envelopes.LocalizationStaysIso}.  For the converse, suppose that $K$ and $C$ are the kernel and cokernel of $M^{\cOne} \to N^{\cOne}$ so that $0 \to K \to M^{\cOne} \to N^{\cOne} \to C \to 0$  is an exact sequence. Note that by \autoref{prop.SomeGeneralFactsAboutCodim1Envelopes.cOneHasExtensionProp}, $M^\cOne$ has codimension-$2$ extension property. In particular, this implies that $H_I^0(M^\cOne)=H_I^1(M^\cOne)=0$ for all ideals $I\subseteq R$ of height $\geq 2$ (and likewise for $N^{\cOne}$). Now for every finitely generated $K' \subseteq K$, $K'_P=0$ for all height $\leq 1$ primes $P$ implies that $K'=H_J^0(K')$ for some $J\subseteq R$ of height $\geq 2$. But then $K'=0$ as $H_J^0(M^\cOne)=0$. 
    It follows that $K=0$ and we have $0\to M^\cOne \to N^\cOne\to C\to 0$. By examining the long exact sequence of local cohomology we obtain that $H_I^0(C)=0$ for all $I\subseteq R$ of height $\geq 2$. But for every finitely generated $C'\subseteq C$, we know that $C'=H_J^0(C')$ for some $J\subseteq R$ of height $\geq 2$ because $C'_P \subseteq C_P=0$ for all height $\leq 1$ primes $P$. Thus $C'=0$ as $H_J^0(C)=0$. Therefore $C=0$ and we have $M^\cOne\to N^\cOne$ is an isomorphism.

    For \autoref{prop.SomeGeneralFactsAboutCodim1Envelopes.TensorProduct}, notice that $M \otimes N \to M^{\cOne} \otimes N$ becomes an isomorphism after localizing at a height $\leq 1$ prime ideal $P$.  Thus, the result follows from \autoref{prop.SomeGeneralFactsAboutCodim1Envelopes.IsoCheckedCodim1}.

    Finally, \autoref{prop.SomeGeneralFactsAboutCodim1Envelopes.Colimit} follows as $\Gamma(U, -)$ commutes with filtered colimits and since colimits commute with colimits.
\end{proof}
}
\begin{lemma}\label{codimension one envelope fractional ideal}
Let $X$ be a Noetherian, locally equidimensional $\mathbf{S}_2$-scheme.
Let $\sI$ be a nondegenerate fractional ideal. Then
$\sI^{\mathrm{c1}}$ is a nondegenerate $\mathbf{S}_2$ fractional ideal and is
the unique such fractional ideal agreeing with $\sI$ at every point
of codimension at most one. In particular, the natural map
$\sI\longrightarrow\sI^{\mathrm{c1}}$ is an isomorphism if and only
if $\sI$ is $\mathbf{S}_2$.

If $\sI$ is locally principal in codimension one, we have a
canonical identification $\sI^{\mathrm{c1}}\cong\sI^{**}$. In
particular, $\sI^{\mathrm{c1}}$ is reflexive, and, under the same
hypothesis, $\sI$ is $\mathbf{S}_2$ if and only if it is reflexive.

\end{lemma}

\begin{proof}
The assertion is local, so we may clear denominators and assume
$\sI\subseteq\sO_X$. Applied to $\sO_X$, the codimension-two extension
property gives $(j_U)_*\sO_U=\sO_X$. Thus
$(j_U)_*(\sI|_U)$ is a quasi-coherent ideal of $\sO_X$
\cite[Tag~01LC]{stacks-project}. Their filtered union
$\sI^{\mathrm{c1}}$ is therefore a quasi-coherent ideal; since $X$ is
Noetherian, it is coherent and the system stabilizes. As a submodule
of $\sK_X$, it is $\mathbf{S}_1$. A stabilizing $U$ may be shrunk to omit any
prescribed point of codimension at least two, so
\cite[Tag~0DX4]{stacks-project} gives depth at least two there. Thus
$\sI^{\mathrm{c1}}$ is $\mathbf{S}_2$. Nondegeneracy and agreement with $\sI$
in codimension at most one are immediate, while the codimension-two
extension property gives uniqueness and the final assertion.

For the additional statement, choose an open immersion
$j\colon U\hookrightarrow X$ whose complement has codimension at
least two and such that $\sI|_U$ is invertible. The sheaf $\sI^{**}$
is a nondegenerate fractional ideal and is $\mathbf{S}_2$ by
\cite[Tag~0AV6]{stacks-project}; its restriction to $U$ is
canonically $\sI|_U$. The codimension-two extension property therefore
identifies $\sI^{**}$ with $j_*(\sI|_U)$, and hence with
$\sI^{\mathrm{c1}}$. Applying this identification to
$\sI^{\mathrm{c1}}$ shows that the latter is reflexive. The final
equivalence now follows from the first part of the lemma.
\end{proof}

\begin{definition}
Let $X$ be a Noetherian, locally equidimensional $\mathbf{S}_2$-scheme.
\begin{enumerate}[(1)]
    \item A \emph{generalized divisor} $D$ is a nondegenerate $\mathbf{S}_2$
    fractional ideal $\sI_D$.
    \item We call a generalized divisor $D$
    \emph{effective} if $\sI_D\subseteq\sO_X$, \emph{principal} if
    $\sI_D$ is principal, and \emph{Cartier} if $\sI_D$ is locally
    principal (which agrees with the usual notion of a Cartier divisor). If $\sI_D=g\sO_X$, we write $D=\Div_X(g)$. The zero
    divisor is given by $\sI_0=\sO_X$.
    \item A generalized divisor is a \emph{Weil generalized divisor}
    if $\sI_D$ is locally principal in codimension
    one. Hartshorne calls these almost Cartier
    divisors \cite[Section~2]{HartshonreGeneralizedDivisorsAndBiliaison};
    we use ``Weil'' because ``almost'' has a different meaning here.
\end{enumerate}
\end{definition}

For generalized divisors $D$ and $E$, define
\begin{equation}
    \label{eq.OperationsOfGeneralizedDivisors}
\sI_{D+E}\coloneqq(\sI_D\sI_E)^{\mathrm{c1}},
\qquad
\sI_{-D}\coloneqq\sHom_X(\sI_D,\sO_X),
\end{equation}
where $\sI_D\sI_E$ is the image of
$\sI_D\otimes_{\sO_X}\sI_E\to\sK_X$. The product
$\sI_D\sI_E$ is a coherent nondegenerate fractional ideal, so
\autoref{codimension one envelope fractional ideal} shows that
$\sI_{D+E}$ is a generalized divisor. On an open set $V\subset X$, the sheaf $\sI_{-D}$ is described by
$\{q\in\sK_X(V)\mid q\sI_D(V)\subseteq\sO_X(V)\}$; it is coherent and
nondegenerate, and is $\mathbf{S}_2$ by
\cite[Tag~0AV6]{stacks-project}. Thus $-D$ is also a generalized
divisor. For any generalized divisor $D$, set
$\sO_X(D)\coloneqq\sI_{-D}$. If $D$ is effective, the inclusion
$\sO_X\subseteq\sO_X(D)$ determines a distinguished global section
$1$ of $\sO_X(D)$. When $X=\Spec R$, set
$I_D=\Gamma(X,\sI_D)$ and write
\[
R(D)=\Gamma(X,\sO_X(D))=\Hom_R(I_D,R)
=\{q\in K(R)\mid qI_D\subseteq R\}.
\]

Now suppose that $D$ and $E$ are Weil generalized divisors, and choose
a common Cartier locus $j\colon U\hookrightarrow X$ whose complement
has codimension at least two. By
\autoref{codimension one envelope fractional ideal}, $\sI_D$ and
$\sI_E$ are reflexive. In particular, $\sO_X(-D)=\sI_D$. That lemma and the codimension-two extension
property give
\[
\sI_{D+E}\cong(\sI_D\sI_E)^{**}
\cong j_*\bigl((\sI_D\sI_E)|_U\bigr),\qquad
\sI_{-D}\cong j_*\bigl((\sI_D|_U)^{-1}\bigr).
\]
In particular, $D+E$ and $-D$ are Weil generalized divisors. These
descriptions are independent of $U$ and show that the Weil generalized
divisors form an abelian group. For $n\in\Z$, write $nD$ for the
corresponding multiple. Define $\Supp(D)$ to be the closure of the
support of $D|_U$; this is independent of the choice of $U$. If $D$
is effective, then $\Supp(D)=\Supp(\sO_X/\sI_D)$.
{Whenever $K_X$ appears, we assume that $X$ admits
a canonical sheaf $\omega_X$. If $X$ is $\mathbf{S}_2$ and $\mathbf{G}_1$, after
choosing a generic trivialization of $\omega_X$, we write $K_X$ for
the corresponding Weil generalized divisor, so that
$\sO_X(K_X)\cong\omega_X$.}

We record the following compatibility with flat pullback.

\begin{proposition}\label{flat keeps reflexive}
Let $R\to R'$ be a flat homomorphism of Noetherian, locally
equidimensional $\mathbf{S}_2$ rings, and let $D$ be a Weil generalized divisor
on $\Spec R$. Then
$R' \tensor_RI_D\subseteq K(R')$ defines a Weil generalized divisor
$D'$ on $\Spec R'$, called the flat pullback of $D$, and there are
canonical identifications
\[
I_{D'}=R'\tensor_R I_D
\qquad\text{and}\qquad
R'(D')\cong R'\tensor_RR(D).
\]
Flat pullback is functorial, defines a homomorphism of groups of Weil
generalized divisors, and preserves effective, Cartier, and principal
divisors. If $D$ is effective, the second identification carries $1$
to $1\tensor1$.
\end{proposition}
\begin{proof}
Let $W$ be the set of nonzerodivisors of $R$. Flatness gives
$R'\tensor_RI_D\hookrightarrow R'\tensor_RK(R)\cong
W^{-1}R'\hookrightarrow K(R')$, since every element of $W$ remains a
nonzerodivisor on $R'$. Thus $R'\tensor_RI_D$ is a fractional ideal.

Choose an open immersion $j\colon U\hookrightarrow\Spec R$ whose
complement has codimension at least two and on which $D$ is Cartier,
and let $j'\colon V\hookrightarrow\Spec R'$ be its inverse image. Flat
going-down \cite[Tag~00HS]{stacks-project} shows that
$\Spec R'\setminus V$ has codimension at least two and that
$R'\tensor_RI_D$ is nondegenerate: its restriction to $V$ is
invertible. By \autoref{codimension one envelope fractional ideal},
$I_D$ is reflexive, so flat base change for $\Hom$ shows that
$R'\tensor_RI_D$ is reflexive. It is therefore $\mathbf{S}_2$ by
\cite[Tag~0AV6]{stacks-project} and defines the Weil generalized
divisor $D'$. The same base change gives the second displayed
identification. Functoriality and the remaining compatibilities follow
from the double-dual description of addition and the definitions.
\end{proof}

We conclude by relating Weil generalized divisors to Mumford divisors.
\begin{definition}
Let $X$ be a Noetherian, locally equidimensional
$\mathbf{S}_2$-scheme. A \emph{prime Mumford divisor} on $X$ is an integral
closed subscheme $Z\subseteq X$ of codimension one such that $X$ is
regular at its generic point $\eta_Z$. A \emph{Mumford divisor}
(respectively, a \emph{Mumford $\Q$-divisor}) is a finite formal sum
$D=\sum_i a_iZ_i$ of prime Mumford divisors with $a_i\in\Z$
(respectively, $a_i\in\Q$) \cite[Definition~1]{KollarMumfordDivisors}.
Its support is $\bigcup_{a_i\neq0}Z_i$, and it is \emph{effective} if
every $a_i\geq0$. A Mumford $\Q$-divisor $D$ is
\emph{$\Q$-Cartier} if $mD$ is Cartier for some integer $m>0$. If $D=\sum_i a_iZ_i$ is a Mumford divisor then we write $D_{\red}$ for the Mumford divisor $\sum_i b_i Z_i$ where $b_i=0$ if $a_i=0$ and $b_i=1$ if $a_i\neq 0$.
\end{definition}

If $Z$ is a prime Mumford divisor, its ideal sequence combined with \cite[Tag 00LX]{stacks-project} shows that
$\sI_Z$ is $\mathbf{S}_2$, while regularity at $\eta_Z$ gives local
principality in codimension one. Thus $\sI_Z$ defines a Weil
generalized divisor, which we again denote by $Z$. The group law then
associates a Weil generalized divisor to every Mumford divisor. For a
Mumford divisor $D$, choose an open immersion $j\colon U\hookrightarrow
X$ whose complement has codimension at least two and such that $D|_U$
is Cartier. Then
$\sI_D=\sO_X(-D)=j_*\sO_U(-D|_U)\subseteq\sK_X$.
This association preserves support and effectivity.
For Mumford $\Q$-divisors, inequalities and convergence
are coefficientwise, while support and reduction have their usual
meanings. For effective Mumford $\Q$-divisors $D'$ and
$D''$ on $X$, write
\begin{equation}\label{def:divisor_order_relation}
D'\prec D''\quad\Longleftrightarrow\quad
D'\leq(1-\varepsilon)D''\text{ for some }0<\varepsilon<1.
\end{equation}

\begin{remark}\label{flat pullback is Mumford}
    In general, the flat pullback of a Mumford divisor need not be a
Mumford divisor. However, flat pullback preserves Mumford divisors -- and
hence extends $\Q$-linearly to Mumford $\Q$-divisors -- when the fibers
are $R_0$ (e.g., when the morphism is regular).

More precisely, let $f\colon R\to S$ be a flat map between excellent
reduced $\mathbf{S}_2$ local rings and let $\Delta$ be a Mumford $\Q$-divisor
on $R$. Then $\Delta_S$ is a Mumford $\Q$-divisor if and only if $S_P$
is regular for every minimal prime $P$ of $\p S$ and every height-one
prime $\p\in\Supp(\Delta)$. In that case, we often abuse notation and write $\Delta=\Delta_S$ as divisors on $S$ since there is no potential for confusion. We also call $\Delta=\Delta_S$ the \emph{flat pullback} of $\Delta$ to $S$.
\end{remark}

Moreover, if $\Delta$ is an effective Mumford $\Q$-divisor on $R$ then one can always construct a finite flat extension $f\colon R\to S$ such that $\Delta$ admits a flat pullback to an effective Mumford divisor on $S$ as the following construction shows.

\begin{construction}\label{construction SnD}
Let $R$ be an excellent
$\mathbf{S}_2$ ring and let $\Delta$ be an
effective Mumford $\Q$-divisor on
$R$. Choose $m>0$ such that
$D_m\coloneqq m\Delta$ is integral.
{Using prime avoidance, choose a nonzerodivisor $f\in R$ that vanishes to order one along every prime component of $D_m$ and does not vanish along any codimension-one component of the singular locus. Then $\Div_R(f)=(D_m)_{\mathrm{red}}+G$ for an effective Mumford divisor $G$ having no common component with $D_m$.} Set $S_{f,m}\coloneqq R[x]/(x^m-f)$,
which is finite free, hence flat, over $R$.
Note that $x$ is a nonzerodivisor since $x^m=f$.

We claim that there is a unique effective Mumford divisor
$\Delta_{f,m}$ on $S_{f,m}$ whose $m$-fold multiple is the flat pullback of $D_m$. Indeed, write $D_m=\sum_i d_iP_i$ where we identify $P_i$ with the associated height-1 prime of $R$. For each $P_i$, the choice of $f$ gives $f=u\pi$ in the DVR $R_{P_i}$, where $u\in R_{P_i}^{\times}$ and $\pi$ is a uniformizer. Thus $x^m-u\pi$ is Eisenstein, so the quotient is a domain. In fact, $R_{P_i}[x]/(x^m-u\pi)$ is a DVR with uniformizer $x$, and there is a unique height one prime $Q_i$ of $S_{f,m}$ over $P_i$. One can then define $\Delta_{f,m}\coloneqq \sum_i d_i Q_i$, which is Mumford. Note that its $m$-fold multiple is indeed the flat pullback of $D_m$ since at $Q_i$, the pullback of a local equation $\pi^{d_i}$ is $x^{md_i}$ up to a unit. To see uniqueness, note that at $Q_i$, this calculation forces the coefficient to be $d_i$ while at every point of height $1$ not in the pullback of the support of $\Delta$, effectivity and the vanishing of its $m$-fold multiple force a coefficient of $0$. We also note that $S_{f,m}$ need not be a domain or even reduced. However, if every irreducible component of $\Spec R$ contains a prime component of $D_m$ then $S_{f,m}$ will be reduced. 

\end{construction}

When $X$ is normal, Mumford divisors and Mumford $\Q$-divisors are
precisely the usual Weil divisors and Weil $\Q$-divisors,
respectively. Moreover, $D\mapsto\sI_D=\sO_X(-D)$ identifies the
usual Weil divisors with the Weil generalized divisors, compatibly
with addition, effectivity, Cartier and principal divisors, support,
and the sheaves $\sO_X(D)$. Under this identification, the
codimension-one envelope of a fractional ideal is its usual reflexive
hull (by \autoref{codimension one envelope fractional ideal}); on affine opens, it is equivalently the intersection of the
height-one localizations
\cite[Tag~0AVB]{stacks-project}.

{

\begin{lemma}\label{canonical module evaluation}
Let $(R,\m)$ be a complete normal local domain of dimension $d$,
fix an effective canonical divisor $K_R$, giving an embedding $R\subset R(K_R)\eqqcolon\omega_R$. Let $C$ be an $R$-algebra, let $q\in C$, and let
$\widetilde q\in(C\tensor_RR(-K_R))^{\mathrm{c1}}$ have the same image as $q$ in
$C^{\mathrm{c1}}$. Then
\[
\im\left(
\Hom_R\!\left((C\tensor_RR(-K_R))^{\mathrm{c1}},R\right)
\xrightarrow{\ \operatorname{ev}_{\widetilde q}\ }R
\right)
=
\Ann_{\omega_R}\ker\left(
H^d_\m(R)\xrightarrow{\ \cdot q\ }H^d_\m(C)
\right)
\subseteq R\subseteq\omega_R.
\]
\end{lemma}

\begin{proof}
The codimension-two extension property and Hom-tensor adjunction give
\[
\Hom_R\!\left((C\tensor_RR(-K_R))^{\mathrm{c1}},R\right)
\cong\Hom_R(C,\omega_R).
\]
Suppose that $\psi$ corresponds to $\varphi$. At a height-one prime
$P$, choose a generator $s$ of $R(-K_R)_P$ and write
$\widetilde q_P=t\tensor s$. Since $\widetilde q$ maps to $q$, one has
$q=st$ in $C_P$, and adjunction gives
$\psi(\widetilde q)_P=s\varphi(t)_P=\varphi(q)_P$. Since both sides are
sections of the reflexive module $\omega_R$, this height-one equality
holds globally. Thus evaluation at $\widetilde q$ corresponds to
evaluation at $q$, which under local duality is the Matlis dual of
$H^d_\m(R)\xrightarrow{\ \cdot q\ }H^d_\m(C)$. Exactness of Matlis
duality gives the asserted equality.
\end{proof}
}

\subsection{Pure and split maps}

\begin{lemma}\label{check purity on codim 1}
Let $R$ be a Noetherian
$\mathbf{S}_2$ ring. Let $M$ be an $R$-module and
$U\subset \Spec R$ an open set whose complement has codimension
$\geq 2$. Then, a fixed map $R\to M$ is pure if and only if the induced
$R\to \Gamma(U,\widetilde M)$ is pure.
\end{lemma}

\begin{proof}
Since $U$ is quasi-compact and quasi-separated, taking sections over
$U$ is additive and commutes with filtered colimits. By
\cite[Tag 058K]{stacks-project}, pure exact sequences of modules are
filtered colimits of split exact sequences, so taking sections over
$U$ preserves pure morphisms. By
the $\mathbf{S}_2$ property
and the codimension-two extension principle,
$\Gamma(U,\sO_U)=R$. Thus, if $R\to M$ is pure, then so is
$R\to\Gamma(U,\widetilde M)$. Conversely, the latter map factors
through $R\to M$. If it is pure, then $R\to M$ is pure: after tensoring
with any $R$-module, the map induced by $R\to M$ is the first map in
an injective composite.
\end{proof}

\begin{lemma}\label{purity and codimension one envelope}
Let $R$ be a Noetherian $\mathbf{S}_2$ ring.
Let $\eta\colon R\to N$ be a map of $R$-modules. Then $\eta$ is pure if
and only if the composite
\[
R\xrightarrow{\eta}N\longrightarrow N^{\mathrm{c1}}
\]
is pure.
\end{lemma}

\begin{proof}
If the composite is pure, then $\eta$ is pure because the composite
factors through $\eta$. Conversely, suppose that $\eta$ is pure.
By \autoref{check purity on codim 1}, the induced map
\[
R\longrightarrow
\Gamma(U,\widetilde N)
\]
is pure for every $U$ occurring in the definition of
$N^{\mathrm{c1}}$. Filtered colimits of pure maps are pure, which
proves the result.
\end{proof}

\begin{definition}
Let $A$ be a ring. A \emph{pointed} $A$-module is a pair consisting 
of an $A$-module $M$ and a distinguished element $\mathbf 1_M \in M$, 
or, equivalently, a homomorphism $A \to M$. 
If $B$ is a ring and $N$ is a pointed $B$-module 
then a homomorphism $(A, M) \to (B, N)$
is a homomorphism of rings $f \colon A \to B$
and a homomorphism of $A$-modules $\phi \colon M \to f_* N$
such that $\phi (\mathbf 1_M) = \mathbf 1_N$.    
\end{definition}

\begin{definition}
\label{def: non-pure ideals}
    Let $R$ be a ring, $A$ be an $R$-algebra, and $M$ be a pointed $A$-module. We use $\np^R (A, M)$ to denote the set
    \(
    \{x \in A \;|\; R \xrightarrow{1 \mapsto x\cdot \mathbf 1_M} M \text{ is not pure} \}.     
    \)

    When $R$ is clear in the context, we will use the shorthand $\np (A, M)$.
    If $M = A$ we will further shorten to $\np^R (A)$ or $\np (A)$. Note that these loci need not be ideals in general. However, when $R$ is Noetherian local, they are ideals by \autoref{lem: injective hull splitting ideal} below and will be referred to as \emph{non-pure ideals}.
\end{definition}

\begin{lemma}\label{lem: injective hull splitting ideal}
Let $R$ be a Noetherian ring and $M$ be an $R$-module.
For any $x\in M$, the map $R \to M, 1 \mapsto x$, is pure if and only if the map
\[
E_R(R/\m)\to E_R(R/\m)\tensor_RM, \quad u\mapsto u\tensor x,
\]
is injective for every maximal ideal $\m$ of $R$. 

Consequently, if $(R,\m)$ is local, $A$ is an $R$-algebra, $M$ is a pointed $A$-module,  and $u$ is a nonzero socle element of $E(R/\m)$, then
\[
\np^R(A, M)=\ker\{ A \to E(R/\m) \tensor_R M, x\mapsto u\tensor x\mathbf 1_M\}
= \Ann_A \left(u \otimes \mathbf 1_M \text{ in } E(R/\m) \tensor_R M\right).
\]
In particular,
$\np^R(A, M)$ is an ideal of $A$ containing $\m A$.
\end{lemma}
\begin{proof}
Purity may be checked locally, and the first assertion
follows from \cite[Proposition~6.11]{HochsterRobertsRingsOfInvariants}.
In the local case, every nonzero submodule of $E$ meets its socle, so
injectivity may be tested on $u$. This gives the kernel description,
from which the final assertions follow.
\end{proof}

We will use a generalization to our setting of the characterization of 
purity for approximately Gorenstein rings originating from \cite{HochsterPurity}. Note that every Noetherian local ring whose completion is reduced is approximately Gorenstein by \cite{HochsterPurity}.

\begin{lemma}
    \label{lem.SplittingApproxGorenstein}
        Let $(R,\m)$ be a Noetherian local ring that is approximately Gorenstein. Let $A$ be an $R$-algebra and $M$ be a pointed $A$-module. 
        Then $\np^R (A, M) = \bigcup_t (I_tM :_{A} u_t \mathbf 1_M),$
        where $I_t$ is a system of Gorenstein ideals cofinal with $\m^t$ and $u_t$ is a socle representative of $R/I_t$.
    \end{lemma}
    \begin{proof}
    Since $R$ is approximately Gorenstein, the injective hull $E$ of the residue field $k=R/\fram$ can be computed as:
\[
E=\bigcup_t \Ann_E{I_t}\cong \varinjlim_{t} E_{R/I_t} \cong \varinjlim_{t} R/I_t.
\]
It follows from \autoref{lem: injective hull splitting ideal} that $R\to M$ is pure if and only if $R/I_t \to M/I_tM$ is injective for all $t$, which happens if and only if the image of $u_t$ in $M/I_tM$ is nonzero for all $t$. Thus $x\in \np^R (A, M)$ if and only if $u_tx\cdot \mathbf 1_M\in I_tM$ for some $t$, i.e., $x\in \bigcup_t (I_tM :_{A} u_t\mathbf 1_M)$. 
    \end{proof}

\begin{lemma}\label{pure + p-comp pure is pure}
Let $R$ be a Noetherian ring with Jacobson radical $J$. Suppose that 
$\phi \colon M \to N$ is such that the induced map $M/J^nM\to N/J^nN$ is pure for every $n>0$ . If $f \colon R \to M$ is a pure map, then the composition $R\xrightarrow{f} M \xrightarrow{\phi} N$ is pure. 

In particular, the same conclusion holds if $A = M \to N = B$  is a $J$-completely faithfully flat map of $R$-algebras (e.g., a $g$-completely faithfully flat map of $R$-algebras for some $g\in J$).  
\end{lemma}

\begin{proof}
Let $x = f(1) \in M$.
By Lemma~\ref{lem: injective hull splitting ideal} it suffices to assume that $R$ is local, with $E \coloneqq E(R/\m)$, and show that 
$E \to E \otimes_R N$, $u \mapsto u \otimes \phi (x)$ is injective. 
By taking direct limit, it suffices to show that $K \to K \otimes_R N$
is injective for every finitely generated submodule $K \subseteq E$. 
But then $J^n K = 0$ for some $n \in \N$, so 
$K \otimes_R N \cong K \otimes_{R/J^n R} N/J^n N$
so the map $K \to K \otimes_R N$ is injective as a composition of injective maps.
\end{proof}    

\begin{lemma}\label{lem: functoriality of non-pure ideals}
    Let $S$ be a ring and $\phi \colon (A, M) \to (B, N)$ be a map between $S$-algebras and pointed modules over these algebras. Then
    \begin{enumerate}
        \item If $f \colon R \to S$ is a homomorphism of rings 
        which is pure as a map of $R$-modules (e.g., $\np^R(A, M)\neq A$), then  $\np^R (A, M) \subseteq \np^S (A, M)$.
        \label{part: ring functoriality}
        \item If $f \colon (R, \m) \to (S, \n)$  
        is a flat map of Noetherian local rings such that $\m S  = \n$, then $\np^R (A, M) = \np^S (A, M)$. 
        \label{part: invariance of field extension}
        \item $\np^S (A, M) \subseteq \np^S (B, N) \cap A$. 
        \label{part algebra functoriality}
        \item If $\{M_i\}_{i \in \N}$ is a directed system of pointed $A$-modules, then $\np^S (A, M_i)$ form an ascending sequence of sets and 
        $\np^S (A, \displaystyle{\lim_{\to}} M_i) = \cup_{i}  \np^S (A,  M_i)$.   \label{part direct limit of modules}    
        \item If $M \xrightarrow{\phi} N$ is pure as $S$-modules, then $\np^S (A, M) = A \cap \np^S (B, N)$. The same conclusion holds
        if $S$ is Noetherian and $J$ is its Jacobson radical such that the induced map $M/J^n M \to N/J^n N$ is pure for all $n \in \N$. \label{part p-comp pure} 
        \item If $(S, \m)$ is Noetherian local, $\phi \colon A \to B$ is 
        $\m$-completely faithfully flat and $N = M \otimes_A B$, then $\np^S (A, M)B = \np^S (B, N)$. 
        \label{part extension for p completely flat}
        \item If $S$ is Noetherian and $\mathbf{S}_2$, then $\np^S (A, M) = \np^S (A, M^{\mathrm{c1}})$. \label{part envelope}
    \end{enumerate}
\end{lemma}
\begin{proof}
    First, if $\np^R (A, M) \neq A$ then there is $x \in A$ such that
    the composition $R \xrightarrow{f} S \xrightarrow{1 \mapsto x\mathbf 1_M} M$ is pure, so $f$ must be pure. 
    Now if $f$ is pure and $x \notin \np^S (A, M)$, then 
    $R \xrightarrow{f} S \xrightarrow{1 \mapsto x\mathbf 1_M} M$
    must be pure as a composition of pure maps, so $x \notin \np^R (A, M)$.

    Lemma~\ref{lem: injective hull splitting ideal} implies \autoref{part: invariance of field extension} after noting that 
    $E (S/\n) = E(R/\m) \otimes_R S$ by the assumptions. 

    Next, take any $x \in A$. If $\phi(x) \notin \np^S (B, N)$
    then, by definition, $S \xrightarrow{1 \mapsto \phi(x)\mathbf 1_N} N$
    is a pure map that factors through $S \xrightarrow{1 \mapsto x \mathbf 1_M} M$. Hence, $x \notin \np^S (A, M)$, showing \autoref{part algebra functoriality}. Now, \autoref{part direct limit of modules} is immediate from \autoref{part algebra functoriality}.

    We now prove \autoref{part p-comp pure}.
    Suppose $x \in A \setminus \np^S (A, M)$.
    By \autoref{part algebra functoriality}, it suffices to show that $\phi(x) \notin \np^S (B, N)$,
    i.e., $S \xrightarrow{\phi(x) \mathbf 1_N} N$ is pure. 
    This is immediate if $M \to N$ is pure and follows 
    from \autoref{pure + p-comp pure is pure} in the more general case.

    For \autoref{part extension for p completely flat}, we note that, 
    since $B/\m^nB$ is faithfully flat over $A/\m^nA$ for every $n$, tensoring with $B$ preserves injections of $\m$-power-torsion $A$-modules.
    Applying this to the injection $A/\np^S (A, M) \cong A (u \otimes \mathbf 1_M) \hookrightarrow E \otimes_S M$ in \autoref{lem: injective hull splitting ideal} shows that the annihilator 
    of $u \otimes \mathbf 1_N$ is $\np^S (A, M)B$. 

    The last claim follows from \autoref{purity and codimension one envelope}.
\end{proof}

\subsection{Compatible ideals}

\begin{definition}[{\cf \cite{FayolleCentersPerfectoidPurity}}]
Let $R$ be a Noetherian ring with $p$ in its Jacobson radical and let $I\subset R$ be an ideal. For a perfectoid $R$-algebra $B$ and $\phi\in \Hom_R(B,R)$, we say that $I$ is $\phi$-compatible if $\phi((IB)_\perfd)\subset I$. 
Similarly, if $\sC\subset \Hom_R(B,R)$ is an $R$-submodule then we say that $I$ is $\sC$-compatible if $\phi((IB)_\perfd)\subset I$ for all $\phi\in \sC$. If $I$ is $\phi$-compatible for all $\phi\in \Hom_R(B,R)$ and all perfectoid $R$-algebras $B$, we say that it is uniformly compatible.
\end{definition}

\begin{remark}
Unless $R$ is complete local, the existence of a nonzero map $\phi\in \Hom_R(B,R)$ is not guaranteed, making the definition useful mostly in the complete local case (this is why the theory was only written for complete local rings in \cite{FayolleCentersPerfectoidPurity}). Indeed, if $(R,\fram,k)$ is complete local and $R\to B$ is pure then  
\[
\Hom_R(B,R)\cong \Hom_R(B,\Hom_R(E,E))\cong \Hom_R(B\tensor_R E, E)\twoheadrightarrow \Hom_R(E,E)=R
\]
where $E=E_R(k)$ and the surjection comes from the injection $E\to E\tensor_R B$ and the fact that $E$ is an injective module. In our applications below, if $R$ is not complete local, we assume a map exists.
\end{remark}

\begin{proposition}\label{basic properties comp ideals}
Let $R$ be a Noetherian ring with $p$ in its Jacobson radical, $B$ a perfectoid $R$-algebra and let $I\subset R$ be an ideal. Fix an $R$-submodule $\sC$ of $\Hom_R(B,R)$.
\begin{enumerate}
    \item $I$ is $\sC$-compatible if and only if $\sC$ descends to a corresponding $R/I$-submodule $\overline\sC$ of $\Hom_{R/I}(B/(IB)_\perfd,R/I)$. If some $\phi\in \sC$ is surjective then so is the induced map on the quotient $\overline{\phi}\in \overline{\sC}$.
    \item If $I$ is $\phi$-compatible for some surjective $\phi$, then it is radical. 
    \item If $I$ is $\sC$-compatible and $J\subset R$ is any ideal such that $B$ admits a compatible system of $p$-th power roots for generators of $(I\colon J)$, then  $(I\colon J)$ is also $\sC$-compatible. 
    \item If $R$ is complete local and perfectoid pure then any uniformly compatible ideal is radical. 
    \item If $R$ is complete local then uniform compatibility can be checked on a single choice of $\raftypfd$ for $A\to R$ as in \autoref{notation RAinfty}.
    \item If $R$ is normal, complete local and $I$ is a radical $\sC$-compatible ideal, then so is $I^{**}$.
\end{enumerate}
\end{proposition}
\begin{proof}
Note that $I$ is $\phi$-compatible for $\phi\in \sC$ if and only if the composition map 
\[
B\xrightarrow{\phi} R \longrightarrow R/I
\]
factors through $B/(IB)_\perfd$ so (a) follows.
In (b), (a) implies that $\phi$ descends to a surjective map $\overline\phi\in \Hom_R(B/(IB)_\perfd,R/I)$ so it follows that $R/I$ injects into the perfectoid hence reduced ring $B/(IB)_\perfd$ and therefore $R/I$ is reduced.
For (c), by linearity and Krull's intersection theorem in the form \cite[Theorem 8.10]{MatsumuraCommutativeRingTheory}, it suffices to show that $\phi(x^{1/p^e}B)j\in I$ for any $\phi\in \sC$, $j\in J$ and generator $x$ of $(I\colon J)$ having a $p^e$-th root in $B$. Then, 
\[
\phi(x^{1/p^e}B)j=\phi((xj^{p^e})^{1/p^e}B)\subset \phi(\sqrt{IB})\subset I
\]
where the last containment follows from the $\phi$-compatibility of $I$ and \autoref{pfd of ideal only depends on radical}.
Item (d) is \cite[Corollary 2.33]{FayolleCentersPerfectoidPurity} and item (e) is \cite[Corollary 2.29]{FayolleCentersPerfectoidPurity}. For item (f), since $R$ is normal, $I^{**}=\bigcap_{\substack{\mathfrak p\in\operatorname{Min}(I)\\
\operatorname{ht}\mathfrak p=1}}\p$. By (a), each $\phi \in \sC$ induces an $R/I$-linear map
$B/(IB)_\perfd\to R/I$, and from \cite[Corollary 2.36 and Proposition 2.31]{FayolleCentersPerfectoidPurity}
every minimal prime of $R/I$ is compatible with these induced maps, so the arguments of \cite[Propositions 2.34, 2.37]{FayolleCentersPerfectoidPurity} give the desired result.
\end{proof}

\section{Perfectoid pure principal pairs}
\label{sec.PerfectoidPurePrincipal}

In this section we develop the theory of perfectoid pure pairs of the form $(R, f^\lambda)$ where $f \in R$ and $\lambda \in \bQ_{> 0}$.  Later, in \autoref{sec.PerfdPureDivisorPairs} we will generalize this to pairs $(R, \Delta)$ where $\Delta$ is a $\bQ$-divisor.  We would recommend the reader to look at this case first before exploring the more technical divisor-pair setting.

We will give our definition of a perfectoid pure pair $(R, f^t)$ shortly, but first we prove some small results to show it is meaningful.

\begin{proposition}[{cf. \cite[Lemma 4.5]{BMPSTWW3}}]
    \label{can choose B integrally closed} 
    {Let $R$ be a Noetherian ring with $p$ in its Jacobson
    radical, let $f\in R$, and let $a,b\in\Z_{>0}$. Suppose that $B$ is a
    perfectoid $R$-algebra and $Y\in(fB)_\perfd$ satisfies $Y^b=f^a$.
    Let $B\to\widetilde B$ be a $p$-completely faithfully flat map with
    $\widetilde B$ absolutely integrally closed and perfectoid, as in
    \autoref{thm: Andre}. If $\widetilde Y$ denotes the image of $Y$, then
    $R\xrightarrow{1\mapsto Y}(fB)_\perfd$ is pure if and only if
    $R\xrightarrow{1\mapsto\widetilde Y}(f\widetilde B)_\perfd$ is pure.}
\end{proposition}

\begin{proof}    The reverse implication is
immediate because the second pointed map factors through the first. For the forward implication,
{write $\mathbin{\widehat{\otimes}}^{\mathbb L}$ for
the {derived $p$-completion of the derived tensor product}
Since
$B\to\widetilde B$ is $p$-completely faithfully flat,
we have { isomorphisms of $\widetilde B$-algebras}
\[
\bigl(B/(fB)_\perfd\bigr)
\mathbin{\widehat{\otimes}}_{B}^{\mathbb L}\widetilde B
\simeq \bigl((B/fB)\mathbin{\widehat{\otimes}}_{B}^{\mathbb L}
\widetilde B\bigr)_\perfd
\simeq (\widetilde B/f\widetilde B)_\perfd
=\widetilde B/(f\widetilde B)_\perfd
\]
using \cite[Propositions~8.5 and~8.13]{BhattScholzePrisms} (as in the displayed calculation in the proof of
\cite[Lemma~2.19]{CaiLeeMaSchwedeTuckerPerfectoidSignature}).
Taking their fibers and using the exactness of
$\mathbin{\widehat{\otimes}}^{\mathbb L} \widetilde{B}$ gives
$(fB)_\perfd\mathbin{\widehat{\otimes}}_{B}^{\mathbb L}\widetilde B
\simeq (f\widetilde B)_\perfd$.
Applying
$-\tensor_{\widetilde B}^{\mathbb L}\widetilde B/p^n\widetilde B$
to this isomorphism removes the completion on the tensor product, since derived $p$-completion does not change the derived tensor
product with a module killed by $p^n$. 
Thus
\[
(f\widetilde B)_\perfd\tensor_{\widetilde B}^{\mathbb L}
\widetilde B/p^n\widetilde B
\simeq
\bigl((fB)_\perfd\tensor_B^{\mathbb L}B/p^nB\bigr)
\tensor_{B/p^nB}^{\mathbb L}\widetilde B/p^n\widetilde B.
\]
Since $\widetilde B/p^n\widetilde B$ is faithfully flat over $B/p^nB$,
taking zeroth cohomology gives
\[
(f\widetilde B)_\perfd/p^n(f\widetilde B)_\perfd
\cong
\bigl((fB)_\perfd/p^n(fB)_\perfd\bigr)
\otimes_{B/p^nB}\widetilde B/p^n\widetilde B.
\]
It follows that
$(fB)_\perfd/p^n(fB)_\perfd\to
(f\widetilde B)_\perfd/p^n(f\widetilde B)_\perfd$ is pure for every
$n>0$. The forward implication now follows from
\autoref{pure + p-comp pure is pure}.}
\end{proof}

\begin{lemma}\label{rk limit of pure is pure}
{Let $R$ be a Noetherian ring with $p$ in its Jacobson
radical, let $f\in R$, and $a,b\in\N_{>0}$. Suppose that $B$ is {a perfectoid $R$-algebra and that
$Y\in(fB)_\perfd$ satisfies $Y^b=f^a$ and admits a compatible system
$\{Y^{1/p^e}\}_{e\in\N_{>0}}$ of $p$-power roots in $B$.} Then the map
$R\xrightarrow{1\mapsto Y}(fB)_\perfd$ is pure if and only if
{$R\xrightarrow{1\mapsto Y^{1-1/p^e}}B$ is pure
for every $e\in\N$, or equivalently for every $e\gg0$.}}

Moreover, if $B$ is absolutely integrally closed, then $R\xrightarrow{1\mapsto Y}(fB)_\perfd$ is pure if and only if for all $c, d \in \N_{>0}$ with $\frac{c}{d}<\frac{a}{b}$ and any $Z\in B$ such that $Z^d=f^c$, the map $R\xrightarrow{1\mapsto Z}B$ is pure.
\end{lemma}

\begin{proof}
{Since $Y^b=f^a$, \autoref{BS - root} \autoref{BS - root.RadicalComputation} gives $(Y^{1/p^\infty})B=\sqrt{YB}=\sqrt{fB}$. Combining \autoref{lem: injective hull splitting ideal} and \autoref{root ideal tensor comparison} shows that $R\to(fB)_\perfd=(YB)_\perfd$, $1\mapsto Y$ is pure if and only if $R\to(Y^{1/p^\infty})B$, $1\mapsto Y$ is pure. \autoref{BS - root} \autoref{BS - root.RadicalComputation} shows that this latter map is the colimit of $R\xrightarrow{1\mapsto Y^{1-1/p^e}}B$. These maps are all pure if and only if their colimit is pure.}

For the moreover part, the ``if" direction follows from the previous paragraph (and we only need to use $\frac{c}{d}=\frac{a(p^e-1)}{bp^e}$ and a particular choice of $Z$, namely $Y^{1-1/p^e}$). For the ``only if" direction, note that 
$$Z^{ad(p^e-1)} = f^{ac(p^e-1)} = (Y^{1-1/p^e})^{bcp^e}.$$
Since $ad>bc$, we have $ad(p^e-1) > bcp^e$ for $e\gg0$. Thus by \autoref{BS - root} \autoref{BS - root.baseCase}, we see that $Y^{1-1/p^e} \in ZB$. It follows from the already established purity of $R\xrightarrow{1\mapsto Y^{1-1/p^e}}B$ that $R\xrightarrow{1\mapsto Z}B$ is also pure.
\end{proof}

\begin{proposition}\label{purity independent choice of root}
Let $R$ be a Noetherian ring with $p$ in its Jacobson radical, and let $f\in R$. Let $a,b\in \Z_{>0}$. Let $B$ be a perfectoid $R$-algebra and suppose that there are $Y,Z\in B$ with $Y^b=Z^b=f^a$. If
$R \xrightarrow[]{1 \mapsto Y} (fB)_\perfd$
is pure, then so is
$R \xrightarrow[]{1 \mapsto Z} (fB)_\perfd$.
\end{proposition}

\begin{proof}
{By \autoref{can choose B integrally closed}, we may
assume that $B$ is absolutely integrally closed. Choose compatible
systems of $p$-power roots of $Y$ and $Z$. By
\autoref{rk limit of pure is pure},
$R\xrightarrow{1\mapsto Y^{1-1/p^e}}B$ is pure for every
$e\gg0$.}
{For $e\gg0$, the equality
\((Y^{1-1/p^{e+1}})^{bp^{e+1}(p^e-1)}
=(Z^{1-1/p^e})^{bp^e(p^{e+1}-1)}\)
and the inequality
$bp^{e+1}(p^e-1)<bp^e(p^{e+1}-1)$ show by
\autoref{BS - root} \autoref{BS - root.baseCase} that
$Y^{1-1/p^{e+1}}\in Z^{1-1/p^e}B$. Thus, the pure map
$R\xrightarrow{1\mapsto Y^{1-1/p^{e+1}}}B$ factors through
$R\xrightarrow{1\mapsto Z^{1-1/p^e}}B$, so the latter is pure.}
{Another application of
\autoref{rk limit of pure is pure} proves that
$R\xrightarrow{1\mapsto Z}(fB)_\perfd$ is pure.}
\end{proof}

\begin{definition}\label{definition pfd pure pair}
Let $R$ be a Noetherian ring with $p$ in its Jacobson radical, and let $f\in R$. Let $\lambda=\frac{a}{b}$ with $a,b\in\N_{>0}$. We say the pair $(R,f^\lambda)$ is \emph{perfectoid pure} if {$f$ is a nonzerodivisor and} there exists a perfectoid $R$-algebra $B$ together with $Y\in (fB)_\perfd$ such that $Y^b=f^a$ and the map
$R\xrightarrow{1\mapsto Y}(fB)_\perfd$
is pure.
\end{definition}

\begin{remark}\label{principal pair choices}
First note if $(R,f^\lambda)$ is perfectoid pure then $R$ itself is perfectoid pure (see \autoref{rk limit of pure is pure}). 
{
\begin{enumerate}
    \item \label{principal pair choices.IndependentOfYLambda}
        By \autoref{can choose B integrally closed} and
        \autoref{purity independent choice of root}, we may take $B$ to
        be absolutely integrally closed.  Furthermore, the definition is independent
        of the chosen $Y$ and the representation $\lambda=a/b$.
        To see the latter, pass to a common denominator.
    \item \label{unit invariance}Purity is not affected if $f$ is replaced by $uf$ for
        a unit $u\in R$: after taking $B$ absolutely integrally closed,
        multiplication by a $b$-th root of $u^a$ carries a witness $Y$ for
        $f^{\lambda}$ to one for $(uf)^{\lambda}$. Thus we may also speak of the perfectoid
        purity of $(R,(f)^\lambda)$ or $(R,\lambda\Div_R(f))$. More
        general divisor pairs will be considered below in \autoref{sec.PerfdPureDivisorPairs}.
    \item By \autoref{rk limit of pure is pure}, after choosing
        compatible $p$-power roots of $Y$ in an absolutely integrally
        closed $B$ as in \autoref{principal pair choices.IndependentOfYLambda}, the condition is equivalent to the purity of
        $R\xrightarrow{1\mapsto Y^{1-1/p^e}}B$ for every $e\gg0$.
        Equivalently, $R$ is perfectoid pure along these elements in the
        sense of \cite[Definition~3.1]{FayolleCentersPerfectoidPurity}.
\end{enumerate}}
\end{remark}

\begin{lemma}
\label{principal pair sequence criterion}
Let $R$ be a Noetherian ring with $p$ in its Jacobson
radical and let $f\in R$ be a nonzerodivisor. The following are equivalent:
\begin{enumerate}
    \item $(R,f^\lambda)$ is perfectoid pure;
    \item $(R,f^\mu)$ is perfectoid pure for all positive rational $\mu<\lambda$;
    \item for all positive rational $\mu<\lambda$, there is a perfectoid R-algebra $B_\mu$ and $Y\in B_\mu$ with $Y=f^\mu$ so that $R\xrightarrow{1\mapsto Y} B_\mu$ is pure;
    \item there is a single perfectoid $R$-algebra $B$, a sequence $0<\lambda_i=a_i/b_i<\lambda$ of rational numbers converging to $\lambda$, and elements $Y_i\in B$ with $Y_i^{b_i}=f^{a_i}$ such that $R\xrightarrow{1\mapsto Y_i}B$ is pure for every $i$. 
\end{enumerate}
\end{lemma}
\begin{proof}
First $(a)\Leftrightarrow(c)$ follows from \autoref{rk limit of pure is pure} (we can take $B$ to be absolutely integrally closed by \autoref{principal pair choices} \autoref{principal pair choices.IndependentOfYLambda}, and conversely we can take $B_\mu$ to be absolutely integrally closed and set $B=\prod B_\mu$ which is perfectoid). Next $(c) \Rightarrow (b)$ is obvious as the map factors through $(fB_\mu)_\perfd$ and $(b)\Rightarrow(c)$ follows from $(a)\Rightarrow(c)$ applied to any $\lambda'$ so that $\mu<\lambda'<\lambda$. $(c)\Rightarrow (d)$ is obvious by noting that we can assume each $B_i$ is absolutely integrally closed and set $B=\prod B_i$. Finally, $(d)$ obviously implies $(R, f^{\lambda_i})$ is perfectoid pure (as the map factors through $(fB)_\perfd$) and thus implies $(b)$ (by applying $(a)\Rightarrow(b)$ to any $\mu<\lambda_i<\lambda$). 
\end{proof}

\subsection{Localization, completion and choice of $B$ for principal pairs}\label{sub:localization_completion_principal}

In this subsection we verify the basic properties of perfectoid purity with regards to localization and completion. 

\begin{proposition}\label{prop:localization_ideal_pairs}
Let $R$ be a Noetherian ring with $p$ in its Jacobson radical, and let $\mathfrak q$ be a prime ideal containing $p$.
Let $f\in R$ be a nonzerodivisor and $\lambda\in\Q_{>0}$. 
If $(R,f^\lambda)$ is perfectoid pure then  $(R_{\mathfrak q}, f^\lambda)$ is perfectoid pure.
\end{proposition}
\begin{proof}
If $B$ is a perfectoid $R$-algebra, set
$B_{\mathfrak q}=B\tensor_RR_{\mathfrak q}$ and $C =B_{\mathfrak q}^{\wedge p}$.
The ring $C$ is perfectoid by \cite[Example~3.8(7)]{BhattIyengarMaRegularRings} and $B_{\mathfrak q}\to C$ is an isomorphism modulo $p^n$ for every $n$.
Consequently, $B_\mathfrak{q} \to C$ is $p$-completely pure.

Suppose now that $(R,f^\lambda)$ is perfectoid pure. By
\autoref{principal pair sequence criterion}, there exist some
perfectoid $R$-algebra $B$, rational numbers
$\lambda_e=a_e/b_e<\lambda$ converging to $\lambda$, and elements
$Y_e\in B$ such that $Y_e^{b_e}=f^{a_e}$ and
$R\to B$, $1\mapsto Y_e$, is pure for every $e$. 
By \autoref{pure + p-comp pure is pure}, $R_{\mathfrak q} \to C$, $1\mapsto Y_e$, is pure for every $e$. Thus
$(R_{\mathfrak q},f^\lambda)$ is perfectoid pure by the same criterion.  
\end{proof}

\begin{proposition}[{\cf \cite[Lemma 4.10]{BMPSTWW3}}]\label{can check purity locally ideals}
Let $R$ be a Noetherian ring with $p$ in its Jacobson radical, let
$f\in R$ be a nonzerodivisor, and let $\lambda\in\Q_{>0}$. Then
$(R,f^\lambda)$ is perfectoid pure if and only if
$(R_\fram,f^\lambda)$ is perfectoid pure for every maximal ideal
$\fram$ of $R$.

\end{proposition}

\begin{proof}
The forward direction follows from \autoref{prop:localization_ideal_pairs}.
For the converse, suppose first that $(R_\fram,f^\lambda)$ is perfectoid pure for every maximal ideal $\fram$, and write $\lambda=a/b$ for some $a,b\in \N_{>0}$. By \autoref{can choose B integrally closed}, for each maximal ideal $\fram$ there is an absolutely integrally closed perfectoid $R_\fram$-algebra $B(\fram)$ as well as $Y_\fram\in B(\fram)$ with $Y_\fram^b=f^a$ such that 
\[R_\fram\longrightarrow(fB(\fram))_\perfd,\qquad 1\longmapsto Y_\fram,\] is pure. Choose compatible $p$-power roots of every $Y_\fram$. By \autoref{rk limit of pure is pure}, the maps
\begin{equation}\label{eqn purity locally}
R_\fram\longrightarrow B(\fram),
\qquad 1\longmapsto Y_\fram^{1-1/p^{e}},
\end{equation}
are pure for every $e\geq1$. Set $B=\prod_\fram B(\fram)$ and
$Z_e=(Y_\fram^{1-1/p^{e}})_\fram\in B$. The ring $B$ is perfectoid by
\cite[Example~3.8(8)]{BhattIyengarMaRegularRings}, and
\[
Z_e^{bp^e}=f^{a(p^e-1)}.
\]
Localizing at any $\fram$, projection onto $B(\fram)$ gives the pure map of \autoref{eqn purity locally} so this localized map is pure.
Since purity can be checked locally on $R$, the map $R\to B$, $1\mapsto Z_e$, is also pure. By, \autoref{principal pair sequence criterion} $(R,f^\lambda)$ is perfectoid pure.
\end{proof}

\begin{proposition}\label{pure iff completion pure principal case}
      Let $(R,\fram)$ be a Noetherian local ring of residue characteristic $p>0$. Let $f\in R$ be a nonzerodivisor and let $\lambda\in \Q_{>0}$. Then, $(R,f^\lambda)$ is perfectoid pure if and only if $(\widehat{R},f^\lambda)$ is perfectoid pure for
    $\widehat R=R^{\wedge\fram}$. 
\end{proposition}
\begin{proof}
    This follows from \autoref{principal pair choices}~(c) and \cite[Proof of proposition 3.4]{FayolleCentersPerfectoidPurity} (\cf \cite[Lemma 4.8]{BMPSTWW3}).
\end{proof}

Now we show, at least if $R$ is complete, that we can choose a $B$ determining perfectoid purity of a pair quite explicitly, similarly to \cite[Lemma 4.24]{BMPSTWW3}.

\begin{proposition}\label{can test principal pairs on RAinfty}
Let $(R,\fram,k)$ be a complete Noetherian local ring of residue
characteristic $p>0$, let $f\in\fram$ be a nonzerodivisor, and let
$\lambda\in\Q_{>0}$. Fix any $\varphi\colon A\to R$ as in
\autoref{notation RAinfty} with $\varphi(x_1)=f$ (such a map may be
obtained by adjoining to a Cohen presentation or Noether--Cohen
normalization a variable mapping to $f$). Then $(R,f^\lambda)$ is
perfectoid pure if and only if
\begin{equation}\label{pfd purity principal on raftypfd}
R\longrightarrow \raftypfd,
\qquad
1\longmapsto(x_1^{1/p^e})^{\lfloor(p^e-1)\lambda\rfloor},
\end{equation}
is pure for every $e\gg0$. 
\end{proposition}

\begin{proof}
For the backwards direction, note that $\lambda_e:= \lfloor(p^e-1)\lambda\rfloor/p^e <\lambda$ and $\lambda_e\to\lambda$ so the purity of \autoref{pfd purity principal on raftypfd} implies that $(R,f^\lambda)$ is
perfectoid pure by \autoref{principal pair sequence criterion}.
Conversely, suppose that $(R,f^\lambda)$ is perfectoid pure. Then by \autoref{rk limit of pure is pure} there is an absolutely integrally closed perfectoid $R$-algebra $B$ with a fixed compatible system of $p$-th power roots of $f$ such that $R\to B$, $1\mapsto (f^{1/p^e})^{\lfloor(p^e-1)\lambda\rfloor}$ is pure for all $e\gg0$.
Following the proof of \cite[Lemma~4.24]{BMPSTWW3}, choose a map $W(k^{1/p^\infty})\to B$ compatible with $C_k\to B$, together with compatible $p$-power roots in $B$ of $p$ and of the images of the variables of $A$. Replacing $B$ by its $\fram B$-adic completion, the universal property of perfectoidization then implies that $R\to B$, $1\mapsto (f^{1/p^e})^{\lfloor(p^e-1)\lambda\rfloor}$ factors through \autoref{pfd purity principal on raftypfd}, which must therefore also be pure.
\end{proof}

{
Finally, we compare perfectoid purity of principal pairs to the characteristic $p > 0$ definition that inspired it, see \cite{HaraWatanabe}.

\begin{proposition}\label{perfectoid purity in characteristic p principal}
Let $R$ be an excellent reduced ring of characteristic $p>0$, let $f\in R$
be a nonzerodivisor, and let $\lambda\in\Q_{>0}$. Then
$(R,f^\lambda)$ is perfectoid pure if and only if the map
\[
R\longrightarrow R^{1/p^e},
\qquad
1\longmapsto
f^{\lfloor\lambda(p^e-1)\rfloor/p^e}
\]
is pure for all $e\gg0$.

\end{proposition}

\begin{proof}
{ 
Using
\autoref{can check purity locally ideals} and
\autoref{pure iff completion pure principal case} above, we see both conditions are local and invariant under completion. We may therefore assume $(R,\fram,k)$ is complete local. If $f$ is a unit, this is the usual equivalence between perfectoid purity and $F$-purity \cite[Remark 4.3]{BMPSTWW3}, so further assume $f \in \fram$. Then for $A$  as in \autoref{can test principal pairs on RAinfty}, $\raftypfd=R_\perf$, so \autoref{can test principal pairs on RAinfty} tells us that $(R,f^\lambda)$ is perfectoid pure if and only if $R\to R_\perf$, $1\longmapsto
f^{\lfloor\lambda(p^e-1)\rfloor/p^e}$ is pure for all $e\gg 0$. If this is true then certainly $R\longrightarrow R^{1/p^e}$, $ 1\mapsto
f^{\lfloor\lambda(p^e-1)\rfloor/p^e}$ is pure for all $e$. 
Conversely, suppose that these maps to $R^{1/p^e}$ are pure for all $e\gg0$. Each factors through the $e$-iterated Frobenius map $F^e\colon R\to R^{1/p^e}$, so $F^e$ is pure and hence Frobenius $F$ is pure. The transition maps $R^{1/p^j}\to R^{1/p^{j+1}}$ are Frobenius twists of $F$ and hence are
pure. It follows that $R^{1/p^e}\to R_\perf=\varinjlim_jR^{1/p^j}$ is pure. Thus $R\to R_\perf$, $1\mapsto f^{\lfloor\lambda(p^e-1)\rfloor/p^e}$, is pure for all $e\gg0$, as desired.
}
\end{proof}

}

\subsection{Adjunction and inversion of adjunction for principal pairs}\label{sec. adjunction principal case}

{\color{black}
\begin{lemma}
\label{lem:coefficient-one-comparison}
Let $(R,\fram,k)$ be a Noetherian local ring of residue characteristic
$p$, let $f,g\in R$ with $f$ a nonzerodivisor, and let
$\lambda\in\Q_{>0}$.  The following are equivalent:
\begin{enumerate}[(a)]
\item $(R,fg^\lambda)$ is perfectoid pure.
\item There is a perfectoid $R$-algebra $B$ admitting a compatible system of $p$-th power roots of $f$ and elements $g_n$ with $g_n^{b_n}=g^{a_n}$ for some sequences $a_n, b_n\in \N_{>0}$ with $\frac{a_n}{b_n}<\lambda $, and $\frac{a_n}{b_n}\to \lambda$ such that
\begin{equation}\label{eqn adj characterization pdf pure}
fR\to (fB)_\perfd, \quad fr \mapsto frg_n
\end{equation}
is pure for all $n$.
\end{enumerate}
\end{lemma}

\begin{proof}
Assume (a) and let $B$ be an absolutely integrally closed perfectoid $R$-algebra witnessing the perfectoid purity of $(R,fg^\lambda)$. Let $e>0$ be arbitrary and fix a sequence $g_n$ in $B$ as in (b). By \autoref{principal pair sequence criterion}, for any $n$, $R\to B$, $1\mapsto f^{1-p^{-e}}g_n$ is pure. Since $e$ was arbitrary, taking the colimit over all $e$ while fixing $n$ and using \autoref{BS - root}\autoref{BS - root.RadicalComputation} and \autoref{root ideal tensor comparison}, we get that $R\to (fB)_\perfd$, $1\mapsto fg_n$ is pure. Part (b) then follows from the isomorphism $R\cong fR$, $1\mapsto f$.

Now, assume (b) holds. 
By \autoref{thm: Andre} and the argument
in the proof of \autoref{can choose B integrally closed},
we may assume $B$ is absolutely integrally closed.
Using the isomorphism $R\cong fR$, $1\mapsto f$, we know that $R\to (fB)_\perfd$, $1\mapsto fg_n$ is pure. This factors through $R\to (fgB)_\perfd$, $1\mapsto f^{(a_n/b_n)/\lambda}g_n$, so this latter map must be pure. Part (a) then follows from another application of \autoref{principal pair sequence criterion}.
\end{proof}
}

\begin{proposition}[{\cf\cite[Lemma 6.5]{BMPSTWW3}}]\label{quotient of pfd pure is pfd pure}
Suppose $(R,\fram, k)$ is a Noetherian local ring of residue characteristic $p$, $f$ is a nonzerodivisor on $R$, and $g \in R$. Denote the image of $g$ in $R/fR$ by $\overline{g}$. If $(R, fg^\lambda)$ is perfectoid pure, then so is $(R/fR,\overline{g}^\lambda)$.
\end{proposition}

\begin{proof}
Let $g_n$ and $B$ be as in \autoref{lem:coefficient-one-comparison}(b). Let $M$ be an $R/fR$-module. Let 
\[
\gamma_M \colon M\tensor_{R/fR} B/(fB)_\perfd\to M\tensor_R (fB)_\perfd, \quad m\tensor \overline{b}\mapsto m\tensor fb.
\]
This is well defined: if $c\in (fB)_\perfd$ then $m\tensor fc=mf\tensor c=0$. Since multiplication by $f$ induces an isomorphism $M\tensor_R fR\cong M$, the composition
\[
M\to M\tensor_{R/fR} B/(fB)_\perfd \xrightarrow{\gamma_M} M\tensor_R (fB)_\perfd, \quad m\mapsto m\tensor \overline{g_n} \mapsto m\tensor fg_n
\]
is a base change of \autoref{eqn adj characterization pdf pure} so it is injective and therefore so is 
\[M\to M\tensor_{R/fR} B/(fB)_\perfd, \quad  m\mapsto m\tensor \overline{g_n}.\] To conclude that $(R/fR,\overline{g}^\lambda)$ is perfectoid pure, we need to show that $\overline{g}$ is a nonzerodivisor on $R/fR$. Suppose not and let $r\in R\setminus fR$ be such that $rg\in fR$. Then, $(rg_n)^{b_n}=r^{b_n}g^{a_n}=0$ in the reduced ring $B/(fB)_\perfd$. Therefore we must have $rg_n\in (fB)_{\perfd}$, contradicting the injectivity of 
\[R/fR\to  B/(fB)_\perfd. \quad  r\mapsto rg_n.\]
\end{proof}

\begin{theorem}[{\cf \cite[Theorem 6.6]{BMPSTWW3}}]\label{quotient of ci pfd pure}
Suppose $(R,\fram, k)$ is a complete intersection of residue characteristic $p$ and let $f\in\fram$ be a nonzerodivisor and $g\in R$. Then, $(R, fg^\lambda)$ is perfectoid pure if and only if $(R/f, \overline{g}^{\lambda})$ is perfectoid pure.
\end{theorem}
\begin{proof}
The forward direction is \autoref{quotient of pfd pure is pfd pure} so we only need to show the backwards one.
By \autoref{pure iff completion pure principal case}, it suffices to show the theorem in the case $(R,\fram)$ is complete. By \autoref{principal pair choices}\autoref{unit invariance}, if $g$ is not in $\fram$ then we can assume $g=1$, in which case this is \cite[Theorem~6.6]{BMPSTWW3} so assume $g\in \fram$. We follow the proof of \cite[Theorem~6.6]{BMPSTWW3}.

Let $A\to R$ be as in \autoref{notation RAinfty} and by potentially adding variables to $A$ that get sent to $f$ and $g$, we can assume $\raftypfd$ has a fixed choice of compatible systems of $p$-th power roots for both $f$ and $g$. Let $g_e\coloneqq (g^{1/p^e})^{\lfloor (p^e-1)\lambda\rfloor}$. Let $d=\dim R$. By \cite[Theorem~4.25]{BMPSTWW3}, the following commutative diagram has exact rows. 
\[
\xymatrix{
0\ar[r]&H_\fram^{d-1}(R/fR)\ar[r]\ar[d]^{\cdot\overline{g_e}}&
H_\fram^d(fR)\ar[r]\ar[d]^{\cdot g_e}&
H_\fram^d(R)\ar[r]\ar[d]^{\cdot g_e}&0\\
0\ar[r]&H_\fram^{d-1}((R/fR)^{A_\infty}_\perfd)\ar[r]&
H_\fram^d((f\raftypfd)_\perfd)\ar[r]&H_\fram^d(\raftypfd)\ar[r]&0.
}
\]
By \autoref{can test principal pairs on RAinfty}, 
\[
R/fR \to (R/fR)^{A_\infty}_\perfd, \quad 1\mapsto \overline{g_e}
\]
is pure for all $e\gg 0$ so the left vertical map in the above diagram is injective. Since $R$ is Gorenstein and $f$ is a nonzerodivisor, $H^d_\fram(fR)\cong E_R(R/\fram)$. In fact, its socle is in the image of $H^{d-1}_\fram(R/fR)$ since multiplication by $f$ kills it. Chasing this socle shows that the middle vertical map is injective so 
\[
fR\tensor_R H^d_\fram(R) \to (f\raftypfd)_\perfd \tensor_R H^d_\fram(R), \quad f \tensor u\mapsto fg_e\tensor u
\]
is injective. The result then follows from combining $H^d_\fram(R)\cong E_R(R/\fram)$, \autoref{lem: injective hull splitting ideal} and \autoref{lem:coefficient-one-comparison}.
\end{proof}

\section{Perfectoid pure divisor pairs}
\label{sec.PerfdPureDivisorPairs}

We now extend the treatment of principal pairs to divisor
pairs $(R,\Delta)$, where $R$ is an excellent $\mathbf{S}_2$ (hence locally equidimensional
\cite[Tag~0FIW]{stacks-project}) ring and $\Delta$ is an effective Mumford $\Q$-divisor on $X=\Spec R$.
We will repeatedly use finite flat $R$-algebras, which must be excellent and $\mathbf{S}_2$ and form a class closed under tensor
products over $R$.

\subsection{Definition and first properties} 
\begin{definition}\label{delta datum}
Let $R$ be an excellent
$\mathbf{S}_2$ ring, $\Delta$
an effective {Mumford $\Q$-divisor} on
$R$, and let $B$ be an $R$-algebra.
\begin{enumerate}
    \item A \emph{divisorial lift of $\Delta$}, or simply
a \emph{$\Delta$-lift}, is a pair $(S,\Delta_S)$, where $S$ is a finite flat $R$-algebra and $\Delta_S$ is an effective Weil generalized divisor on $S$ such that, for some integer $n>0$ with $n\Delta$ integral, the flat pullback of $n\Delta$ to $S$ is $n\Delta_S$. We often refer to such an $n$ as a \emph{witness} of the lift.
    \item A \emph{$B$-realization} of the
$\Delta$-lift $(S,\Delta_S)$ is an $R$-algebra map $\iota\colon S\to B$, i.e., such that $R \to B$ factors as $R \to S \xrightarrow{\iota} B$.
    \item A \emph{$\Delta$-datum on $B$} is a triple $\mathfrak d=(S,\Delta_S,\iota)$
consisting of a $\Delta$-lift together with a $B$-realization.
    \item Associated to $\mathfrak d$ is the pointed $B$-module
\[
M_B(\mathfrak d)
\coloneqq
B\tensor_{S,\iota}S(\Delta_S),
\qquad
\mathbf 1_{\mathfrak d}\coloneqq1\tensor1,
\]
and the distinguished map
\[
\eta_{\mathfrak d}\colon R\longrightarrow M_B(\mathfrak d),
\qquad
1\longmapsto\mathbf 1_{\mathfrak d}.
\]
Here $\mathbf 1_{\mathfrak d}$ is well-defined because the
effectivity of $\Delta_S$ gives an inclusion
$S\subseteq S(\Delta_S)$.
\end{enumerate}
\end{definition}

\begin{remark}
    Because our rings are typically not normal, the divisors that appear in divisorial lifts need not be unique. Suppose that
$(S,\Delta_S)$ and $(S,\Delta'_S)$ are two $\Delta$-lifts on the same
algebra $S$, witnessed by positive integers $n$ and $n'$, respectively.
If $N$ is any common multiple of $n$ and $n'$, then both
$N\Delta_S$ and $N\Delta'_S$ are the flat pullback of $N\Delta$ to
$S$. Thus, while $\Delta_S$ and $\Delta'_S$ need not be equal, their difference
is torsion in the group of Weil generalized divisors on~$S$ (if $S$ is normal, there is no such torsion, and so $\Delta_S$ and $\Delta'_S$ agree).
Likewise, distinct $R$-algebra maps
$\iota,\iota'\colon S\to B$ give distinct $\Delta$-data
\[
\mathfrak d=(S,\Delta_S,\iota)
\qquad\text{and}\qquad
\mathfrak d'=(S,\Delta_S,\iota')
\]
on $B$. The pointed modules $M_B(\mathfrak d)$ and
$M_B(\mathfrak d')$, as well as the purity of their distinguished
maps, may therefore depend on the chosen $B$-realization.
\end{remark}

\begin{example}\label{example.DeltaDatum}
Let $R$ be an excellent {reduced
$\mathbf{S}_2$ ring} with $p$ in its
Jacobson radical, let $f\in R$ be a nonzerodivisor
{that does not vanish along any codimension-one component
of the singular locus}, and let
$a,b\in\Z_{>0}$. Set
$\Delta=\frac{a}{b}\Div_R(f)$ and $S=R[x]/(x^b-f^a)$. The algebra $S$
is finite free over $R$, hence flat
{. Moreover, $x$ is a
nonzerodivisor since $x^b=f^a$ and $f$ is a nonzerodivisor.
Consequently,}
$b\Div_S(x)=a\Div_S(f)$ is the flat pullback of
$b\Delta=a\Div_R(f)$ to $S$. Thus $(S,\Div_S(x))$ is a
$\Delta$-lift.

Any $R$-algebra $B$ containing an element $Y$ with $Y^b=f^a$ realizes this lift via $\iota_Y\colon S\to B$, $x\mapsto Y$ and
$\mathfrak d_Y=(S,\Div_S(x),\iota_Y)$ is a $\Delta$-datum on $B$. Multiplication by $x$
induces an isomorphism
\[
M_B(\mathfrak d_Y)
=B\tensor_{S,\iota_Y}S(\Div_S(x))
\xrightarrow{\sim}B
\]
under which $\mathbf 1_{\mathfrak d_Y}$ corresponds to $Y$. Hence
$\eta_{\mathfrak d_Y}$ corresponds to the map
$R\xrightarrow{1\mapsto Y}B$. Different choices of $Y$ give different
$B$-realizations of the same $\Delta$-lift, and hence different
$\Delta$-data on $B$.
\end{example}

\begin{example}\label{datum associated to SnD construction}
In \autoref{construction SnD}, $(S_{f,m},\Delta_{f,m})$ is a $\Delta$-lift. For any $R$-algebra
$B$ admitting an element $z$ with $z^m=f$ (for instance any absolutely integrally closed $B$), the map
$\iota_z\colon S_{f,m}\to B$ given by $x\mapsto z$ is a
$B$-realization and
$(S_{f,m},\Delta_{f,m},\iota_z)$ is a $\Delta$-datum on $B$.
Different choices of $z$ give different $B$-realizations of the same
$\Delta$-lift, and hence different $\Delta$-data on $B$.
\end{example}

The purity of the maps $\eta_{\mathfrak{d}}:R\to M_B(\mathfrak{d})$ will appear in our definition of perfectoid purity.  Before giving that definition, we record some compatibilities of this purity.

\begin{lemma}\label{properties of M_B}
Let $R$ be an excellent reduced
$\mathbf{S}_2$ ring with $p$ in its Jacobson radical $J$, let $\Delta$ be an effective
Mumford $\Q$-divisor on $R$,
let $B$ be a perfectoid $R$-algebra, and let
$\mathfrak d=(S,\Delta_S,\iota)$ be a $\Delta$-datum on $B$. 
\begin{enumerate}
    \item\label{part functoriality of M_B} If $f \colon B\to C$ is a map of perfectoid $R$-algebras, 
    then \[
    \mathfrak d_C
    \coloneqq
    (S,\Delta_S,S\xrightarrow{\iota}B\xrightarrow{f} C).
    \]
    is a $\Delta$-datum on $C$ and there is a 
    natural pointed isomorphism $M_C(\mathfrak d_C)\cong C\tensor_B M_B(\mathfrak d)$. 
    Moreover, if $f$ is $J$-completely pure, then 
    so is the natural map $M_B(\mathfrak d) \to M_C(\mathfrak d_C)$. In particular, $\eta_{\mathfrak d}$ is pure if and
    only if $\eta_{\mathfrak d_{C}}$ is pure.
    \label{part M_B under p-comp pure map}
    \item $B$ admits a
    $p$-completely faithfully flat map to an absolutely integrally closed
    perfectoid algebra $C$ for which $\eta_{\mathfrak d}$ is pure if and
    only if $\eta_{\mathfrak d_{C}}$ is pure. \label{part assume B is abs ic for M_B}
    \item Suppose that
    $\iota$ factors as
    \[
    S\longrightarrow T\xrightarrow{\iota_T}B,
    \]
    where $S\to T$ is finite flat. If $\Delta_T$ is the flat pullback of $\Delta_S$
    to $T$, then $\mathfrak d_T=(T,\Delta_T,\iota_T)$ is a $\Delta$-datum on
    $B$, and there is a natural pointed isomorphism
    \[
    M_B(\mathfrak d)\xrightarrow{\sim}M_B(\mathfrak d_T).
    \]
    In particular, $\eta_{\mathfrak d}$ is pure if and only if
    $\eta_{\mathfrak d_T}$ is pure. \label{part flat pullback absorbtion}
    \item Suppose that $R \to B$ factors through a flat excellent reduced $\mathbf{S}_2$ $R$-algebra $R'$ and that there is an effective Mumford divisor $\Delta'$ on $R'$ which is the flat pullback of $\Delta$. 
    Write $S' \coloneqq S \otimes_R R'$, and let $\Delta_{S'}$ be the flat pullback of $\Delta_S$.
    Then, the natural map $S' \to B$ makes 
    $\mathfrak d' \coloneqq (S', \Delta_{S'}, S' \to B)$ a $\Delta'$-datum on $B$ and 
    there is a natural pointed isomorphism $M_B (\mathfrak d') \cong M_B (\mathfrak d)$. \label{part flat base change of M_B}
\end{enumerate}
\end{lemma}
\begin{proof}
The first part of \autoref{part functoriality of M_B} is immediate from definition. The second part follows by noting that $M_B(\mathfrak d) \to M_C(\mathfrak d_C)$ is a base change of $f$ and using \autoref{pure + p-comp pure is pure}.  Part \autoref{part assume B is abs ic for M_B} is immediate from 
\autoref{thm: Andre}  and \autoref{part M_B under p-comp pure map}.
Part \autoref{part flat pullback absorbtion} follows from \autoref{flat keeps reflexive} and associativity of tensor products. Part \autoref{part flat base change of M_B} follows by putting every divisor to a sufficiently high multiple of itself and \autoref{flat keeps reflexive}.
\end{proof}

Recall from \autoref{def:divisor_order_relation} that for effective Mumford divisors $\Delta$ and $\Delta'$, the notation $\Delta\prec \Delta'$ means $\Delta\leq (1-\epsilon)\Delta'$ for some $0<\epsilon<1$.

\begin{lemma}\label{divisorial comparison on big open}
Let $R$ be an excellent $\mathbf{S}_2$ ring with $p$ in its Jacobson radical and let $\Delta' \prec \Delta$ be effective Mumford $\Q$-divisors on $R$. Let $B$ be a perfectoid $R$-algebra admitting $\Delta$ and $\Delta'$-data $\mathfrak d=(S,D,\iota)$ and $\mathfrak d'=(S',D',\iota')$, respectively. Let $J_{\mathfrak d'}\subseteq B$ be the radical of the image of $B\tensor_{S'}S'(-D')\to B$.
Let $U\subseteq \Spec R$ be a big open such that both $D$ and $D'$ are Cartier on the respective inverse images of $U$. 
Then, there
is a unique pointed
$\left.\widetilde B\right|_U$-linear map of quasicoherent sheaves
\[
\left.\widetilde{M_B(\mathfrak d')}\right|_U
\longrightarrow
\left.\widetilde{M_B(\mathfrak d)}\right|_U
\]
whose image is contained in
$\left.\reallywidetilde{J_{\mathfrak d'}\cdot
M_B(\mathfrak d)}\right|_U$.
\end{lemma}

\begin{proof}
The images in $\Spec R$ of the Cartier loci of $D$ and $D'$ are big opens since the two induced maps $\Spec S, \Spec S' \to \Spec R$ are finite flat and preserve heights. Thus open subsets $U$ as in the statement exist. Set $T=S\tensor_RS'$, with the realization induced by $\iota$ and $\iota'$. By \autoref{properties of M_B}\autoref{part flat pullback absorbtion}, it suffices to show the lemma after pulling both data back to $T$. We continue to write $D$ and $D'$ for the pulled-back divisors. Let $\rho\colon \Spec T\to \Spec R$ be the natural morphism and let $V\coloneqq {\rho}^{-1}(U)$. By hypothesis, $D$ and $D'$ are Cartier on $V$.

Let $0<\varepsilon<1$ be such that $\Delta'\leq(1-\varepsilon)\Delta$ and let $n$  be a witness to both divisorial lifts. Fix $a<b \in \N_{>0}$ such that $1-\varepsilon<a/b<1$. { Let $H=n(aD-bD')$, which is the flat pullback of the integral Mumford divisor $n(a\Delta-b\Delta')$} and is effective 
because $b\Delta'\leq b(1-\varepsilon)\Delta\leq a\Delta$.

Since $V$ is a big open, it admits a  principal affine cover $\{V_\alpha=\Spec T_\alpha=\Spec T[1/h_\alpha]\}_\alpha$, where each $h_\alpha$ is a nonzerodivisor (since $V$ contains the associated primes), such that
\[
T_\alpha(-D')=x_\alpha T_\alpha
\qquad\text{and}\qquad
T_\alpha(-D)=y_\alpha T_\alpha.
\]
Multiplying by powers of $h_\alpha$, we may take $x_\alpha,y_\alpha\in T$. Write $B_\alpha\coloneqq B\tensor_TT_\alpha$, $M_B(\mathfrak{d})_\alpha\coloneqq M_B(\mathfrak{d})\tensor_T T_\alpha$ and $M_B(\mathfrak{d'})_\alpha\coloneqq M_B(\mathfrak{d'})\tensor_T T_\alpha$. Since $T_\alpha(-H)=y_\alpha^{an}x_\alpha^{-bn}T_\alpha$ and $H$ is effective, $y_\alpha^{an}B_\alpha\subseteq x_\alpha^{bn}B_\alpha$. Applying \autoref{BS - root} gives a unique $c_\alpha\in\sqrt{x_\alpha B_\alpha}$ such that $y_\alpha=c_\alpha x_\alpha$.
On the other hand, $T_\alpha(D')=x_\alpha^{-1}T_\alpha$ and $T_\alpha(D)=y_\alpha^{-1}T_\alpha$, so
$M_B(\mathfrak{d'})_\alpha$ and
$M_B(\mathfrak{d})_\alpha$ are freely generated over
$B_\alpha$ by $x_\alpha^{-1}$ and $y_\alpha^{-1}$, respectively.
Thus
\(
x_\alpha^{-1}\longmapsto c_\alpha y_\alpha^{-1}
\)
defines a pointed $B_\alpha$-linear map between them.

To glue these maps, set $V_{\alpha\beta}=V_\alpha\cap V_\beta=\Spec T_{\alpha\beta}$ and
$B_{\alpha\beta}=B\tensor_TT_{\alpha\beta}$.
By \autoref{BS - root}, on $V_{\alpha\beta}$, the element $c_\alpha$ is the unique element of $\sqrt{x_\alpha B_{\alpha\beta}}$ satisfying $y_\alpha=c_\alpha x_\alpha$.
If $x_\beta=u_{\alpha\beta}x_\alpha$ and $y_\beta=v_{\alpha\beta}y_\alpha$ for units $u_{\alpha\beta},v_{\alpha\beta}\in T_{\alpha\beta}$, uniqueness gives $c_\alpha=v_{\alpha\beta}^{-1}u_{\alpha\beta}c_\beta$. 
Hence by restricting the map on $V_{\alpha}$ to $V_{\alpha \beta}$, we see $x_{\alpha}^{-1} \mapsto c_{\alpha} y_{\alpha}^{-1} = v_{\alpha \beta}^{-1} u_{\alpha \beta} c_{\beta} y_{\alpha}^{-1} = c_{\beta} u_{\alpha\beta} y_{\beta}^{-1}$ and likewise restricting the map on $V_\beta$, we see $x_{\alpha}^{-1} = u_{\alpha \beta} x_{\beta}^{-1}$ is also sent to $c_{\beta} u_{\alpha\beta} y_{\beta}^{-1}$.
This implies that the local maps glue to a pointed morphism
\[
\widetilde{M_B(\mathfrak d')}\big|_V
\longrightarrow
\widetilde{M_B(\mathfrak d)}\big|_V
\]
of $T$-modules. Pushing this map forward along the affine morphism $\rho|_V\colon V\to U$ gives the desired map of sheaves on $U$. Since \autoref{flat keeps reflexive}, $T(-D')\cong T\tensor_{S'}S'(-D')$, $J_{\mathfrak d'}B_\alpha=\sqrt{x_\alpha B_\alpha}$. Moreover, since the localized modules are free, any local pointed map whose image is contained in $J_{\mathfrak d'}M_B(\mathfrak{d})_\alpha$ is determined by an element $c\in\sqrt{x_\alpha B_\alpha}$ such that $y_\alpha=cx_\alpha$. Uniqueness in \autoref{BS - root} shows that it is the map constructed above.
\end{proof}

We now prove the divisor analog of the first part of \autoref{principal pair sequence criterion}.  Also compare with \autoref{common perfectoid algebra for divisor pairs} and \autoref{perfectoid purity decreases with divisor}.

\begin{proposition}
\label{BS application to divisor purity}
Let $R$ be an excellent  $\mathbf{S}_2$ ring with $p$ in its Jacobson radical and let $\Delta' \prec \Delta$ be effective Mumford $\Q$-divisors on $R$. 
Let $B$ be a perfectoid $R$-algebra admitting $\Delta$ and $\Delta'$-data, $\mathfrak d=(S,D,\iota)$ and $\mathfrak d'=(S',D',\iota')$, respectively. If $\eta_{\mathfrak d}$ is pure, then so is
$\eta_{\mathfrak d'}$.
\end{proposition}

\begin{proof}
Using \autoref{divisorial comparison on big open} and its notation, we have a pointed map
\[
\Gamma\bigl(U,\widetilde{M_B(\mathfrak d')}\bigr) \longrightarrow \Gamma\bigl(U,\widetilde{M_B(\mathfrak d)}\bigr).
\]
Since $\eta_{\mathfrak d}$ is pure, so is $\Gamma(U,\eta_{\mathfrak d})$ by \autoref{check purity on codim 1}. By the above, this factors through $\Gamma(U,\eta_{\mathfrak d'})$, which must therefore also be pure and the result follows from another application of \autoref{check purity on codim 1}.
\end{proof}

We are finally ready to define perfectoid purity for pairs $(R, \Delta)$.

\begin{definition}\label{def perfd pure div pair}
Let $R$ be an excellent reduced $\mathbf{S}_2$ ring with $p$ in its Jacobson radical, and let $\Delta$ be an effective Mumford $\Q$-divisor on $R$. We say that $(R,\Delta)$ is \emph{perfectoid pure} if, for every rational number $0<\varepsilon<1$, there exist a perfectoid $R$-algebra $B_\varepsilon$ and a $(1-\varepsilon)\Delta$-datum $\mathfrak d_\varepsilon$ on $B_\varepsilon$ such that $\eta_{\mathfrak d_\varepsilon}\colon R\to M_{B_\varepsilon}(\mathfrak d_\varepsilon)$ is pure.
\end{definition}

For $\Delta=0$, this definition recovers the notion of a perfectoid
pure ring from \cite[Definition~4.1(a)]{BMPSTWW3}. Indeed, a pure map
$R\to B$ to a perfectoid $R$-algebra gives a trivial $0$-datum on
$B$ that works for every $\varepsilon$, while the converse follows
from \autoref{pure full boundary datum implies perfectoid pure}.

\begin{remark}[$\bR$-divisors]
    Suppose $R$ is an excellent reduced $\mathbf{S}_2$ ring with $p$ in its Jacobson radical. Further suppose that $\Delta$ is a Mumford $\bR$-divisor.  We could then define $(R, \Delta)$ to be perfectoid pure if for every sequence of effective Mumford $\bQ$-divisors $\{\Delta_n \}_{n \geq  1}$ with $\Delta_n \leq \Delta$ and $\Delta_n \to \Delta$ coefficientwise, we have that $(R, \Delta_n)$ is perfectoid pure.  Many results in this paper likely can be extended to this setting.  
\end{remark}

\begin{corollary}\label{pure full boundary datum implies perfectoid pure}
Let $R$ and $\Delta$ be as in
\autoref{def perfd pure div pair}. Suppose that there are a
perfectoid $R$-algebra $B$ and a $\Delta$-datum $\mathfrak d$ on
$B$ such that $\eta_{\mathfrak d}$ is pure. Then, the $R$-algebra map
$R\to B$ is pure and $(R,\Delta)$ is perfectoid pure. In particular,
$R$ is perfectoid pure.
\end{corollary}

\begin{proof}
Effectivity gives a natural map $B\to M_B(\mathfrak d)$, $b\mapsto b\tensor1$, through which $\eta_{\mathfrak d}$ factors. Hence $R\to B$ is pure so $R$ is perfectoid pure.
Using \autoref{properties of M_B}\autoref{part assume B is abs ic for M_B}, we may assume that $B$ is absolutely integrally closed. Then, by \autoref{datum associated to SnD construction}, for every rational number $0<\varepsilon<1$, there is a $(1-\varepsilon)\Delta$-datum $\mathfrak d_\varepsilon$ on $B$. Choose the trivial datum when $\Delta=0$. Applying \autoref{BS application to divisor purity} with $\Delta'=(1-\varepsilon)\Delta$ shows that $\eta_{\mathfrak d_\varepsilon}$ is pure, as desired.
\end{proof}

\begin{theorem} \label{common perfectoid algebra for divisor pairs}
Let $R$ and $\Delta$ be as in \autoref{def perfd pure div pair}. Then
$(R,\Delta)$ is perfectoid pure if and only if there is an absolutely integrally closed perfectoid $R$-algebra $B$ such that for some sequence of rational numbers $0<\{\varepsilon_n\}<1$ converging to $0$, there is a $(1-\varepsilon_n)\Delta$-datum $\mathfrak d_{\varepsilon_n}$ on $B$ for all $n$ for which $\eta_{\mathfrak d_{\varepsilon_n}}$ is pure. Moreover, $B$ can be chosen so that every effective Mumford $\Q$-divisor $\Delta'$ satisfying $\Delta'\prec\Delta$
admits a datum on $B$, and $\eta_{\mathfrak d'}$ is pure for every
$\Delta'$-datum $\mathfrak d'$ on $B$. 
\end{theorem}

\begin{proof}
The reverse implication is immediate from \autoref{BS application to divisor purity} so suppose that $(R,\Delta)$ is
perfectoid pure. If $\Delta=0$, the result follows from \cite[Lemma~4.5]{BMPSTWW3} and \autoref{BS application to divisor purity}, so we may assume that $\Delta\neq0$. Fix $f\in R$ as in \autoref{construction SnD} and note that since $(m\alpha\Delta)_{\mathrm{red}}=\Delta_{\mathrm{red}}$ 
for any $\alpha\neq 0$, this same $f$ works for any $(1-\varepsilon_n)\Delta$.

By \autoref{properties of M_B}\autoref{part assume B is abs ic for M_B}, we may assume that every $B_\varepsilon$ in the definition that $(R,\Delta)$ being perfectoid pure is absolutely integrally closed. Now for every rational $0<\varepsilon<1$, there is $m_\varepsilon$ so that $m_\varepsilon(1-\varepsilon)\Delta$ is integral. Let $(S_{f,m_\varepsilon}, D_\varepsilon)$ be the $(1-\varepsilon)\Delta$-lift supplied by \autoref{datum associated to SnD construction}. Since $B_\lambda$ is absolutely integrally closed for each $0<\lambda<1$, there exists a $(1-\varepsilon)\Delta$-datum $\mathfrak{d}_{\varepsilon,\lambda}=(S_{f,m_\varepsilon}, D_\varepsilon, \iota_{\varepsilon,\lambda})$ on $B_\lambda$. Then $\mathfrak{d}_\epsilon := \prod_\lambda\mathfrak{d}_{\varepsilon,\lambda}$ is a $(1-\varepsilon)\Delta$-datum on $B{:=\prod_{0<\lambda<1}B_\lambda}$.

$B$ is perfectoid by \cite[Example~3.8(8)]{BhattIyengarMaRegularRings}, and it is absolutely integrally closed because monic polynomials can be solved coordinatewise. For each $0<\lambda<1$, we have a composition pointed map (where the second map is simply projection onto the $\lambda$-factor)
\[
\eta_{\mathfrak d_{\varepsilon,\lambda}}\colon R\to M_B(\mathfrak d_\varepsilon)\longrightarrow M_{B_\lambda}(\mathfrak d_{\varepsilon,\lambda}).
\]
Note that for each $0<\lambda<\varepsilon$, $\eta_{\mathfrak d_{\varepsilon,\lambda}}$ is pure: this is because $(1-\varepsilon)\Delta < (1-\lambda)\Delta$ and thus \autoref{BS application to divisor purity} shows that $\eta_{\mathfrak d_{\varepsilon,\lambda}}$ is pure. It follows that the pointed map $\eta_{\mathfrak d_{\varepsilon}}$: $R\to M_B(\mathfrak d_{\varepsilon})$ is pure for every rational $0<\varepsilon<1$ and the first statement follows. The second part of the statement follows from \autoref{datum associated to SnD construction}.
\end{proof}

\begin{corollary}\label{perfectoid purity decreases with divisor}
Let $R$ be an excellent $\mathbf{S}_2$ ring with $p$ in its Jacobson radical, and let $\Delta'\leq\Delta$ be effective Mumford $\Q$-divisors on $R$. If $(R,\Delta)$ is perfectoid pure, then $(R,\Delta')$ is perfectoid pure. Moreover, if $\{\Delta_n\}_{n\geq1}$ is a sequence of effective Mumford $\Q$-divisors such that $\Delta_n\leq\Delta$, $\Delta_n\to\Delta$ coefficientwise, and $(R,\Delta_n)$ is perfectoid pure for all $n$, then $(R,\Delta)$ is perfectoid pure.
\end{corollary}

\begin{proof}
Direct from \autoref{common perfectoid algebra for divisor pairs} after noting that for any $0<\varepsilon<1$, there is $n\gg 0$ with $\Delta_n \geq (1-\varepsilon/2)\Delta$ so that $(1-\varepsilon)\Delta \prec \Delta_n$.
\end{proof}

\begin{corollary}\label{perfectoid pure coefficients at most one}
Let $R$ and $\Delta$ be as in \autoref{def perfd pure div pair}. If $(R,\Delta)$ is perfectoid pure, then every coefficient of $\Delta$ is at most one.
\end{corollary}

\begin{proof}
Suppose that $\Delta$ has coefficient $c>1$ along a prime Mumford divisor $P$ and let $0<\varepsilon<1$ be a small enough rational number so that $(1-\varepsilon)c>1$. By \autoref{construction SnD} and \autoref{common perfectoid algebra for divisor pairs}, there a pure datum $\mathfrak d_\varepsilon=(S_\varepsilon,D_\varepsilon,\iota_\varepsilon)$ for some $m\in \N_{>0}$, nonzerodivisor $f\in R$ and $S_\varepsilon\coloneqq R[x]/(x^m-f)$ mapping to a perfectoid ring $B$. In fact, we have $S_\epsilon(D_\epsilon)\tensor_RR_P=x^{-m(1-\varepsilon)c}(S_\varepsilon \tensor_RR_P)$. Write $\pi$ for the uniformizer of $R_P$ and $f=u\pi$ for some unit $u\in R_P$. 
Since $m(1-\varepsilon)c>m$, $1/\pi = ux^{-m}\in S_\epsilon(D_\epsilon)\tensor_RR_P$ so the distinguished element of $M_B(\mathfrak d_\epsilon)_P$ is in $\pi M_B(\mathfrak d_\epsilon)_P$. Base changing $\eta_{\mathfrak d_\varepsilon}$ by $R_P/\pi R_P$ gives a pure map that sends $1$ to $0$, a contradiction.
\end{proof}

\subsection{The limiting module from below}
\label{subsec:frombelow}
Fix $R$ and $\Delta$ as in \autoref{def perfd pure div pair} and a perfectoid $R$-algebra $B$. In this subsection, we combine \autoref{codimension one envelope} and \autoref{divisorial comparison on big open} to show that perfectoid purity can be tested on only one module. Throughout, codimension-one envelopes are taken over $R$.

\begin{lemma}\label{coherent divisorial comparison}
Let $R$ be an excellent $\mathbf{S}_2$ ring with $p$ in its Jacobson radical, and let $\Delta'\prec\Delta$ be effective Mumford $\Q$-divisors on $R$. Let  $\mathfrak d=(S,D,\iota)$ and $\mathfrak d'=(S',D',\iota')$ be $\Delta$ and $\Delta'$-data on a perfectoid ring $B$, respectively. 
There is a canonical pointed $B$-linear map
\[
\Phi_{\mathfrak d',\mathfrak d}\colon
M_B(\mathfrak d')^{\mathrm{c1}}
\longrightarrow M_B(\mathfrak d)^{\mathrm{c1}}
\]
that satisfies the following properties.
\begin{enumerate}
    \item \label{coherent divisorial comparison.SheafMapVsExtendedDatum}
        Suppose that $\iota$ and $\iota'$ factor through finite flat
        $S$- and $S'$-algebras $T$ and $T'$, respectively, and let
        $\mathfrak d_{T}$ and $\mathfrak d'_{T}$ be the pullback data
        given by \autoref{properties of M_B}\autoref{part flat pullback absorbtion}. The natural pointed
        isomorphisms of that lemma identify
        $\Phi_{\mathfrak d',\mathfrak d}$ with
        $\Phi_{\mathfrak d'_{T},\mathfrak d_{T}}$.
    \item \label{coherent divisorial comparison.FunctorialityOfSheafComparison}
        If $\Delta''$ is an effective
        Mumford $\Q$-divisor on $R$,
        $\mathfrak d''=(S'',D'',\iota'')$ is a $\Delta''$-datum on $B$,
        and $\Delta''\prec\Delta'$, then
        \[
        \Phi_{\mathfrak d'',\mathfrak d}
        =
        \Phi_{\mathfrak d',\mathfrak d}
        \circ\Phi_{\mathfrak d'',\mathfrak d'}.
        \]
\end{enumerate}
\end{lemma}

\begin{proof}
Note that if both $U$ and $U'$ are big opens as in \autoref{divisorial comparison on big open}, then so is $U\cap U'$. Moreover, the maps of \autoref{divisorial comparison on big open} on $U$ and $U'$ must agree on $U\cap U'$ by uniqueness. Since big opens as in \autoref{divisorial comparison on big open} are cofinal in the system defining the codimension-one envelope, we can pass to the filtered colimits to obtain a canonical map $\Phi_{\mathfrak d',\mathfrak d}$.

For \autoref{coherent divisorial comparison.SheafMapVsExtendedDatum}, \autoref{flat keeps reflexive} and associativity show that the ideals $J_{\mathfrak d'}$ (which we recall is defined as $ \sqrt{\mathrm{Image}(B \tensor_{S'} S'(-D') \to B)}$) and $J_{\mathfrak d'_{T}}$ (defined likewise) agree. This implies that on any common admissible big open $U$, the pointed isomorphism of \autoref{properties of M_B}\autoref{part flat pullback absorbtion} identifies the maps $\widetilde{M_B(\mathfrak d')}|_U \to \widetilde{M_B(\mathfrak d)}|_U$ and $\widetilde{M_B(\mathfrak d'_T)}|_U \to \widetilde{M_B(\mathfrak d_T)}|_U$ of \autoref{divisorial comparison on big open} so the result follows.
Part \autoref{coherent divisorial comparison.FunctorialityOfSheafComparison} similarly follows from the uniqueness in \autoref{divisorial comparison on big open}.
\end{proof}

\begin{definition}
\label{universal divisorial module} 
Let $R$ be an excellent $\mathbf{S}_2$ ring with $p$ in its Jacobson radical and $B$ a perfectoid $R$-algebra. Suppose that, for every $e\geq1$, there are effective Mumford $\Q$-divisors $\Delta_e$ admitting $\Delta_e$-data $\mathfrak d_e$ on $B$ such that
\[
\Delta_e\prec\Delta_{e+1}\prec\Delta
\quad\text{for every }e
\qquad \text{and}\qquad
\Delta_e\longrightarrow\Delta
\quad\text{coefficientwise}.
\]
We call $(\Delta_e, \mathfrak d_e)$ a limiting sequence for $\Delta$. The maps of \autoref{coherent divisorial comparison} form a direct system, and we define
\[
M_B({<}\Delta)
\coloneqq
\varinjlim_e M_B(\mathfrak d_e)^{\mathrm{c1}}.
\]
The distinguished elements give $\mathbf 1_{{<}\Delta}\in M_B({<}\Delta)$ and a map $ \rho_{B,\Delta}\colon R\to M_B({<}\Delta) $, $1\mapsto \mathbf 1_{{<}\Delta}$.

The next lemma shows that this pointed module and map do not depend on the chosen sequence.
\end{definition}

\begin{lemma}\label{divisorial module computed by towers}
Let $R$ be an excellent $\mathbf{S}_2$ ring with $p$ in its Jacobson radical and $B$ a perfectoid $R$-algebra. For any effective Mumford divisor $\Delta$ admitting limiting sequences $(\Delta_e,\mathfrak d_e)$ and $(\Gamma_j,\mathfrak d'_j)$ on $B$, write
\[
\lambda_e\colon M_B(\mathfrak d_e)^{\mathrm{c1}} \longrightarrow\varinjlim_eM_B(\mathfrak d_e)^{\mathrm{c1}}, 
\qquad
\mu_j\colon M_B(\mathfrak d'_j)^{\mathrm{c1}} \longrightarrow\varinjlim_jM_B(\mathfrak d'_j)^{\mathrm{c1}}
\]
for the maps defining the colimits. There is a canonical pointed isomorphism
\[
\Theta\colon
\varinjlim_eM_B(\mathfrak d_e)^{\mathrm{c1}}
\xrightarrow{\ \sim\ }
\varinjlim_jM_B(\mathfrak d'_j)^{\mathrm{c1}},
\]
which is uniquely characterized by $ \Theta\circ\lambda_e = \mu_j\circ\Phi_{\mathfrak d_e,\mathfrak d'_j}$ whenever $\Delta_e\prec\Gamma_j$
for $\Phi_{\mathfrak d_e, \mathfrak d'_j}$ as in \autoref{coherent divisorial comparison}. In particular \[M_B({<}\Delta) \cong \varinjlim_{\substack{\Gamma\prec\Delta\\
\mathfrak d'\text{ a }\Gamma\text{-datum on }B}} M_B(\mathfrak d')^{\mathrm{c1}}.\]
Moreover, for any $D\leq\Delta$ admitting a limiting sequence on $B$, there is a canonical pointed map
\[
M_B({<} D)\longrightarrow M_B({<}\Delta).
\]
\end{lemma}

\begin{proof}
Fix $e\geq 1$ and consider $\Delta_e\prec \Delta$. Since $\Delta$ has finite support and $\Gamma_j\to\Delta$ coefficientwise, it is easy to see that for $j\gg0$ we have $\Delta_e \prec \Gamma_j$.
We may therefore choose strictly increasing sequences of indices $e_1<e_2<\cdots$ and $j_1<j_2<\cdots$ recursively so that
\[
\Delta_{e_1}\prec\Gamma_{j_1}\prec\Delta_{e_2}
\prec\Gamma_{j_2}\prec\cdots.
\]
Since this is an interleaved direct system, the colimit in $\Gamma$ must agree with the colimit in $\Delta$.

For each $e$, the transitivity of 
\autoref{coherent divisorial comparison} implies that the composite $\mu_j\circ\Phi_{\mathfrak d_e,\mathfrak d'_j}$ is independent of the choice of $j$ satisfying $\Delta_e\prec\Gamma_j$. These composites are also compatible as $e$ varies, and hence induce the displayed map $\Theta$. Note that $\Theta$ is independent of the chosen interleaving. Moreover, reversing the roles of the two sequences gives its inverse so it is indeed an isomorphism. All the maps are pointed and therefore so is $\Theta$. For the last statement, let $(D_e, \mathfrak{d_e'})$ be a limiting sequence for $D$. For any $e$, $D_e\prec D \leq \Delta$ so $D_e$ is part of a limiting sequence for $\Delta$ and we get a canonical pointed map $M_B(\mathfrak{d_e'})\to M_B({<}\Delta)$. This is compatible with transitions so going to the colimit gives the desired $M_B({<} D)\longrightarrow M_B({<}\Delta)$.
\end{proof}
Although any limiting sequence works to define $M_B({<}\Delta)$, two of them are particularly useful.

\begin{construction}\label{construct.LimitingConstructionWithComparisons}
Let $R$ be an excellent $\mathbf{S}_2$ ring and let $\Delta$ be an effective Mumford $\bQ$-divisor. Let $m\in \N_{>0}$ be such that $m\Delta$ has integer coefficients and fix $E = \Delta_{\reduced}$.  Then there are sequences of data for $e \gg 0$
\[
    \frad_e = (S_e, \Delta_e, \iota_e) \quad \text{and} \quad  \frad_e' = (S_e',\Delta_e', \iota_e')
\]
with the following properties. 
\begin{enumerate}
    \item $\Delta_e$ and $\Delta_e'$ are effective Mumford divisors.
        \item $\Delta_e$ is the flat pullback of the effective divisor $\Delta - \frac 1 {m p^e} E \prec \Delta$ and $\Delta_e'$ is the flat pullback of the effective divisor $t_e\Delta$, where $t_e=\frac{m}{p^e}(\left\lfloor\frac{p^e}{m}\right\rfloor-1)$. Moreover, 
         \[\lim_{e\to \infty} t_e\Delta=\lim_{e\to \infty}\left(\Delta - \frac 1 {m p^e} E\right)=\Delta,\]
        where the limit is to be understood coefficientwise.
        \item For each $e > 0$, there are finite flat $R$-algebra maps $S_e \to S_{e+1}$ and $S_e' \to S_{e+1}'$ such that the following diagrams commute
            \[
                \xymatrix{
                    S_e \ar@{_{(}->}[d] \ar[r]^{\iota_e} & B\\
                    S_{e+1} \ar[ur]_{\iota_{e+1}}
                } \qquad
                 \xymatrix{
                    S_e' \ar@{_{(}->}[d] \ar[r]^{\iota_e'} & B'\\
                    S_{e+1}' \ar[ur]_{\iota_{e+1}'}
                }.
            \]
            \item If $h_e \colon \Spec S_{e+1} \to \Spec S_e$ and $h_e' \colon \Spec S'_{e+1} \to \Spec S'_e$ are induced from the previous maps, then $h_e^* \Delta_e \prec \Delta_{e+1}$ and $(h_e')^*\Delta_e' \prec \Delta'_{e+1}$ for any $e\gg 0$. In particular there are natural maps $S_e(\Delta_e) \otimes_{S_e} S_{e+1} \to S_{e+1}(\Delta_{e+1})$ and  $S_e'(\Delta_e') \otimes_{S_e'} S_{e+1}' \to S_{e+1}'(\Delta'_{e+1})$ which are injective. \label{construct.LimitingConstructionWithComparisons.CompatibleMapsDivisorLevel}
        \end{enumerate}
        We now perform the constructions.  Essentially we are simply going to form iterated versions of \autoref{construction SnD}.
        Pick $f \in R$ such that $\Div(f) = E + G = {\Delta}_{\reduced} + G$ and where $G$ has no common components with $E$.  
        Set $S_e = R[X_e]/(X_e^{mp^e} - f)$ and $S_e' = R[Y_e]/(Y_e^{p^e} - f)$.  Just as in \autoref{construction SnD}, then the flat pullback of $\Delta$ to $S_e$ is a Mumford divisor, $\Delta_{S_e}$.
        We set 
        \[
            \Delta_e \coloneqq \Delta_{S_e} - (\Delta_{S_e})_{\reduced}.
        \]
        By construction, $\Delta_e$ is the flat pullback of $\Delta - \frac 1 {mp^e} E$ to $S_e$.
        Similarly, we set $\Delta_e'$ to be the flat pullback of $t_e\Delta$ to $S_e'$, which is possible since $t_ep^e\Delta= m(\lfloor \frac{p^e}{m}\rfloor-1)\Delta$ has integer coefficients.
        There are certainly finite flat maps 
        \[
            S_e =  R[X_e]/(X_e^{mp^e} - f) \xrightarrow{X_e \mapsto X_{e+1}^p} R[X_{e+1}]/(X_{e+1}^{mp^{e+1}} - f).
        \]
        and 
\[
            S_e' =  R[Y_e]/(Y_e^{p^e} - f) \xrightarrow{Y_e \mapsto Y_{e+1}^p} R[Y_{e+1}]/(Y_{e+1}^{p^{e+1}} - f).
        \]
        Fix $B$ a perfectoid $R$-algebra such that $f$ has an $m$-th root $Z \in B$ ($Z^m = f$) and such that $Z$ has a compatible system of $p$-power roots, $Z^{1/p^e}\in B$. Fix $B'$ a perfectoid $R$-algebra such that $f$ has a compatible system of $p$-th power roots $W^{1/p^e}\in B'$.
        For instance, this holds if $B$ and $B'$ are absolutely integrally closed. 
        Sending $X_e \mapsto Z^{1/p^e}$ and $Y_e\mapsto W^{1/p^e}$ produces maps $\iota_e\colon S_e \to B$ and $\iota_e'\colon S_e'\to B'$ so that 
        \[
            \frad_e = (S_e, \Delta_e, \iota_e) \quad \text{and} \quad  \frad_e' = (S_e',\Delta_e', \iota_e')
        \]
        are datum for $\Delta - \frac 1{mp^e} E$ and $t_e\Delta$, respectively.  Write $a_e\coloneqq \lfloor\frac{p^e}{m}\rfloor -1$ and note that $m\leq p^e-ma_e<2m$ for all $e$ so $p^{e+1}-ma_{e+1}<2m\leq p(p^e-ma_e)$, which gives $t_e<t_{e+1}$. Then, the properties outlined above all follow formally.

        Set $S_{\infty} \coloneqq \varinjlim_e S_e$ where the transition maps are the ones above and note that the induced map $S_e \to S_{\infty} \to B$ is $\iota_e$.  
        Set $M_{S_{\infty}}({<}\Delta) \coloneqq \varinjlim_e (S_{\infty}\otimes_{S_{e}} S_e(\Delta_e) )^{\cOne}$.
        Since the formation of $(-)^{\cOne}$ commutes with filtered colimits, we have canonical $R$- (and $S_{e}$)-module isomorphisms
        \[
            M_{S_{\infty}}({<}\Delta) \cong (\varinjlim_e S_e(\Delta_e))^{\cOne}
        \]
        and also a canonical pointed $B$-module isomorphism
        \[
            M_B({<}\Delta) \cong (M_{S_{\infty}}({<}\Delta) \otimes_{S_{\infty}} B)^{\cOne}.
        \]
    \end{construction}

    \begin{remark}
        \label{rem.ConstructionCompatibilityWithFinitelyManyDivisors}
        It is worth noting that if $0 \leq \Delta \leq D$ is such that both $mD$ and $m\Delta$ have integer coefficients, then we may use the same $f$ and sequence of $S_e$'s for both of them in \autoref{construct.LimitingConstructionWithComparisons}.  Furthermore, setting $E' = D_{\reduced}$ and $D_{e}$ to be the flat pullback of $D - \frac 1{mp^e} E'$, we have $\Delta_e \leq D_e$ and that this is compatible with the inclusions of \autoref{construct.LimitingConstructionWithComparisons.CompatibleMapsDivisorLevel}. One can also use the same $S_e'$ and work with the flat pullbacks of $t_e\Delta$ and $t_eD$ and keep the inclusions.
        This can also be done with any finite collection of divisors $\Delta$s such that $m \Delta$ has integer coefficients for a common $m$. 
    \end{remark}

\begin{lemma}\label{universal divisorial module is local in the divisor}
Let $R$ be an excellent $\mathbf{S}_2$ ring with $p$ in its Jacobson radical, let $\Delta$ and $\Delta'$ be effective Mumford $\Q$-divisors on $R$, and let $B$ be a perfectoid $R$-algebra admitting limiting sequences that define both $M_B({<}\Delta)$ and $M_B({<}\Delta')$. 
If $\Delta|_V=\Delta'|_V$ on an open subset $V\subseteq\Spec R$, then there is a canonical pointed $\widetilde{B|_V}$-linear isomorphism
\[
\left.\widetilde{M_B({<}\Delta)}\right|_V
\xrightarrow{\ \sim\ }
\left.\widetilde{M_B({<}\Delta')}\right|_V,
\]
which is compatible with restriction to smaller open subsets.
\end{lemma}

\begin{proof}
Choose limiting sequences $(\Delta_e,\mathfrak d_e)_{e\geq1}$ and $ (\Delta'_j,\mathfrak d'_j)_{j\geq1}$ for \(M_B({<}\Delta)\) and \(M_B({<}\Delta')\), respectively. Now, $\Delta|_V=\Delta'|_V$ has finite support so coefficientwise convergence implies that $\Delta_e|_V$ and $\Delta_e'|_V $ are mutually cofinal with respect to $\prec$. Since restriction commutes with filtered colimits, the result follows from the fact that the construction of the map in \autoref{divisorial comparison on big open} is local.
\end{proof}

\begin{proposition}\label{universal module detects perfectoid purity}
Let $R$ be an excellent $\mathbf{S}_2$ ring with $p$ in its Jacobson radical and let $\Delta$ be an effective Mumford $\Q$-divisor on $R$. Let $B$ be a perfectoid $R$-algebra admitting a limiting sequence for $\Delta$.
Then, 
\[
    \rho_{B,\Delta}\colon R\to M_B({<}\Delta) \quad 1\mapsto \mathbf 1_{{<}\Delta},
\]
is pure if and only if $\eta_{\mathfrak d'}$ is pure for every $\Delta'$-datum $\mathfrak d'$ on $B$ with $\Delta'\prec\Delta$. 
Consequently, $(R,\Delta)$ is perfectoid pure if and only if $\rho_{B,\Delta}\colon R\to M_B({<}\Delta)$ is pure for some absolutely integrally closed perfectoid $R$-algebra $B$.
\end{proposition}

\begin{proof}
If every $\eta_{\mathfrak d'}$ is pure then \autoref{purity and codimension one envelope} implies that all the maps $R\to M_B(\mathfrak d')^{\mathrm{c1}}$ are pure.
Taking the filtered colimit shows that $\rho_{B,\Delta}$ is pure. Conversely, any $(\Delta',\mathfrak d')$, is part of a limiting sequence for $\Delta$ so \autoref{divisorial module computed by towers} implies that $\rho_{B,\Delta}$ factors through $R\to M_B(\mathfrak d')^{\mathrm{c1}}$. The purity of $\rho_{B,\Delta}$ therefore forces that of $R\to M_B(\mathfrak d')^{\mathrm{c1}}$, and hence that of $\eta_{\mathfrak d'}$.
The last assertion follows from \autoref{common perfectoid algebra for divisor pairs}.
\end{proof}

\begin{proposition}\label{radical ideal computes universal divisorial module}
    Let $R$ be an excellent $\mathbf{S}_2$ ring with $p$ in its Jacobson radical, let $D$ be an effective integral Mumford divisor and let $\Gamma$ be an effective Mumford $\bQ$-divisor.  Fix $f \in R$ an element vanishing to order $1$ at the generic point of every prime Mumford divisor in the support of $D$ or $\Gamma$ and such that $\Div(f)$ is a Mumford divisor.  Choose $m > 0$ such that $m\Gamma$ has integer coefficients and pick $B$ a perfectoid $R$-algebra such that $f$ has an $m$-th root $Y \in B$ (i.e., $Y^m = f$) and such that $Y$ has a compatible system of $p$-power roots in $B$ {(e.g., $B$ is absolutely integrally closed)}.  
    Set $I_D=R(-D)$, and let $T_{B,D}\coloneqq \left(\sqrt{I_DB}\tensor_RR(D)\right)^{\mathrm{c1}}.$
    \begin{enumerate}
        \item\label{isom MB and TBD} If $\Gamma$ and $D$ share no common components, then there is a canonical pointed $B$-linear
            isomorphism
            \[
            (M_B({<}(D+\Gamma)),\mathbf{1}_{{<}(D+\Gamma)})
            \xrightarrow{\ \sim\ }
            (\left(T_{B,D}\tensor_BM_B({<}\Gamma)\right)^{\mathrm{c1}},\mathbf 1_D\tensor\mathbf 1_{{<}\Gamma}), 
            \]
        where $\mathbf 1_D$ is induced by the canonical identification $\left(I_D\tensor_RR(D)\right)^{\mathrm{c1}}\cong R$.
            \item \label{def alpha D Gamma}There is a canonical pointed $B$-linear map
			\[
    			\alpha_{D,\Gamma}\colon
    			\left(
    			M_B({<}(D+\Gamma)),
    			\mathbf 1_{{<}(D+\Gamma)}
    			\right)
    			\longrightarrow
    			\left(
    			\left(
    			(B\tensor_RR(D))\tensor_BM_B({<}\Gamma)
    			\right)^{\mathrm{c1}},
    			\mathbf 1_D\tensor\mathbf 1_{{<}\Gamma}
    			\right).
            \]
            When $\Gamma$ and $D$ have no common component, this is given by composing the map in \autoref{isom MB and TBD} with the natural pointed map \[
            			T_{B,D} \longrightarrow(B\tensor_RR(D))^{\mathrm{c1}}
            \]
            coming from $\sqrt{I_DB}\subseteq B$.
    \end{enumerate}
\end{proposition}

\begin{proof}
    Apply the construction of \autoref{construct.LimitingConstructionWithComparisons} to $D + \Gamma$ noting that it also immediately applies to both $D$ and $\Gamma$ by \autoref{rem.ConstructionCompatibilityWithFinitelyManyDivisors}.  Set $E = \Gamma_{\reduced}$ and $E' = D_{\reduced}$.   On $\Spec S_e$, we have $D_e$ (the flat pullback of $D - \frac 1{mp^e} E'$), $\Gamma_e$ (the flat pullback of $\Gamma - \frac 1 {mp^e} E)$ and $\Delta_e \coloneqq \Gamma_e + D_e$ (the flat pullback of $D + \Gamma - \frac 1 {mp^e}(E+E')$). Write $\mathfrak d_e$ and $\mathfrak d_e'$ for the corresponding $\Gamma - \frac 1 {mp^e} E$ and $D + \Gamma - \frac 1 {mp^e}(E+E')$-data, respectively.
    We first show \autoref{isom MB and TBD}. 
    Recall our construction also provides pointed isomorphisms
    \[
        M_B({<}(\Gamma + D)) \cong \big(M_{S_{\infty}}({<}(\Gamma + D)) \otimes_{S_{\infty}} B\big)^{\cOne}.
    \]
    In fact, by construction and since $D$ and $E$ have no common components in their support, we have that 
    \begin{equation}\label{SeDelta is sqrt}
        S_e(\Delta_e) \cong \big(\sqrt{I_{D+E}S_e} \otimes_{S_e} S_e(\Gamma + D)\big)^{\cOne} \cong \big((\sqrt{I_{D}S_e} \otimes_R R(D) )^{\cOne}\otimes_{S_e} S_e(\Gamma_e) \big) ^{\cOne}.
    \end{equation}
    Likewise $S_e(D_e) = \big( \sqrt{I_D S_e} \otimes_{S_e} S_e(D) \big)^{\cOne}$.  By taking filtered colimits and using the associativity of tensor product, we obtain a canonical pointed map 
    \[
        M_B({<}(\Gamma + D)) \cong \varinjlim_e \big((\sqrt{I_{D}S_e} \otimes_{R} R(D))^{\cOne} \otimes_{S_e} S_e(\Gamma_e) \otimes_{S_e} B\big) ^{\cOne} \to \big(T_{B,D} \otimes_B M_B({{<}\Gamma})\big)^{\cOne}.
    \]
    As tensor products and codimension-1-envelopes commute with filtered colimits, and using the associativity of tensor product yet again, it suffices to show that 
    \[
        \varinjlim_e \left(\left(\sqrt{I_D S_e} \otimes_R R(D)\right) \otimes_{S_e} B\right)^{\cOne} \to T_{B,D}
    \]
    is an isomorphism.  Similarly, the result then follows if we can show that 
    \[
        \varinjlim_e \left(\sqrt{I_D S_e} \otimes_{S_e} B\right)^{\cOne} \to \left(\sqrt{I_D B}\right)^{\cOne}
    \]
    is an isomorphism.  By \autoref{prop.SomeGeneralFactsAboutCodim1Envelopes}, it suffices to check this at codimension-1 points of $\Spec R$, 
    at which it suffices to prove the result without the $(-)^{\cOne}$.  
    At those points outside of the support of $D$, there is nothing to show, and hence we may assume that $R$ and all $S_e$ are DVRs, $I_D = f^a R$ and $X_e^{mp^e} = f \in S_e$ where $X_e$ is a uniformizer for $S_e$ and $f$ is a uniformizer for $R$.  We write $h$ for the image of $X_1^p$ in $B$ (thought of as $f^{1/m}$).  We thus need to show that 
    \[
        \varinjlim_e \left(X_e S_e \otimes_{S_e} B\right) \to \sqrt{h B} = \sqrt{I_D B}\]
    is an isomorphism.  As $S_e$ is a DVR, we have isomorphisms $X_e S_e \otimes_{S_e} B \cong B$ and with this identification our direct system becomes 
    \[
        B \xrightarrow{h^{1 - {1/p}}} B \xrightarrow{h^{1/p - {1/p^2}}} B \xrightarrow{h^{1/p^2 - {1/p^3}}} \dots .
    \]
    That colimit is identified with $\sqrt{h B}$ by \autoref{BS - root} \autoref{BS - root.RadicalComputation}, and \autoref{isom MB and TBD} follows. 

    For \autoref{def alpha D Gamma}, note that for any $e$, there is a natural map 
    \[
        M_B(\mathfrak d_e')^{\mathrm{c1}}\to (R(D) \tensor_R M_B(\mathfrak d_e))^{\mathrm{c1}} \cong ((B\tensor_R R(D))\tensor_B M_B(\mathfrak d_e))^{\mathrm{c1}}
    \]
    coming from the natural morphism 
    \begin{equation}\label{eqn lim of sum to lim of Gamma}
    (S_e(D_e+\Gamma_e))^{\mathrm{c1}} \to (S_e(D+\Gamma_e))^{\mathrm{c1}}\cong (S_e(\Gamma_e)\tensor_R R(D))^{\mathrm{c1}}.
    \end{equation}
    This is compatible with transitions and codimension $1$-envelopes commute with filtered colimits so we can go to the colimit and obtain $\alpha_{D,\Gamma}$. To see the last statement, combine \autoref{SeDelta is sqrt} and \autoref{eqn lim of sum to lim of Gamma} into 
    \begin{equation}
    \big((\sqrt{I_{D}S_e} \otimes_R R(D) )^{\cOne}\otimes_{S_e} S_e(\Gamma_e) \big) ^{\cOne} \to  (S_e(\Gamma_e)\tensor_R R(D))^{\mathrm{c1}}
    \end{equation}
    and go to the colimit.
\end{proof}

\subsection{Non-pure ideals}
\label{subsec:nonpure}

\begin{definition}\label{nonpure ideal below delta}
Let $(R,\m)$ be an excellent local $\mathbf{S}_2$ ring of residue characteristic $p$ and let $\Delta$ be an effective Mumford $\Q$-divisor on $R$. Let $B$ be a perfectoid $R$-algebra admitting a limiting sequence for $\Delta$. Let $E=E_R(R/\m)$ and fix a socle element $u\in E$. We define
\[
    \np(B, {<}\Delta)
    \coloneqq \np^R (B, M_B({<}\Delta)) = 
    \Ann_B\left(
    u\tensor\mathbf 1_{{<}\Delta}
    \;\text{ in }\;E\tensor_RM_B({<}\Delta)
    \right).
\]
{Recall that the non-pure ideal $\np^R (B, M)$ corresponding to an $R$-algebra $B$ and pointed $B$-module $M$ was introduced in \autoref{def: non-pure ideals}.}
\end{definition}

\begin{corollary}\label{properties nonpure ideal below delta}
Let $(R,\m)$ be an excellent local $\mathbf{S}_2$ ring of residue characteristic $p$ and let $\Delta$ be an effective Mumford $\Q$-divisor on $R$. Let $B$ be a perfectoid $R$-algebra admitting a limiting sequence for $\Delta$. Then, any limiting sequence $\{\mathfrak d_e\}$ for $\Delta$ computes $\np(B, {<}\Delta)$ by
\begin{equation}\label{nonpure ideal tower formula}
    \np(B, {<}\Delta)= \bigcup_e \np(B, M_B(\mathfrak d_e)) = \bigcup_e
    \Ann_B\left(u\tensor\mathbf 1_{\mathfrak d_e}
    \;\text{ in }\;E\tensor_RM_B(\mathfrak d_e)\right).
\end{equation}
where the unions are ascending. 
Consequently, the pair $(R,\Delta)$ is perfectoid pure if and only if $\np(B, {<}\Delta)\neq B$ for some absolutely integrally
closed perfectoid $R$-algebra $B$.
\end{corollary}

\begin{proof} 
    Follows immediately from \autoref{lem: functoriality of non-pure ideals}\autoref{part direct limit of modules}, \autoref{part envelope} and \autoref{universal module detects perfectoid purity}.
\end{proof}

\begin{proposition}\label{base change nonpure ideal below delta}
With notation as in \autoref{nonpure ideal below delta}, let $B\to C$ be a map of perfectoid $R$-algebras. Then, $C$ admits a limiting sequence for $\Delta$ and $\np(B, {<}\Delta)\subseteq\np(C, {<}\Delta)\cap B.$
If, additionally, $B\to C$ is $\m$-completely faithfully flat, then
\[
\np(C, {<}\Delta)=\np(B, {<}\Delta)C,
\qquad
\np(C, {<}\Delta)\cap B=\np(B, {<}\Delta)
\]
and $\rho_{B,\Delta}$ is pure if and only if $\rho_{C,\Delta}$ is pure.
\end{proposition}

\begin{proof}
By \autoref{properties of M_B}, we can compose the realizations of a limiting sequence for $\Delta$ in $B$ with $B\to C$ to obtain realizations on $C$, so $C$ also admits a limiting sequence for $\Delta$. Then, we use \eqref{nonpure ideal tower formula} to compute the non-pure ideals. In this way the containment and the two equalities follow from \autoref{properties of M_B} and \autoref{lem: functoriality of non-pure ideals}.
The purity assertion follows from \autoref{properties nonpure ideal below delta}, since the contraction equality shows that $\np(B, {<}\Delta)$ is proper if and only if $\np(C, {<}\Delta)$ is proper.
\end{proof}

\begin{remark}\label{universal divisorial module principal case}
Let $R$ be an excellent $\mathbf{S}_2$ ring with $p$ in its Jacobson radical and $B$ a perfectoid $R$-algebra. Suppose that $f\in R$ is a nonzerodivisor that does not vanish along any codimension-one component of the singular locus (thus $\Delta=\lambda\Div_R(f)$ is a Mumford $\mathbb{Q}$-divisor). Write $\lambda=a/b$ with $a,b\in\Z_{>0}$, and suppose there is $Y\in B$ with $Y^b=f^a$ that admits a fixed compatible system of $p$-power roots in $B$. 
By \autoref{divisorial module computed by towers}, we may compute $M_B({<}\Delta)$ using the $(1-1/p^e)\Delta$-data $\mathfrak d_e=(S_e, \Delta_e, \iota_e)$ as described in \autoref{example.DeltaDatum}, where $S_e=R[x]/(x^{bp^e}-f^{a(p^e-1)})$, $\Delta_e=\Div(x)$, and $\iota_e$ sends $x$ to $Y^{1-1/p^e}$. Thus for each $e$, we have a pointed identification $M_{B}(\mathfrak d_e) \cong B$ with distinguished element $Y^{1-1/p^e}$ in $B$, and the natural map $M_B(\mathfrak d_e) \to M_B(\mathfrak d_{e+1})$ is  identified with multiplication by
$Y^{1/p^e-1/p^{e+1}}$ on $B$. It follows from \autoref{BS - root} \autoref{BS - root.RadicalComputation} that we have a pointed isomorphism $$M_B({<}\Delta)
\cong
\left((Y^{1/p^\infty})B\right)^{\mathrm{c1}}$$
which sends $\mathbf 1_{{<}\Delta}$ to $Y$. Since $\sqrt{YB}=\sqrt{fB}$, by \autoref{BS - root} \autoref{BS - root.RadicalComputation} and \autoref{pfd of ideal only depends on radical} we have that 
\[
\left((Y^{1/p^\infty})B\right)^-=(fB)_\perfd.
\]
It follows from \autoref{purity and codimension one envelope} and \autoref{rk limit of pure is pure} that $\rho_{B,\Delta}$ is pure if and only if $R\to(fB)_\perfd$ sending $1$ to $Y$ is pure. 
In other words, $(R,f^\lambda)$ is perfectoid pure in the sense of \autoref{definition pfd pure pair} if and only if $(R,\Delta)$ is perfectoid pure in the sense of \autoref{def perfd pure div pair}\footnote{If
$(R,f^\lambda)$ is perfectoid pure for some $\lambda>0$, then
\autoref{normalization is pfd pure pair}~(a), together with
\autoref{perfectoid pure coefficients at most one}, shows that $f$
cannot vanish along any codimension-one component of the singular
locus.}. 

In fact, by \autoref{lem: functoriality of non-pure ideals} \autoref{part p-comp pure} and unwinding the definition of pointed $B$-modules we have
\[
\np(B, {<}\Delta)
=
\left\{
b\in B\ \middle|\
R\xrightarrow{\,1\mapsto bY\,}(fB)_\perfd
\text{ is not pure}
\right\}.
\]
\end{remark}

\begin{remark}\label{perfectoid pure triples}
Let $R$ be an excellent $\mathbf{S}_2$ ring with $p$ in its Jacobson radical, let $\Delta$ be an effective Mumford $\Q$-divisor on $R$, let $f\in R$ be a nonzerodivisor that does not vanish along any codimension-one component of the singular locus, and let $t\in\Q_{>0}$. In light of \autoref{universal divisorial module principal case}, it is natural to say that the triple $(R,\Delta,f^t)$ is \emph{perfectoid pure} if $(R,\Delta+t\Div_R(f))$ is perfectoid pure. Equivalently, there is a perfectoid $R$-algebra $B$ that admits both a fixed compatible system of $p$-th power roots of $f$ and a limiting sequence for $\Delta$ such that for every $e \gg 0$ (equivalently $e\geq 1$), 
\begin{equation}
    \label{eq.perfectoid pure triples}
R\longrightarrow M_B({<}\Delta),
\qquad
1\longmapsto(f^{1/p^e})^{\lfloor t(p^e-1)\rfloor}\mathbf 1_{{<}\Delta},
\end{equation}
is pure: this can be seen with a similar argument as in \autoref{universal divisorial module principal case}. The sequence $\lfloor t(p^e-1)\rfloor/p^e$ may be replaced by any sequence $t_i = a_i/b_i$ of positive rational numbers, with $a_i, b_i \in \bZ_{>0}$, such that   $t_i<t$ and converges to $t$, as long as $f^{t_i}$ has a representative in $B$ (for instance, if $B$ is absolutely integrally closed). 
If $(R,\mathfrak m)$ is local, this is also equivalent to $(f^{1/p^e})^{\lfloor t(p^e-1)\rfloor}\notin\np(B, {<}\Delta)$ for every $e\gg0$. 
\end{remark}

\begin{corollary}\label{mixed characteristic compatible splitting via radical ideal}
Keep the hypotheses and notation of
\autoref{radical ideal computes universal divisorial module}.
Let $\phi\in\Hom_R(B,R)$. 
Consider the following conditions.
\begin{enumerate}
    \item $I_D$ is $\phi$-compatible.
    \item $\phi(\sqrt{I_DB})\subseteq I_D$.
    \item There is
    $\widetilde\phi\in\Hom_R(M_B({<}D),R)$ whose composite with the
     pointed map $\iota_{B,D}\colon B\to M_B({<}D)$ is $\phi$.
\end{enumerate}
Then (a)~$\Longrightarrow$~(b)~$\Longleftrightarrow$~(c). All three
conditions are equivalent whenever
$(I_DB)_{\mathrm{perfd}}=(\sqrt{I_DB})^-$. By
\autoref{pfd of ideal only depends on radical}, this holds if $I_D$
has a finite generating set whose elements admit compatible systems
of $p$-power roots in $B$. Whenever any of these conditions holds and $\phi$ is surjective, $(R,D)$ is perfectoid pure.
\end{corollary}

\begin{proof}
Suppose first that (b) holds, i.e., $\phi(\sqrt{I_DB})\subseteq I_D$. The rule
$k\tensor r\mapsto r\phi(k)$ defines a map
$\sqrt{I_DB}\tensor_RR(D)\to R$. Taking codimension-one envelopes
and applying \autoref{radical ideal computes universal divisorial module}
gives $\widetilde\phi$. If $h$ is a local equation for $D$, then
$\iota_{B,D}(a)=ah\tensor h^{-1}$, so
$\widetilde\phi\circ\iota_{B,D}=\phi$ in codimension one and hence
everywhere, which shows (b)~$\implies$~(c).

Conversely, suppose that $\widetilde\phi$ exists as in
(c). At any
codimension-one point {$P$} where $h$ is a local equation for $D$, every
$k\in\sqrt{I_DB}$ satisfies
$\phi(k)=h\widetilde\phi(k\tensor h^{-1})\in hR_P$.
 As $I_D$ is $\mathbf{S}_2$, this gives
$\phi(\sqrt{I_DB})\subseteq I_D$.
This proves (c)~$\Longleftrightarrow$~(b). Moreover, if $\phi$ is surjective with $\phi(b)=1$ for some $b\in B$,
$m\mapsto\widetilde\phi(bm)$ sends $\mathbf1_{{<}D}$ to
$\phi(b)=1$, and hence splits the distinguished map
$R\to M_B({<}D)$. Thus $(R,D)$ is perfectoid pure by
\autoref{universal module detects perfectoid purity}.

Finally, 
the implication (a)$\implies$(b) follows  as $\sqrt{I_D B} \subseteq (I_D B)_{\perfd}$ (see \autoref{pfd of ideal only depends on radical}); when
$(I_DB)_{\mathrm{perfd}}=(\sqrt{I_DB})^-$, the converse follows as in
\cite[proof of Lemma~2.21]{FayolleCentersPerfectoidPurity}, since
$I_D$ is $p$-adically closed.
\end{proof}

\subsection{Localization, completion, and choice of $B$ for divisor pairs}
\label{subsec.PerdPureCompletionLocalizationChoice}

We now show that perfectoid purity of pairs can be detected locally (\autoref{localization of purity} and \autoref{can check purity locally}) or after completion (\autoref{completion of divisor pairs}).  In \autoref{perfectoid purity in characteristic p divisor} we also show that our notions coincide with notion of $F$-purity introduced in \cite{HaraWatanabe}.
In the complete local case, we show how to construct a specific $B$ that can determine perfectoid purity for pairs (\autoref{can test divisor pairs on RAinfty}).  
Finally, we show that perfectoid pure pairs remain so under normalization in \autoref{normalization is pfd pure pair}.

\begin{proposition}\label{localization of purity}
Let $R$ be an excellent $\mathbf{S}_2$ ring with $p$ in its Jacobson radical, and let $\mathfrak q$ be a prime ideal containing $p$. For an effective Mumford $\Q$-divisor $\Delta$ on $R$, if $(R,\Delta)$ is perfectoid pure then so is $(R_{\mathfrak q},\Delta_{\mathfrak q})$.
\end{proposition}
\begin{proof}
 Suppose that $(R,\Delta)$ is perfectoid
pure, and fix a rational number $0<\varepsilon<1$. Choose a
perfectoid $R$-algebra $B$ and a $(1-\varepsilon)\Delta$-datum
$\mathfrak d=(S,D,\iota)$ on $B$ for which $\eta_{\mathfrak d}$ is
pure. 
As in the proof of \autoref{prop:localization_ideal_pairs}, we set $
B_{\mathfrak q}=B\tensor_RR_{\mathfrak q}$, $C =B_{\mathfrak q}^{\wedge p}$, and note that $C$ is perfectoid and  $B_\mathfrak{q} \to C$ is $p$-completely pure.
\autoref{properties of M_B}
gives a $(1-\varepsilon)\Delta_{\mathfrak q}$-datum
$\mathfrak d_{\mathfrak q}$ on $C$ such that $\eta_{\mathfrak d_{\mathfrak q}}$ is pure. Since $\varepsilon$ was arbitrary,
$(R_{\mathfrak q},\Delta_{\mathfrak q})$ is perfectoid pure.
\end{proof}

\begin{proposition}[{\cf \cite[Lemma 4.10]{BMPSTWW3}}]\label{can check purity locally}
Let $R$ be an excellent $\mathbf{S}_2$ ring with $p$ in its Jacobson radical. Let $\Delta$ be an effective Mumford $\Q$-divisor on $R$. Then $(R,\Delta)$ is perfectoid pure if and only if $(R_\fram,\Delta_\fram)$ is perfectoid pure for every maximal
ideal $\fram$ of $R$.
\end{proposition}

\begin{proof}
The forward direction follows from \autoref{localization of purity}.

For the converse, suppose that $(R_\fram,\Delta_\fram)$ is perfectoid pure for every maximal ideal $\fram$, and fix a rational number $0<\varepsilon<1$. Use \autoref{construction SnD} to construct a $(1-\varepsilon)\Delta$-lift $(S,D)$, taking the trivial lift when $\Delta=0$.
By \autoref{common perfectoid algebra for divisor pairs}, for every $\fram$, there is an absolutely integrally closed perfectoid $R_\m$-algebra $B(\fram)$ admitting the datum $\mathfrak d(\fram) = (S_\m,D_\m, \iota_\m)$ with $R\to M_B(\mathfrak d_\m)$ pure. Combining these gives an $R$-algebra map $S\to B\coloneqq\prod_\fram B(\fram)$ where $B$ is perfectoid. Let $\mathfrak d$ be the resulting $(1-\varepsilon)\Delta$-datum on $B$.
Projecting onto the $\fram$-th factor gives a pointed map $M_B(\mathfrak d)_\fram\to M_{B(\fram)}(\mathfrak d(\fram))$ whose composite with $(\eta_{\mathfrak d})_\fram$ is the pure map $\eta_{\mathfrak d(\fram)}$ so $(\eta_{\mathfrak d})_\fram$ is also pure. Since purity can be checked locally and $\varepsilon$ was arbitrary, $(R,\Delta)$ is perfectoid pure.
\end{proof}

\begin{proposition}[{\cf \cite[Lemma 4.8]{BMPSTWW3}}]\label{completion of divisor pairs}
    Let $(R,\fram)$ be an excellent $\mathbf{S}_2$ ring of residue characteristic $p>0$. Let $\Delta$ be an effective Mumford $\Q$-divisor on $R$. Then $(R,\Delta)$ is perfectoid pure if and only if $(\widehat R,\widehat\Delta)$ is perfectoid pure, where $\widehat\Delta$ is the flat pullback of $\Delta$ along the regular map $R\to\widehat R$.
    More precisely, the following comparisons hold.
    \begin{enumerate}
        \item If $C$ is an absolutely integrally closed perfectoid
            $\widehat R$-algebra, 
            $
            \np^R(C, {<}\Delta)=\np^{\widehat R}(C, {<}\widehat\Delta).
            $
            \label{completion of divisor pairs.NonsplittingIdeal1}
        \item Let $B$ be an absolutely integrally closed perfectoid
            $R$-algebra, set $\widehat B=B^{\wedge\fram}$, which is
            perfectoid, and let
            $\widehat B\to C$ be a $p$-completely faithfully flat map to an
            absolutely integrally closed perfectoid ring. Then $C$ is naturally
            an $\widehat R$-algebra and
            \[
            \np^{\widehat R}(C, {<}\widehat\Delta)
            =\np^R(B, {<}\Delta)C,
            \qquad
            \np^{\widehat R}(C, {<}\widehat\Delta)\cap B
            =\np^R(B, {<}\Delta).
            \]
            \label{completion of divisor pairs.NonsplittingIdeal2}
    \end{enumerate}
\end{proposition}

\begin{proof}
Suppose first that $(R,\Delta)$ is perfectoid pure, and fix a rational
number $0<\varepsilon<1$. Choose a perfectoid $R$-algebra $B$ and a
$(1-\varepsilon)\Delta$-datum
$\mathfrak d=(S,D,\iota)$ on $B$ such that $\eta_{\mathfrak d}$ is
pure. Let $C$ be the $\fram B$-adic completion of $B$. Since
$p\in\fram B$, \autoref{prop: properties of perfectoids} shows that
$C$ is perfectoid and agrees with the derived $\fram B$-adic
completion. The map $R\to C$ extends to $\widehat R\to C$, and hence,
by \autoref{part functoriality of M_B} and \autoref{part flat base change of M_B} of \autoref{properties of M_B}, there is a pointed isomorphism
$M_C\bigl(\,\widehat{\mathfrak d}\,\bigr)\cong C\tensor_B M_B(\mathfrak d)$.
Since $\varepsilon$ was arbitrary \autoref{completion of divisor pairs.NonsplittingIdeal2} follows from \autoref{properties nonpure ideal below delta} together with \autoref{lem: functoriality of non-pure ideals} \autoref{part: invariance of field extension} \autoref{part p-comp pure} and \autoref{part extension for p completely flat}.
In particular, $(\widehat R,\widehat\Delta)$ is perfectoid pure.

Conversely, suppose that $(\widehat R,\widehat\Delta)$ is perfectoid
pure, and fix a rational number $0<\varepsilon<1$. Choose a
$(1-\varepsilon)\Delta$-lift $(S,D)$ supplied by
\autoref{construction SnD}, taking the trivial lift when $\Delta=0$.
Its flat pullback
$(\widehat S,\widehat D)$, where
$\widehat S=S\tensor_R\widehat R$, is a
$(1-\varepsilon)\widehat\Delta$-lift. Choose an absolutely integrally
closed perfectoid $\widehat R$-algebra $C$ as in
\autoref{common perfectoid algebra for divisor pairs}. By \autoref{datum associated to SnD construction}, $C$ realizes $(\widehat S,\widehat D)$, say with datum $\widehat{\mathfrak d}$. Note that $\eta_{\widehat{\mathfrak d}}$ is pure by \autoref{common perfectoid algebra for divisor pairs} since $(1-\varepsilon)\widehat\Delta\prec\widehat\Delta$. Via $S\to\widehat S\to C$, this realization also gives a $(1-\varepsilon)\Delta$-datum $\mathfrak d$ on $C$. 
By \autoref{properties of M_B} $M_C (\mathfrak d) = M_C(\widehat{\mathfrak d})$, so
\autoref{lem: functoriality of non-pure ideals} \autoref{part: invariance of field extension}
then implies \autoref{completion of divisor pairs.NonsplittingIdeal1}.
In particular, $(R, \Delta)$ is perfectoid pure. 
\end{proof}

We now compare perfectoid purity with the characteristic $p$ definition of \cite{HaraWatanabe}.  The purity that appears in \autoref{eq.perfectoid purity in characteristic p} was originally called \emph{$F$-purity} of the pair $(R, \Delta)$ in \cite{HaraWatanabe}.  Sometimes it is also called \emph{weak $F$-purity} to distinguish it from stronger notions.

\begin{proposition}\label{perfectoid purity in characteristic p divisor}
Let $R$ be an excellent reduced $\mathbf{S}_2$ ring of characteristic $p>0$,  and $\Delta$ an effective Mumford $\Q$-divisor on $R$. Then $(R,\Delta)$ is perfectoid pure if and only if the map
\begin{equation}
    \label{eq.perfectoid purity in characteristic p}
    R\longrightarrow
    R\bigl(\lfloor(p^e-1)\Delta\rfloor\bigr)^{1/p^e},
    \qquad 1\longmapsto1,
\end{equation}
is pure for all $e\gg0$. When $R$ is $F$-finite, purity here is
equivalent to splitting.
\end{proposition}

\begin{proof}

Set
$\Delta_e=p^{-e}\lfloor(p^e-1)\Delta\rfloor$.  \autoref{construct.LimitingConstructionWithComparisons} shows that $R_\perf=\varinjlim_jR^{1/p^j}$ admits a limiting sequence for $\Delta$ so if $(R,\Delta)$ is perfectoid pure, $R\to M_{R_\perf}({<}\Delta)$ is pure.
Since for any $e$, $\Delta_e\prec \Delta$, we must have that $R\to R_\perf \tensor_{R^{1/p^e}} R(p^e\Delta_e)^{1/p^e}$ is pure. Since this factors through $R\to R(p^e\Delta_e)^{1/p^e}$, this is also pure, as desired. 

Conversely, suppose that $R\to R(p^e\Delta_e)^{1/p^e}$ is pure for all $e\gg 0$. Since this implies that $R$ is $F$-pure, $R^{1/p^e}\to R_\perf$ is pure. Tensoring this with $R(p^e\Delta_e)^{1/p^e}$ and precomposing with the pure map  $R\to R(p^e\Delta_e)^{1/p^e}$ gives that $R\to R_\perf\tensor_{R^{1/p^e}} R(p^e\Delta_e)^{1/p^e}$ is pure. Since a subsequence of the $\Delta_e$s gives a limiting sequence for $\Delta$, $(R,\Delta)$ is perfectoid pure by \autoref{universal module detects perfectoid purity}.
\end{proof}

We also obtain the analogous result to \autoref{can test principal pairs on RAinfty} for divisor pairs.

\begin{proposition}\label{can test divisor pairs on RAinfty}
Let $(R,\fram,k)$ be a complete Noetherian $\mathbf{S}_2$ local ring of residue characteristic $p>0$, and let $\Delta$ be an effective Mumford $\Q$-divisor on $R$. Choose $f\in\fram$ for $\Delta$ as in \autoref{construction SnD}. For any $\varphi\colon A\to R$ as in \autoref{notation RAinfty} with $\varphi(x_1)=f$ (such a map may be obtained by adjoining a variable mapping to $f$ to any $A$ as in \autoref{notation RAinfty}), there are rational numbers $0<t_e<t_{e+1}<1$ for $e\gg0$, with $t_e\to1$, and $t_e\Delta$-data $\mathfrak d_e=(S_e,D_e,\iota_e)$ on $\raftypfd$. For these data, $(R,\Delta)$ is perfectoid pure if and only if
$\eta_{\mathfrak d_e}\colon
R\longrightarrow M_{\raftypfd}(\mathfrak d_e)$
is pure for every $e\gg0$.
Equivalently, as these data also form a limiting sequence for $\Delta$, $(R,\Delta)$ is perfectoid pure if and only if $\rho_{\raftypfd,\Delta}\colon R\to M_{\raftypfd}({<}\Delta)$ is pure.
\end{proposition}

\begin{proof}
The existence of the $t_e$ data is \autoref{construct.LimitingConstructionWithComparisons} and the backwards direction is direct from the definition. To show the forward direction, suppose that $(R,\Delta)$ is perfectoid pure and use \autoref{common perfectoid algebra for divisor pairs} to fix an absolutely integrally closed perfectoid $R$-algebra $B$ on which we can test the perfectoid purity of $(R,\Delta)$. Just as in \autoref{can test principal pairs on RAinfty}, $R\to B$ factors through $\raftypfd$ and base changing allows us to realize every $(S_e,D_e)$ on $B$, say with data $\mathfrak{d}_{e,B}$. Since each resulting $\eta_{\mathfrak{d}_{e,B}}$ is pure and factors through $\eta_{\mathfrak{d}_{e}}$ by \autoref{properties of M_B}\autoref{part functoriality of M_B}, these latter maps must also be pure, as desired.
\end{proof}

Finally, we generalize \cite[Corollary 5.5]{FayolleCentersPerfectoidPurity} to the context of pairs.

\begin{proposition}\label{normalization is pfd pure pair}
    Let $R$ be an excellent $\mathbf{S}_2$ ring with $p$ in its Jacobson radical, $R^N$ its normalization, $\mathfrak c$ the conductor of $R$ in $R^N$. Let $C$ be the divisorial part of $\mathfrak c$ so that $R^N(-C)=\mathfrak c^{**}\eqqcolon I_C$.
    \begin{enumerate}
        \item If $f\in R$ is a nonzerodivisor, $\lambda\in\Q_{>0}$, and
            $(R,f^\lambda)$ is perfectoid pure, then
            $(R^N,C+\lambda\Div_{R^N}(f))$ is perfectoid pure and $\Div_{R^N}(f)$ and $C$ have no component in common.  \label{normalization is pfd pure pair.PrincipalCase}
        \item If
            $\Delta=\sum_i a_iP_i$ is an effective Mumford $\Q$-divisor, let
            $P_i^N$ be the unique prime divisor of $R^N$ lying over $P_i$ and set
            $\Delta^N=\sum_i a_iP_i^N$. If $(R,\Delta)$ is perfectoid pure, then
            $(R^N,\Delta^N+C)$ is perfectoid pure.  \label{normalization is pfd pure pair.DivisorCase}
    \end{enumerate}
\end{proposition}

\begin{proof}

By \autoref{universal divisorial module principal case}, part~\autoref{normalization is pfd pure pair.DivisorCase} for
$\Delta=0$ follows from \autoref{normalization is pfd pure pair.PrincipalCase} with $f=1$; hence assume
$\Delta\neq0$. 

The formation of $R^N$ and $C$ commutes with localization.
If $(R,\fram)$ is local and $\widehat R=R^{\wedge\fram}$, then $R^N\tensor_R\widehat R=(\widehat R)^N$ by \cite[Scholie 7.8.3(vii)]{EGAIV2} and {in \autoref{normalization is pfd pure pair.PrincipalCase}}
principal divisors commute with this base change.
In \autoref{normalization is pfd pure pair.DivisorCase}, the ring $R_{P_i}$ is a DVR for every $i$, and hence $R^N\tensor_RR_{P_i}=R_{P_i}$. Thus $P_i^N$ is well-defined, and the conductor is the unit ideal there. In particular, $\Delta^N$ and $C$ have no common prime component. 
If 
$\widehat\Delta$ is the flat pullback of $\Delta$, the coefficient-wise 
definition gives 
$(\widehat\Delta)^N=\widehat{\Delta^N}$; here a component may split, 
but every prime above it receives the same coefficient. 
Flat base commutes with the construction of the conductor and, by \autoref{flat keeps reflexive}, so does the construction of $I_C$. Hence \autoref{can check purity locally}, \cite[Proposition~3.4]{FayolleCentersPerfectoidPurity}, and \autoref{completion of divisor pairs}, applied also to the local factors of $R^N$, reduce both \autoref{normalization is pfd pure pair.PrincipalCase} and \autoref{normalization is pfd pure pair.DivisorCase} to the case where
$(R,\fram)$ is complete local.

If $(R,\Delta)$ or $(R,f^\lambda)$ is perfectoid pure, $R$ is perfectoid pure so $\mathfrak c$ is radical in both $R$ and $R^N$ by \cite[Corollary~2.33, Proposition~5.2, and Corollary~5.5] {FayolleCentersPerfectoidPurity} and $C$ is reduced hence has integer coefficients. 
Note that the localization of $\mathfrak c$ at every prime component $Q$ of $C$ is the maximal ideal of the DVR $(R^N)_Q$. On the other hand, the localization of $\mathfrak c$ at every $P_i^N$ in \autoref{normalization is pfd pure pair.DivisorCase} is the unit ideal.  In the semilocal ring obtained by localizing $R^N$ at these primes and at the generic points of the remaining irreducible components, $\mathfrak c$ is invertible, hence principal. Clearing a denominator from a generator gives a nonzerodivisor $g\in\mathfrak c\subseteq R$ having order one at every such $Q$ and, in \autoref{normalization is pfd pure pair.DivisorCase}, order zero at every $P_i^N$. Thus $\Div_{R^N}(g)=C+G$ for an effective divisor $G$ having no component in common with $C$. 
For \autoref{normalization is pfd pure pair.DivisorCase}, since $\Delta$ and $C$  have no component in common, we can similarly choose $h\in R$ with $\Div_{R^N}(h)=(\Delta^N+C)_{\red}+G'$ for an effective divisor $G'$ having no component in common with $\Delta+ C$: instead of looking at $\mathfrak{c}$ in the semilocal ring above, look at $\mathfrak{c}\prod_i P_i^N$.

Let $A\to R$ be as in \autoref{notation RAinfty}. By potentially enlarging $A$, we can assume $\raftypfd$ contains compatible systems of $p$-th power roots for $g$, generators of $I_C$, and in the case of \autoref{normalization is pfd pure pair.DivisorCase}, $h$. 
By \autoref{can test divisor pairs on RAinfty} and \autoref{can test principal pairs on RAinfty}, in both \autoref{normalization is pfd pure pair.PrincipalCase} and \autoref{normalization is pfd pure pair.DivisorCase},
we can assume our pure (hence split since $R$ is complete) maps
are with the perfectoid ring $\raftypfd$. By \cite[Corollary 5.5]{FayolleCentersPerfectoidPurity}, any $\phi\in \Hom_R(\raftypfd,R)$ extends to $\phi^N\in \Hom_{R^N}(R^{N,A_\infty}_\perfd, R^N)$. Moreover, \cite[Lemma 5.1]{FayolleCentersPerfectoidPurity} shows that $\mathfrak{c}$ is compatible with such extension.
By \cite[Proposition 2.35]{FayolleCentersPerfectoidPurity} applied on each complete local factor of $R^N$, all the height $1$ associated primes of $\mathfrak{c}$ are compatible with such extensions and therefore so is their intersection, $I_C$ (\cite[Proposition 2.34]{FayolleCentersPerfectoidPurity}).
 
To prove \autoref{normalization is pfd pure pair.PrincipalCase}, fix $e \gg 0$. By \autoref{can test principal pairs on RAinfty}, there is 
$\phi_e\in\Hom_R(\raftypfd,R)$ with
$$
\phi_e\bigl((f^{1/p^e})^{\lfloor\lambda(p^e-1)\rfloor}\bigr)=1.
$$
This map extends to $
\phi_e^N\in
\Hom_{R^N}(R^{N,A_\infty}_{\mathrm{perfd}},R^N)
$
with $
\phi_e^N\bigl((f^{1/p^e})^{\lfloor\lambda(p^e-1)\rfloor}\bigr)=1.
$
Since $I_C$ is $\phi_e^N$-compatible, 
\autoref{mixed characteristic compatible splitting via radical ideal}
gives
$
\widetilde\phi_e\in
\Hom_{R^N}\bigl(
M_{R^{N,A_\infty}_{\mathrm{perfd}}}({<}C),R^N
\bigr)
$
 such that
$$
\widetilde\phi_e\bigl(
(f^{1/p^e})^{\lfloor\lambda(p^e-1)\rfloor}
\mathbf1_{{<}C}
\bigr)=1.
$$
Applying \autoref{perfectoid pure triples} proves (a). If $f$ and
$C$ had a common height-one component $Q$, then localizing
$\phi_e^N$ at $Q$ would give $1\in Q(R^N)_Q$, a contradiction.

For \autoref{normalization is pfd pure pair.DivisorCase}, if $C=0$ then since $R$ is $\mathbf{S}_2$, we see that $R=R^N$ and there is nothing to show.  Thus we may assume that $C\neq0$. 
Since $R$ is complete, perfectoid purity of $(R,\Delta)$ gives a splitting
\[
\widetilde\phi_\Delta\colon M_{\raftypfd}({<}\Delta)\longrightarrow R
\]
of $\rho_{\raftypfd,\Delta}$
and so set 
$
\phi
\coloneqq
\widetilde\phi_\Delta\circ\iota_{\raftypfd,\Delta}\colon
\raftypfd\longrightarrow R.
$
Let $\phi^N$ be its extension to $R^{N,A_\infty}_\perfd$. Since $I_C$ is $\phi^N$-compatible, 
\autoref{mixed characteristic compatible splitting via radical ideal}
gives a splitting
\[
\widetilde\phi_C\colon
M_{R^{N,A_\infty}_{\mathrm{perfd}}}({<}C)\longrightarrow R^N,
\qquad
\widetilde\phi_C\circ
\iota_{R^{N,A_\infty}_{\mathrm{perfd}},C}=\phi^N.
\]

Set
\[
U_\Delta=\Spec R^N\setminus V(\mathfrak{c}),
\qquad
U_C=\Spec R^N\setminus\Supp(\Delta^N),
\qquad
U=U_\Delta\cup U_C.
\]
Over $U_\Delta$, $C$ vanishes and the pullback square in \cite[Lemma~5.1]{FayolleCentersPerfectoidPurity} shows that $\raftypfd|_{U_\Delta}\to R^{N,A_\infty}_{\mathrm{perfd}}|_{U_\Delta}$ is an isomorphism. By \autoref{universal divisorial module is local in the divisor}, 
\[
\reallywidetilde {M_{R^{N,A_\infty}_{\mathrm{perfd}}}({<}\Delta^N+C)}\Big|_{U_{\Delta}} \cong \reallywidetilde {M_{\raftypfd}({<}\Delta)}\Big|_{U_{\Delta}} 
\]
so $\widetilde\phi_\Delta$ gives a splitting over $U_{\Delta}$. Similarly, 
\[
\reallywidetilde {M_{R^{N,A_\infty}_{\mathrm{perfd}}}({<}\Delta^N+C)}\Big|_{U_C} \cong \reallywidetilde {M_{R^{N,A_\infty}_{\mathrm{perfd}}}({<}C)}\Big|_{U_C} 
\]
so $\widetilde\phi_C$ gives a splitting over $U_C$.
Moreover, since both divisors vanish on $U_\Delta\cap U_C$,
\[
\reallywidetilde {M_{R^{N,A_\infty}_{\mathrm{perfd}}}({<}\Delta^N+C)}\Big|_{U_\Delta\cap U_C} \cong \reallywidetilde {(R^{N,A_\infty}_{\mathrm{perfd}})^{\mathrm{c1}}}\Big|_{U_\Delta\cap U_C},
\]
and $\widetilde\phi_C\big|_{U_\Delta\cap U_C}=\widetilde\phi_\Delta \big|_{U_\Delta\cap U_C} = \widetilde{(\phi^N)}^{\cOne}\big|_{U_\Delta\cap U_C}$. In particular, $\widetilde\phi_C$ and $\widetilde\phi_\Delta$ glue to a splitting over the big open $U_\Delta\cup U_C$ so the result follows from \autoref{check purity on codim 1}  and \autoref{universal module detects perfectoid purity}.
{
The compositions of the two splittings with $R^{N,A_\infty}_{\mathrm{perfd}}\to \bigl(R^{N,A_\infty}_{\mathrm{perfd}}\bigr)^{\mathrm{c1}}$ both equal $\phi^N$. Since this map is an isomorphism in codimension at most one and $\cO_{U_\Delta\cap U_C}$ is $\mathbf{S}_1$, the two splittings agree. They therefore glue to a splitting of the distinguished map over $U$.
Since $C$ and $\Delta^N$ have no common prime component, $\Spec R^N\setminus U$ has codimension at least two and thus $\Gamma(U,\cO_U)=R^N$, \autoref{check purity on codim 1} shows that
\[
\rho_{R^{N,A_\infty}_{\mathrm{perfd}},\Delta^N+C}\colon
R^N\longrightarrow
M_{R^{N,A_\infty}_{\mathrm{perfd}}}({<}\Delta^N+C)
\]
is pure. The result follows from
\autoref{universal module detects perfectoid purity}.
}
\end{proof}

\section{Comparison with log canonical singularities}
\label{subsec.PerfdPureVsLC}

{
In this subsection, we show that a perfectoid pure pair $(R,\Delta)$
for which $K_R+\Delta$ is $\Q$-Cartier {and $R$ is deminormal} 
is {semi-}log canonical. We briefly sketch the idea of the overall argument here to make it a bit easier to follow. After reducing to the complete local
mixed characteristic setting, let $0 \neq r \in R$ be such that $(R[r^{-1}],\Delta R[r^{-1}])$ is log canonical. By
\autoref{different characterization log canonical}, it is enough to
show $r\in\sJ(R,\Delta+(1-p^{-e})\Div_R(r))$ for every
$e\gg0$. Furthermore, as the BCM test ideal $\tau_B$ of \cite[Definitions~6.2 and~6.9]{MaSchwedeSingularitiesMixedCharBCM} is contained in the multiplier ideal \cite[Theorem~6.21]{MaSchwedeSingularitiesMixedCharBCM}, it suffices to show $r \in \tau_B(R,\Delta+(1-p^{-e})\Div_R(r))$ for a sufficiently large perfectoid big Cohen--Macaulay $R^+$-algebra $B$.
For carefully chosen $r = fg$ as in \autoref{notation lc comparison}, this is verified in two steps, proving the membership (\autoref{damped canonical module comparison}) and containment (\autoref{damped test ideal membership}) in
\[
r\in
\im\left(
\Hom_R\!\left((C\tensor_RR(-K_R))^{\mathrm{c1}},R\right)
\xrightarrow{\ \operatorname{ev}_{\widetilde q_e}\ }R
\right)\subseteq
\tau_B\left(R,\Delta+(1-p^{-e})\Div_R(r)\right),
\]
where $C$ is an appropriately chosen absolutely integrally closed perfectoid $R$-algebra for which $\rho_{C,\Delta}$ is pure and $\widetilde{q}_e$ is to be understood as some $K_R+\Delta$ twist of $f^{1-p^{-e}}g^{1-p^{-2e}}$.
}

\begin{lemma}\label{different characterization log canonical}
Let $X=\Spec R$ be an excellent normal scheme {with a dualizing complex}, and let $\Delta$ be an effective $\Q$-divisor on $X$ such that $K_X+\Delta$ is $\Q$-Cartier.
Suppose that there is a nonzerodivisor $r\in R$ such that for every proper birational morphism $\pi\colon Y\to X$ with $Y$ normal,
\begin{equation}\label{lc mult ideal}
r\in \pi_*\sO_Y\left(
\left\lceil K_Y-\pi^*(K_X+\Delta)
-(1-\varepsilon)\Div_Y(r)\right\rceil
\right)
\end{equation}
for arbitrarily small $\varepsilon>0$. Then $(R,\Delta)$ is log canonical.
\end{lemma}

\begin{proof}
For any such $\pi$, \autoref{lc mult ideal} is equivalent to
\[
\begin{aligned}
1\in{}&\pi_*\sO_Y\left(\Div_Y(r)+\left\lceil K_Y-\pi^*(K_X+\Delta)-(1-\varepsilon)\Div_Y(r)\right\rceil\right)\\
&\qquad=\pi_*\sO_Y\left(\left\lceil K_Y-\pi^*(K_X+\Delta)+\varepsilon\Div_Y(r)\right\rceil\right).
\end{aligned}
\]
This could happen only if, at every prime divisor $E$ on $Y$, the coefficient of $\lceil K_Y-\pi^*(K_X+\Delta)+\varepsilon\Div_Y(r)\rceil$ at $E$ is $\geq 0$. Equivalently, the coefficient $b_{E,\varepsilon}$ of $K_Y-\pi^*(K_X+\Delta)+\varepsilon\Div_Y(r)$ at $E$ is $>-1$. Letting $a_E$ be the coefficient of $K_Y-\pi^*(K_X+\Delta)$ at $E$, we want to show $a_E\geq -1$. If $a_E< -1$ then for any sufficiently small $\varepsilon$, $b_{E,\varepsilon}=a_E+\varepsilon \ord_E(r)<-1$, a contradiction.
\end{proof}

{

\begin{notation}\label{notation lc comparison}
Let $(R,\m)$ be a complete normal local domain of mixed
characteristic $(0,p)$, and let $\Delta\geq0$ be a $\Q$-divisor. We fix an effective canonical divisor
$K_R$ and assume that $K_R+\Delta$ is $\Q$-Cartier. 
\begin{description}
\item[\fbox{$n,h,f$}]
Choose $n>0$ and $h\in R$ such that
$\Div_R(h)=n(K_R+\Delta)$, taking $n=1$ and $h=1$ if
$K_R+\Delta=0$. Choose a nonzero $f\in\m$ such that
$\Div_R(f)=(K_R+\Delta)_{\red}+G$, where $G\geq0$ has no component
in common with $K_R$ or $\Delta$.

\item[\fbox{$A, g$}] Choose a Noether--Cohen normalization $A\subseteq R$ and
$0\neq g\in\m\cap A$ such that
$A[g^{-1}]\to R[g^{-1}]$ is finite \'etale,
and $p,n,f,h$ are units in $R[g^{-1}]$. Form $\raftypfd$ as in \autoref{notation RAinfty}. 
\item[\fbox{$C,y$}]
Choose a $p$-completely faithfully flat map ${\raftypfd}\to C$,
where $C$ is absolutely integrally closed and perfectoid, and choose
$y\in C$ with $y^n=h$. For every $e\geq1$, put
$q_e=yf^{1-p^{-e}}g^{1-p^{-2e}}\in C$. { Since $C$ is absolutely integrally closed, it admits a limiting sequence for $\Delta$ by \autoref{construct.LimitingConstructionWithComparisons} so we can define $M_C({<}\Delta)$ with distinguished map $\rho_{C,\Delta}\colon R\to M_C({<}\Delta)$. }
\end{description}
\end{notation}

\begin{lemma}\label{C is R torsion free}
Let $C$ and $R$ be as in \autoref{notation lc comparison}. Then every nonzero element of $R$ is a nonzerodivisor on $C$. In particular, $(C\tensor_R R(-K_R))^\cOne$ is isomorphic to its image in $C^\cOne$,  $(R(-K_R)C)^\cOne$.
\end{lemma}
\begin{proof}
Let $0\neq r\in R$ be arbitrary.
By \autoref{rafty is R torsion free}, 
there exists an integer $N \geq 0$ such that
\[
(p^n\raftypfd:_{\raftypfd}r)
\subseteq p^{n-N}\raftypfd 
\]
for every $n>N$. Since $\raftypfd \to C$ is faithfully flat
modulo $p^n$, we get 
\[
(p^nC:_C r)\subseteq p^{n-N}C.
\]
In particular, if $rz = 0$ in $C$, then $z \in \bigcap_n p^{n-N}C = 0$ so $C$ is $r$-torsion free.
To show the second part of the lemma, it suffices to show that $(C\tensor_R R(-K_R))^\cOne\to C^\cOne$ is injective. The proof of \autoref{prop.SomeGeneralFactsAboutCodim1Envelopes}\autoref{prop.SomeGeneralFactsAboutCodim1Envelopes.IsoCheckedCodim1} shows that it is enough to check that this map is injective at every height $\leq 1$ prime of $R$. Let $P\in \Spec R$ be any such prime and let $r\in R$ be such that $R(-K_R)_P\cong rR_P$. Then, $(C\tensor_R R(-K_R))_P\cong C\tensor_R r R_P$, which injects into $C_P$ since $C$ is $R$-torsion free.
\end{proof}

\begin{lemma}\label{damped canonical module comparison}
In the setting of \autoref{notation lc comparison}, for every $e\geq1$, there are $C$-linear
maps
\[
M_C({<}\Delta)
\xrightarrow{\ \theta_e\ }
\left( R(-K_R)C\right)^{\mathrm{c1}}
\xrightarrow{\ \beta_e\ }
M_C({<}\Delta)
\]
such that $\theta_e( \mathbf 1_{<\Delta})=yf^{1-1/p^e}g^{1-{1/p^{2e}}}$ and $\beta_e\theta_e=fg\,\mathrm{id}$. In particular, if $(R,\Delta)$ is
perfectoid pure and if we set $\tld{q_e}$ to be the element corresponding to $yf^{1-1/p^e}g^{1-{1/p^{2e}}}$ in $(C \otimes_R R(-K_R))^{\cOne} \cong (R(-K_R)C)^{\cOne}$, then
\[
fg\in
\im\left(
\Hom_R\!\left((C\tensor_RR(-K_R))^{\mathrm{c1}},R\right)
\xrightarrow{\ \operatorname{ev}_{\tld q_e}\ }R
\right).
\]
\end{lemma}

{

\begin{proof}
Work with $S_i=R[x_i]/(x_i^{mp^i}-f)$ as in \autoref{construct.LimitingConstructionWithComparisons}, using the $f$ chosen in \autoref{notation lc comparison} and $m$ so that $m\Delta$ has integer coefficients. 

Define $\theta_{e,i}\colon M_C(\mathfrak{d}_i)^\cOne \to (R(-K_R)C)^\cOne$ by $\alpha\tensor t \mapsto \alpha t yf^{1-1/p^e}g^{1-{1/p^{2e}}}$. 
We need to check that the image lands in $(R(-K_R)C)^\cOne$, which can be done locally at any prime of height $\leq 1$ of $R$. Let $P$ be such a prime. If $P$ is not in the support of $K_R+\Delta$ then both $M_C(\mathfrak{d}_i)_P$ and $(R(-K_R)C)_P$ are isomorphic to $C_P$ so there is nothing to check. Suppose now that $P$ is in the support of $K_R+\Delta$ and write $k,\delta,s$  for the respective coefficients of $K_R, \Delta, \Div_R(g)$ at $P$. Note that the coefficient of $f$ at $P$ is $1$.
Then, \autoref{construct.LimitingConstructionWithComparisons} applied with our choice of $f$ gives $(S_i(-K_R)C)_P=C_Px_i^{mp^ik}$ and 
$M_C(\mathfrak{d}_i)_{P}=C\tensor_{S_i}S_{i,P} x_i^{-(\delta mp^i-\epsilon_{P})}$ { where $\epsilon_P = 1$ if $P$ is in the support of $\Delta$ and $\epsilon_P = 0$ if not, in which case $\delta = 0$ so that $M_C(\mathfrak{d}_i)_{P}=C_P$}. 
Abusing notation and denoting by $x_i$ its image in $C$, we claim that there is $c_i\in C_P$ with
\begin{equation}\label{nts multiplication}
yf^{1-1/p^e}g^{1-1/p^{2e}}=c_{{i}}x_i^{\delta mp^i{-\epsilon_P}+mp^ik}.
\end{equation} 
Then, $\theta_{e,i}(\alpha \tensor t)= c_i \alpha t x_i^{\delta mp^i-{\epsilon_P}+mp^ik}$ which must be in $R(-K_R)C_P$ so $\theta_{e,i}$ has the correct target. Using that $\Div h=n(K_R+\Delta)$, we have
\[
(yf^{1-1/p^e}g^{1-1/p^{2e}})^{np^{2e}}=(u_hx_i^{mp^i(k+\delta)}u_f x_i^{mp^i(1-1/p^e)} u_g x_i^{mp^i s(1-1/p^{2e})})^{np^{2e}}=ux_i^b
\]
for some integer $b>(\delta mp^i-{\epsilon_P}+mp^ik)(np^{2e})$ and units $u,\ u_h,\ u_f,\ u_g$. 
Applying \autoref{BS - root}~(c) gives a unique $c_{i}\in \sqrt{x_i^{\delta mp^i-{\epsilon_P}+mp^ik}C_P}$ that satisfies \autoref{nts multiplication}, as desired. Note that $\theta_{e,i}$ behaves well with the transition maps $M_C(\mathfrak{d}_i)\to M_C(\mathfrak{d}_j)$ so we can go to the colimit to define $\theta_e$.

We proceed similarly and construct $\beta_{e,i}\colon (R(-K_R)C)^\cOne\to M_C(\mathfrak d_i)^\cOne$ by $\alpha\mapsto \alpha y^{-1}f^{1/p^e}g^{1/p^{2e}}\tensor 1$ { assuming $i \gg e$}. 
We need to show that the image indeed lands in $M_C(\mathfrak d_i)^\cOne$, which can be done locally at any prime of height $\leq 1$ of $R$. Let $P$ be such a prime. If $P$ is not in the support of $K_R+\Delta$ then there is nothing to check. Else, we claim that there is $d_i\in C_P$ with $x_i^{mp^i(k+\delta)-{\epsilon_P}}f^{1/p^e}g^{1/p^{2e}}=d_iy$. This will allow us to rewrite our morphism as $\alpha \mapsto \alpha  d_i x_i^{-(mp^ik)} \tensor x_i^{-(mp^i\delta-{\epsilon_P})}$ whose image is clearly in $M_C(\mathfrak d_i)_P$.
As before, raising everything to the $np^{2e}$ gives 
\[
(x_i^{mp^i(k+\delta)-{\epsilon_P}}f^{1/p^e}g^{1/p^{2e}})^{np^{2e}}=(x_i^{mp^i(k+\delta)-{\epsilon_P}} u_f x_i^{mp^i/p^e} u_g x_i^{mp^i s/p^{2e}})^{np^{2e}}=ux_i^b
\]
for some $b>mp^i(k+\delta)np^{2e}$ since $i \gg e$. Similarly, $
y^{np^{2e}}=u_hx_i^{mp^i(k+\delta)np^{2e}}$ so the existence of $d_i$ follows from \autoref{BS - root}~(c). Again, this behaves well with the transition maps $M_C(\mathfrak{d}_i)\to M_C(\mathfrak{d}_j)$ so we can pass to the colimit to define $\beta_e$. 

{
For the final statement, suppose that $(R,\Delta)$ is perfectoid  pure. Consider $A_1 = A \llbracket X \rrbracket \xrightarrow{X \mapsto f}R$, and set
$C'=\left(
C\tensor_{\raftypfd }R^{(A_1)_\infty}_\perfd
\right)^{\wedge p}$ which is perfectoid by 
\cite[Proposition 2.1.11(b)]{CesnaviciusScholzePurityFlatCohomology} and $R^{(A_1)_\infty}_\perfd\to C'$ remains $p$-completely faithfully
flat since $\raftypfd \to C$ is. By \autoref{can test divisor pairs on RAinfty},
$\rho_{R^{(A_1)_\infty}_\perfd,\Delta}$ is pure, hence
so is $\rho_{C',\Delta}$ by \autoref{properties of M_B}, and so $\rho_{C,\Delta}$ is pure by
\autoref{base change nonpure ideal below delta}. As $R$ is complete, we have an $R$-linear map $\psi \colon M_C(<\Delta) \to R$ sending ${\bf 1}_{<\Delta} \mapsto 1$.
Hence $fg = \psi \beta_e (\theta_e ({\bf 1}_{<\Delta}))$.  But now we are done thanks to the isomorphism $(C \otimes_R R(-K_R))^{\cOne} \cong (R(-K_R)C)^{\cOne}$.}
\end{proof}
}

\begin{proposition}\label{damped test ideal membership}
In the setting of \autoref{notation lc comparison}, suppose in addition that $(R,\Delta)$
is perfectoid pure. Then, for every $e\geq1$ and every integral
perfectoid balanced big Cohen--Macaulay $R^+$-algebra $B$,
\[
fg\in\tau_B\left(R,\Delta+(1-p^{-e})\Div_R(fg)\right).
\]
\end{proposition}
\begin{proof}
Fix $e$ and. Using \autoref{damped canonical module comparison} and its notation, we want to show that
\[
\im\left(
\Hom_R\!\left((C\tensor_RR(-K_R))^{\mathrm{c1}},R\right)
\xrightarrow{\ \operatorname{ev}_{\widetilde q_e}\ }R
\right)
\subseteq
\tau_B\big(R,\Delta+(1-p^{-e})\Div_R(fg)\big).
\]
By \cite[Definitions 6.2, 6.9]{MaSchwedeSingularitiesMixedCharBCM} and \autoref{canonical module evaluation}, it suffices to show 
\begin{equation}\label{coh kernel containment}
\ker\left(
H^d_\m(R)\xrightarrow{\ \cdot q_e^+\ }H^d_\m(B)
\right)
\subseteq
\ker\left(
H^d_\m(R)\xrightarrow{\ \cdot q_e\ }H^d_\m(C)
\right),
\end{equation}
for some specific choice of $q_e^+\in B$ with
$(q_e^+)^{np^{2e}}=h^{p^{2e}}((fg)^{1-1/p^e})^{np^{2e}}$.
Choose $y_+\in R^+$ with $y_+^n=h$ and $u_e,v_{2e}\in R^+$ with
$u_e^{p^e}=f$ and $v_{2e}^{p^{2e}}=g$, and put
$q_e^+=y_+u_e^{p^e-1}v_{2e}^{p^{2e}-p^e}\in R^+$. Set
\[
T_e=R[U,V,Y]/(U^{p^e}-f,\;V^{p^{2e}}-g,\;Y^n-h)
\]
and put $a_e=YU^{p^e-1}V^{p^{2e}-p^e}$ and
$b_e=YU^{p^e-1}V^{p^{2e}-1}$. Mapping $U,V,Y$ to
$f^{1/p^e},g^{1/p^{2e}},y$ in $C$ sends $b_e$ to $q_e$, while
mapping them to $u_e,v_{2e},y_+ \in R^+ \subseteq B$ sends $a_e$ to $q_e^+$.
By construction, the finite free $R$-algebra $T_e$ is
complete local with residue field\footnote{this follows from \autoref{notation lc comparison} because $g, f \in \fram$ and if $h \notin \fram$ then $h =1$ and $n = 1$} $R/\m$  and
$A [g^{-1}]\to T_e[g^{-1}]$ is finite étale (this is where we use that the fraction field has characteristic $0$).
Thus
\autoref{componentwise BCM transfer} applies to $A \to T_e$.

Let $\eta\in H^d_\m(R)$, and suppose that $q_e^+\eta=0$ in
$H^d_\m(B)$. To show \autoref{coh kernel containment}, we need to show $q_e\eta=0$ in $H^d_\fram(C)$. We first apply \autoref{componentwise BCM transfer} to
$N=T_e\tensor_RH^d_\m(R)$ and $a_e\tensor\eta$ to show that $(gA_\infty)_\perfd (1\tensor (a_e\tensor\eta))=0$ in $(T_e)^{A_\infty }_\perfd\tensor_{T_e}N$.
For every minimal
prime $P_i$ of $T_e$, finite freeness gives $P_i\cap R=(0)$, so
$R_i=T_e/P_i$ is a complete local domain and a finite extension of $R$. Choose an
$R$-embedding $R_i\to R^+$. Then $R_i^+=R^+$, and $B$ is balanced
big Cohen--Macaulay over $R_i$ by
\cite[Corollary~4.6]{DietzBigCohenMacaulayAlgebrasSeeds}. 
The images of $U,V,Y$ differ from
$u_e,v_{2e},y_+$ by roots of unity, so under the identification $B\tensor_{T_e}N\cong H^d_\m(B)$, the image of
$a_e\tensor\eta$ is a unit multiple of $q_e^+\eta$, and is therefore
zero. This verifies the componentwise hypothesis of
\autoref{componentwise BCM transfer} so $(gA_\infty)_\perfd (1\tensor (a_e\tensor\eta))=0$ in $(T_e)^{A_\infty }_\perfd\tensor_{T_e}N$.

The maps $A_\infty\to C$ and $T_e\to C$ together with the universal property of perfectoidization induce
$(T_e)^{A_\infty}_{\perfd}\to C$ so $(gA_\infty)_{\perfd}C$ annihilates $a_e\eta$ in
$H^d_\m(C)$. By
\cite[Lemma~2.19]{CaiLeeMaSchwedeTuckerPerfectoidSignature}, the
$p$-adic closure of this ideal is $(gC)_{\perfd}$. Since $a_e\eta$
is killed by a power of $p$, this closure also annihilates it, and
\autoref{root ideal tensor comparison}, applied to the cyclic
$R$-module generated by $a_e\eta$, shows that
$(g^{1/p^\infty})C$ annihilates $a_e\eta$. In particular,
$b_e\eta=V^{p^e-1}a_e\eta=0$, and hence
$q_e\eta=0$ in $H^d_\m(C)$, as desired.
\end{proof}

}

The next theorem generalizes \cite[Corollary 5.11]{BMPSTWW3} to the pair setting and removes the index not divisible by $p$ assumption in op.cit.

\begin{theorem}
\label{pfd pure implies lc}
Let $R$ be an excellent normal ring {admitting a dualizing complex}, with $p$ in its Jacobson radical. Let $\Delta$ be an effective $\Q$-divisor on $R$ such that $K_R+\Delta$ is $\Q$-Cartier.
If $(R,\Delta)$ is perfectoid pure, then it is log canonical.
\end{theorem}

\begin{proof}
{
By \autoref{can check purity locally} and
\autoref{completion of divisor pairs}, we may pass to the completion
at each maximal ideal. Log canonicity is local and descends from completion: base changing any proper birational modification along the
faithfully flat regular map $R\to\widehat R$ preserves its discrepancy divisor. We may therefore assume that $R$ is complete local.

{If $R$ has characteristic $p$, we use Hochster's
$\Gamma$-construction as developed in
\cite[Construction~3.1, Lemma~3.2, and Theorem~3.4]{MurayamaGammaConstruction}.
Write $\Delta=\sum_i a_iP_i$ and choose $\Gamma$ sufficiently small
simultaneously for $R$ and the finitely many $R/P_i$. Then $R^\Gamma$
is $F$-finite and normal and each $P_iR^\Gamma$ is prime.
Thus $\Delta^\Gamma=\sum_i a_iP_iR^\Gamma$ is the flat pullback of
$\Delta$. By \autoref{perfectoid purity in characteristic p divisor} and
\cite[Lemma~3.5]{CarvajalRojasSchwedeTuckerBertini},
$(R^\Gamma,\Delta^\Gamma)$ is $F$-pure, and hence log canonical by
\cite[Proposition~2.2(4) and Theorem~3.3]{HaraWatanabe}. If $(R,\Delta)$ were not log
canonical, choose a normal proper birational map
$\mu\colon Y\to\Spec R$ and a prime divisor $E$ on $Y$ with
discrepancy less than $-1$, and shrink $\Gamma$ further so that the
construction also applies to $Y$ and $E$. Then $Y^\Gamma$ is normal and
$E^\Gamma$ is a reduced prime divisor by the simultaneous
construction used in
\cite[Corollary~3.7(ii) and the proof of Theorem~5.6]{MurayamaGammaConstruction}.
Normalized dualizing complexes commute with $R\to R^\Gamma$ and this
flat proper base change
\cite[Tags~0AU3, 0AWD, and 0AA8]{stacks-project}, while reducedness of
$E^\Gamma$ shows that the induced local map of DVRs at the generic
points sends a uniformizer to a uniformizer. Consequently, writing
$\mu^\Gamma\colon Y^\Gamma\to\Spec R^\Gamma$ for the base change, one
has
\[
\coeff_{E^\Gamma}\!\left(
K_{Y^\Gamma}-(\mu^\Gamma)^*(K_{R^\Gamma}+\Delta^\Gamma)
\right)
=
\coeff_E\!\left(K_Y-\mu^*(K_R+\Delta)\right)<-1,
\]
contradicting log canonicity of $(R^\Gamma,\Delta^\Gamma)$.}

If $R$ has mixed characteristic $(0,p)$, we work in \autoref{notation lc comparison} and put $X=\Spec R$ and
$r=fg$. 
Fix $e\geq1$ and a proper birational map $\pi\colon Y\to X$ with
$Y$ normal. For a sufficiently large integral perfectoid balanced
big Cohen--Macaulay $R^+$-algebra $B$,
\autoref{damped test ideal membership} and
\cite[Theorem~6.21]{MaSchwedeSingularitiesMixedCharBCM} give
\[
r\in\tau_B\bigl(R,\Delta+(1-p^{-e})\Div_R(r)\bigr)
\subseteq
\pi_*\sO_Y\left(
\left\lceil K_Y-
\pi^*\bigl(K_X+\Delta+(1-p^{-e})\Div_R(r)\bigr)
\right\rceil\right).
\]
Letting $e\to\infty$, the result follows from
\autoref{different characterization log canonical}.
}
\end{proof}

\begin{remark}
Note that one can run the proof \cite[Theorem 4.12]{FayolleCentersPerfectoidPurity} using \autoref{pfd pure implies lc}. This shows that any log canonical center of a $\Q$-Gorenstein normal complete local perfectoid pure ring of residue characteristic $p>0$ is a center of perfectoid purity in the sense of \cite{FayolleCentersPerfectoidPurity}. 
\end{remark}

{

\begin{corollary}\label{pfd pure implies slc}
Let $R$ be an excellent demi-normal ring with $p$ in its Jacobson
radical and admitting a dualizing complex, and let $\Delta$ be an effective Mumford
$\Q$-divisor on $R$ such that $K_R+\Delta$ is $\Q$-Cartier. If
$(R,\Delta)$ is perfectoid pure, then it is semi-log canonical.
\end{corollary}

\begin{proof}
Let $R^N$, $C$, and $\Delta^N$ be as in
\autoref{normalization is pfd pure pair}~(b). By that proposition,
$(R^N,\Delta^N+C)$ is perfectoid pure. For compatible choices of
canonical divisors, $K_{R^N}+\Delta^N+C$ is the pullback of
$K_R+\Delta$ under the normalization, and hence is $\Q$-Cartier. Thus
\autoref{pfd pure implies lc} shows that $(R^N,\Delta^N+C)$ is log
canonical. By the normalization characterization of semi-log
canonicity \cite[Definition--Lemma~5.10; see also the assumptions on page~188]{KollarSingularitiesMMP},
$(R,\Delta)$ is semi-log canonical.
\end{proof}
}

\section{Adjunction for divisor pairs}
\label{subsec.PerfdPureAdjunction}

In this section, we prove a version of adjunction for divisor pairs.  We begin by noting that a weak version of adjunction already follows from  \autoref{mixed characteristic compatible splitting via radical ideal}.

\begin{proposition}\label{prop:baby_baby_case}
Let $R$ be an excellent $\mathbf{S}_2$ ring of residue characteristic $p>0$, and let $D$ be a reduced effective Mumford divisor on $\Spec R$.  If $(R,D)$ is perfectoid pure, then $R/I_D$ is perfectoid pure.
\end{proposition}

\begin{proof}
Without loss of generality, we may assume that $R$ is complete and local by \autoref{can check purity locally} and \autoref{completion of divisor pairs}. By \autoref{universal module detects perfectoid purity}, there is an absolutely integrally closed $B$ with $R \to B \to M_B({<}D)$ split, say by $\widetilde\phi$.
By \autoref{mixed characteristic compatible splitting via radical ideal} $I_D$ is $\phi$-compatible so the result follows by \autoref{basic properties comp ideals}.
\end{proof}

For the rest of the section, we aim to incorporate a boundary and the corresponding different. We work in the following setup.

\begin{setup}
\label{set:mumford-adjunction}
Let $R$ be an excellent equidimensional demi-normal ring with $p$ in its Jacobson radical. In particular, $R$ is reduced, $\mathbf{S}_2$, and $\mathbf{G}_1$.
Suppose that $R$ admits a dualizing complex $\omega^\bullet_R$ and let $\omega_R$ be the canonical module associated to it. 
Let $D$ be a reduced effective Mumford divisor on $R$ and let $i\colon D=V(I_D)\hookrightarrow \Spec R$. Assume that $S\coloneqq R/I_D$ is deminormal. 
Write $\omega_S\coloneqq \Ext^{1}_R(S,\omega_R^\bullet{[-d]})$, which is a canonical sheaf for $S$ by \cite[Tag~0AX1]{stacks-project}.
Let $\Delta$ be an effective Mumford $\Q$-divisor on $X$ having no prime component in common with $D$.
Since $R$ is excellent, it has at most finitely many singular primes of height $1$ so we can fix a canonical divisor $K_R=-D+G$ for some effective integral Mumford divisor $G$ having no common component with $D$. Use this choice of $K_R$ to fix an embedding $\omega_R\subseteq K(R)$. Assume that $K_R+D+\Delta$ is $\Q$-Cartier { at every height $1$-prime of $S$} and that $(R,D+\Delta)$ is perfectoid pure.
\end{setup}
\begin{lemma}[{\cf \cite[Definition 4.1]{KollarSingularitiesMMP}}]
\label{lem:codim_two_adjunction}
Work in \autoref{set:mumford-adjunction} and let $\omega_R(D)\coloneqq \Hom_R(I_D,\omega_R) = R(K_R+D)$. There is an exact sequence of $R$-modules
\[
0\longrightarrow \omega_R \longrightarrow \omega_R(D) \longrightarrow \omega_S.
\]
Furthermore, the map $\omega_R(D)\to \omega_S$ is surjective at any prime of height ${\leq} 2$ in the support of $D$. In other words, viewing $\omega_R(D)/\omega_R$ as an $S$-module, we have an isomorphism  \[(\omega_R(D)/\omega_R)^{\cOne} \xrightarrow{\sim} \omega_S.\]
\end{lemma}

\begin{proof}
Since $\omega_R$ is $\mathbf{S}_2$ and $I_D$ has height $1$, we have $\depth_{I_D}\omega_{R}=1$ and thus $\Hom_R(S,\omega_R)=0$. 
Applying $\Hom_R(-,\omega_R)$ to 
\[
0\longrightarrow I_D \longrightarrow R \longrightarrow S \longrightarrow 0
\]
gives 
\[
0\longrightarrow \omega_R \longrightarrow \omega_R(D) \xrightarrow{\delta} \Ext^1_R(S,\omega_R)\longrightarrow 0,
\]
Let $\operatorname{CM}(-)$ denote the open subset on which the given ring is Cohen--Macaulay and note that $U\coloneqq \operatorname{CM}(S)\cap i^{-1} \operatorname{CM}(R)$ is a big open subset of $\Spec S$ since $R$ is $\mathbf{S}_2$. On $U$, \cite[Definition~11.13 and (11.13.1)]{KollarFamiliesVarietiesGeneralType} gives an isomorphism $\Ext^1_R(S,\omega_R)|_U\cong \omega_S|_U$. Since $\omega_S$ is $\mathbf{S}_2$, it has the codimension $2$ extension property so we obtain an isomorphism $(\Ext^1_R(S,\omega_R))^{\cOne}\cong \omega_S$. 
\end{proof}

\begin{lemma}\label{lem:mumford-pluricanonical-residue}
The hypotheses of \cite[Definition 4.2]{KollarSingularitiesMMP} are satisfied on $S$. In particular, we can define the different, $\Delta_D=\operatorname{Diff}_D(\Delta)$, which is an effective Mumford $\Q$-divisor on $S$. Moreover, for any $m\in \N_{>0}$ for which $m(K_R+D+\Delta)$ is Cartier { at every height $1$ prime of $S$}, there is a canonical isomorphism
\[
\left(
\omega_R^{[m]}(mD+m\Delta)\tensor_RS
\right)^{\mathrm{c1},S}
\xrightarrow{\ \sim\ }
\omega_S^{[m]}(m\Delta_D),
\]
where $\omega_R^{[m]}(mD+m\Delta)=R(m(K_R+D+\Delta))$. 
\end{lemma}
\begin{proof}
The fact that the hypotheses of \cite[Definition 4.2]{KollarSingularitiesMMP} hold is direct from \autoref{set:mumford-adjunction}: since $S$ is $\mathbf{G}_1$, one can take $S'=S$ in \cite[Definition 4.2]{KollarSingularitiesMMP}. The second part of the statement is in the proof of \cite[Definition 4.2]{KollarSingularitiesMMP}: following the notation in that proof, $S_0$ is a big open of $S$.
\end{proof}

\begin{theorem}[Adjunction for perfectoid pure pairs]
\label{thm:perfectoid-pure-adjunction} Keeping the notation of \autoref{lem:mumford-pluricanonical-residue}, $(S,\Delta_D)$ is perfectoid pure.
\end{theorem}

\begin{proof}
By \autoref{can check purity locally}, we can assume that $R$ and $S$ are local, say with maximal ideals $\fram$ and $\fran$, respectively. Let $I=I_D$ and $d=\dim R$.  The case $d\leq1$ is immediate, so assume $d\geq2$.
Since $\omega_R(D)=R(G)$, \autoref{lem:codim_two_adjunction} gives an isomorphism $\left(R(G)/\omega_R\right)^{\mathrm{c1},S}\cong \omega_S$.
Let $B$ be an absolutely integrally closed perfectoid $R$-algebra such that $\rho_{B,D+\Delta}$ is pure.
Define
\[
 \begin{aligned}
 N_I&=\left(((IB)_\perfd\tensor_RR(G))\tensor_BM_B({<}\Delta)\right)^{\mathrm{c1}},\\
 N_B&=\left((B\tensor_RR(G))\tensor_BM_B({<}\Delta)\right)^{\mathrm{c1}},\\
 L&=\left(\im(N_I\longrightarrow N_B)\right)^{\mathrm{c1}} \subseteq N_B,\\ W & =N_B/L,
 \end{aligned}
\]
where the codimension-$1$ envelopes are taken over $R$. On any common big open used to define $N_I$ and $N_B$, $(IB)_\perfd N_B\subseteq\im(N_I\to N_B)$ so $(IB)_\perfd W=0$. Thus $W$ and $W^{\mathrm{c1},S}$ are naturally $C\coloneqq B/(IB)_\perfd$-modules.
Define a map $R(G)\to N_B$ by sending $\gamma\in R(G)$ to $1\tensor \gamma\tensor \mathbf 1_{<\Delta}$ and taking the codimension-$1$ envelope. The restriction of this map to $\omega_R=(IR(G))^\cOne$ has image in $L$. In particular, it defines a morphism $R(G)/\omega_R\to W$ which, after using $\left(R(G)/\omega_R\right)^{\mathrm{c1},S}\simeq\omega_S$ , yields $ q\colon\omega_S\longrightarrow W^{\mathrm{c1},S}$.

To sum up, we have a commutative diagram of solid arrows
\begin{equation}\label{first diagram adj div pair}
\xymatrix{
\omega_R\ar[r]\ar[dd]&
R(G)\ar[r]\ar[dd]&
\omega_S\ar[dd]_{q}
  \ar[dr]^-{\eta\mapsto\eta\tensor\mathbf1_{<\Delta_D}}&\\
&&&\left(\omega_S\tensor_SM_C({<}\Delta_D)\right)^{\mathrm{c1},S}
  \ar@{-->}[dl]^{\chi}\\
N_I\ar[r]&N_B\ar[r]&W^{\mathrm{c1},S}.&
}
\end{equation}

\begin{proposition}\label{prop:adjunction-third-column}
The dashed arrow in the diagram exists. More precisely, there is a $C$-linear map
\[
 \chi\colon \left(\omega_S\tensor_SM_C({<}\Delta_D)\right)^{\mathrm{c1},S} \longrightarrow W^{\mathrm{c1},S}
\]
which makes the triangle on the right of the preceding diagram commute.
\end{proposition}

\begin{proof}
We construct this map locally around every height $1$ point of $S$ before showing that it glues. For an arbitrary $P\in \Supp(\Delta_D)$, $S_P$ is a DVR since $\Delta_D$ is Mumford. Moreover, since $S$ is $\mathbf{G}_1$, $\omega_{S,P}$ is free of rank $1$. In a sufficiently small neighborhood of $P$, prime avoidance allows us to find a generator $\kappa$ of $\omega_S$ along with a lift $r$ of $\kappa$ in $R(G)\cap K(R)^\times$ such that $G_r\coloneqq G+ \Div_R (r)\geq 0$ has no component in common with $D$ and is Mumford. 
Note that $R(G_r)=r^{-1}R(G)$ and that $G_r+\Delta=K_R+D+\Delta+\Div_R(r)$ is $\Q$-Cartier. Choose a sufficiently divisible $n>0$ and $h\in R$ such that $n(G_r+\Delta)=\Div_R(h)$ sufficiently close to $P$.
The image $\overline h\in S$ is a nonzerodivisor and \autoref{lem:mumford-pluricanonical-residue} gives $n\Delta_D=\Div_S(\overline h)$.
Fix an $n$-th root $z$ of $h$ in $B$.
{By \autoref{universal divisorial module principal case} and \autoref{universal divisorial module is local in the divisor}, on that neighborhood of $P$, there are pointed isomorphisms $M_B(<G_r+\Delta)\cong (z^{1/p^{\infty}}B)^\cOne$, $\mathbf1_{<G_r+\Delta}\mapsto z$ and $M_C(<\Delta_D)\cong (\overline{z}^{1/p^{\infty}}C)^{\cOne,S}$, $\mathbf1_{<\Delta_D}\mapsto \overline z$ . 
{By \autoref{BS - root}~(b), there is a pointed isomorphism
	$C\tensor_B(z^{1/p^{\infty}}B)
	\cong(\overline{z}^{1/p^{\infty}}C)$.} 
Recall from \autoref{radical ideal computes universal divisorial module}\autoref{def alpha D Gamma} that there is a canonical pointed $B$-linear map
\begin{equation}\label{alpha r adjunction}
 \alpha_r\coloneqq\alpha_{G_r,\Delta}\colon M_B({<}(G_r+\Delta)) \longrightarrow \left((B\tensor_RR(G_r))\tensor_BM_B({<}\Delta)
 \right)^{\mathrm{c1}}.
\end{equation}
Composing it with multiplication by $r$ and quotienting by $L$ gives a $B$-linear map $\beta_r\colon M_B({<}(G_r+\Delta)) \to N_B \to W$. Since $(IB)_\perfd W=0$,
{the composition \[(z^{1/p^{\infty}}B)\to M_B({<}(G_r+\Delta))\to W\] descends to \((\overline{z}^{1/p^\infty})C\to W\).  Taking codimension-1 envelopes with respect to $S$ gives}
a map $\psi_\kappa \colon M_C(<\Delta_D)\to W^{\cOne,S}$.

Chasing the maps around,} we see that $\psi_\kappa(\mathbf1_{<\Delta_D})=q(\kappa)$.
We can finally define a $C$-linear map on the neighborhood of $P$
\[
\chi_P\colon  \omega_S\tensor_SM_C({<}\Delta_D)\longrightarrow W^{\mathrm{c1},S},
 \qquad \kappa\tensor m\longmapsto\psi_\kappa(m).
\]
Note that for any $\eta=s\kappa$ in $\omega_S$, $\chi_P(\eta\tensor \mathbf 1_{<\Delta_D})=s\psi_\kappa(\mathbf1_{<\Delta_D})=sq(\kappa)=q(\eta)$.
We want to show that it glues. Let $P_1,\ldots, P_m$ be the finitely many height-$1$ points in $\Supp(\Delta_D)$ and for each $i$, let $f_i\in (\bigcap_{j\neq i} P_j)\setminus P_i$ where the empty intersection is $S$. For any choice of $i$, we can pick $U_i\subset D(f_i)$ so $\chi_P$ as above can be defined on $U_i$. Let $U_0=D\setminus\bigcup_iV(P_i)$ and note that $\bigcup_{m\geq i\geq 0} U_i$ is a big open in $S$. Moreover, since $U_0$ does not contain any prime in the support of $\Delta_D$, one can define $\chi_0$ on $U_0$ by $\eta\tensor c\mapsto c q(\eta)$.

By the choice of the $f_i$, no height-one point of
$\Supp(\Delta_D)$ lies in an overlap. At every height-$1$ point of an overlap, the restrictions of $\chi_{P_i}$ and $\chi_{P_j}$ (or $\chi_0$) must therefore agree so they glue over $U$, note that by construction, the transition maps are equalities. Taking codimension $1$-envelopes finishes the proposition.
\end{proof}

We return to the proof of the theorem. Let $N_0\coloneqq \left(  (\sqrt{IB}\tensor_RR(G))\tensor_BM_B({<}\Delta) \right)^{\mathrm{c1}}$. The pointed map
\[
 \rho_{B,D+\Delta}\colon R\longrightarrow M_B({<}(D+\Delta)),
 \qquad 1\longmapsto\mathbf1_{{<}(D+\Delta)},
\]
was chosen to be pure.  Tensoring this map with $\omega_R$,
using $\left(\omega_R\tensor_RR(D)\right)^{\mathrm{c1}}=R(G)$, \autoref{radical ideal computes universal divisorial module},
and passing to top local cohomology gives an injection
\begin{equation}\label{eqn inj loc coh}
 E_R\lhook\joinrel\longrightarrow H^d_{\fram}(N_0).
\end{equation}
Since the generators of $I$ have compatible $p$-power roots in $B$,
the proof of \autoref{root ideal tensor comparison}, together with
\autoref{pfd of ideal only depends on radical}, gives, for every
$p^\infty$-torsion $R$-module $T$,
\[
 T\tensor_R \sqrt{IB}\xrightarrow{\ \sim\ }T\tensor_R(IB)_\perfd.
\]
In particular, this is the case for $T=H^d_\fram(R)\tensor_R R(G)$. Apply the above isomorphism and tensoring with $M_B(<\Delta)$ over $B$ gives an isomorphism
\begin{equation*}
 H^d_{\fram}\!\left((\sqrt{IB}\tensor_RR(G))\tensor_BM_B({<}\Delta) \right)\cong H^d_{\fram}\!\left(((IB)_\perfd\tensor_RR(G))\tensor_BM_B({<}\Delta) \right).
\end{equation*}
Taking the codimension-one envelopes does not change top local cohomology so
\[
 H^d_{\fram}(N_0)\xrightarrow{\ \sim\ }H^d_{\fram}(N_I).
\]
Note that for any minimal prime $\p\in \Spec R$, $IR_\p=R_\p$ so $\ker (N_I\to L)$ has support of dimension at most $d-1$. 
On the other hand, $\coker (N_I\to L)$ has support of dimension at most $d-2$ hence $ H^d_{\fram}(N_I)\xrightarrow{\ \sim\ }H^d_{\fram}(L)$.
Thus the injection in (\ref{eqn inj loc coh}) can be rewritten as
\[
 a\colon E_R\lhook\joinrel\longrightarrow H^d_{\fram}(N_0)
 \xrightarrow{\ \sim\ }H^d_{\fram}(N_I)
 \xrightarrow{\ \sim\ }H^d_{\fram}(L).
\]

With the same maps as in \autoref{first diagram adj div pair}, the following diagram has exact rows and commutes
\begin{equation}\label{diagram adjunction ses}
 \xymatrix{
 0\ar[r]&\omega_R\ar[r]\ar[d]&R(G)\ar[r]\ar[d]&
 R(G)/\omega_R\ar[r]\ar[d]&0\\
 0\ar[r]&L\ar[r]&N_B\ar[r]&W\ar[r]&0.
 }
\end{equation}

The codimension-$1$ envelope (over $S$) of the last vertical arrow is the map $q$ constructed above. 
Since $K_R=-D+G$,
\[
 \Hom_R(R(G),\omega_R)=R(-D)=I,
 \qquad
 \Hom_R(\omega_R,\omega_R)=R.
\]
It follows from duality that the natural map $H_\m^d(\omega)\to H_\m^d(R(G))$ can be identified with the Matlis dual of $I\hookrightarrow R$.  On the other hand since $(R(G)/\omega_R)^{\cOne,S}\cong \omega_S$, invariance of top local cohomology under codimension-$1$ envelope and local duality give $H^{d-1}_{\fram}(R(G)/\omega_R)\cong E_S$. Consider the exact sequence
$$H_\m^{d-1}(R(G))\xrightarrow{\delta} H_\m^{d-1}(R(G)/\omega_R) \to H_\m^d(\omega_R) \to H_\m^d(R(G)) \to 0.$$
We claim that $\delta=0$, to see this, note that the Matlis dual of the sequence can be identified with $0 \to \hat{I} \to \hat{R} \to \hat{S} \xrightarrow{\delta^\vee} H_\m^{d-1}(R(G))^\vee$. Since $R(G)$ is $\mathbf{S}_2$, $\dim(H_\m^{d-1}(R(G))^\vee)\leq d-3$ while the cokernel of the induced map $\hat{R}/\hat{I}\to \hat{S}$ (which is necessarily multiplication by some $r\in \hat{S}$), if nonzero, must have dimension $\geq d-2$ and thus cannot embed into $H_\m^{d-1}(R(G))^\vee$. It follows that $\delta^\vee=0$ and thus so is $\delta$.

Applying local cohomology to
(\ref{diagram adjunction ses}) therefore gives
\[
 \xymatrix{
 E_S\ar@{^{(}->}[r]^-j\ar[d]_{\widetilde q}&E_R\ar[d]^a\\
 H^{d-1}_{\fram}(W^{\mathrm{c1},S})
 \ar[r]_-{\partial}&H^d_{\fram}(L).
 }
\]
The top and right arrows are injective; hence $\widetilde q$ is
injective. The map $\chi$ of
\autoref{prop:adjunction-third-column} factors $\widetilde q$ which gives
\begin{equation}\label{adjunction map from Es}
 E_S\longrightarrow
 H^{d-1}_{\fram}
 \left(
   (\omega_S\tensor_S M_C({<}\Delta_D))^{\mathrm{c1},S}
 \right)
 \longrightarrow H^{d-1}_{\fram}(W^{\mathrm{c1},S}).
\end{equation}
Again, by invariance of top local cohomology under codimension-$1$ envelopes,
\[
 H^{d-1}_{\fram}
 \left(
   (\omega_S\tensor_S M_C({<}\Delta_D))^{\mathrm{c1},S}
 \right)
 \cong H^{d-1}_{\fram}(\omega_S)\tensor_S M_C({<}\Delta_D)
 \cong E_S\tensor_S M_C({<}\Delta_D).
\]
So the first arrow in \autoref{adjunction map from Es} can be rewritten as
\[
 E_S\longrightarrow E_S\tensor_S M_C({<}\Delta_D),\qquad
 u\longmapsto u\tensor\mathbf1_{<\Delta_D}
\]
and is injective since $\widetilde q$ is injective. Now
\autoref{lem: injective hull splitting ideal} shows that the
 pointed map
\[
 S\longrightarrow M_C({<}\Delta_D),\qquad
 1\longmapsto\mathbf1_{<\Delta_D},
\]
is pure so the result follows from \autoref{universal module detects perfectoid purity}.
\end{proof}

Finally we show that by imposing natural geometric hypotheses one can remove the assumptions on $S$ from \autoref{thm:perfectoid-pure-adjunction}.

\begin{theorem}

\label{cor:geometric-adjunction}
Let $R$ be an equidimensional demi-normal ring essentially of
finite type over an excellent DVR, with $p$ in its
Jacobson radical. Let $D$ be a reduced effective
Mumford divisor on $\Spec(R)$, and let $\Delta$ be an effective Mumford
$\Q$-divisor having no prime component in common with $D$.
Assume that $(R,D+\Delta)$ is perfectoid pure and that
$K_R+D+\Delta$ is $\Q$-Cartier. Put $S=R/I_D$.
Then $S$ is at worst nodal in codimension one.
Furthermore, if either $S$ is $\mathbf{S}_2$ or $D$ is $\Q$-Cartier, then $S$ is
demi-normal and $(S,\Delta_D)$ is perfectoid pure.
\end{theorem}

\begin{proof}
Letting $R^N$ be the normalization of $R$, and using \autoref{normalization is pfd pure pair}~(b) and its notation, $(R^N, C+D^N + \Delta^N)$ is perfectoid pure. 
Put $E=C+D^N$ with its reduced divisor structure. Applying \autoref{prop:baby_baby_case} at the maximal ideals shows that $D$ and $E$ are perfectoid pure. Thus $R$, $D$ and $E$ are weakly normal by  \cite[Corollary~4.30]{BMPSTWW3}. Note that \autoref{pfd pure implies slc} shows that $(R,D+\Delta)$ is semi-log canonical and so $(R^N, C+D^N+\Delta^N)$ is log canonical.

Fix a codimension-one point $x$ of $D$. Let $B$ be its strict henselization of $R$ at $x$, and put $A=B/I_DB=\sO_{D,x}^{\mathrm{sh}}$ and $k=B/\mathfrak m_B$.
Weak normality is preserved by localization and this ind-\'etale base change \cite[Proposition~B.6]{RydhSubmersionsEffectiveDescent}. Normalization commutes with strict henselization \cite[Tag~0CBM]{stacks-project}; log canonicity is preserved as well, since it is \'{e}tale local (see \cite[Section~2.16.1]{KollarFamiliesVarietiesGeneralType}).
Write $B^N=\prod_{i=1}^r B_{i}$. The rings $B_i$ are normal strictly henselian local surfaces, and their residue fields $k_i$ are finite purely inseparable extensions of the separably closed field $k$.

We show that there are at most two branches of $D$ at $x$. 
Let $C_i$ and $D_i^{N}$ be $C$ and $D^N$ on $\Spec B_i$, respectively, and let $c_i$ and $d_i$ be their respective numbers of branches.
Since $C_i+D_i^{N}$ is the reduced boundary on a log canonical surface, it is nodal by \cite[Theorem~2.31]{KollarSingularitiesMMP}, and since it is weakly normal, every node is separable \cite[Lemma~2.8]{PosvaGluingStableFamiliesMixed}, and hence split.
Consequently $c_i+d_i\leq 2$. Note that $B$ is a Cohen-Macaulay local ring of dimension two, thus the punctured spectrum of $B$ is connected. Since each $B_i$ corresponds to a minimal prime $\p_i$ of $B$, the Hochster-Huneke graph of $B$ has at least $r-1$ edges. For an edge connecting $\p_i$ and $\p_j$, corresponding to a height $1$ prime $\p$ of $B$, we know that $B_\p$ has exactly two minimal primes since $B$ is deminormal (in particular, different edges correspond to different height $1$ primes). It follows that $\p$ gives a component of $C_i$ in $B_i$ and a component of $C_j$ in $B_j$. Therefore the $r-1$ edges contributes to at least $2(r-1)$ components of various $C_i$'s and thus $\sum_{i=1}^r c_i\geq 2(r-1)$.
Hence we have 
\[\sum_{i=1}^r d_i\leq 2r-\sum_{i=1}^r c_i\leq 2\] as required.

Firstly, if there is one branch, the normalization $D^N\to D$ is a universal homeomorphism and so $D$ is normal (hence regular) by weak normality. Suppose now there are two boundary branches, with normalization \(A^N=T_1\times T_2\) where $T_i$ are DVRs. Since the node on $B_i$ is separable and $k$ is separably closed, the residue fields of the branches are a common finite purely inseparable extension \(L/k\). 
The natural map
\[
\operatorname{Spec}(T_1\times_L T_2)\longrightarrow\operatorname{Spec}A
\]
is a finite birational universal homeomorphism.  Weak normality therefore gives \(A=T_1\times_L T_2\). Comparing residue fields yields \(L=k\), so \(A=T_1\times_k T_2\).

It is a quotient of a regular ring by \cite[Lemma~3.3.1(c)]{PosvaGluingTheorySLC}, and so \cite[Lemma~3.2.1]{PosvaGluingTheorySLC} shows that $A$ is nodal. Finally \cite[Corollary~3.3.2]{PosvaGluingTheorySLC} shows that the original $D$ (before strict henselization) was also nodal.

It now remains to show that $S$ is $\mathbf{S}_2$ in the case where $D$ is $\mathbb{Q}$-Cartier.
 We may localize to assume that $R$ is local. 
  By
\autoref{perfectoid purity decreases with divisor}, the pair $(R,D)$ is
perfectoid pure.  Therefore, by \autoref{prop:baby_baby_case}, the
underlying scheme $D$ is perfectoid pure, and hence weakly normal by
\cite{BMPSTWW3}.  Since $R$ is $\mathbf{S}_2$, it is topologically $\mathbf{S}_2$
\cite[Definition~6.8]{Kollar.MapsBetweenLocalPicardGroups}.  Choose
$m>0$ such that $mD$ is Cartier.  The connectedness theorem for Cartier
divisors
\cite[Lemma~2.93]{KollarFamiliesVarietiesGeneralType}, applied
\'etale-locally to $mD$, shows that $mD$ is topologically $\mathbf{S}_2$.
Since $D$ and $mD$ have the same underlying topological space, $D$ is
topologically $\mathbf{S}_2$ as well.  Finally,
a weakly normal scheme is $\mathbf{S}_2$ if
and only if it is topologically $\mathbf{S}_2$
\cite[Definition~6.8]{Kollar.MapsBetweenLocalPicardGroups}, so $D$ is
$\mathbf{S}_2$.  Since the maximal ideal $\mathfrak n$ is arbitrary, $S$ is
$\mathbf{S}_2$, and \autoref{thm:perfectoid-pure-adjunction} applies.
\end{proof}

\section{Perfectoid pure thresholds}
\label{sec.PerfectoidPureThresholds}

\subsection{A threshold for perfectoid rings}
\label{subsec.ThresholdForPerfectoidRings}

We present an analogue of F-threshold using perfectoid rings, 
in an approach similar to BCM thresholds defined in \cite{RVS}.  Compare also with \cite{Yoshikawa.ComputationMethodForPerfdPurity} as well as \cite{HunekeMustataTakagiWatanabe.FThresholdsTCICAndMult}.

\begin{definition}
    Let $R$ be a ring and $\fra$ an ideal. 
    For a positive rational number $a/r$ we define
    \[
    \overline{\fra^{a/r}} \coloneqq \{x \in R \mid x^r \in \overline{\fra^a}\}.
    \]
\end{definition}

It is easy to see from the equation of integral dependence that 
$\overline{\fra^{a/r}}$ does not depend on the presentation of the rational number $a/r$ as 
the proof of \cite[Proposition 10.5.2(1)]{HunekeSwansonIntegralClosure} does not use that $R$ is Noetherian.

\begin{definition}[\cite{RVS}]
    Let $R$ be a ring and $I, \ba$ be ideals (we allow $I$ to be the unit ideal). 
    If $\ba \subseteq \sqrt{I}$, then we define\footnote{We recall that $\sup \emptyset = -\infty$.} the \emph{rational powers threshold} of {$\fra$ with respect to $I$} as
    \[
    \newpt{\ba}{I} = \sup \{r \in \mathbb{Q}_{\geq 0} \mid  \overline{\ba^r} \not \subseteq I\}.
    \]
\end{definition}

We state some elementary properties.

\begin{lemma}\label{lem: elementary threshold}
    Let $R$ be a ring and $\ba, \bb, I$ be ideals with $\ba, \bb \subseteq \sqrt{I}$.
    Then 
    \begin{enumerate}
        \item if $\ba \subset \bb$ then $\newpt{\ba}{I}  \leq \newpt{\bb}{I}$; 
        \item if $\overline{\ba} = \overline{\bb}$ then $\newpt{\ba}{I}  = \newpt{\bb}{I}$;
        \item $\newpt{\ba^n}{I}  = n^{-1} \newpt{\ba}{I}$ for $n \in \bZ_{\geq 1}$; 
        \item if $I \subset J$ is an ideal, then 
        $\newpt{\ba}{I} \geq \newpt{\ba}{J}$.
    \end{enumerate}
\end{lemma}
\begin{proof}
    These properties are immediate from the definition.
\end{proof}

We will now switch to the case where $R$ is perfectoid so that we benefit from the following Brian{\c c}on--Skoda result that will allow us to replace rational powers by regular powers. 

\begin{proposition}[{\cite[Corollary~3.20]{MMRGSBrianconSkoda}}]\label{BS property}
Let $R$ be a perfectoid ring and let $I\subseteq R$ be an ideal
generated by $n$ elements. For every $k\geq1$, one has
$I_{\perfd}\overline{I^{n+k-1}}\subseteq I^k$ and
$\overline{I^{n+k}}\subseteq I^k$.
\end{proposition}

The above proposition provides the following characterization of the rational powers threshold in perfectoid rings.

\begin{theorem}\label{thm: from rational powers to monomials}
    Let $R$ be a perfectoid ring and $x_1, \ldots, x_d$ be elements that admit compatible systems of $p$-power roots in $R$.
    Set $\ba = (x_1, \ldots, x_d)$.
    If $I$ is an ideal such that 
    $x_1, \ldots, x_d \in \sqrt{I}$ then
    \begin{align*}
    \newpt{\ba}{I}
    &= \sup \left \{\frac n{p^e} \mid n, e \in \mathbb{Z}_{\geq 0}, \overline{(x_1^{1/p^e}, \ldots, x_d^{1/p^e})^n} \not \subseteq I
    \right \}
    \\ &=
    \sup 
    \left \{
    \frac{|a|}{p^{e}} \mid a \in \mathbb{Z}_{\geq 0}^d, x_1^{a_1/p^e}\cdots x_d^{a_d/p^e}\notin I
    \right \}
    \\ &=
    \lim_{e \to \infty} 
    p^{-e}\sup \left \{
    |a| \mid a \in \mathbb{Z}_{\geq 0}^d, x_1^{a_1/p^e}\cdots x_d^{a_d/p^e}\notin I
    \right \},    
    \end{align*}
    where $|a| = \sum_{i = 1}^d a_i$.
\end{theorem}
\begin{proof}
    We start by observing that 
    \[
    \sup \left \{r \in \mathbb{Q}_{\geq 0} \mid  \overline{\ba^r} \not \subseteq I \right\}
    = \sup \left \{\frac{n}{p^e} \mid n, e \in \mathbb{Z}_{\geq 0}, 
    \overline{\ba^{\frac n {p^e}}} \not \subseteq I\right \}.
    \]
    Namely, the $\geq$ inequality is obvious and its converse holds because 
    we may approximate from below any positive real $r$ by a sequence of  rational numbers of the form $n/p^e$, 
    explicitly, by taking $n = \lfloor p^e r \rfloor$.

   The first equality follows from 
    $
    \overline{\ba^{n/p^e}} = \overline{(x_1^{1/p^e}, \ldots, x_d^{1/p^e})^n}.
    $
    To prove the second equality, note that 
    { $\geq$} is obvious. The other inequality follows from \autoref{BS property}
    after noting that for every sufficiently large $e' \geq 0$
    \[
    \overline{(x_1^{1/p^e}, \ldots, x_d^{1/p^e})^n}
    = \overline{(x_1^{1/p^{e+e'}}, \ldots, x_d^{1/p^{e+e'}})^{p^{e'}n}}
    \subseteq (x_1^{1/p^{e+e'}}, \ldots, x_d^{1/p^{e+e'}})^{p^{e'}n - d}.
    \]
    The last equality is obvious since the sequence 
    \[
    e \mapsto p^{-e} \left \{
    |a| \mid a \in \mathbb{Z}_{\geq 0}^d, x_1^{a_1/p^e}\cdots x_d^{a_d/p^e}\notin I\right \}   
    \]
    is non-decreasing.
\end{proof}

\begin{corollary}
\label{cor: ppt sum of ideals}
 Let $R$ be a perfectoid ring and $x_1, \ldots, x_s$, $y_1,\dots,y_t$ be elements that admit compatible systems of $p$-power roots in $R$. Set $\ba=(x_1,\dots,x_s)$ and $\bb=(y_1,\dots,y_t)$. If $I\subseteq R$ is an ideal such that $\ba + \bb \subseteq \sqrt{I}$, then 
 $$\newpt{\ba + \bb}{I} \leq \newpt{\ba}{I}  + \newpt{\bb}{I}.$$
\end{corollary}
\begin{proof}
Consider a monomial of the form $x_1^{a_1/p^e} \cdots x_s^{a_s/p^e} y_1^{b_1/p^e}\cdots y_t^{b_t/p^e}$. If $\frac{|a|}{p^e} + \frac{|b|}{p^e} > \newpt{\ba}{I}  + \newpt{\bb}{I},$
then either $\frac{|a|}{p^e} > \newpt{\ba}{I}$ or $\frac{|b|}{p^e} > \newpt{\bb}{I}$, hence, by \autoref{thm: from rational powers to monomials}, 
\[x_1^{a_1/p^e} \cdots x_s^{a_s/p^e} y_1^{b_1/p^e}\cdots y_t^{b_t/p^e} \in I.\] 
Applying \autoref{thm: from rational powers to monomials} again we find that $\newpt{\ba + \bb}{I} \leq \newpt{\ba}{I}  + \newpt{\bb}{I}$ as wanted.
\end{proof}

\begin{corollary}\label{cor: basic functoriality of threshold}
    Let $R \to S$ be a map of perfectoid rings and 
    $\ba \subseteq \sqrt{I} \subseteq R$ be ideals.
    Suppose that $\ba$ is generated by finitely many elements admitting compatible systems of $p$-power roots. 
    Then
    \begin{enumerate}
        \item $\newpt{\ba}{I} \geq \newpt{\ba S}{IS}$;
        \item if $J \subseteq S$ is an ideal such that $I = J\cap R$
        then $\newpt{\ba}{I} = \newpt{\ba S}{J}$. 
    \end{enumerate}
\end{corollary}
\begin{proof}
    Both properties are immediate from \autoref{thm: from rational powers to monomials}.
\end{proof}

\begin{proposition}[Reduction to principal ideals]\label{prop: reduce to principal}
    Let $R$ be a perfectoid ring and $\ba = (f_1, \ldots, f_n)$ an ideal generated by elements $f_i$ admitting compatible systems of $p$-power roots.
    If $I$ is an ideal such that $\ba \subseteq \sqrt{I}$, then
    \[
    \newpt{\ba}{I} 
    = \sup \big\{|\alpha| \newpt{(f_1^{\alpha_1}\cdots f_n^{\alpha_n})}{I} \mid 0\neq \alpha \in \mathbb{Z}_{\geq 0}^n\big\}.
    \]
\end{proposition}
\begin{proof}
    Let $\mathcal{M}_k$ be the set of degree-$k$ monomials on the $f_i$. 
    For any $f \in \mathcal{M}_k$ 
    \[
    k \newpt{(f)}{I} \leq k \newpt{\ba^{k}}{I} = \newpt{\ba}{I},  
    \]
    since $f \in \ba^k$. The opposite inequality follows from \autoref{thm: from rational powers to monomials} and the definition of $\pt$:
    \[
    \newpt{\ba}{I} = \sup 
    \left \{
    \frac{k}{p^{e}} \mid f \in \mathcal{M}_k, f^{1/p^e} \notin I
    \right \} 
    \leq \sup 
    \left \{
    k \newpt{(f)}{I} \mid f \in \mathcal{M}_k
    \right \}.
    \]
\end{proof}

\subsection{Perfectoid pure thresholds}
In this subsection we introduce perfectoid thresholds and prove some of their basic properties.

\begin{definition}
    Let $(R, \m)$ be a Noetherian local ring of residual characteristic $p > 0$ and $\ba$ be a proper ideal of $R$. 
    If $B$ is a perfectoid $R$-algebra such that 
    $\ba B = (x_1, \ldots, x_g)$ where the elements $x_1, \ldots, x_g$ have compatible systems of $p$-power roots,
    then we will call $\ppt (\ba; B) \coloneqq \newpt{\ba {B}}{\np(B)}$ the \emph{perfectoid pure threshold of} $\ba$ with respect to $B$.
    
    We set the \emph{perfectoid pure threshold} of $\ba$, $\ppt(\ba)$, to be 
    $\sup \{\ppt (\ba; B)\}$, where $B$ ranges through all perfectoid $R$-algebras such that $\ba B$ can be generated by elements that have compatible systems of $p$-power roots. By Proposition~\ref{prop: properties of perfectoids} the latter condition is automatic if $p \in \ba$ (e.g., when $R$ is local and $\ba$ is the maximal ideal), or when $B$ is absolutely integrally closed.  
\end{definition}

\begin{definition}
\label{def.ThresholdForDivisorPairs}
    Suppose $(R, \m)$ is an excellent reduced $\mathbf{S}_2$ local ring of residual characteristic $p > 0$, $\ba$ a proper ideal of $R$,  and $\Delta$ is an effective Mumford $\Q$-divisor on $R$. 
    If $B$ is a perfectoid $R$-algebra such that $M_B ({<}\Delta)$ is defined (\autoref{universal divisorial module}) and 
    $\ba B = (x_1, \ldots, x_g)$ where the elements $x_1, \ldots, x_g$ have compatible systems of $p$-power roots
    in $B$, then from \autoref{nonpure ideal below delta} we define $\ppt (\ba; \Delta; B) \coloneqq \newpt{\ba {B}}{\np (B, {<}\Delta)}$.
    We further define 
    \(
    \ppt (\ba; \Delta) = \sup \{\ppt (\ba; \Delta; B)\}, 
    \)
    where $B$ ranges through perfectoid $R$-algebras such that $\ppt (\ba; \Delta; B)$ is defined.    
\end{definition}

\begin{remark} We collect some basic observations.
\begin{enumerate}
    \item The definition of $\ppt(\ba; \Delta; B)$ does not depend on the choice of generators of $\ba B$.\footnote{The quantity $\newpt{\mathfrak{a}B}{\mathcal{I}(B)}$ makes sense even without the hypothesis that $\mathfrak a B$ can be generated by elements that have compatible systems of $p$-power roots. We make this assumption so that we can freely apply \autoref{thm: from rational powers to monomials}. Note that, in practice we will often assume (or reduce to the case) that $B$ is absolutely integrally closed so this assumption is automatic.}  
    \item The pair $(R,\Delta)$ is perfectoid pure if and only if $\ppt(\ba; \Delta)\neq -\infty$ for some (equivalently, all) ideals $\ba\subseteq R$.
    \item It follows from the above that for two perfectoid $R$-algebras $B_1, B_2$, it may happen that $\ppt(\ba; B_1) \neq \ppt(\ba; B_2)$: e.g., take a perfectoid pure $R$ that is not a splinter, and take $B_1=\raftypfd$ and $B_2=\widehat{R^+}$. 
    \item On the other hand, we do not know whether it is possible that $\ppt(\ba; B_1) \neq \ppt(\ba; B_2)$ when $R\to B_1$ and $R\to B_2$ are both pure. In fact, we will show that this is not possible in positive characteristic, see \autoref{prop: recover char p}. 
\end{enumerate}
\end{remark}

\begin{remark}\label{rmk: limiting sequence for thresholds}
    In the setting of \autoref{universal divisorial module}, 
    it easily follows from \autoref{properties nonpure ideal below delta}
    that 
    \[\ppt (\ba; \Delta; B) = \liminf_{e \to \infty} \newpt{\ba B} {\np(B, M_B(\mathfrak d_e))} = 
    \lim_{e \to \infty} \newpt{\ba B}{\np(B, M_B(\mathfrak d_e))}.
    \]
    Namely, 
    since $\np(B, M_B(\mathfrak d_e)) \subseteq \np(B, {<}\Delta)$, 
    the sequence $\newpt{\ba B}{\np(B, M_B(\mathfrak d_e))}$
    is non-increasing and is bounded below by 
    $\ppt (\ba; \Delta; B)$.
    The existence of the limit and  $\leq$ inequality follow, 
    and the opposite inequality is an easy consequence of 
    Theorem~\ref{thm: from rational powers to monomials}.
\end{remark}

As corollaries to \autoref{cor: basic functoriality of threshold} we get comparison results for  perfectoid pure thresholds.

\begin{corollary}\label{cor: ppt in algebra extensions}
    Let $(R, \m)$ be a Noetherian local ring of residual characteristic $p > 0$ and suppose that $f\colon A \to B$ is a map of perfectoid $R$-algebras such that $R\to B$ is pure. 
    Suppose that $x_1, \ldots, x_d \in R$ have compatible systems of $p$-power roots in $A$ and set $\ba = (x_1, \ldots, x_d)$.
    Then 
    $\ppt (\ba; A) \geq 
    \ppt (\ba; B)$ and the equality holds if $f$ is $\m$-completely pure.

    Moreover, if $R$ is an excellent reduced $\mathbf{S}_2$ ring and $\Delta$ is an effective Mumford $\Q$-divisor on $R$, then 
    $\ppt (\ba; \Delta; A) \geq 
    \ppt (\ba; \Delta; B)$ whenever both are defined 
    and the equality holds if $f$ is $\m$-completely pure. 
\end{corollary}
\begin{proof}
    The first part follows from 
    \autoref{cor: basic functoriality of threshold} and \autoref{lem: functoriality of non-pure ideals} { \autoref{part p-comp pure}} and the second part from 
    \autoref{base change nonpure ideal below delta}. 
\end{proof}

\begin{corollary}\label{cor: ppt in ring extensions}
     Let $f \colon (R, \m) \to (S, \n)$ be a pure map of Noetherian local rings of residual characteristic $p > 0$ and let $A$ be a perfectoid $S$-algebra. Suppose that $x_1, \ldots, x_d \in R$ have compatible systems of $p$-power roots in $A$ and set $\ba = (x_1, \ldots, x_d)$.
     Suppose $g \colon A \to B$ is any map of perfectoid $S$-algebras (e.g., $A = B)$.
     Then 
     \[
     \ppt (\ba; A) \geq \ppt (\ba S; B).
     \]
     Moreover, the equality holds if $f$ is faithfully flat, $g$ is $p$-completely pure, and $\m S = \n$.
\end{corollary}
\begin{proof}
    For the first claim, \autoref{part: ring functoriality} and \autoref{part algebra functoriality} of \autoref{lem: functoriality of non-pure ideals}
    show that $\np^R (A)B \subseteq \np^R (B) \subseteq \np^S (B)$, proving the required inequality by \autoref{cor: basic functoriality of threshold}. 
    For the second claim, we use instead \autoref{part: invariance of field extension} and \autoref{part p-comp pure} of \autoref{lem: functoriality of non-pure ideals}
    to deduce that $\np^S (B) \cap A = \np^R (A)$ .
\end{proof}

We record the following version of \autoref{cor: ppt in ring extensions}
for pairs.

\begin{corollary}\label{cor: ppt of pairs in ring extensions}
     Let $f \colon (R, \m) \to (R', \m')$ be a flat map between excellent reduced $\mathbf{S}_2$ local rings of residual characteristic $p > 0$. Let $\Delta$ be an effective Mumford $\Q$-divisor on $R$ 
     and suppose that its flat pullback $\Delta'$ is also a Mumford $\Q$-divisor (\autoref{flat pullback is Mumford}).    
     Suppose that $P$ is a perfectoid $R'$-algebra such that $M_P({<} \Delta)$ is defined (e.g., $P$ is absolutely integrally closed) and $x_1, \ldots, x_d \in R$ have compatible $p$-power roots in $P$. Set $\ba = (x_1, \ldots, x_d)$.
     Then 
     \[
     \ppt (\ba; \Delta; P) \geq \ppt (\ba R'; \Delta'; P).
     \]
     Moreover, the equality holds if $\m R' = \m'$. 
\end{corollary}
\begin{proof}
By \autoref{properties of M_B},  $M_P ({<} \Delta) \cong M_P ({<} \Delta')$. 
\autoref{lem: functoriality of non-pure ideals} \autoref{part: ring functoriality} then shows that $\np (P, {<}\Delta) \subseteq \np (P, {<}\Delta')$ and the first claim follows. 
Likewise, \autoref{lem: functoriality of non-pure ideals} \autoref{part: invariance of field extension} implies the second claim.
\end{proof}

\begin{corollary}\label{cor: ppt and completion}
    Let $(R, \m)$ be a Noetherian local ring of residual characteristic $p > 0$ and $\ba$ be an ideal. 
    Then  $\ppt (\ba) = \ppt (\widehat{\ba})$. 

    Moreover, if $R$ is an excellent reduced $\mathbf{S}_2$ ring and $\Delta$ is an effective Mumford $\Q$-divisor on $R$, then $\ppt (\ba; \Delta) = \ppt (\widehat{\ba}; \widehat{\Delta})$, 
    where $\widehat{\Delta}$ is the flat pullback of $\Delta$. 
\end{corollary}
\begin{proof}
    Any $\widehat{R}$-pure perfectoid $\widehat{R}$-algebra $B$ 
    is also an $R$-algebra, so $\ppt (\ba) \geq \ppt (\widehat{\ba})$ 
    by \autoref{cor: ppt in ring extensions}. 
    Furthermore,  if $B$ is a perfectoid $R$-algebra then $B^{\wedge \m}$
    is a perfectoid $\widehat{R}$-algebra thus 
    $\ppt (\ba; B) = \ppt (\widehat{\ba}; B^{\wedge \m})$
    by \autoref{cor: ppt in ring extensions}. 
    Therefore, $\ppt (\ba) \leq \ppt (\widehat{\ba})$. 

    We now proceed to pairs. For every absolutely integrally closed
    perfectoid $\widehat R$-algebra $C$,  \autoref{completion of divisor pairs} \autoref{completion of divisor pairs.NonsplittingIdeal1} gives that 
    $\ppt(\widehat\ba;\widehat\Delta;C)
    =\ppt(\ba;\Delta;C)$, and hence
    $\ppt(\widehat\ba;\widehat\Delta)\leq\ppt(\ba;\Delta)$.
    Conversely, for an absolutely integrally closed perfectoid
    $R$-algebra $B$, set $\widehat B=B^{\wedge\m}$ and choose
    $\widehat B\to C$ as in \autoref{completion of divisor pairs} \autoref{completion of divisor pairs.NonsplittingIdeal2}. Then \autoref{completion of divisor pairs} \autoref{completion of divisor pairs.NonsplittingIdeal2} and
    \autoref{cor: basic functoriality of threshold} show that
    $\ppt(\ba;\Delta;B)
    =\ppt(\widehat\ba;\widehat\Delta;C)$; taking suprema over $B$ gives
    the converse inequality.
\end{proof}

\begin{proposition}\label{prop: perfectoid threshold from one algebra}
    Let $(R, \m, k)$ be a perfectoid pure local ring of residual characteristic $p > 0$, let $\ba$ be an ideal of $R$, and let $\Delta$ be an effective Mumford $\Q$-divisor on $R$. Assume that either $\Delta = 0$ or $R$ is an excellent reduced $\mathbf{S}_2$ ring.
        Then 
    \begin{enumerate}
    \item  $\ppt(\ba) = \ppt (\ba; \widehat{R}^{A_\infty}_\perfd)$
    whenever $\ba \widehat{R}^{A_\infty}_\perfd$ can be generated by elements admitting compatible systems of $p$-power roots (here $A$ and $A_{\infty}$ are as in \autoref{notation RAinfty});
    \item likewise, if $f = f_{\Delta} \in R$ is chosen as in \autoref{construction SnD}, 
    then we have $\ppt(\ba; \Delta) = \ppt (\ba; \Delta; \widehat{R}^{A_\infty}_\perfd)$
    whenever $f$ admits a compatible system of $p$-power roots in 
    $\widehat{R}^{A_\infty}_\perfd$ and $\ba \widehat{R}^{A_\infty}_\perfd$ can be generated by elements admitting compatible systems of $p$-power roots;
    \item  
    there is an absolutely integrally closed perfectoid $R$-algebra $B$ such that $\ppt(\ba; \Delta) = \ppt (\ba; \Delta; B)$ for all $\ba\subseteq R$.
    \end{enumerate}
\end{proposition}
\begin{proof}
    First, we note that, by \autoref{cor: ppt and completion}, we may assume that $R$ is complete and that there always exist relative perfectoidizations 
    $R^{A_\infty}_{\perfd}$ on which $\ppt (\ba; R^{A_\infty}_{\perfd})$
    is defined. 
    Namely, we can take any coefficient ring $\coff(k) \subseteq R$ and build a  surjection $\phi \colon A \coloneqq \coff(k)\llbracket x_1, \ldots, x_n\rrbracket \twoheadrightarrow R$ so that $\ba = (x_1, \ldots, x_g)R$ for some integer $g \in \{1, \ldots, n\}$. Then $x_i$ admit compatible systems of $p$-power roots in $R^{A_\infty}_{\perfd}.$ 
    The same argument applies to any finite number of ideals and to pairs
    via \autoref{can test divisor pairs on RAinfty}, by setting $x_{g+1} = f$. 

    We now show the first two claims by proving that $\ppt (\ba; \Delta; R^{A_\infty}_{\perfd}) \geq \ppt(\ba; \Delta; P)$ for any 
    $R$-pure perfectoid ring $P$ and $A$ constructed as in the first paragraph. 
    By \autoref{thm: Andre} and \autoref{cor: ppt in algebra extensions}, we may assume that $P$ is absolutely integrally closed, so, as explained in \cite[Lemma~4.23]{BMPSTWW3}, 
    there is a map $R^{A_\infty}_{\perfd} \to P$. Thus 
    $\ppt (\ba; \Delta; R^{A_\infty}_{\perfd}) \geq \ppt (\ba; \Delta; P)$ by \autoref{cor: ppt in algebra extensions}. By symmetry, this argument also shows that any two appropriate relative perfectoidizations have equal thresholds. 

    The final claim follows from \autoref{cor: ppt in algebra extensions} after applying \autoref{thm: Andre} to build a $p$-completely {faithfully flat} map $R^{A_\infty}_{\perfd} \to P$ to an absolutely integrally closed perfectoid $R$-algebra.  
\end{proof}

\begin{corollary}\label{cor: perfectoid threshold via pure pairs}
    Let $(R, \m)$ be an excellent reduced $\mathbf{S}_2$ ring of residual characteristic $p > 0$, $\Delta$ an effective Mumford $\Q$-divisor on $R$, and $\ba$ an ideal { not contained in any minimal prime}. Then we have 
    \begin{enumerate}
        \item If $\ba \subseteq \p$ for some height one prime $\p$ such that $R_\p$ is singular, then we have
        $$\begin{cases}
            \ppt(\ba; \Delta)=0 \text{ if $(R,\Delta)$ is perfectoid pure;}\\
            \ppt(\ba; \Delta)=-\infty \text{ if $(R,\Delta)$ is not perfectoid pure.} 
        \end{cases}$$\label{cor: perfectoid threshold via pure pairs.parta}
        \item Otherwise, there are minimal generators $(x_1, \ldots, x_m) = \ba$
    such that
    \[
    \ppt (\ba; \Delta) = \sup
    \left \{
    \frac{|a|}{p^{e}} \mid e \in \mathbb{Z}_{\geq 0}, a \in \mathbb{Z}_{\geq 0}^m, 
    (R, \Delta + \frac{1}{p^e} \Div (x_1^{a_1}\cdots x_m^{a_m}))
    \text{ is perfectoid pure}
    \right \}.
    \]    \label{cor: perfectoid threshold via pure pairs.partb}    
    \end{enumerate}
\end{corollary}
\begin{proof}
For \autoref{cor: perfectoid threshold via pure pairs.parta}, it is clear that $\ppt(\ba; \Delta)=-\infty$ if $(R,\Delta)$ is not perfectoid pure. So assume $(R,\Delta)$ is perfectoid pure. If $\ppt(\ba; \Delta)\neq 0$, then $\ppt(\ba; B)\neq 0$ for some perfectoid $R$-algebra $B$. By \autoref{prop: reduce to principal} there is some $f\in \ba$ so that $\ppt((f); B)\neq 0$. This implies that for all $0<\lambda <\ppt((f); B)$, the pair $(R,f^\lambda)$ is perfectoid pure.\footnote{Note that $f$ must be a nonzerodivisor as otherwise $f$ becomes zero in $R_Q$ for some minimal prime $Q$, and $R_Q\to B_Q$ sending $1\to f^{1/p^e}=0$ cannot split for any $e$. } But then \autoref{normalization is pfd pure pair}  \autoref{normalization is pfd pure pair.PrincipalCase} implies that $\Div(f)$ cannot contain any height one prime $\p$ so that $R_\p$ is singular; this is a contradiction to $f\in \ba$.

We next treat \autoref{cor: perfectoid threshold via pure pairs.partb}. Since $R$ is reduced, there are finitely many height-one prime ideals 
in the singular locus of $R$. Let $\p_1, \ldots, \p_N$
be these ideals {along with the minimal primes of $R$}. Since $\ba$ is not contained in any $\p_i$, by prime avoidance we can choose a generating set $\{x_1,\dots,x_m\}$ of $\ba$ so that $x_i$ is not in $\p_j$ for all $j$ and $i$. 
Thus a monomial on $\{x_1, \ldots, x_m\}$ is not contained in any $\p_i$. 
By \autoref{prop: perfectoid threshold from one algebra} 
we may compute $\ppt (\ba; \Delta)$ in a fixed perfectoid absolutely integrally closed algebra $B$. The conclusion then follows 
from \autoref{thm: from rational powers to monomials} and \autoref{perfectoid pure triples}. (Note that $\ppt(\ba; \Delta)$ could still be $-\infty$ in case \autoref{cor: perfectoid threshold via pure pairs.partb} since when $(R,\Delta)$ is not perfectoid pure the right hand side is the supremum over an empty set, so it equals $-\infty$ by convention).
\end{proof}

\begin{proposition}[Elementary properties]\label{prop: elementary properties}
    Let $(R, \m)$ be a perfectoid pure local ring of residual characteristic $p > 0$ and $\ba, \bb$ be ideals. In what follows, either assume $\Delta=0$ or assume that $R$ is an excellent reduced $\mathbf{S}_2$ ring and $\Delta$ is an effective Mumford $\Q$-divisor on $R$.  
    Then 
    \begin{enumerate}
        \item $\ppt (\ba + \bb; \Delta) \leq \ppt (\ba; \Delta)  + \ppt (\bb; \Delta)$; \label{part: semiadditive}
        \item if $\ba \subset \bb$ then $\ppt (\ba; \Delta)  \leq \ppt (\bb; \Delta)$; \label{part: containment}
        \item if $\overline{\ba} = \overline{\bb}$ then $\ppt (\ba; \Delta)  = \ppt (\bb; \Delta)$; \label{part: integral closure}
        \item $\ppt (\ba^n; \Delta)  = n^{-1} \ppt (\ba; \Delta) $ for $n \in \bZ_{\geq 1}$; \label{part: power scaling}
        \item if $\Delta'$ is an
        effective Mumford $\QQ$-divisor with $\Delta'\leq\Delta$, then
        $\ppt(\ba; \Delta) \leq \ppt (\ba; \Delta')$;
        \item if $(p, \ba) \subseteq \mathfrak{q} \in \Spec (R)$ then $\ppt (\ba; \Delta)  \leq \ppt (\ba R_\mathfrak{q}; \Delta R_\mathfrak{q}) $.
    \end{enumerate}
\end{proposition}
\begin{proof}
    By \autoref{prop: perfectoid threshold from one algebra} there is an absolutely integrally closed perfectoid ring $B$
    that computes $\ppt(\ba; \Delta), \ppt(\bb; \Delta)$, and $\ppt(\ba+\bb; \Delta)$. 
    The first four properties now follow from \autoref{lem: elementary threshold} and \autoref{cor: ppt sum of ideals}. The next property follows immediately from the definitions and \autoref{divisorial module computed by towers}.
    
    The last property follows immediately from \autoref{cor: perfectoid threshold via pure pairs} and \autoref{localization of purity} when $\Delta \neq 0$ and from a similar argument when $\Delta = 0$. 
\end{proof}

\begin{corollary}
    Let $(R, \m)$ be a perfectoid pure local ring and $\ba$ be a proper ideal. 
    Furthermore, either assume $\Delta=0$ or assume that $R$ is an excellent reduced $\mathbf{S}_2$ ring and $\Delta$ is an effective Mumford $\Q$-divisor on $R$ {such that $(R, \Delta)$ is perfectoid pure}.
    Then for every $\varepsilon > 0$ there exists $n \in \mathbb{Z}_{> 0}$
    such that for every ideal $\bb \equiv \ba \mod {\m^n}$ we have 
    $$|\ppt (\ba; \Delta) - \ppt (\bb; \Delta)| < \varepsilon.$$ 
    In particular, for any $f \in \m$ and any $\varepsilon > 0$ there exists $n \in \mathbb{Z}_{> 0}$
    such that for all $g \in \m^n$
    we have 
    $
     |\ppt (f) - \ppt (f + g)| \leq \varepsilon.
    $
\end{corollary}
\begin{proof}
    Suppose that $\bb \equiv \ba \mod {\m^n}$. 
    Then we have $\ba \subseteq \bb + \m^n$.  Hence by \autoref{prop: elementary properties}
    \[
    \ppt (\ba; \Delta) \leq \ppt (\bb + \m^n; \Delta) \leq \ppt (\bb; \Delta) + \ppt (\m^n; \Delta)
    = \ppt (\bb; \Delta) + n^{-1}\ppt (\m; \Delta),
    \]
    and so by symmetry $|\ppt (\ba; \Delta) - \ppt (\bb; \Delta)| \leq n^{-1} \ppt (\m; \Delta)$.
\end{proof}

\begin{corollary}\label{cor: compute regular}
    Let $(R, \m, k)$ be a regular local ring of residual characteristic $p > 0$. Then $\ppt (\m) = \dim (R)$.  
\end{corollary}
\begin{proof}
    We may assume that $R$ is complete. 
    If $R$ is unramified of mixed characteristic then $R \cong \coff(k)\llbracket x_1, \ldots, x_d\rrbracket $, so 
    we may compute 
    \[
    \ppt(\m) = 
     \sup 
    \left \{
    \frac{|a|}{p^{e}} \mid a \in \mathbb{Z}_{\geq 0}^{d+1}, x_1^{a_1/p^e}\cdots x_d^{a_d/p^e}p^{a_{d+1}/p^e }\notin (x_1, \ldots, x_d, p)R_\infty
    \right \} = d+1 = \dim (R), 
    \]
    since $p, x_1, \ldots, x_d$ form a regular sequence in $R_\infty$.
    
    Otherwise, $R$ can be written as 
    $R \cong S/(p - f(x_1, \ldots, x_d))$ 
    where $S = \coff(k)\llbracket x_1, \ldots, x_d\rrbracket $.
    Since $\m = (x_1, \ldots, x_d)$ in this case,  we may compute 
    \[
    \ppt (\m) = 
    \sup 
    \left \{
    \frac{|a|}{p^{e}} \mid a \in \mathbb{Z}_{\geq 0}^{d}, x_1^{a_1/p^e}\cdots x_d^{a_d/p^e} \notin (x_1, \ldots, x_d)R^{S_\infty}_\perfd 
    \right \} = d = \dim (R),
    \]
    since again $x_1, \ldots, x_d$ form a regular sequence on $R^{S_\infty}_{\perfd}$.  
\end{proof}

We end this subsection by showing that in characteristic $p>0$, $\ppt (\ba; B)$ does not depend on the choice of $B$ {\it as long as $R\to B$ is pure}.
\begin{proposition}\label{prop: recover char p}
If $(R, \m)$ is a Noetherian $F$-finite local ring of characteristic $p > 0$, $\ba = (x_1, \ldots, x_d)$, and $R \to B$ is a pure map to a perfect ring $B$, then
\[
\fpt (\ba) = \ppt (\ba; B) = \ppt  (\ba; R_\perf).
\]
Furthermore, if $R$ is a normal domain and $\Delta$ is an effective $\Q$-divisor on $R$, then 
\[
\fpt(\ba; \Delta) = \ppt (\ba; \Delta; B).
\]
\end{proposition}
\begin{proof}
We may assume that $R$ is F-pure. 
Note that the purity of $R \to B$ implies purity of 
$R_\perf \to B$ since $B$ is perfect. 
Thus 
$\ppt (\ba; \Delta; {R_\perf}) = 
\ppt (\ba; \Delta; B)$ by \autoref{cor: ppt in algebra extensions}.
Hence, it suffices to show the statements for $B = R_\perf$. 

First, assume that $\Delta = 0$. 
Recall that by \cite{dSNB} 
we have 
\[
\fpt (\ba) = 
\lim_{e \to \infty} \frac{\sup \{m \mid \ba^m \not \subseteq I_e\}}{p^e}
= 
\lim_{e \to \infty} \frac{\sup \{|a| \mid a \in \ZZ_{\geq 0}^{d}, x_1^{a_1}\cdots x_d^{a_d} \notin I_e\}}{p^e},
\]
where $I_e = \{x \in R \mid \phi(x) \in \m, \forall \phi \in \Hom_R (F_*^e R, R)\}$, so $F_*^e I_e = \np (F_*^e R)$. 
By \autoref{lem: functoriality of non-pure ideals} 
\[
\np(R_\perf) \cap F_*^e R = \np(F_*^e R) = F_*^e I_e,
\]
thus the conclusion follows from \autoref{thm: from rational powers to monomials}.

We now assume $\Delta \neq 0$. Set $\Delta_e = p^{-e} \lfloor (p^e - 1)\Delta \rfloor$.
Adapting \cite{dSNB}, we may compute
\begin{align*}
\fpt (\ba; \Delta) &= \lim_{e \to \infty} p^{-e}\sup \left\{|a| \mid a \in \ZZ_{\geq 0}^d, x_1^{a_1/{p^e}}\cdots x_d^{a_d/{p^e}} \notin \np^{R} (R^{1/p^e}, R(p^e \Delta_e)^{1/p^e})\right \}.
\end{align*}
Set $J_e = \np(R_\perf, R_\perf \tensor_{R^{1/p^e}} R(p^e\Delta_e)^{1/p^e})$. 
Since $J_e \cap R^{1/p^e} = \np^{R} (R^{1/p^e}, R(p^e \Delta_e)^{1/p^e})$,
it follows that 
\[
\fpt (\ba; \Delta) = \lim_{e \to \infty} p^{-e}\sup \left\{|a| \mid a \in \ZZ_{\geq 0}^d, x_1^{a_1/{p^e}}\cdots x_d^{a_d/{p^e}} \notin J_e\right \}.
\]
The argument of \autoref{perfectoid purity in characteristic p divisor}
shows that a subsequence $\Delta_e, e \in S \subseteq \N$
forms a limiting sequence for $\Delta$, so 
$\np (R_\perf, {<}\Delta) = \cup_{e \in S} J_e$. Thus, 
\autoref{thm: from rational powers to monomials} immediately shows that 
\begin{align*}
\ppt (\ba; \Delta) 
&= \lim_{n \to \infty} p^{-n}\sup 
    \left \{
    |a| \mid a \in \mathbb{Z}_{\geq 0}^d, x_1^{a_1/p^n}\cdots x_d^{a_d/p^n}\notin J_e \text{ for all } e \in S
    \right \}
\\ &\leq   \lim_{n \in S} p^{-n} \sup 
    \left \{
    |a| \mid a \in \mathbb{Z}_{\geq 0}^d, x_1^{a_1/p^n}\cdots x_d^{a_d/p^n}\notin J_n
    \right \} = \fpt (\ba; \Delta).
\end{align*}

On the other hand,  by \autoref{rmk: limiting sequence for thresholds}, we may compute
\(
\ppt (\ba; \Delta) = \lim_{e \in S} \newpt{\ba {R_{\perf}}}{J_e},
\)
so it is immediate from \autoref{thm: from rational powers to monomials} that 
\[
\ppt (\ba; \Delta) = \lim_{e \in S} 
\sup 
    \left \{
    \frac{|a|}{p^n} \mid a \in \mathbb{Z}_{\geq 0}^d, x_1^{a_1/p^n}\cdots x_d^{a_d/p^n}\notin J_e
    \right \}
\geq 
\fpt (\ba; \Delta).
\]
\end{proof}

\begin{remark}
    The proof of \autoref{prop: recover char p} shows that
    the perfectoid threshold generalizes the Frobenius closure version of F-threshold: if $R$ has characteristic $p > 0$ and $\ba, J$ are ideals,  then 
    \[
    \newpt{\ba R_{\perf}}{J R_{\perf}}
    = \lim_{e \to \infty} \frac 1 {p^e} {\sup \{m \mid \ba^m \not \subseteq (J^{[p^e]})^F\}}.
    \]
    In the same spirit \cite[Corollary~3.9]{RVS} shows that a tight closure version of F-threshold can be recovered as $\newpt{\ba B}{JB}$ for certain BCM algebras $B$.
\end{remark}

As we mentioned, we do not know a good analog of \autoref{prop: recover char p} in mixed characteristic. 

\begin{question}
Let $(R, \m)$ be a Noetherian local ring of mixed characteristic $(0,p>0)$ that is perfectoid pure. Then is $\ppt(\ba) = \ppt (\ba; B)$ whenever $R \to B$ is pure and $B$ is perfectoid? 
\end{question}

An affirmative answer to this question will be a generalization of the well-known fact that in a strongly F-regular ring the strongly F-regular threshold coincides with F-pure threshold. We record the following partial result towards this question.

\begin{proposition}
      If $(R, \m, k)$ is a complete regular local ring {of residual characteristic $p > 0$ }then $\ppt (\ba) = \ppt (\ba; \widehat{R^+})$ for any ideal $\ba \subset R$.
\end{proposition}
\begin{proof}
Let $C_k$ be a coefficient ring of $R$. Set $R'= R\widehat{\otimes}_{C_k}W(k^{1/p^\infty})$. Note that the map $C_k\to W(k^{1/p^\infty})$ is the $p$-adic completion of an integral map and thus $W(k^{1/p^\infty})$ and hence $R'$ admit a map to $\widehat{R^+}$. Let $x_1,\dots,x_d$ be a regular system of parameters of $R$ and fix a compatible system of $p$-power roots of them in $R^+$. Set $R_n=R'[x_1^{1/p^n},\dots,x_d^{1/p^n}]$ and $R_\infty =\widehat{\varinjlim} R_n$. Then $R_\infty$ is perfectoid and each $R_n$ is regular local. Let $R''$ be the $\m$-adic completion of $R'$; then by \cite[Theorem 5.16]{Bhatt}, the natural map $R_n\to \widehat{R''^+}$ is faithfully flat and hence $R_\infty\to \widehat{R''^+}$ is $p$-completely faithfully flat. Set $\ba = (f_1, \ldots, f_n)$ and let $S=R[y_1,\dots,y_n]$. Consider the natural map $S\to R$ sending $y_i$ to $f_i$ and  define $R^{S_\infty}_\perfd$ as in \autoref{notation RAinfty}.
    The ring $B:=\widehat{R''^+} \widehat{\otimes}^{\mathbb{L}}_{R_\infty} R^{S_\infty}_\perfd$ is then perfectoid and $p$-completely faithfully flat 
    over $R^{S_\infty}_\perfd$. 
    Then we have 
    $$\ppt(\ba)\geq \ppt(\ba; \widehat{R^+}) \geq \ppt(\ba; \widehat{R''^+})\geq \ppt (\ba; B) = \ppt (\ba; R^{S_\infty}_\perfd) = \ppt (\ba) $$
    by \autoref{cor: ppt in algebra extensions} 
    and \autoref{prop: perfectoid threshold from one algebra}. Thus we must have equalities throughout which completes the proof.
\end{proof}

\subsection{Localization and regular extensions}
In this subsection we will strengthen and generalize \cite[Corollary 4.21 (a)]{BMPSTWW3} by proving a faithfully flat ascent result for pairs, \autoref{thm: ppt in regular fibers for pairs}. In particular, we will show that perfectoid purity is preserved under faithfully flat extensions with regular closed fiber (thus answering \cite[Conjecture 4.22 (a)]{BMPSTWW3}). 

We start with a lemma treating a partial case. 

\begin{lemma}\label{lem: adjoining variable for pairs}
    Let $(R, \m)$ be a complete local ring of residual characteristic $p > 0$ and $P$ be a perfectoid $R$-algebra. Set $T = P[x^{1/p^{\infty}}]^{\wedge{(\m, x)}}$.
    If $M$ is a pointed $P$-module, then 
    \[
    \np^{R\llbracket x\rrbracket } (T, M \otimes_P T) = \np^R (P, M)T + xT.
    \] 
\end{lemma}
\begin{proof}
Set $M_T = M \otimes_P T$ and note that $x$ is a regular element on 
$M_T/(N \otimes_P T)$ for any submodule $N \subset M$.
By \autoref{lem: functoriality of non-pure ideals} and \autoref{lem: injective hull splitting ideal}, it suffices to show that $\np^{R\llbracket x\rrbracket } (T, M_T) \subseteq \np^R (P, M)T + xT$. We may further assume that $R \to P$
is pure, so $R$ is reduced. 

Let $I_t$ be an approximately Gorenstein sequence in $R$ 
and $u_t$ be elements such that $u_t R/I_t$ is the socle ideal. 
By \autoref{lem.SplittingApproxGorenstein} we may compute 
\[
\np^R (T, M_{T}) = \bigcup I_tM_T :_T u_t \mathbf 1_{M_T} \text{ and } 
\np^{R\llbracket x\rrbracket } (T, M_T) = \bigcup (I_t + x^t)M_T :_T u_tx^{t-1}\mathbf 1_{M_T}.
\]
Let $a \in (I_n + x^n) M_T :_T u_nx^{n-1} \mathbf 1_{M_T}$. 
By definition, there is $f \in M_T$ such that
\[
x^{n-1} (au_n \mathbf 1_{M_T} + fx) \in I_n M_T.
\]
{Note $x$ is a regular element here since $I_n M_T \subseteq M_T$ is obtained by base change from $P$,} and so $( au_n \mathbf 1_{M_T} + fx) \in I_n M_T$, i.e., 
\[
fx \in (I_n M_T, T u_n \mathbf 1_{M_T}) = (I_n M, P u_n \mathbf 1_{M}) \otimes_P T.
\]
Hence, $f \in (I_n M_T, T u_n \mathbf 1_{M_T})$, 
so there is $b \in T$ such that $f - bu_n \mathbf 1_{M_T} \in I_n M_T$.
This implies that 
$
(a {+} b x)u_n \mathbf 1_{M_T} \in I_n M_T,  
$
thus  $a { +} bx \in I_n M_T :_T u_n \mathbf 1_{M_T} \subseteq \np^R (T, M_T)$.
It remains to note that $\np^R (T, M_T) \subseteq \np^R (P, M)T + xT$, which follows from \autoref{lem: functoriality of non-pure ideals} \autoref{part algebra functoriality} and \autoref{part extension for p completely flat} applied to the $\m$-completely faithfully flat map $P\to T/xT$.
\end{proof}

\begin{corollary}\label{cor: ppt after adding a variable}
Let $(R, \m)$ be a complete local perfectoid pure ring of residue characteristic $p$, $\ba$ be an ideal of $R$, and $\Delta$ be an effective Mumford $\Q$-divisor on $R$. If $\Delta \neq 0$, then further assume that $R$ is an excellent reduced $\mathbf{S}_2$ ring. Let $\Delta\llbracket x\rrbracket $ be the flat pullback of $\Delta$
to $R\llbracket x\rrbracket $. Then 
\[
 \ppt ((\ba+x) R\llbracket x\rrbracket ; \Delta\llbracket x\rrbracket ) = \ppt (\ba; \Delta) + 1.
\]
\end{corollary}
\begin{proof}
Fix a surjection $A \twoheadrightarrow R$ such that $\ppt (\ba; \Delta; R^{A_\infty}_\perfd) = \ppt (\ba; \Delta)$ (\autoref{prop: perfectoid threshold from one algebra}). 
It follows from the universal property 
that there is a natural map
\[
R\llbracket x\rrbracket ^{(A\llbracket x\rrbracket )_\infty}_\perfd \to 
(R^{A_\infty}_\perfd [x^{1/p^{\infty}}])^{\wedge{(\m, x)}} \cong \big(R\llbracket x\rrbracket ^{(A\llbracket x\rrbracket )_\infty}_\perfd\big)^{\wedge(\m, x)} =:T.
\]
By \autoref{cor: ppt in algebra extensions} and \autoref{prop: perfectoid threshold from one algebra} we know that 
$$\ppt ((\ba+x) R\llbracket x\rrbracket ; \Delta\llbracket x\rrbracket )=\ppt ((\ba+x) R\llbracket x\rrbracket ; \Delta\llbracket x\rrbracket ; T).$$ 
Applying \autoref{lem: adjoining variable for pairs} to 
the sequence $(\raftypfd, M_{\raftypfd}(\mathfrak d_e))$
from \autoref{universal divisorial module} as in \autoref{base change nonpure ideal below delta}, 
we obtain that 
\begin{equation}
\label{eq1.cor: ppt after adding a variable}
\np^{R\llbracket x\rrbracket }(T, {<}\Delta\llbracket x\rrbracket ) = \np^R (R^{A_\infty}_\perfd, {<}\Delta)T + xT.
\end{equation}
By \autoref{thm: from rational powers to monomials} it remains 
to show that, writing $\ba=(y_1,\dots,y_n)$ so that each $y_i$ admits a compatible system of $p$-power roots in $\raftypfd$, if
a monomial $m = y_1^{a_1/p^e}\cdots y_n^{a_n/p^e} \in R^{A_\infty}_\perfd$ 
is not contained in $\np^R (R^{A_\infty}_\perfd, {<}\Delta)$
then $x^{1 - 1/p^e}m \notin \np^{R\llbracket x\rrbracket }(T, {<}\Delta\llbracket x\rrbracket )$. 
If not, then 
\[
x^{1 - 1/p^e} m - xs \in \np^R (R^{A_\infty}_\perfd, {<}\Delta)T
\]
for some $s \in T$.
Since $x$ is a regular element modulo $IT$ for every $I \subseteq R^{A_\infty}_\perfd$, we must have 
\begin{equation}
\label{eq2.cor: ppt after adding a variable}
m - sx^{1/p^e} \in \np^R (R^{A_\infty}_\perfd, {<}\Delta) T, 
\end{equation}
i.e., 
$sx^{1/p^e} \in (m, \np^R (R^{A_\infty}_\perfd, {<}\Delta)) T$.
Thus, there is $t \in T$ such that 
$s - mt \in \np^R (R^{A_\infty}_\perfd, {<}\Delta) T$.
Plugging { this}  into {\autoref{eq2.cor: ppt after adding a variable} }
we obtain that 
\[
m (1 - tx^{1/p^e}) \in \np^R (R^{A_\infty}_\perfd, {<}\Delta) T.
\]
Since $T$ is $(x)$-complete, $(1 - tx^{1/p^e})$ is invertible, so by
\autoref{lem: functoriality of non-pure ideals} {and \autoref{eq1.cor: ppt after adding a variable}}
\[
m \in \np^R (R^{A_\infty}_\perfd, {<}\Delta) T \cap R^{A_\infty}_\perfd
\subseteq \np^R (T, {<}\Delta) \cap R^{A_\infty}_\perfd = 
\np^R (R^{A_\infty}_\perfd, { {<}\Delta}),
\]
a contradiction.     
\end{proof}

\begin{theorem}\label{thm: ppt in regular fibers for pairs}
Let $(R, \m) \to (S, \n)$ be a faithfully flat {local} map of Noetherian local rings of residue characteristic $p$ such that $S/\m S$ is regular. Let $f_1,\dots,f_n \in S$ be elements whose images form a regular system of parameters on $S/\m S$. 
Let $\ba$ be an ideal of $R$ and $\Delta$ be an effective Mumford $\Q$-divisor on $R$. 
If $\Delta \neq 0$, further assume that $R$ and $S$ are excellent reduced $\mathbf{S}_2$ rings and that the flat pullback $\Delta_S$ of $\Delta$ is a Mumford $\Q$-divisor (\autoref{flat pullback is Mumford}).
Then 
\[
\ppt ((\ba S+(f_1,\dots,f_n))S; \Delta_S) = \ppt (\ba; \Delta) + \dim (S/\m S).
\]
In particular, if $(R,\Delta)$ is perfectoid pure, then so is $(S,\Delta_S)$. 
\end{theorem}
\begin{proof}
We may assume that $R$ is perfectoid pure, because if 
it is not then $S$ is also not perfectoid pure (since an $S$-algebra is an $R$-algebra) so all the thresholds involved are $-\infty$. We may also assume that both rings are complete by \autoref{cor: ppt and completion}. Using \autoref{cor: ppt after adding a variable}, by adjoining power series variables $x_1,\dots,x_n$ to $R$ and mapping them to $f_1,\dots,f_n$ in $S$, we reduce the theorem to the case where $\n = \m S$, in which case we need to show that $\ppt(\ba; \Delta)=\ppt(\ba S; \Delta_S)$.

Let us choose coefficient rings $V_R \subseteq R$ and $V_S \subseteq S$. 
Fix a surjection $B':=V_S \llbracket y_1, \ldots, y_h\rrbracket  \twoheadrightarrow S$
such that $\ppt (\ba S; \Delta_S; S^{B'_\infty}_{\perfd}) = \ppt (\ba S; \Delta_S)$ (\autoref{prop: perfectoid threshold from one algebra}). 
By \cite[IV, 19.8.6.i]{EGA64}, the composition $V_R \to S$ 
can be lifted to a homomorphism  $V_R \to B'=V_S \llbracket y_1, \ldots, y_h\rrbracket $. 
We now take $A = V_R \llbracket x_1, \ldots, x_g\rrbracket  \twoheadrightarrow R$ to be 
any surjection such that 
$\ppt (\ba; \Delta; R^{A_\infty}_{\perfd}) = \ppt (\ba; \Delta)$ (\autoref{prop: perfectoid threshold from one algebra}) 
and define $B = V_S \llbracket x_1, \ldots, x_g, y_1, \ldots, y_h\rrbracket $. 
We have a faithfully flat map $A \to B$. Also note that $\ppt (\ba S; \Delta_S; S^{B'_\infty}_{\perfd}) = \ppt (\ba S; \Delta_S; S^{B_\infty}_{\perfd})$ by \autoref{cor: ppt in algebra extensions} since $S^{B'_\infty}_{\perfd}\to S^{B_\infty}_{\perfd}$ is $p$-completely faithfully flat by Andr\'{e}'s lemma \cite[Proposition~7.11]{BhattScholzePrisms}. 

Consider the perfectoid ring $P = (R^{A_\infty}_{\perfd} \widehat{\otimes}^\mathbb{L}_R S \widehat{\otimes}^\mathbb{L}_B B_\infty)_\perfd$. 
By \cite[Proposition 8.13]{BhattScholzePrisms}, for any diagram $T_1\leftarrow T_3 \to T_2$ we know that $(T_1 \widehat{\otimes}^\mathbb{L}_{T_3} T_2)_\perfd = (T_{1, \perfd} \widehat{\otimes}^\mathbb{L}_{T_3} T_2)_\perfd$,
thus we have
\[
P
= (A_\infty \widehat{\otimes}^\mathbb{L}_A R \widehat{\otimes}^\mathbb{L}_R S \widehat{\otimes}^\mathbb{L}_B B_\infty)_\perfd
= (A_\infty \widehat{\otimes}^\mathbb{L}_A S \widehat{\otimes}^\mathbb{L}_B B_\infty)_\perfd
= (A_\infty \widehat{\otimes}^\mathbb{L}_A S^{B_\infty}_{\perfd})_\perfd.
\]
By Andr{\'e}'s lemma (\cite[Proposition~7.11]{BhattScholzePrisms}), we find that
$S^{B_\infty}_{\perfd} \to P$  is $p$-completely faithfully flat. 
Therefore, by \autoref{cor: ppt in algebra extensions} we have 
\[
\ppt(\ba; \Delta) = \ppt (\ba; \Delta; R^{A_\infty}_{\perfd}) 
\quad \text{and} \quad \ppt (\ba S; \Delta_S; P) = \ppt (\ba S; \Delta_S; S^{B_\infty}_{\perfd}) = \ppt (\ba S; \Delta_S).
\]
By \autoref{cor: ppt of pairs in ring extensions}
we also have $\ppt (\ba; \Delta; P) = \ppt (\ba S; \Delta_S; P)$, so it follows immediately that 
\[
\ppt(\ba; \Delta) = 
\ppt (\ba; \Delta; R^{A_\infty}_{\perfd}) 
\geq \ppt (\ba; \Delta; P) = \ppt (\ba S; \Delta_S).
\]
Thus, it suffices to show that $\ppt (\ba; \Delta; R^{A_\infty}_{\perfd}) 
\leq \ppt (\ba; \Delta; P)$.

We claim that the $p$-completed base change $R^{A_\infty}_{\perfd} \to R^{A_\infty}_{\perfd} \widehat{\otimes}^\mathbb{L}_{R} S \widehat{\otimes}^\mathbb{L}_{B} B_\infty$ is quasisyntomic. 
First, since $R \to S$ is a complete intersection morphism,
\cite[Theorem~1.2]{Avramov} shows that
the Tor dimension of the cotangent complex $L_{S/R}$ is at most one. 
This shows that the $p$-completely faithfully flat base change $R^{A_\infty}_{\perfd} \to R^{A_\infty}_{\perfd} \widehat{\otimes}_R^\mathbb{L} S$
is quasisyntomic. Second, $B \to B_\infty$ is ind-syntomic, so a further base change
$R^{A_\infty}_{\perfd} \widehat{\otimes}^\mathbb{L}_{R} S \to R^{A_\infty}_{\perfd} \widehat{\otimes}^\mathbb{L}_{R} S \widehat{\otimes}^\mathbb{L}_{B} B_\infty$ is quasisyntomic.

Since $R^{A_\infty}_{\perfd} \to R^{A_\infty}_{\perfd} \widehat{\otimes}^\mathbb{L}_{R} S \widehat{\otimes}^\mathbb{L}_{B} B_\infty$ is quasisyntomic, 
by \cite[Proposition~7.11]{BhattScholzePrisms}
there exists a perfectoid $P$-algebra $T$ such that the composition 
$R^{A_\infty}_{\perfd} \to P \to T$ is 
$p$-completely faithfully flat. 
Thus, \autoref{cor: ppt in algebra extensions} yields the desired
\(
\ppt (\ba; \Delta; P) \geq 
\ppt (\ba; \Delta; T) 
= \ppt (\ba; \Delta; R^{A_\infty}_\perfd). 
\)
\end{proof}

We end this subsection by {describing} the behavior of perfectoid pure threshold of the maximal ideal under localizations.

\begin{proposition}\label{prop: localization inequality}
    Let $(R, \m)$ be an excellent local ring of residual characteristic $p > 0$, $\Delta$ be an effective Mumford $\Q$-divisor on $R$. 
    If $\Delta \neq 0$, we further assume that $R$ is reduced and $\mathbf{S}_2$.  
    If $p \in \q$ is a prime ideal such that 
    $\dim (R_\q) + \dim (R/\q) = \dim (R)$, 
    then 
\[
    \ppt (\m; \Delta) \leq \dim(R/\q)+ \ppt (\q R_\q; \Delta_\q).
    \]
\end{proposition}

\begin{proof}
By \autoref{cor: ppt and completion} and \autoref{thm: ppt in regular fibers for pairs}, we may assume that $R$ is complete. By a straightforward induction we may assume that $\dim(R/\q)=1$. 

Since $R/\q$ is a complete local domain of dimension one, its normalization $V$ is a DVR module-finite over $R/\q$. 
Let $1,f_1,\dots,f_n$, where $f_i\in V$, generate $V$ as a module over $R/\q$. Consider the surjective map 
$$(S,\n):=R[x_1,\dots,x_n]_{\n}  \twoheadrightarrow V$$
sending $x_i$ to $f_i$ where $\n$ is the inverse image of $\m_V$. Let $Q$ be the kernel of this map; then $Q\in\Spec(S)$ lies over $\q\in \Spec(R)$. In particular, we have an induced commutative diagram: 
\[\xymatrix{
R \ar[d] \ar[r] & S  \ar[d] \\
R_\q \ar[r] & S_Q
}
\]
where both horizontal maps are faithfully flat with regular closed fibers of dimension $n$: this is obvious for $R\to S$, and for $R_\q\to S_Q$, note that $\dim(S_Q/\q S_Q)=n$ while $Q(S_Q/\q S_Q)$ can be generated by $(x_1-f_1,\dots,x_n-f_n)$ since $f_i\in V \subseteq R_\q/\q R_\q\subseteq S_Q/\q S_Q$. 

Let $\Delta_S$ and $\Delta_{S_Q}$ be the obvious pullbacks of $\Delta$ to $S$ and $S_Q$ respectively. By \autoref{thm: ppt in regular fibers for pairs}, we have that 
$$\ppt(\m; \Delta) = \ppt(\n; \Delta_S) -n \text{ and } \ppt(\q R_\q; \Delta_\q) = \ppt(QS_Q; \Delta_{S_Q}) -n.$$
Therefore, in order to prove the proposition, we may further replace $R$ by $S$ and $\q$ by $Q$ to assume that $R/\q$ is a DVR, i.e., $\m = (\q, x)$ for some $x\in \m$. Now by \autoref{prop: elementary properties} we have    
    \[
    \ppt (\m; \Delta) =\ppt ((\q, x); \Delta) \leq \ppt (\q; \Delta) + \ppt (x; \Delta) 
    \leq \ppt (\q R_\q; \Delta_\q) + 1.
    \]
\end{proof}

\subsection{Measuring singularities}
In this subsection we connect both $\ppt(\ba)$, respectively $\ppt(\ba; \Delta)$, with the singularities of $R$, respectively $(R,\Delta)$. We start by proving some bounds on the perfectoid pure threshold using analytic spread.

\begin{corollary}
    \label{cor.AnalyticSpreadPPT}
    Let $(R, \m)$ be a Noetherian local ring of residual characteristic $p > 0$ and $\ba { \subseteq \m}$ be an ideal.
    Then $\ppt (\ba) \leq \ell (\ba)$, where $\ell (\ba)$ denotes the analytic spread of $\ba$. 
    Moreover, 
    $$\ppt (\ba) \leq \min \{\dim (R_Q) \mid (p, \ba) \subseteq Q \in \Spec (R)\}.$$
\end{corollary}
\begin{proof}
    We may assume that $R$ is perfectoid pure as otherwise $\ppt(\ba)={-}\infty$.  
    First, if $\ba = xR$ is principal, then $\ppt (\ba) \leq 1$ 
    by \autoref{thm: from rational powers to monomials}
    since $xB \subseteq \np (B)$ for any $R$-algebra $B$. 
    
    We may replace $R$ with $R(T)$ by \autoref{thm: ppt in regular fibers for pairs}
    to assume that the residue field is infinite. 
    Hence, if $\lambda = \ell (\ba)$ then there are elements $x_1, \ldots, x_\lambda \in R$ such that 
    $\overline{(x_1, \ldots, x_\lambda)} { = \overline{\ba}}$. 
    It now follows from \autoref{prop: elementary properties} 
    that
    \[
    \ppt (\ba) \leq \ppt (x_1R) + \cdots + \ppt (x_\lambda R) \leq \lambda = \ell (\ba).  
    \]
    
    For the second bound, since 
    $\ppt (\ba) \leq \ppt ((p, \ba)) \leq \ppt ((p, \ba)R_Q)$, 
    we may assume that $p \in \ba$ and $\ba$ is $\m$-primary. 
    But then $\ell (\ba) = \dim (R)$ and the conclusion follows from the first claim.  
\end{proof}

We now extend an argument of Takagi--Watanabe \cite[Proposition~2.6]{TakagiWatanabe} to our setting. 

\begin{proposition}\label{prop: Takagi-Watanabe}
    Let $(R, \m, k)$ be a perfectoid pure Noetherian local ring with infinite 
    residue field $k$. 
    Let $\ba$ be an ideal of analytic spread $\lambda$ and   
    $x_1, \ldots, x_\lambda \in R$ form 
    a minimal reduction of $\ba$.
    If $\ppt (\ba) > \lambda - r$ for some integer $\lambda \geq r \geq 1$, then 
    $\ba^r \subseteq (x_1, \ldots, x_\lambda)$.
\end{proposition}
\begin{proof}
We may assume that $R$ is complete by \autoref{cor: ppt and completion}. Let $B$ be an absolutely integrally closed perfectoid $R$-algebra such that $\ppt (\ba) = \ppt (\ba; B)$.
    By the assumption on $\ppt(\ba)$ we may find an integer $n$ and 
    an $R$-linear map $\phi\colon B \to R$
    such that 
    $\phi (\overline{\ba^{1/n} {B}} \cdot \overline{\ba^{\lambda - r}{B}} ) = R$.
    By \autoref{pfd of ideal only depends on radical}, $\overline{\ba^{1/n}{B}} \subseteq (\overline{\ba^{1/n}{B}})_\perfd = (\ba {B})_\perfd$.
    Thus by \autoref{BS property} applied to ${I}=(x_1,\dots,x_\lambda)$, we have
    \[
    \ba^r = \ba^r \phi ((\ba {B})_\perfd \overline{\ba^{\lambda - r} {B}})
    \subseteq \phi (\ba_\perfd \overline{\ba^{\lambda}}{B})
    \subseteq \phi ((x_1, \ldots, x_\lambda){B}) \subseteq (x_1, \ldots, x_\lambda).
    \]
\end{proof}

As a consequence, we show that the perfectoid pure threshold of the maximal ideal detects regularity, generalizing the analogous characteristic $p>0$ result for $F$-pure thresholds, \cite[Theorem 2.7(1)]{TakagiWatanabe}.

\begin{corollary}\label{cor: characterize regular}
    Let $(R, \m, k)$ be a perfectoid pure Noetherian local ring of residual characteristic $p > 0$. 
    Then $\ppt(\m) >  \dim (R) - 1$ if and only if $R$ is regular (in which case $\ppt(\m) = \dim (R)$).    
\end{corollary}
\begin{proof}
    We may assume that $R$ is complete by \autoref{cor: ppt and completion}.  By \autoref{thm: ppt in regular fibers for pairs} we may assume that $k$ is infinite. Taking $r = 1$ in \autoref{prop: Takagi-Watanabe} shows that $\fram \subseteq (x_1, \dots, x_{\dim R}) \subseteq \fram$ and so $R$ is regular.  The converse is \autoref{cor: compute regular}.
\end{proof}

We next prove some behavior of the perfectoid pure threshold under adjunction for \emph{compatible} ideals (generalizing \cite[Lemma~2.19]{DSNBS}).

\begin{proposition}\label{ppt under quotients}
Let $(R,\m)$ be a Noetherian complete local ring of residue characteristic $p>0$ and let $B$ be a perfectoid $R$-algebra. If $I\subseteq R$ is $\phi$-compatible for all $\phi\in \Hom_R(B,R)$, then for all $\ba\subseteq R$ such that $\ba B$ can be generated by elements admitting compatible systems of $p$-power roots we have the inequality
$$\ppt(\ba(R/I); (B/IB)_\perfd) \geq \ppt(\ba; B).$$
\end{proposition}
\begin{proof}
    By \autoref{basic properties comp ideals}, any $\phi\in \Hom_R(B,R)$ descends to $\overline{\phi}\in \Hom_R((B/IB)_\perfd, R/I)$. 
    Since a surjection would descend to a surjection, it follows 
    immediately that 
    $$\np^R (B) (B/IB)_\perfd \supseteq \np^{R/I} ((B/IB)_\perfd)$$ 
    and the proposition follows from \autoref{thm: from rational powers to monomials}.
\end{proof}

This result does not have an obvious extension to divisor pairs, 
since, in general, there is no appropriate divisor $\Delta_\ba$ on $R/\ba$ 
even in positive characteristic.  {In characteristic $p > 0$ though, there is a correct Cartier algebra-type statement and that is what we will mimic below.}
We will provide two results: one generalizing the notion of pairs 
and one putting an extra assumption ensuring that 
$\Delta_\ba$ is defined.

\begin{definition}
Let $(R, \m)$ be a Noetherian complete local ring of residual characteristic $p > 0$ that is reduced and $\mathbf{S}_2$, and let $\Delta$ be an effective Mumford $\Q$-divisor on $R$. We say that \emph{$\fra$ is $(R, \Delta)$-compatible} if for any perfectoid $R$-algebra $B$ on which $M_B({<}\Delta)$ can be defined and every $\phi\in \Hom_R(M_B({<}\Delta),R)$, $(\phi\circ\iota_{B,\Delta})((\fra B)_\perfd )\subset \fra$. 
\end{definition}

\begin{remark}\label{rem from pairs to generalized pairs}
With notation as above, $\phi\circ\iota_{B,\Delta}\in \Hom_R(B,R)$. The set of all maps obtained this way, \ie the image of $\Hom_R(M_B({<}\Delta),R)$ in $\Hom_R(B,R)$ via precomposition with $\iota_{B,\Delta}$ forms a $B$-submodule, which we call $\sC(B, {<}\Delta)$. The $B$-module structure comes from precomposition by multiplication by an element of $B$.
Thus $\fra$ is $(R, \Delta)$-compatible if and only if it is $\sC(B, {<}\Delta)$-compatible.
In particular, if $\fra$ is uniformly compatible then it is $(R, \Delta)$-compatible for all $\Delta$. 

Furthermore, if $\ba$ is $\Delta$-compatible, then by \autoref{basic properties comp ideals},
$\sC(B, {<}\Delta)$ descends to a $(B/\ba)_\perfd$-submodule 
$\overline{\sC(B, {<}\Delta)} \subseteq \Hom_{R/\fra} ((B/\ba)_\perfd, R/\fra)$.
\end{remark}

\begin{definition}
Let $(R, \m)$ be a Noetherian complete local ring, $B$ a perfectoid $R$-algebra, and $\sC \subseteq \Hom_R (B, R)$ be a $B$-submodule (i.e., $\sC$ is stable under precomposition with elements in $B$). 
Then we call
\[
\np (B, \sC) \coloneqq \{x \in B \mid \phi(x) \in \m  \, \  \forall \phi\in \sC\}
\]
the \emph{non-pure ideal} of $\sC$. 
\end{definition}
The verification that $\np (B, \sC)$ is an ideal is straightforward.
First, if $x \in \np (B, \sC)$ and $b \in B$, then for any $\phi \in \sC$, 
the composition $\phi \circ \ev_b \in \sC$, so 
$\phi (bx) \in \m$. 
Second, if $x, y \in \np (B, \sC)$ then $\phi(x+y) = \phi (x) + \phi (y) \in \m$ for all $\phi \in \sC$.

\begin{lemma}\label{lem generalized pairs ppt}
Let $(R, \m)$ be a Noetherian complete local ring of residual characteristic $p > 0$ that is reduced and $\mathbf{S}_2$, and let $\Delta$ be an effective Mumford $\Q$-divisor on $R$. Let $B$ be a perfectoid $R$-algebra such that $M_B ({<}\Delta)$ is defined.
If $\ba$ is an $(R, \Delta)$-compatible ideal, then
\[
\ppt (\bb; \Delta; B) \leq \newpt{\bb (B/\ba B)_\perfd}{\np ((B/\ba B)_\perfd, \overline{\sC(B, {<}\Delta)})}
\]
for all $\bb\subseteq R$ (such that $\bb B$ can be generated by elements admitting compatible systems of $p$-power roots).
\end{lemma}
\begin{proof}
By \autoref{rem from pairs to generalized pairs}, $\np (B, {<}\Delta) = \np (B, \sC(B, {<}\Delta))$. 
Since $\ba$ is $(R, \Delta)$-compatible, for any $\phi\in \sC(B, {<}\Delta)$, $\overline{\phi} (\overline{x}) \in \m (R/\ba)$
if and only if $\phi (x) \in \m$.
It follows immediately that 
\[
\np (B, {<}\Delta) (B/\ba B)_\perfd = \np ({(B/\ba B)_\perfd}, \overline{\sC(B, {<}\Delta)}).
\]    
The conclusion follows from \autoref{thm: from rational powers to monomials}.
\end{proof}

\begin{proposition}
Let $(R,\fram)$ be a Noetherian complete local normal domain of residual characteristic $p>0$ and $\Delta$ be an integral effective Weil divisor on $R$ such that 
$(R,\Delta)$ is perfectoid pure. 
Let $\fra$ be an $(R, \Delta)$-compatible ideal of $R$ with $R/\fra$ normal and  
define $\Delta_\fra$ to be the divisor corresponding to $I_{\Delta_\fra}\coloneqq (((I_\Delta\colon \fra)+\fra)/\fra)^{**}$ (which we assume to be nonzero) on $R/\fra$.
If $\bb$ is an ideal of $R$, then 
\[
\ppt (\bb; \Delta) \leq \ppt (\bb (R/\ba); \Delta_\ba). 
\]
\end{proposition}
\begin{proof}
By \autoref{prop: perfectoid threshold from one algebra}, 
we may find an absolutely integrally closed perfectoid $R$-algebra $B$ such that $\ppt (\bb; \Delta) = \ppt (\bb; \Delta; B)$. 
In particular, $R\to M_B({<}\Delta)$ splits, say by $\tilde{\phi}$.
It follows that $\phi\coloneqq \widetilde\phi\circ\iota_{B,\Delta} \in \sC(B, {<}\Delta)$ is surjective. Since $I_\Delta$ is $\sC(B, {<}\Delta)$-compatible by \autoref{mixed characteristic compatible splitting via radical ideal}, so is $(I_\Delta\colon \fra)$ by \autoref{basic properties comp ideals}. It follows from \cite[Proposition 2.34]{FayolleCentersPerfectoidPurity} that $J=(I_\Delta\colon \fra)+\fra$ is $\sC(B, {<}\Delta)$-compatible and in particular radical by \autoref{basic properties comp ideals} (we assumed that $(R,\Delta)$ is perfectoid pure). Since every element of $(B/\fra B)_\perfd$ has a compatible system of $p$-th power roots, $M_{(B/\fra B)_\perfd}({<}\Delta_\fra)$ exists.

We now claim that 
\[
\overline{\sC(B, {<}\Delta)} \subseteq \sC ((B/\fra B)_\perfd, {<} \Delta_\fra).
\]
Indeed,
the image $J(R/\fra)$ is radical and by \cite[Proposition 2.20]{FayolleCentersPerfectoidPurity}, it is $\overline{\sC(B, {<}\Delta)}$-compatible. Then, \autoref{basic properties comp ideals} gives that $(J(R/\fra))^{**} = I_{\Delta_\fra}$ is $\overline{\sC(B, {<}\Delta)}$-compatible.
The claim follows from \autoref{mixed characteristic compatible splitting via radical ideal}. 

The claim then implies that 
$\np ((B/\fra B)_\perfd, {<}\Delta_\ba) \subseteq 
\np ((B/\fra B)_\perfd, \overline{\sC(B, {<}\Delta)})$.
Thus by \autoref{lem generalized pairs ppt} and \autoref{lem: elementary threshold} we obtain that
\begin{align*}
\ppt (\bb; \Delta) = 
\ppt (\bb; \Delta; B) & \leq \newpt{\bb (B/\ba B)_\perfd}{\np ((B/\fra B)_\perfd, \overline{\sC(B, {<}\Delta)})}
 \\
& \leq \newpt{\bb (B/\ba B)_\perfd}{\np ((B/\fra B)_\perfd, {<}\Delta_\ba)}\\
 & =
\ppt (\bb (R/\ba); \Delta_\ba; (B/\fra B)_\perfd) \leq \ppt(\bb(R/\ba); \Delta_\ba).
\end{align*}
\end{proof}

We now present a quantification of the inversion of adjunction 
in \cite[Theorem~6.6]{BMPSTWW3}. 

\begin{proposition}\label{prop: ppt inversion}
Let $(R, \m, k)$ be a perfectoid pure complete intersection 
of residual characteristic $p > 0$ and $\ba$ be an ideal.
If $f \in \ba$ is a regular element, then $\ppt (\ba) \geq \ppt (\ba (R/fR)) + 1$. 
\end{proposition}
\begin{proof}
Take a surjection $S:=\coff(k)\llbracket x, x_1, \ldots, x_d\rrbracket  \to R$ such that $x \mapsto f$ and $\ba R^{S_\infty}_\perfd$ is generated by elements $(y_1,\dots,y_n)$ that admit compatible systems of $p$-power roots (e.g., we could arrange that $\{y_1,\dots, y_n\}$ is a subset of the images of $\{x_1,\dots,x_d\}$). 

If the image of $y_1^{a_1/p^m}\cdots y_n^{a_n/p^m}$ is not in $\np((R/fR)^{S_\infty}_\perfd)$, then as in the proof of \autoref{quotient of ci pfd pure}, we know that $R\xrightarrow{\cdot fy_1^{a_1/p^m}\cdots y_n^{a_n/p^m}} ((f)^{1/p^\infty}\rsftypfd)^-$ is pure. Thus by \autoref{rk limit of pure is pure}, we have
$$f^{(p^e-1)/p^{e}}y_1^{a_1/p^m}\cdots y_n^{a_n/p^m} \notin  \np(R^{S_\infty}_\perfd)$$ for all $e \gg 0$. 
The desired inequality then follows from \autoref{thm: from rational powers to monomials}.\end{proof}

We end this subsection by comparing the perfectoid pure threshold with the log canonical threshold.

\begin{proposition}
    \label{prop.pptVsLCT}
    Let $(R, \m)$ be an excellent normal local ring of residual characteristic $p > 0$ admitting a dualizing complex and let $\ba\subseteq R$ be an ideal. Let $\Delta$ be an effective $\Q$-divisor on $R$ such that $K_R+\Delta$ is $\Q$-Cartier.
    Then $\ppt(\ba; \Delta) \leq \lct(\ba; \Delta).$
\end{proposition}
\begin{proof}
    We can choose a generating set $\ba = (x_1, \ldots, x_m)$ that satisfies \autoref{cor: perfectoid threshold via pure pairs}. By \autoref{perfectoid pure triples} and \autoref{pfd pure implies lc}, we have
    \begin{align*}
    \ppt (\ba; \Delta) &= \sup
    \left \{
    \frac{|a|}{p^{e}} \mid a \in \mathbb{Z}_{\geq 0}^m, 
    (R, \Delta + \frac{1}{p^e} \Div (x_1^{a_1}\cdots x_m^{a_m}))
    \text{ is perfectoid pure}
    \right \}
    \\
    &\leq 
    \sup
    \left \{
    \frac{|a|}{p^{e}} \mid a \in \mathbb{Z}_{\geq 0}^m, 
    (R, \Delta + \frac{1}{p^e} \Div (x_1^{a_1}\cdots x_m^{a_m}))
    \text{ is log canonical}
    \right \}
    \\
    &\leq 
    \sup
    \left \{
    \frac{n}{p^{e}} \mid n \in \mathbb{Z}_{\geq 0}, 
    (R, \ba^{n/p^e}, \Delta)
    \text{ is log canonical}
    \right \}
    = \lct (\ba; \Delta).
    \end{align*}    
\end{proof}

\subsection{Examples} In this subsection we collect some computations on the perfectoid pure threshold, using various properties we developed earlier this section.

\begin{example}[$A_n$ singularities]
    Let $R = \ZZ_p \llbracket a, b\rrbracket $ and $f = xy + z^{n+1}$ where 
    $x,y,z$ coincide with $p, a, b$ up to a permutation.
    Then $1 \geq \ppt (\m (R/fR))$ by \autoref{cor: characterize regular} and 
    $\ppt (\m (R/fR)) \geq 1$ by \autoref{prop: ppt inversion} because $R/(f, z)R$ is perfectoid pure.
    Therefore $\ppt (\m(R/fR)) = 1$ and thus coincides with the characteristic $p$ value computed in \cite{TakagiWatanabe}.
\end{example}

\begin{example}
    \label{exam.PPTExamplesEllipticCurve}
Consider the ring $R = \mathbb{Z}_p\llbracket x,y,z\rrbracket /(x^3 + y^3 + z^3)$. 
First, if $p \equiv 1 \mod 3$, then $R/pR$ is F-pure. From this we obtain that $\ppt (\m) = 1$: the lower bound holds by \autoref{prop: ppt inversion} and 
the upper bound by \autoref{prop: Takagi-Watanabe}.

Second, we claim that $\ppt(\m) = 1 - p/(p^2 - 1)$ if 
$p \equiv 2 \mod 3$. To see this, first note that by  \cite[Proposition~6.10]{Yoshikawa.ComputationMethodForPerfdPurity}, we have
$\ppt(p) = 1 - p/(p^2-1).$
However, $(x,y,z)$ is the multiplier ideal of $R$ and hence is uniformly compatible by \cite[Corollary~4.14]{FayolleCentersPerfectoidPurity}. 
In particular, $(x,y,z)_\perfd \subseteq \np(R^{A_\infty}_\perfd)$ and thus by \autoref{prop: reduce to principal}, we have
$\ppt (\m) = \ppt (p)$ since only monomials in $p$ have nonzero thresholds.   
\end{example}

\begin{example}
    \label{exam.PPTExamplesPVarialbe}
In a similar way,  $R=\mathbb{Z}_p\llbracket x, y\rrbracket /(p^3 + x^3 + y^3)$
is not perfectoid pure if $p \equiv 2 \mod 3$ (one can see this via \autoref{quotient of ci pfd pure}: let $A=\mathbb{Z}_p\llbracket x, y\rrbracket$; we have $\ppt ((p^3 + x^3+y^3)A) < 1$ by \cite[Theorem~4.5]{BJPRMSS}).
On the other hand, if $p \equiv 1 \mod 3$ then $R$ is perfectoid pure by \cite[Example 7.4]{BMPSTWW3} and in fact we have
$
\ppt(\m) = 0
$ in this case: this is because the multiplier ideal is the entire $\m$ and thus only units can split.
\end{example}

\begin{example}
More generally, consider
$R = \mathbb{Z}_p\llbracket x_1, \ldots, x_n\rrbracket /F_d(x_1, \ldots, x_n),$
where $F_d$ is a degree $d$ form so that $F_d\neq 0$ mod $p$ and such that $R/p R$ is F-pure. Then 
by \autoref{prop: ppt inversion} we have that 
$\ppt (\m) \geq 1 + \fpt (\m (R/ p R)) = n - d + 1$, where the equality is due to \cite[Theorem B]{dSNB}. On the other hand, if $J$ is a homogeneous minimal reduction, then the socle degree of $R/J$ is $\eh(R) - 1 = d - 1$. Thus 
$\m^{d-1} \not\subset J$ and by \autoref{prop: Takagi-Watanabe} we obtain that 
$ n - d + 1 = \dim(R) - (d-1)  \geq \ppt(\m)$. It follows that $\ppt(\m) = n-d+1$.     
\end{example}

\section{Perfectoid signature for pairs}
\label{sec.PerfectoidSignatureForPairs}

Let $(A,\m_A,k)$ be a complete regular local ring of
dimension $d$ with perfect residue field $k$ of
characteristic $p>0$. Consider the following two cases:
\begin{itemize}
    \item If $\mathrm{char} A=p>0$, the
    \emph{equicharacteristic $p>0$ case}, then we set
    $A=\Lambda\llbracket x_1,\dots,x_d\rrbracket$,
    where $\Lambda=k$.

    \item If $\mathrm{char} A=0$, the
    \emph{unramified mixed characteristic case}, then we
    define $A=\Lambda\llbracket x_2,\dots,x_d\rrbracket$,
    where $\Lambda=W(k)$. In this case we set $x_1=p$.
\end{itemize}
In either case, $x_1,\dots,x_d$ form a regular system of
parameters, and we define
\[
A_n:=A[x_1^{1/p^n},x_2^{1/p^n},\dots,x_d^{1/p^n}].
\]
These rings form a filtered system, and we set
\[
A_\infty^{nc}:=\varinjlim_n A_n.
\]
Faltings's normalized length $\lambda_\infty(-)$ is then
defined on $\m_A$-power torsion $A_\infty^{nc}$-modules
and has the following properties:
\begin{itemize}
    \item It is additive on short exact sequences;
    \item It is defined on almost categories;
    \item It commutes with filtered colimits with
    injective transition maps.
\end{itemize}

Our primary use of normalized length will be in the
following setting. Let $(R,\m,k)$ be a complete Noetherian
normal local domain with perfect residue field of positive
characteristic $p>0$, and choose a Cohen--Gabber
normalization $A\subset R$.
With $A_n$ and $A_\infty^{nc}$ as above, the
\emph{perfectoid signature} of $R$ is then defined as
\[
s^{\underline{x}}_\perfd(R)
=\lambda_\infty(R^{A_\infty}_\perfd/I_\infty),
\]
where
\[
I_\infty:=\np^R(R^{A_\infty}_\perfd)
=
\{z\in R^{A_\infty}_\perfd
\mid R\xrightarrow{1\mapsto z}R^{A_\infty}_\perfd
\mbox{ is not pure}\}
\]
is the non-splitting ideal over $R$, in the notation of
\autoref{def: non-pure ideals}.
We retain here $\underline{x}$ in the notation because
independence of the choice of parameters is not known in mixed characteristic.
For more properties of the normalized length and
perfectoid signature we refer the reader to
\cite{CaiLeeMaSchwedeTuckerPerfectoidSignature}
and the references within.

In equicharacteristic $p>0$, the perfectoid signature agrees with the $F$-signature, giving a new interpretation of the $F$-signature via perfection \cite[Theorem 3.7]{CaiLeeMaSchwedeTuckerPerfectoidSignature}. Analogously, to motivate our definition of the perfectoid signature of divisor pairs, we first give below a similar interpretation of the $F$-signature of divisor pairs from \cite{BSTfsigpairs1} via perfection. Throughout this section, we will make use of the following notation. For a finite extension $R \subseteq S$ of normal domains, let $\pi \colon \Spec(S) \to \Spec(R)$ be the induced morphism. If $\Delta_R$ is an effective $\bQ$-divisor on $\Spec(R)$ whose pullback $\pi^* \Delta_R$ has integral coefficients, by writing $S(\Delta_R)$ we mean $S(\pi^* \Delta_R)$. As we are working with finite extensions of normal rings, such abuses of notation are harmless. Moreover, we also write $R^+(\Delta_R)$ to mean the union over all $S(\Delta_R)$ where $R \subseteq S \subseteq R^+$ are finite normal extensions such that the pullback of $\Delta_R$ to $S$ has integral coefficients.

{
\begin{theorem}
\label{thm.FSigPairsPosCharViaPerfection}
Let $(R,\m,k)$ be a complete Noetherian normal local domain of characteristic $p>0$ with perfect residue field. Fix a Noether-Cohen-Gabber normalization $A \subseteq R$ and a regular system of parameters $\underline{x}$ for $A$ as above. Let $\Delta_R \geq 0$ be a $\bQ$-divisor on $\Spec(R)$, and let $(R,\m) \subseteq (S,\n)$ be a finite generically \'etale extension of complete normal local domains such that the pullback of $\Delta_R$ to $S$ has integral coefficients. If 
\[
\np^R(S_{\perf},\Delta_R) = \{ z \in S_{\perf} \mid R \xrightarrow{1 \mapsto z \otimes 1} S_{\perf}\otimes_S S(\Delta_R) \mbox{ is not pure} \}
\]
is the non-splitting ideal over $R$, then 
\[
s(R,\Delta_R) = \frac{1}{r}\cdot\lambda_\infty(S_\perf/\np^R (S_\perf, \Delta_R))
\]
where $r = \rank_R S$.
\end{theorem}

\begin{proof}
    Write $d = \dim R$, and for $q = p^e$ let 
    \[
    \begin{array}{ll}
    J_e &= \{ z \in S \mid R \xrightarrow{1 \mapsto z^{1/q}\otimes 1} S^{1/q} \otimes_S S(\Delta_R) \mbox{ is not pure}\} \\
    & = \{ z \in S \mid \varphi(z^{1/q}) \in \m \mbox{ for all } \varphi \in \Hom_R(S(q\Delta_R)^{1/q},R)\}.
    \end{array}
    \] We first claim that
    \begin{equation}
    \label{eq.limitgivesnormlength}
    \lambda_\infty (S_\perf / \np^R(S_{\perf},\Delta_R) ) = \lim_{e \to \infty}\ell_R(S / J_e) / q^d.
    \end{equation} To see this, choose $0 \neq g \in A$ such that $g^{1/q} S_\perf \subseteq S^{1/q}[A_\perf] = S^{1/q} \otimes_{A^{1/q}} A_\perf$ for all $e$ (see \cite[Lemma 2.3]{PolstraTuckerCombinedApproach} for further details). We have
    \[
    J_e^{1/q} \cdot S^{1/q}[A_\perf] \subseteq \np^R(S_{\perf},\Delta_R) \cap S^{1/q}[A_\perf] \subseteq (J_e : g)^{1/q} \cdot S^{1/q}[A_\perf]
    \]
    where the first containment follows from the string of inclusion $S^{1/q} \subseteq S^{1/q}[A_\perf] \subseteq S_\perf$ and the second from the argument in \cite[Lemma 3.6]{CaiLeeMaSchwedeTuckerPerfectoidSignature}, both after applying $\blank \otimes_S S(\Delta_R)$.
    Consequently,
    \[ 0 \leq \frac{\ell_R(S/J_e)}{q^d} - \lambda_\infty \left(\frac{S^{1/q}[A_\perf]}{\np^R(S_{\perf},\Delta_R)\cap S^{1/q}[A_\perf]} \right) \leq \frac{\ell_R(S/(J_e,g))}{q^d}
    \]
    which limits to $0$ as $e \to \infty$ as $\m^{[q]} \cdot S \subseteq J_e$ and $\dim S/(g) < d$. As $S_\perf = \bigcup_e S^{1/q}[A_\perf]$, \eqref{eq.limitgivesnormlength} then follows from the definition of the normalized length as in \cite[proof of Theorem~3.7]{CaiLeeMaSchwedeTuckerPerfectoidSignature}.

    Let $\operatorname{Ram}$ be the ramification divisor of the generically \'etale extension $R \subseteq S$, consider now the sequence of ideals $J_e' = \{ z \in S \mid \varphi(z^{1/q}) \in \m \mbox{ for all } \varphi \in \Hom_R(S(q\Delta_R + \operatorname{Ram})^{1/q},R)\}$. To conclude, it suffices to observe
    \[
     \lim_{e \to \infty} \frac{\ell_R(S/J_e)}{q^d} =  \lim_{e \to \infty} \frac{\ell_R(S/J'_e)}{q^d} = r \cdot s(R,\Delta_R).
    \]
    The first equality follows from \cite[proof of Lemma 2.8] {CarvajalSchwedeTuckerEtaleFundFsignature} after noting that $J_e \subseteq J'_e \subseteq (J_e : x)$ for any $x \in R$ with $\Div_S(x) \geq \operatorname{Ram}$. The second equality is precisely \cite[Lemma 4.5]{CarvajalSchwedeTuckerEtaleFundFsignature}, whose argument notably does not make use of the additional hypothesis $\pi^* \Delta_R - \operatorname{Ram} \geq 0$.
\end{proof}

}
\subsection{Definition and first properties}

\begin{definition}
\label{def.perfectoidsigforpairs}
Let $(R,\fram,k)$ be a complete Noetherian normal local domain of mixed characteristic $(0,p)$ with perfect residue field.
Fix a Noether--Cohen--Gabber normalization $A\subseteq R$ and a regular
system of parameters $\underline{x}$ for $A$ as above.
Let $\Delta_R\geq 0$ be a $\bQ$-divisor on $\Spec(R)$, and let
$f\colon (R,\fram)\to (S,\frn)$ be a finite
{ generically \'etale} extension of complete normal
local domains such that the pullback of $\Delta_R$ to $S$
has integral coefficients. Following the notation from \autoref{def: non-pure ideals} with some abbreviation, we set
\begin{align}
\np^R (S^{A_\infty}_\perfd, \Delta_R) &= \{x \in S^{A_\infty}_{\perfd} \mid R \xrightarrow{1 \mapsto x \otimes 1} S^{A_\infty}_{\perfd} \otimes_S S(\Delta_R) \mbox{ is not pure}\} \label{eq.FirstDescriptionOfIinfty}
\\
&=\{x \in S^{A_\infty}_{\perfd} \mid \psi(x \otimes 1)\in \m_R \text{ for any }\psi\in\Hom_R(S^{A_\infty}_{\perfd} \otimes_S S(\Delta_R),R) \}.
\label{eq.SecondDescriptionOfIinfty}
\end{align}
The \emph{perfectoid signature} of the pair $(R,\Delta_R)$
is then defined by
$$s^{\underline{x}}_\perfd(R,\Delta_R)=\frac{1}{r}\cdot\lambda_\infty(S^{A_\infty}_\perfd/\np^R (S^{A_\infty}_\perfd, \Delta_R))$$
where $r=\rank_R S$.
\end{definition}

We now prove that the definition is independent of the choice of $S$.

\begin{lemma}
\label{lem.extendingI_infty}
In the setting of \autoref{def.perfectoidsigforpairs}, let  $f'\colon (S,\n)\to (S',\n')$ be another { generically \'etale} finite extension of complete normal local domains. Choose $0\neq g\in A$ such that $f[1/g]$ and $f'[1/g]$ are \'etale. Then the canonical map 
$${S'}^{A_\infty}_\perfd/\np^R (S^{A_\infty}_\perfd, \Delta_R){S'}^{A_\infty}_\perfd \to  {S'}^{A_\infty}_\perfd/\np^R (S'^{A_\infty}_\perfd, \Delta_R)$$
is a $g$-almost isomorphism, i.e., the ideal $\np^R (S^{A_\infty}_\perfd, \Delta_R)$ is independent (in the sense of extensions) of the choice of $S$ up to almost isomorphism. 
\end{lemma}
\begin{proof}

By \autoref{lem.SplittingApproxGorenstein}, we have 
$$\np^R (S^{A_\infty}_\perfd, \Delta_R)=\bigcup_t (I_t{S}^{A_\infty}_\perfd\otimes_S S(\Delta_R) :_{{S}^{A_\infty}_\perfd} u_t\otimes 1).$$
By \cite[Theorem 10.9]{BhattScholzePrisms}, ${S}^{A_\infty}_\perfd \to {S'}^{A_\infty}_\perfd$ is $g$-almost \'etale. Hence by \cite[Proposition 2.25]{CaiLeeMaSchwedeTuckerPerfectoidSignature}, it is $p$-complete $g$-almost faithfully flat. Since one can write $I_t{S}^{A_\infty}_\perfd\otimes_S S(\Delta_R) :_{{S}^{A_\infty}_\perfd} u_t\otimes 1$ as 
$$\ker({S}^{A_\infty}_\perfd\xrightarrow{1\mapsto u_t\otimes 1}{S}^{A_\infty}_\perfd\otimes_S S(\Delta_R)/I_t{S}^{A_\infty}_\perfd\otimes_S S(\Delta_R))$$
we have 
$$\np^R (S^{A_\infty}_\perfd, \Delta_R){S'}^{A_\infty}_\perfd\overset{a}{\simeq}\bigcup_t (I_t{S'}^{A_\infty}_\perfd\otimes_S S(\Delta_R) :_{{S'}^{A_\infty}_\perfd} u_t\otimes 1) $$
by $p$-complete $g$-almost faithfully flatness of ${S'}^{A_\infty}_\perfd$. Now it suffices to show that 
$$\bigcup_t (I_t{S'}^{A_\infty}_\perfd\otimes_S S(\Delta_R) :_{{S'}^{A_\infty}_\perfd} u_t\otimes 1)=\bigcup_t (I_t{S'}^{A_\infty}_\perfd\otimes_{S'} S'(\Delta_R) :_{{S'}^{A_\infty}_\perfd} u_t\otimes 1)$$
since the latter is $\np^R (S'^{A_\infty}_\perfd, \Delta_R)$ by \autoref{lem.SplittingApproxGorenstein} again. In other words, we need to show that $R \to {S'}^{A_\infty}_\perfd\otimes_S S(\Delta_R)$ sending $1$ to $x\otimes 1$ is pure if and only if $R \to {S'}^{A_\infty}_\perfd\otimes_{S'} S'(\Delta_R)$ sending $1$ to $x\otimes 1$ is pure. By local duality it suffices to show that $H^d_\fram(R(K_R))\to H^d_\fram(R(K_R)\otimes_R{S'}^{A_\infty}_\perfd\otimes_S S(\Delta_R))$ is injective if and only if $ H^d_\fram(R(K_R))\to H^d_\m(R(K_R)\otimes_R{S'}^{A_\infty}_\perfd\otimes_{S'} S'(\Delta_R))$ is injective. We claim that $ H^d_\m(R(K_R)\otimes_R{S'}^{A_\infty}_\perfd\otimes_S S(\Delta_R))\simeq H^d_\m(R(K_R)\otimes_R{S'}^{A_\infty}_\perfd\otimes_{S'} S'(\Delta_R))$. Note that top degree local cohomology of two modules are the same if the difference is supported on a codimension $\ge 2$ subset of $\Spec(S)$.
Now if we restrict to the regular locus $U$ of $S$, then ${S'}^{A_\infty}_\perfd\otimes_{S'} S'(\Delta_R) \cong {S'}^{A_\infty}_\perfd\otimes_{S} S(\Delta_R)$ because $S(\Delta_R)$ is a Cartier divisor when restricting to $U$. Since the complement of $U$ has codimension at least two on $\Spec(S)$ as $S$ is normal, we obtain the isomorphism on top local cohomology as wanted.
\end{proof}

\begin{lemma}
\label{lem.additionofuntilt}
Let $R$ be a perfectoid algebra and $R^\flat$ be its tilt. Denote by $\sharp:R^\flat\to R$ the untilt map. Then for any $x,y\in R^\flat$ we have 
$$(x+y)^\sharp=x^\sharp+y^\sharp+p\cdot(x^{1/p}y^{1/p})^\sharp \cdot u+p^2\cdot v$$ 
for some $u,v\in R$.
\end{lemma}
\begin{proof}
Factor the untilt map as 
$$\sharp\colon R^\flat \xrightarrow{[\cdot]} A_{\inf}(R) \xrightarrow{\theta} R$$
where $[\cdot]$ is the Teichmüller lift and $\theta$ is Fontaine’s map. Hence it suffices to show that $[x+y]=[x]+[y]+p\cdot[x^{1/p}y^{1/p}] \cdot u'+p^2\cdot v'$ for some $u',v'\in A_{\inf}(R)$. This follows from a simple Witt vector computation. 
\end{proof}

Recall from the beginning of this section that $A_\infty$ is the $p$-completion of the union of $A_n=A[x_1^{1/{p^n}},x_2^{1/{p^n}},\dots,x_d^{1/{p^n}}]$ where $A=W(k)\llbracket x_2, \dots, x_d \rrbracket$.

\begin{lemma}
\label{lem.sopinA_infty}
Let $t = p^{\flat}, x_{2}^{\flat}, \ldots, x_d^{\flat} \in A_{\infty}^{\flat}$ satisfy $t^{\sharp} = p, (x_2^{\flat})^{\sharp} = x_2, \ldots, (x_d^{\flat})^{\sharp} = x_d$, respectively, and put 
$(A_{\infty}^{\flat})^{\un} = k\llbracket p^{\flat},x^{\flat}_{2}, \ldots, x^{\flat}_{d}\rrbracket _{\perf}$ so that $k\llbracket p^{\flat},x^{\flat}_{2}, \ldots, x^{\flat}_{d}\rrbracket  \subseteq (A_{\infty}^{\flat})^{\un} \subseteq A_{\infty}^{\flat}$.
For any $0\neq f\in \m_{A}$, we can find elements $f_2, \ldots, f_d$ inside the maximal ideal of $k\llbracket p^{\flat},x^{\flat}_{2}, \ldots, x^{\flat}_{d}\rrbracket $ (so also $f_2^\sharp,\dots,f_d^\sharp \in \fram_A A_\infty$)
such that $\sqrt{(f,f_2^\sharp,\dots,f_d^\sharp)}=\m_{A_\infty}$.
\end{lemma}

\begin{proof}
We first recall the following fact which we shall use repeatedly throughout the proof: for any element $z \in A_\infty$, there is some $y \in (A_{\infty}^{\flat})^{\un}$ and $y' \in A_\infty$ so that $z = y^\sharp + p \cdot y'$ in $A_\infty$. Indeed, to see this, begin by writing $z = y_1^\sharp + p \cdot y_1'$ for some $y_1 \in A_{\infty}^{\flat}$ and $y_1' \in A_\infty$. Since $(A_{\infty}^{\flat})^{\un} / (t) = A_{\infty}^{\flat}/(t)$, we may write $y_1 = y + t \cdot e$ for some $y \in (A_{\infty}^{\flat})^{\un}$ and $e \in A^\flat_\infty$. This gives $y_1^\sharp = y^\sharp + p \cdot y_2'$ (e.g., by  \autoref{lem.additionofuntilt}) for some $y_2' \in A_\infty$, and putting $y' = y_1' + y_2'$ we have $z = y^\sharp + p \cdot y'$ for $y \in (A_{\infty}^{\flat})^{\un}$ and $y' \in A_\infty$.

Proceeding with the proof, write $f=p^n a$ in $A$ such that $p\nmid a$ with $n \in \mathbb{Z}_{\geq 0}$ 
and $a \in A$ non-zero. If $a$ is a unit, then we may take $f_2 := x_2^{\flat}, \ldots, f_d :=x_d^{\flat}$. Thus, we may assume $a$ is not a unit, so that the height of $(p,a)$ in $A$ is two. 
Choose $\tilde{f_3},\ldots, \tilde{f_d}$ in $A$ so that $p,a,\tilde{f_3},\ldots, \tilde{f_d}$ is a system of parameters of $A$. Write $a = x^\sharp + p \cdot b$ for $x \in (A_{\infty}^{\flat})^{\un}$ and $b \in A_\infty$.

Assume first that $b$ is not a unit. In this case, we first argue that $a,p \in \sqrt{(f,(x+t)^\sharp)}$. By \autoref{lem.additionofuntilt} we have $(x+t)^\sharp=x^\sharp+p+p^{(p+1)/p}\cdot c$ for some $c\in A_\infty$. This gives
$$p^n\cdot(x+t)^\sharp-f=p^{n+1}\cdot(1+p^{1/p}\cdot c -b).$$
Since $1+p^{1/p}\cdot c -b$ is a unit, we see $p \in \sqrt{(f,(x+t)^\sharp)}$. As $(x+t)^\sharp=x^\sharp+p+p^{(p+1)/p}\cdot c$, this gives $x^\sharp \in \sqrt{(f,(x+t)^\sharp)}$ and hence also $a = x^\sharp + p \cdot b \in \sqrt{(f,(x+t)^\sharp)}$ as claimed. 
For $j > 2$, we can then write $\tilde{f}_j = g_j^\sharp + pv_j$ for $g_j \in (A^\flat_\infty)^{\un}$ and $v_j \in A_\infty$. Since $A_\infty / (p) = A_\infty^\flat / (t)$, note also that $g_j \in \sqrt{(t,x_2^\flat,\ldots,x_d^\flat) A^\flat_\infty}$.
Putting $g_2 = x+t$, again noting that $g_2 \in \sqrt{(t,x_2^\flat,\ldots,x_d^\flat) A^\flat_\infty}$ as $f \in \fram_A$, we have that $\fram_{A_\infty} = \sqrt{(p,a,\tilde{f}_3,\ldots,\tilde{f}_d)} = \sqrt{(f,g_2^\sharp,\ldots,g_d^\sharp)}$. Taking $N = p^e$ for $e \gg 0$ so that $f_2 = g_2^N, \ldots, f_d = g_d^N$ are all in $(t,x^\flat_{2}, \ldots, x^\flat_{d}) \subseteq k\llbracket t,x^\flat_{2}, \ldots, x^\flat_{d}\rrbracket $, this gives $\sqrt{(f,f_2^\sharp,\dots,f_d^\sharp)}=\sqrt{(f,g_2^\sharp,\dots,g_d^\sharp)}=\m_{A_\infty}$ as desired.

It remains to treat the case where $b$ is a unit, where we begin by arguing that $a,p \in \sqrt{(f,(x+t^p)^\sharp)}$. Again by \autoref{lem.additionofuntilt} we have $(x+t^p)^\sharp=x^\sharp+p^p+p^2\cdot d$ for some $d\in A_\infty$. 
This gives
$$p^n\cdot(x+t^p)^\sharp-f=p^{n+1}\cdot(p^{p-1}+p\cdot d -b).$$
Since $p^{p-1}+p\cdot d -b$ is a unit, we see $p \in \sqrt{(f,(x+t^p)^\sharp)}$. As $(x+t^p)^\sharp=x^\sharp+p^p+p^2\cdot d$, this gives $x^\sharp \in \sqrt{(f,(x+t^p)^\sharp)}$ and hence also $a = x^\sharp + p \cdot b \in \sqrt{(f,(x+t^p)^\sharp)}$ as claimed. 
For $j > 2$, we can then write $\tilde{f}_j = g_j^\sharp + pv_j$ for $g_j \in (A^\flat_\infty)^{\un}$ and $v_j \in A_\infty$. Since $A_\infty / (p) = A_\infty^\flat / (t)$, note also that $g_j \in \sqrt{(t,x_2^\flat,\ldots,x_d^\flat) A^\flat_\infty}$.
Putting $g_2 = x+t^p$, again noting that $g_2 \in \sqrt{(t,x_2^\flat,\ldots,x_d^\flat) A^\flat_\infty}$ as $f \in \fram_A$, we have that $\fram_{A_\infty} = \sqrt{(p,a,\tilde{f}_3,\ldots,\tilde{f}_d)} = \sqrt{(f,g_2^\sharp,\ldots,g_d^\sharp)}$. Taking $N = p^e$ for $e \gg 0$ so that $f_2 = g_2^N, \ldots, f_d = g_d^N$ are all in $(t,x^\flat_{2}, \ldots, x^\flat_{d}) \subseteq k\llbracket t,x^\flat_{2}, \ldots, x^\flat_{d}\rrbracket $, this gives $\sqrt{(f,f_2^\sharp,\dots,f_d^\sharp)}=\sqrt{(f,g_2^\sharp,\dots,g_d^\sharp)}=\m_{A_\infty}$ as desired.
\end{proof}

\begin{proposition}
\label{prop.transformationrulefordefinition}
With notations as in \autoref{lem.extendingI_infty}. Let $r'=\rank_S S'$, then we have
$$\lambda_\infty({S'}^{A_\infty}_\perfd/\np^R (S'^{A_\infty}_\perfd, \Delta_R))={r'}\cdot\lambda_\infty(S^{A_\infty}_\perfd/\np^R (S^{A_\infty}_\perfd, \Delta_R)).$$
\end{proposition}
\begin{proof}
We follow the proof of \cite[Proposition 6.12]{CaiLeeMaSchwedeTuckerPerfectoidSignature}. Choose $0\neq g\in \m_A$ such that $A[1/g]\to S[1/g]\to S'[1/g]$ are finite \'etale maps and $S'[1/g]$ is a finite free $S[1/g]$-module. By \cite[Theorem 10.9]{BhattScholzePrisms}, we have that 
$\Ainfty/p \to S^{\Ainfty}_{\perfd}/p \to {S'}^{\Ainfty}_{\perfd}/p$ are $g$-almost finite \'etale maps. By \autoref{lem.sopinA_infty}, we can pick $f_2^\sharp,\dots,f_d^\sharp \in \fram_A A_\infty$ such that $\sqrt{(pg,f_2^\sharp,\dots,f_d^\sharp)}=\m_{A_\infty}$. 
Thus, modulo $(f_2^\sharp,\dots, f_d^\sharp)_\perfd=((f_2^\sharp)^{1/p^\infty},\dots, (f_d^\sharp)^{1/p^\infty})^-\subseteq \Ainfty$, we have
$$ \Ainfty/(p, (f_2^\sharp,\dots, f_d^\sharp)_\perfd) \to S^{\Ainfty}_\perfd/(p,(f_2^\sharp,\dots, f_d^\sharp)_\perfd) \to {S'}^{\Ainfty}_\perfd/(p,(f_2^\sharp,\dots, f_d^\sharp)_\perfd)$$
are $p$-almost finite \'etale maps since $\sqrt{p}=(g)_\perfd=\m_{A_\infty}$ modulo $(f_2^\sharp,\dots, f_d^\sharp)_\perfd$. Note that for any perfectoid ring, its $p$-torsion is exactly $\sqrt{p}$-torsion which is $p$-almost zero. Hence by modding out $p$-torsion if necessary, we may assume that $\Ainfty/(f_2^\sharp,\dots, f_d^\sharp)_\perfd$, $S^{\Ainfty}_\perfd/(f_2^\sharp,\dots, f_d^\sharp)_\perfd$ and ${S'}^{\Ainfty}_\perfd/(f_2^\sharp,\dots, f_d^\sharp)_\perfd$ are all $p$-torsion free.
Removing $p$-torsion gives $p$-almost isomorphisms, so it
does not change normalized length by
\cite[Proposition 4.10]{CaiLeeMaSchwedeTuckerPerfectoidSignature}.
The reductions modulo $p$ are also $p$-almost isomorphic,
so it suffices to prove uniform $p$-almost freeness of
rank $r'$ after removing $p$-torsion.

Now by the tilting equivalence (see \cite[Theorem 5.25]{ScholzePerfectoidSpaces}), we know that 
$$ (\Ainfty^\flat/(f_2,\dots, f_d)_\perf)^\wedge \to (S^{\Ainfty}_\perfd/(f_2^\sharp,\dots,f_d^\sharp)_\perfd)^\flat \to (S'^{\Ainfty}_\perfd/(f_2^\sharp,\dots,f_d^\sharp)_\perfd)^\flat$$
is $t$-almost finite \'etale, where the completion is $t$-adic. We then claim that the natural map 
$$k\llbracket t\rrbracket _\perf^\wedge\to(\Ainfty^\flat/(f_2,\dots, f_d)_\perf)^\wedge \simeq (k\llbracket t,x_2^\flat,\ldots,x_d^\flat\rrbracket /(f_2,\dots, f_d))_\perf^\wedge$$
is $t$-almost finite \'etale. It suffices to prove without the completion. By the almost purity theorem in characteristic $p$, it suffices to show that $k\llbracket t\rrbracket _\perf[1/t]\to (k\llbracket t,x_2^\flat,\ldots,x_d^\flat\rrbracket /(f_2,\dots, f_d))_\perf[1/t]$ is finite \'etale. 
Since perfection commutes with localization, we may proceed the localization first. We claim that
$$k\llbracket t\rrbracket [1/t]_\perf\to (k\llbracket t,x_2^\flat,\ldots,x_d^\flat\rrbracket /(f_2,\dots, f_d))[1/t]_\perf$$
is a finite \'etale extension: before perfection, $(k\llbracket t,x_2^\flat,\ldots,x_d^\flat\rrbracket /(f_2,\dots, f_d))[1/t]$ is a finite extension of a field, so $(k\llbracket t,x_2^\flat,\ldots,x_d^\flat\rrbracket /(f_2,\dots, f_d))[1/t]_{\red}$ is a finite product of finite field extension of $k\llbracket t\rrbracket [1/t]$, thus after perfection, it is a finite product of finite (\'etale) extension of the perfect field $k\llbracket t\rrbracket [1/t]_\perf$. This proves the claim.

It follows that, in particular,  $(\Ainfty^\flat/(f_2,\dots, f_d)_\perf)^\wedge[1/p^\flat]$ and thus $(S^{\Ainfty}_\perfd/(f_2^\sharp,\dots,f_d^\sharp)_\perfd)^\flat[1/p^\flat]$ are product of perfect fields, hence $S^{\Ainfty}_\perfd/(f_2^\sharp,\dots,f_d^\sharp)_\perfd[1/p]$ is a finite product of perfectoid fields. 
But we also know that $(S^{\Ainfty}_\perfd/(f_2^\sharp,\dots,f_d^\sharp)_\perfd)[1/p] \to (S'^{\Ainfty}_\perfd/(f_2^\sharp,\dots,f_d^\sharp)_\perfd)[1/p]$ is finite \'etale and finite free (of constant rank $r'$). Hence the latter is also a finite product of perfectoid fields. By tilting again, we know that 
\[ 
  (S^{\Ainfty}_\perfd/(f_2^\sharp,\dots,f_d^\sharp)_\perfd)^\flat[1/p^\flat] \to (S'^{\Ainfty}_\perfd/(f_2^\sharp,\dots,f_d^\sharp)_\perfd)^\flat[1/p^\flat].
\]
is a finite free extension of constant rank $r'$ (and both are finite product of perfect fields). 

By \cite[Lemma 6.8]{CaiLeeMaSchwedeTuckerPerfectoidSignature}, we know that $(S'^{\Ainfty}_\perfd/(f_2^\sharp,\dots,f_d^\sharp)_\perfd)^\flat$ is uniformly $p^\flat$-almost finite free of rank $r'$ over $(S^{\Ainfty}_\perfd/(f_2^\sharp,\dots,f_d^\sharp)_\perfd)^\flat$, and thus $S'^{\Ainfty}_\perfd/(p, (f_2^\sharp,\dots,f_d^\sharp)_\perfd)$ is uniformly $p$-almost finite free of rank $r'$ over $S^{\Ainfty}_\perfd/(p, (f_2^\sharp,\dots,f_d^\sharp)_\perfd)$. Also note that $(g)_\perfd =(p^{1/p^\infty})$ in $A_\infty/{(f_2^\sharp,\dots, f_d^\sharp)_\perfd}$. We choose a sequence of elements $\epsilon_n\in (g)_\perfd$ so that  $\epsilon_n = p^{1/p^n} \mod {(f_2^\sharp,\dots, f_d^\sharp)_\perfd}$. By uniformly $p$-almost finite freeness (see \cite[Definition 6.7]{CaiLeeMaSchwedeTuckerPerfectoidSignature}),
       \begin{align*}
        {S'}^{\Ainfty}_\perfd\big/(p, (f_2^\sharp,\dots, f_d^\sharp)_\perfd) &\xrightarrow{\phi_{\epsilon_n}}  \left(S^{\Ainfty}_\perfd\big/(p, (f_2^\sharp,\dots, f_d^\sharp)_\perfd)\right)^{r'} \\ & \xrightarrow{\psi_{\epsilon_n}}{S'}^{\Ainfty}_\perfd\big/(p, (f_2^\sharp,\dots, f_d^\sharp)_\perfd)
    \end{align*} 
    where the two-way compositions are both multiplication by $p^{1/p^n}$. Since ${S'}^{\Ainfty}_\perfd/p$ is $g$-almost finite projective over $S^{\Ainfty}_\perfd/p$ by \cite[Theorem 10.9]{BhattScholzePrisms}, we can apply \cite[Lemma 6.1.1]{CaiLeeMaSchwedeTuckerPerfectoidSignature} to obtain the following commutative diagram:
    {\small
    \[
        \xymatrix{
            {S'}^{\Ainfty}_{\perfd}/p \ar[r]^-{\widetilde{\phi}_{\epsilon_n}}  \ar[d]^-{can}& (S^{\Ainfty}_{\perfd}/p)^{r'} \ar[r]^-{\widetilde{\psi}_{\epsilon_n}} \ar[d]^-{can}& {S'}^{\Ainfty}_{\perfd}/p \ar[d]^-{can} \\
            {S'}^{\Ainfty}_\perfd/(p,(f_2^\sharp,\dots, f_d^\sharp)_\perfd) \ar[r]^-{{\epsilon_n} \cdot\phi_{\epsilon_n}}  & (S^{\Ainfty}_\perfd/(p,(f_2^\sharp,\dots, f_d^\sharp)_\perfd))^{r'} \ar[r]^-{\psi_{\epsilon_n}}  & {S'}^{\Ainfty}_\perfd/(p,(f_2^\sharp,\dots, f_d^\sharp)_\perfd)
        }
    \]
    }

        Below we will only show $$\lambda_\infty({S'}^{A_\infty}_\perfd/\np^R (S'^{A_\infty}_\perfd, \Delta_R)) \le r'\cdot\lambda_\infty(S^{A_\infty}_\perfd/\np^R (S^{A_\infty}_\perfd, \Delta_R))$$ as the other direction is similar (and slightly easier). 
        
        We set $C_{\epsilon_n}=\coker(\widetilde{\psi}_{\epsilon_n})$. Since $(p,(f_2^\sharp,\dots, f_d^\sharp)_\perfd)=(p,{(f_2^\sharp)}^{1/p^\infty},\dots,{(f_d^\sharp)}^{1/p^\infty})$ and $C_{\epsilon_n}$ is almost finitely generated, we can find $m\gg n$ such that ${\epsilon^3_n}\cdot C_{\epsilon_n}\subset ({(f_2^\sharp)}^{1/p^m},\dots,{(f_d^\sharp)}^{1/p^m})C_{\epsilon_n}$, i.e., 
    \begin{equation}
    \label{eqn.EqnInProofofGenTransRule}
        {\epsilon^3_n} \cdot ({S'}^{\Ainfty}_{\perfd}/p)\subset \im(\widetilde{\psi}_{\epsilon_n})+({(f_2^\sharp)}^{1/p^m},\dots,{(f_d^\sharp)}^{1/p^m})({S'}^{\Ainfty}_{\perfd}/p).
    \end{equation} 
    Increasing $m$ if necessary, we may assume that  ${\epsilon_n} =p^{1/p^n}$  in $\Ainfty/({(f_2^\sharp)}^{1/p^m},\dots,{(f_d^\sharp)}^{1/p^m})$ up to units. By \autoref{lem.extendingI_infty}, we have
\begin{align*}
\lambda_\infty({S'}^{A_\infty}_\perfd/\np^R (S'^{A_\infty}_\perfd, \Delta_R)) &=\lambda_\infty({S'}^{A_\infty}_\perfd/\np^R (S^{A_\infty}_\perfd, \Delta_R){S'}^{A_\infty}_\perfd)\\
& \le r'\cdot\lambda_\infty(S^{A_\infty}_\perfd/\np^R (S^{A_\infty}_\perfd, \Delta_R))+\lambda_\infty(C_{\epsilon_n}/\np^R (S^{A_\infty}_\perfd, \Delta_R) C_{\epsilon_n})
\end{align*}
Therefore it suffices to show that for any $\delta>0$, there exists some $n>0$ such that 
$$\lambda_\infty(C_{\epsilon_n}/\np^R (S^{A_\infty}_\perfd, \Delta_R) C_{\epsilon_n})< \delta.$$ 
Now we estimate the normalized length:
        \begin{align*}
        &\lambda_\infty(C_{\epsilon_n}/\np^R (S^{A_\infty}_\perfd, \Delta_R) C_{\epsilon_n})\\ &\le \lambda_\infty({S'}^{\Ainfty}_{\perfd}/((p,f_2^\sharp,\dots,f_d^\sharp){S'}^{\Ainfty}_{\perfd}+\im(\widetilde{\psi}_{\epsilon_n})))\\
        &\le p^{m(d-1)}\cdot  \lambda_\infty({S'}^{\Ainfty}_{\perfd}/((p,{(f_2^\sharp)}^{1/p^m},\dots,{(f_d^\sharp)}^{1/p^m}){S'}^{\Ainfty}_{\perfd}+\im(\widetilde{\psi}_{\epsilon_n})))\\
        &\le p^{m(d-1)} \cdot  \lambda_\infty({S'}^{\Ainfty}_{\perfd}/(\epsilon^3_n,{(f_2^\sharp)}^{1/p^m},\dots,{(f_d^\sharp)}^{1/p^m}))\\
        &=p^{m(d-1)} \cdot (\rank_A S') \cdot \lambda_\infty(\Ainfty/(\epsilon^3_n,{(f_2^\sharp)}^{1/p^m},\dots,{(f_d^\sharp)}^{1/p^m}))\\
        &=p^{m(d-1)} \cdot (\rank_A S') \cdot \lambda_\infty(\Ainfty/(p^{3/p^n},{(f_2^\sharp)}^{1/p^m},\dots,{(f_d^\sharp)}^{1/p^m}))\\
        &=(3/p^n)\cdot (\rank_A S') \cdot \lambda_\infty(\Ainfty/(p,{(f_2^\sharp)},\dots,{(f_d^\sharp)}))
        \end{align*}
    where the inequality on the third line follows from \autoref{eqn.EqnInProofofGenTransRule}, the equality on the fourth line follows from \cite[Proposition 4.16]{CaiLeeMaSchwedeTuckerPerfectoidSignature}; the equality in the fifth line follows from our choice that ${\epsilon_n} =p^{1/p^n}$ up to units in $\Ainfty/({(f_2^\sharp)}^{1/p^m},\dots,{(f_d^\sharp)}^{1/p^m})$. The equality in the sixth line follows from taking a filtration and noting that $p, {f_2^\sharp},\dots,{f_d^\sharp}$ is a regular sequence on $\Ainfty$: after modulo $p$, $({f_2^\sharp},\dots,{f_d^\sharp}) (A_\infty/p)$ has the same radical with $\m_A (A_\infty/p)$, so the images of ${f_2^\sharp},\dots,{f_d^\sharp}$ in $A_\infty/p$, which we may assume comes from $A_n$ for $n\gg0$, form a system of parameters and thus a regular sequence on $A_n/p$, and hence they form a regular sequence on $A_\infty/p$. 
\end{proof}

\begin{remark}
\label{rem.generalizedtransformationrule}
Using the method of \autoref{prop.transformationrulefordefinition}, one can generalize \cite{CaiLeeMaSchwedeTuckerPerfectoidSignature}[Proposition 6.12] to get rid of the issue that one needs to change $A_\infty$ whenever considering a new extension $R\to S$. Explicitly, one can run through the proof of \cite{CaiLeeMaSchwedeTuckerPerfectoidSignature}[Proposition 6.12] while replacing the canonical choice of system of parameters in $A_\infty$ by the system of parameters given by \autoref{lem.sopinA_infty}.
\end{remark}

\begin{corollary}
\autoref{def.perfectoidsigforpairs} is well-defined, i.e., independent of the choice of $S$.
\end{corollary}
\begin{proof}
Suppose $f_1\colon R\to S_1$ and $f_2\colon R\to S_2$ are two extensions of normal domains such that $\pi_1^*\Delta_R$ and $\pi_2^*\Delta_R$ are integral. Then we can find a third one $S_3$ that dominates $S_1$ and $S_2$. Now we are done by \autoref{prop.transformationrulefordefinition}.
\end{proof}

Now we establish a transformation rule for pairs using the relative canonical divisor. 

\begin{theorem}
\label{thm.transformationruleforpairs}
Let $\phi\colon (R,\m,k)\to (R',\n,l)$ be a finite {generically \'etale} extension of complete normal local domains of mixed characteristic $(0,p)$ with perfect residue fields and with fraction fields $K = K(R)$ and $L = K(R')$. Let $\tau\colon \Spec(R')\to \Spec(R)$ be the corresponding morphism of schemes. Suppose $\Delta_R\ge 0$ is a $\bQ$-divisor such that $\Delta_{R'}\coloneqq \tau^*\Delta_R-K_{R'/R}\ge 0$. Let $t$ be the maximal number of direct $\mathrm{Tr}$-summands\footnote{{That is, a rank-1 free $R$-module summand of $R'$ such that if $\rho \colon R' \to R$ is projection on that summand, then $\rho = \Tr \cdot s$ for some $s \in R'$, see also \cite[Definition 2.12]{CarvajalSchwedeTuckerEtaleFundFsignature}.}} of $R'$ and $r=[L:K]$. Then we have
\[ 
    t\cdot  s^{\underline{x}}_\perfd(R',\Delta_{R'})={r}\cdot s^{\underline{x}}_\perfd(R,\Delta_R).
\]
\end{theorem}

\begin{proof}
Let $f \colon R' \to T$ be a finite generically \'etale extension of complete normal local domains such that $\Delta_R$ and $\Delta_{R'}$ are integral after pullback to $\Spec(T)$. Set $M = K(T)$ and $\pi = \Spec(f)$. 

Suppose first that $t = 0$, so that the trace map is not surjective and $\Tr(R') \subseteq \m$. By \autoref{lem.I_inftycomparison} below, every $\varphi \in \Hom_R(T^{A_\infty}_\perfd\otimes_T T(\Delta_R),R)$ can be expressed as $\varphi = \Tr \circ \psi$ for some $\psi \in \Hom_{R'}(T^{A_\infty}_\perfd\otimes_T T(\Delta_{R'}),R')$. In particular,
\[
\varphi(z \otimes 1) = \Tr(\psi(z \otimes 1)) \in \Tr(R') \subseteq \m
\]
for all $z \in T^{A_\infty}_\perfd$. Thus, $\np^R(T^{A_\infty}_\perfd,\Delta_R)=T^{A_\infty}_\perfd$ by
\autoref{eq.SecondDescriptionOfIinfty}, and hence
$s^{\underline{x}}_\perfd(R,\Delta_R)=0$.
The desired equality therefore holds in this case.

Now assume that the trace map is surjective. By \cite[Lemma 2.13]{CarvajalSchwedeTuckerEtaleFundFsignature}, we know that $t=[l:k]$. 
Notice that 
\[
\lambda_{\infty}^{A_\infty}(T^{A_\infty}_\perfd/\np^{R}(T^{A_\infty}_\perfd,\Delta_R)) = r\cdot[M:L]  \cdot s_{\perfd}^{\underline{x}}(R, \Delta_R)
\] by the independence of our definition, where the normalized length is computed with respect to the Noether-Cohen-Gabber normalization $A=W(k)\llbracket x_2,\dots,x_d\rrbracket  \to R$. Similarly, we see that 
$$\lambda_{\infty}^{A'_\infty}(T^{A'_\infty}_\perfd/\np^{R'}(T^{A'_\infty}_\perfd,\Delta_{R'})) = [M:L] \cdot s^{\underline{x}}_{\perfd}(R', \Delta_{R'}).$$
where the normalized length is computed with respect to the Noether-Cohen-Gabber normalization $A'=W(l)\llbracket x_2,\dots,x_d\rrbracket \to R'$. Using the argument behind \cite[Equation (6.2)]{CaiLeeMaSchwedeTuckerPerfectoidSignature}, we have 
$$t\cdot \lambda_{\infty}^{A'_\infty}(T^{A'_\infty}_\perfd/\np^{R'}(T^{A'_\infty}_\perfd,\Delta_{R'})) = \lambda_{\infty}^{A_\infty}(T^{A_\infty}_\perfd/\np^{R'}(T^{A_\infty}_\perfd,\Delta_{R'}))$$
and keeping in mind that $t=[l:k]$.
Therefore, it suffices to show that
\[ 
\np^{R}(T^{A_\infty}_\perfd,\Delta_R)= \np^{R'}(T^{A_\infty}_\perfd,\Delta_{R'}).
\]
This follows by using \autoref{lem.I_inftycomparison} below and \cite[Lemma 2.10]{CarvajalSchwedeTuckerEtaleFundFsignature}, applied to the descriptions of these ideals in \autoref{eq.SecondDescriptionOfIinfty}.  Thus the proof is complete.
\end{proof}

\begin{lemma}
\label{lem.I_inftycomparison}
With notations as in the proof of \autoref{thm.transformationruleforpairs}, we have
$$\Hom_R(T^{A_\infty}_{\perfd} \otimes_T T(\Delta_R),R)\simeq \Hom_{R'}(T^{A_\infty}_{\perfd} \otimes_T T(\Delta_{R'}),R')$$
via the composition by the trace map $\Tr\colon R'(K_{R'/R})\to R$.
\end{lemma}
\begin{proof}
Note that we have
\begin{align*}
    \Hom_R(T^{A_\infty}_{\perfd} \otimes_T T(\Delta_R),R)&\simeq \Hom_R(T^{A_\infty}_{\perfd} \otimes_T T(\Delta_R)\otimes_{R'} R',R)\\
    &\simeq \Hom_{R'}(T^{A_\infty}_{\perfd} \otimes_T T(\Delta_R),\Hom_R(R',R))\\
    &\simeq \Hom_{R'}(T^{A_\infty}_{\perfd} \otimes_T T(\Delta_R),R'(K_{R'/R}))
\end{align*}
where the last isomorphism follows by identifying $\Hom_{R}(R',R)$ with $R'(K_{R'/R})$. Note that under this identification, the evaluation at $1$ map $\Hom_R(R',R)\to R$ is identified with the trace map $\Tr\colon R'(K_{R'/R})\to R$.
Then we are done after we have observed the following isomorphism
\begin{align*}
\Hom_{R'}(T^{A_\infty}_{\perfd} \otimes_T T(\Delta_R),R'(K_{R'/R}))&\simeq \Hom_T(T^{A_\infty}_{\perfd} ,\Hom_{R'}( T(\Delta_R),R'(K_{R'/R})))\\
&\simeq \Hom_T(T^{A_\infty}_{\perfd} ,\Hom_{R'}( T(\Delta_R- \pi^*K_{R'/R}),R'))\\
&= \Hom_T(T^{A_\infty}_{\perfd} ,\Hom_{R'}( T(\Delta_{R'}),R'))\\
&\simeq \Hom_{R'}(T^{A_\infty}_{\perfd} \otimes_T T(\Delta_{R'}),R').
\end{align*}
The second isomorphism is true by using the fact that any morphism between two reflexive sheaves is determined by any open subset whose complement has codimension $\ge 2$. Hence we may restrict ourselves to the regular locus which clearly has the desired isomorphism.
\end{proof}

\subsection{BCM-regularity}
In this subsection we will show that a pair $(R,\Delta)$ is BCM-regular in the sense of \cite{MaSchwedeSingularitiesMixedCharBCM} if and only if the perfectoid signature of pair $s_\perfd^{\underline{x}}(R,\Delta)>0$, generalizing the result in \cite[Section 5]{CaiLeeMaSchwedeTuckerPerfectoidSignature} to the pair version. Throughout this subsection we work in the following setting.

\begin{setting}
    \label{set.BCMRegularityPairSetting}
    Suppose $(R, \m)$ is a Noetherian complete local normal domain of mixed characteristic $(0,p)$ with perfect residue field.  Fix $\Delta \geq 0$ a $\bQ$-divisor on $\Spec R$ such that $K_R + \Delta$ is $\bQ$-Cartier.  Choose $R \subseteq S \subseteq R^+$ a finite extension of normal domains such that $S(K_R + \Delta) \cong S$ (that is, the pullback of $K_R + \Delta$ is Cartier).  

    We also fix a Noether--Cohen normalization $A \subseteq R$ and $0 \neq g \in A$ such that $A[g^{-1}] \subseteq R[g^{-1}]$ and $A[g^{-1}] \subseteq S[g^{-1}]$ are \'etale.
\end{setting}

\begin{definition}
    With notation as in \autoref{set.BCMRegularityPairSetting} we say that $(R, \Delta)$ is \emph{perfectoid BCM-regular} if for each perfectoid big Cohen-Macaulay $R^+$-algebra $B$, we have that 
    \[
        R \to B \otimes_S S(\Delta)
    \]
    is pure where $S$ is a normal domain finite over $R$ so that the pullback of $\Delta$ to $S$ is a Weil divisor.  
\end{definition}

\begin{remark}
\label{rem. BCM-regular via injective hull}
    Set $E$ to be the injective hull of the residue field.  Notice that $(R, \Delta)$ is perfectoid BCM-regular if and only if for all perfectoid big Cohen-Macaulay $B$, we have that
    \[ 
        E \to B \otimes_S S(\Delta) \otimes_R E
    \]
    is injective. Since $E = H^d_{\m}(R(K_R))$, we see that $S(\Delta) \otimes_R E \cong H^d_{\m}(S(K_R + \Delta)) \cong H^d_\m(S)$.  Hence, we have that $(R, \Delta)$ is perfectoid BCM-regular if and only if 
    \begin{equation}
        H^d_{\m}(R(K_R)) \to H^d_{m}(B \otimes_S S(K_R + \Delta)) \cong H^d_{m}(B )
    \end{equation}
    is injective.  
\end{remark}

Our goal in this section is to prove that positivity of perfectoid signature is equivalent to BCM-regularity of pairs $(R, \Delta)$.  We recall the following result from \cite{CaiLeeMaSchwedeTuckerPerfectoidSignature}, also see \cite[Claim 7.11]{BMPSTWW2}.  

\begin{lemma}[{\cite[Lemma 5.6]{CaiLeeMaSchwedeTuckerPerfectoidSignature}}]
\label{lem.AlmostKernel}
Suppose $(S, \frn, L)$ is a complete Noetherian local domain of mixed characteristic and fix $W(k)\llbracket x_2, \dots, x_d \rrbracket = A \subseteq S$ a regular subring such that $S$ is a finite $A$-module (we do not assume that the induced $k \subseteq L$ is an equality). Let $0 \neq g \in A$ with $A[1/g] \to S[1/g]$ finite \'etale.   Suppose $M$ is a $p^{\infty}$-torsion $S$-module. Consider the following submodules of $M$:
\[
\begin{array}{rcl}
    M_1 := \{\eta \in M & | & (g)_{\perfd} \otimes \eta =0 \text{ in } S^{\Ainfty}_{\perfd} \otimes_S M\}.\\
    M_2 := \{\eta \in M  & | &  (g^{1/p^\infty}) \otimes \eta =0 \text{ in } S^+ \otimes_S M\}. \\
    M_3 := \{\eta \in M & | & (g^{1/p^\infty}) \otimes \eta =0 \text{ in } B \otimes_S M \text{ for some perfectoid big Cohen-Macaulay}\\ & & \text{$S^+$-algebra $B$}\}.\\
    M_4 := \{\eta \in M & | & 1 \otimes \eta =0 \text{ in } B \otimes_S M \text{ for some perfectoid big Cohen-Macaulay} \\ & & \text{$S^+$-algebra $B$} \} 
\end{array}
\]
Then we have $M_1=M_2=M_3\supseteq M_4$. Moreover, if $M$ is finitely generated, then we have $M_3=M_4$.
\end{lemma}

The purpose of this lemma is seen in the next result.

\begin{proposition}
\label{prop.perfectoid BCM-regularity equivalence}
With notation as in \autoref{set.BCMRegularityPairSetting}, the following are equivalent.
\begin{enumerate}
    \item $(R, \Delta)$ is perfectoid BCM-regular.
    \item For some $c \in (g)_{\perfd} \subseteq S_{\perfd}^{A_{\infty}}$ we have that the following map is pure:
    \[
        R \xrightarrow{1 \mapsto c} S^{A_{\infty}}_{\perfd} \otimes_S S(\Delta).
    \]
    \item For some $e \gg 0$, the following map is pure:
    \[ R \xrightarrow{1 \mapsto g^{1/p^e}} \widehat{R^+}(\Delta). \]
    {\item For some $e \gg 0$, and for every perfectoid big Cohen-Macaulay $R^+$-algebra $B$, the following map is pure: \[ R \xrightarrow{1 \mapsto g^{1/p^e} \otimes 1} B \otimes_S S(\Delta).\]} 
\end{enumerate}
\end{proposition}
\begin{proof}
Let $E$ be the injective hull of the residue field of $R$. Then by \autoref{rem. BCM-regular via injective hull}, we know that $(R,\Delta)$ is BCM-regular if and only if for each perfectoid { big Cohen-Macaulay} $R^+$-algebra $B$, the map 
$$ E \to E\otimes_R B\otimes_S S(\Delta)$$
is injective, namely the following is zero:
$$\{ \eta \in E\ |\ \eta \otimes 1 =0 \text{ in }E\otimes_R B\otimes_S S(\Delta) \text{ for some perfectoid { big Cohen-Macaulay} $R^+$-algebra }B \}.$$
Similarly, we also note that (b) holds if and only if
$$ \{ \eta \in E\ |\ \eta \otimes (g)_\perfd =0 \text{ in }E\otimes_R S^{A_{\infty}}_{\perfd} \otimes_S S(\Delta)\}=0$$
and (c) holds if and only if 
$$ \{ \eta \in E\ |\ \eta \otimes (g^{1/p^\infty}) =0 \text{ in }E\otimes_R R^+ \otimes_S S(\Delta)\}=0.$$
{ Finally (d) holds if and only if \[
\left\{
\eta\in E \left|
\begin{array}{l} \qquad \quad
\eta\otimes(g^{1/p^\infty})=0
\text{ in } E\otimes_R B\otimes_S S(\Delta)\\
\text{for some perfectoid big Cohen--Macaulay }R^+\text{-algebra }B
\end{array} \right.
\right\}=0.
\] 
}
With the above observation the equivalence between (a), (b), (c) { and (d)} follows from \autoref{lem.AlmostKernel}. Note that the equality of $M_3$ and $M_4$ requires $M$ to be finitely generated in \autoref{lem.AlmostKernel}. However, since $K_R+\Delta$ is $\mathbb{Q}$-Cartier, by \autoref{rem. BCM-regular via injective hull} we know that $$E\otimes_R B \otimes_S S(\Delta) \simeq H^d_\m (B)\simeq H^d_\m(R)\otimes_R B.$$ The proof of \cite[Lemma 5.6]{CaiLeeMaSchwedeTuckerPerfectoidSignature} holds verbatim for the top local cohomology module by the Čech-complex description. Alternatively, choose a system of parameters $y_1,\ldots,y_d$ of $R$.
Using $S(K_R+\Delta)\cong S$, represent the image of
$\eta\in E$ in $H^d_\m(S)$ by
$u\in S/(y_1^N,\ldots,y_d^N)$ for some $N$.
We may assume that $B$ is balanced, so
$B/(y_1^N,\ldots,y_d^N)\longrightarrow H^d_\m(B)$ is injective. Thus, if every $g^{1/p^e}$ annihilates the
image of $\eta$ in $H^d_\m(B)$, they all annihilate the image of $u$
in $B/(y_1^N,\ldots,y_d^N)$. Since $S/(y_1^N,\ldots,y_d^N)$ is a finite
$p^\infty$-torsion $S$-module, \autoref{lem.AlmostKernel}
gives a perfectoid big Cohen--Macaulay algebra $B'$
such that $u$ maps to zero in $B'/(y_1^N,\ldots,y_d^N)$.
Hence $\eta$ maps to zero in $H^d_\m(B')$.
\end{proof}

We can now state and prove the main result of this subsection.

\begin{theorem}
    \label{thm.BCMRegular=>PositiveSignatureForPairs}
    With the notation as in \autoref{set.BCMRegularityPairSetting}, we have $(R, \Delta)$ is BCM-regular if and only if
    \[
        s^{\underline{x}}_{\perfd}(R, \Delta) > 0.
    \]
\end{theorem}
\begin{proof}  The proof closely follows the one in \cite[Proposition 5.8 and 5.9]{CaiLeeMaSchwedeTuckerPerfectoidSignature}.  
    We first prove the ``if" direction. We will prove that for every nonzero $g\in R^+$, there exists $e \gg 0$ such that for any perfectoid big Cohen-Macaulay $R^+$-algebra $B$, the map $R \to B \otimes_S S(\Delta)$ sending $1 \mapsto g^{1/p^e}$ is pure. Replacing $g$ with a multiple if necessary, we may assume that $g\in A$ and that $A[g^{-1}] \subseteq S[g^{-1}]$ is \'etale.
    By \autoref{prop.perfectoid BCM-regularity equivalence}, it suffices to show that there exists $g_1\in (g)_{\perfd}$ such that $R \to S^{\Ainfty}_{\perfd} \otimes_{S} S(\Delta)$ sending $1$ to $g_1$ is pure.  
    If this is not the case, then $(g)_{\perfd}\subseteq \np^R (S^{A_\infty}_\perfd, \Delta)$ and hence $s^{\underline{x}}_{\perfd}(R, \Delta)=\frac{1}{[K(S) : K(R)]} \lambda_{\infty}\big(S^{\Ainfty}_{\perfd}/ \np^R (S^{A_\infty}_\perfd, \Delta))  = 0$ by \cite[Proposition 4.10]{CaiLeeMaSchwedeTuckerPerfectoidSignature} which is a contradiction.

    We next prove the ``only if" direction. We prove this by contradiction. So suppose
    $s^{\underline{x}}_{\perfd}(R, \Delta) = {(1 / \rank_R S)}\lambda_{\infty}(S^{A_{\infty}}_{\perfd} /\np^R (S^{A_\infty}_\perfd, \Delta))=0$.   Set $v$ to be the $\bZ[1/p]$-valued valuation that sends $f \in A_{\infty}$ to the lowest degree (in terms of the rational-power monomials in $p, x_2, \dots, x_d$) of a term (see \cite[Definition 4.11]{CaiLeeMaSchwedeTuckerPerfectoidSignature}).   By \cite[Proposition 2.15]{ShimomotoFrobeniusActionLocalCohomology} or \cite[Lemma 4.13]{CaiLeeMaSchwedeTuckerPerfectoidSignature}, there is a sequence of elements $\{c_i\} \subseteq \Ainfty$ such that $v(c_i)\to 0$ and with $c_i \in \np^R (S^{A_\infty}_\perfd, \Delta)$.  In particular, the map $R\to S^{\Ainfty}_{\perfd} \otimes_S S(\Delta)$ sending $1 \mapsto c_i \otimes 1$ is not pure.  
    
    Suppose $\eta \in H^d_\m(R(K_R)) = E$ generates the socle. Then $c_i\otimes 1 \otimes \eta =0$ inside $S^{\Ainfty}_{\perfd} \otimes_S S(\Delta)  \otimes_R E$. Since $R\to \widehat{R^+} \otimes_S S(\Delta)$ factors through $S^{\Ainfty}_{\perfd} \otimes_S S(\Delta)$, we have that $c_i \otimes 1 \otimes \eta =0$ inside $\widehat{R^+} \otimes_S S(\Delta) \otimes_R E \cong R^+ \otimes_S S(\Delta) \otimes_R E$. It follows that $c_i\otimes 1 \otimes \eta=0$ inside
    \[
        R^+ \otimes_S S(\Delta) \otimes_RE \cong R^+ \otimes_S S(\Delta)\otimes_R H^d_{\m}(R(K_R))  \cong H^d_{\m}(R^+(K_R + \Delta)).
    \]
    To see the second isomorphism, note that while the canonical map $R^+ \otimes_R R(K_R) \otimes_S S(\Delta) \to R^+(K_R + \Delta)$ need not be an isomorphism, it is an isomorphism over the complement of a codimension $\geq 2$ subset of $\Spec R$ (where all the divisors considered are Cartier), and hence the induced map 
    \[ 
        H^d_{\m}(R(K_R)) \otimes_R R^+ \otimes_S S(\Delta) \cong H^d_{\m}\big(R^+ \otimes_R R(K_R) \otimes_S S(\Delta)\big) \to H^d_{\m}\big(R^+(K_R + \Delta) \big)
    \]
    is an isomorphism.

    Since $K_R + \Delta$ is $\bQ$-Cartier and $R$ is local, we see that $R^+(K_R + \Delta) \cong R^+$.  It follows from an argument of Gabber \cite{GabberMSRINotes} in the form of \cite[Corollary 2.28]{CaiLeeMaSchwedeTuckerPerfectoidSignature} that there exists a perfectoid big Cohen-Macaulay $R^+$-algebra $B$ such that the image of $\eta$ is zero under the map 
    \[
        H^d_{\m}(R^+(K_R + \Delta)) \cong H^d_{\m}(R^+) \to H^d_\m(B).
    \]   
    Hence $(R, \Delta)$ is not BCM-regular which is a contradiction.  
\end{proof}

\subsection{Continuity and Convexity}

\begin{theorem}
  \label{thm:continuity}
  Suppose $(R,\fram,k)$ is a complete Noetherian normal local domain of mixed characteristic $(0,p)$ with perfect residue field and $f \in \fram$ is non-zero. Let $D = \divisor(f)$ on $\Spec(R)$. The function $\QQ_{\geq 0} \to \mathbb{R}$ given by
  \begin{equation*}
    t \mapsto s^{\underline{x}}_\perfd(R,tD)
  \end{equation*}
  is continuous. In particular, it extends naturally to a function $\mathbb{R}_{\geq 0} \to \mathbb{R}$.
\end{theorem}

\begin{proof}
  Let $y_1 :=f,y_2, \ldots, y_d$ be a fixed system of parameters of $R$ beginning with $f$.
  Suppose $t \in \mathbb{Q}_{\geq 0}$ and $\epsilon \in \mathbb{Q}_{>0}$. Choose $(S,\n)$ a finite complete normal local extension domain of $(R,\m)$ so that $f^{t},f^{t+\epsilon} \in S$. Observe that, after tracing through the definitions, we have
  \begin{gather*}
    \np^R(S^{A_\infty}_\perfd,tD)  = \np^R(S^{A_\infty}_\perfd,0) :_{S^{A_{\infty}}_{\perfd}} f^{t} \; \mbox{ and } \; \np^R(S^{A_\infty}_\perfd,(t+\epsilon)D)  = \np^R(S^{A_\infty}_\perfd,0) :_{S^{A_{\infty}}_{\perfd}} f^{t+\epsilon}
  \end{gather*}
   so that in particular $\np^R(S^{A_\infty}_\perfd,(t+\epsilon)D) = \np^R(S^{A_\infty}_\perfd,tD):_{S^{A_{\infty}}_{\perfd}} f^{\epsilon}$. If $r = \rank_R(S)$, then we have 
   \[\begin{array}{l} \displaystyle
    s^{\underline{x}}_\perfd(R,tD) - s^{\underline{x}}_\perfd(R,(t+\epsilon)D) \\[8pt] \displaystyle = \frac{1}{r}\left( \lambda_\infty\left( S^{A_\infty}_{\perfd} / \np^R(S^{A_\infty}_\perfd,tD)\right) - \lambda_\infty\left( S^{A_\infty}_{\perfd} / \np^R(S^{A_\infty}_\perfd,(t+\epsilon)D)\right)\right) \\[8pt] \displaystyle =
    \frac{1}{r}\lambda_\infty\left( \frac{S^{A_\infty}_{\perfd}}{(\np^R(S^{A_\infty}_\perfd,tD),f^{\epsilon})S^{A_\infty}_{\perfd}} \right) \\[15pt] \displaystyle
     \leq \frac{1}{r}\lambda_\infty \left( \frac{S^{A_\infty}_{\perfd}}{(y_1^\epsilon,y_2, \ldots, y_d)S^{A_\infty}_{\perfd}} \right).
   \end{array}\]
  {Set $\ell = [S/\frn : k]$.}  Now $y_1^\epsilon,y_2, \ldots, y_d$ is also a system of parameters for $S$, and so by \cite[Proposition 4.14]{CaiLeeMaSchwedeTuckerPerfectoidSignature} it follows that
  \begin{align*}
    \lambda_\infty \left( \frac{S^{A_\infty}_{\perfd}}{(y_1^\epsilon,y_2, \ldots, y_d)S^{A_\infty}_{\perfd}} \right) & ={\ell \cdot } e_{\perfd}^{\underline{x}}(S,(y_1^\epsilon,y_2, \ldots, y_d)) \\ 
    & = {\ell \cdot }e(S,(y_1^\epsilon,y_2, \ldots, y_d)) \\ 
    & = \epsilon \cdot r \cdot e(R,(y_1,y_2, \ldots, y_d)).
  \end{align*}
  Putting these together gives that 
  \begin{equation*}
    0 \leq s^{\underline{x}}_\perfd(R,tD) - s^{\underline{x}}_\perfd(R,(t+\epsilon)D) \leq \epsilon \cdot e(R,(y_1,y_2, \ldots, y_d)).
  \end{equation*}
  Thus, we have that the function $\QQ_{\geq 0} \to \mathbb{R}$ given by $
    t \mapsto s^{\underline{x}}_\perfd(R,tD)$
  is Lipschitz with constant $e(R,(y_1,y_2, \ldots, y_d))$, and in particular continuous as claimed.
\end{proof}

\begin{theorem}
    \label{thm.PerfdSignatureIsConvex}
  Suppose $(R,\fram,k)$ is a complete Noetherian normal local domain of mixed characteristic $(0,p)$ with perfect residue field and $f \in \fram$ is non-zero with $D = \divisor(f)$ a divisor on $\Spec R$.  The function $\mathbb{R}_{\geq 0} \to \mathbb{R}$ given by
  \begin{equation*}
    t \mapsto s^{\underline{x}}_\perfd(R,tD)
  \end{equation*}
  is convex.
\end{theorem}

\begin{proof}
  If $0 \leq t < t_1 < t_2$, we need to verify that
  \begin{equation}
    \label{eq:diffquotients}
    D(t,t_1) = \frac{s^{\underline{x}}_\perfd(R,t_1D) - s^{\underline{x}}_\perfd(R,tD)}{t_1 - t} \leq \frac{s^{\underline{x}}_\perfd(R,t_2D) - s^{\underline{x}}_\perfd(R,t_1D)}{t_2-t_1} = D(t_1,t_2).
  \end{equation}
  By continuity, it suffices to check this inequality on the dense subset of rational numbers. Taking a common denominator $n \in \mathbb{Z}_{>0}$, we may assume that 
  \begin{equation*}
    t = \frac{a}{n} \qquad t_1 = \frac{a_1}{n} \qquad t_2 = \frac{a_2}{n}
  \end{equation*}
  where $a,a_1,a_2 \in \mathbb{Z}$ with $0 \leq a < a_1 < a_2$. Moreover, it suffices to treat the case $a_1 = a+1$ and $a_2 = a+2$.
  
  Choose $(S,\n)$ a complete normal local finite extension domain of $(R,\fram)$ so that $f^{1/n} \in S$. If $r = \rank_R(S)$, the inequality \eqref{eq:diffquotients} reduces to the following
  \begin{multline*}
    \frac{n}{r} \left( \lambda_\infty\left( 
      \frac{S^{A_\infty}_{\perfd}}{\np^R(S^{A_\infty}_\perfd,tD):_{S^{A_\infty}_{\perfd}}f^{1/n}}
     \right) - \lambda_\infty\left( 
      \frac{S^{A_\infty}_{\perfd}}{\np^R(S^{A_\infty}_\perfd,tD)}
     \right)\right) \\ \leq \frac{n}{r} \left( \lambda_\infty\left( 
      \frac{S^{A_\infty}_{\perfd}}{\np^R(S^{A_\infty}_\perfd,t_1D):_{S^{A_\infty}_{\perfd}}f^{1/n}}
     \right)- \lambda_\infty\left( 
      \frac{S^{A_\infty}_{\perfd}}{\np^R(S^{A_\infty}_\perfd,t_1D)}
     \right) \right)
  \end{multline*}
With similar reductions as in the proof of \autoref{thm:continuity} the above is  equivalent to
\begin{align*}
  -\frac{n}{r} \left( \lambda_\infty\left( 
  \frac{S^{A_\infty}_{\perfd}}{(\np^R(S^{A_\infty}_\perfd,tD),f^{1/n})S^{A_\infty}_{\perfd}}
 \right) \right) &\leq -\frac{n}{r} \left( \lambda_\infty\left( 
  \frac{S^{A_\infty}_{\perfd}}{(\np^R(S^{A_\infty}_\perfd,t_1D),f^{1/n})S^{A_\infty}_{\perfd}}
 \right) \right) 
\end{align*}
or simply
\begin{align*}
  \lambda_\infty\left( 
  \frac{S^{A_\infty}_{\perfd}}{(\np^R(S^{A_\infty}_\perfd,tD),f^{1/n})S^{A_\infty}_{\perfd}}
 \right)  &\geq  \lambda_\infty\left( 
  \frac{S^{A_\infty}_{\perfd}}{(\np^R(S^{A_\infty}_\perfd,t_1D),f^{1/n})S^{A_\infty}_{\perfd}}
 \right)
\end{align*}
which follows from the inclusion $\np^R(S^{A_\infty}_\perfd,tD) \subseteq \np^R(S^{A_\infty}_\perfd,t_1D) = \np^R(S^{A_\infty}_\perfd,tD) :_{S^{A_\infty}_{\perfd}} f^{1/n}$.
\end{proof}

\bibliographystyle{amsalpha}
\bibliography{preprints,main}

\end{document}